\documentclass[11pt,fleqn]{article}
\usepackage[obeyspaces,hyphens,spaces]{url}
\usepackage{xcolor}
\usepackage{etoolbox}
\usepackage{amsmath}
\usepackage{amssymb}
\usepackage{multicol}
\usepackage{newtxtext}
\usepackage{newtxmath}
\usepackage{euscript}
\usepackage{mathrsfs}  
\usepackage{pifont}    
\usepackage{mathtools}
\usepackage{srcltx}    
\usepackage{graphicx}
\usepackage{newcent}
\usepackage{mathrsfs} 
\usepackage[twoside,
	paperheight=247mm,
	paperwidth=174mm,
	marginratio={3:5,2:3},
	left=18mm,
	top=20mm]{geometry}
\usepackage[UKenglish]{babel} 
\DeclareOldFontCommand{\rm}{\normalfont\rmfamily}{\mathrm}
\DeclareOldFontCommand{\sf}{\normalfont\sffamily}{\mathsf}
\DeclareOldFontCommand{\tt}{\normalfont\ttfamily}{\mathtt}
\DeclareOldFontCommand{\bf}{\normalfont\bfseries}{\mathbf}
\DeclareOldFontCommand{\it}{\normalfont\itshape}{\mathit}
\DeclareOldFontCommand{\sl}{\normalfont\slshape}{\@nomath\sl}
\DeclareOldFontCommand{\sc}{\normalfont\scshape}{\@nomath\sc}
\DeclareRobustCommand*\cal{\@fontswitch\relax\mathcal}
\DeclareRobustCommand*\mit{\@fontswitch\relax\mathnormal}
\definecolor{red3}{rgb}{0.9,0.15,0.0}
\definecolor{dgreen}{rgb}{0.00 0.35 0.00}
\definecolor{labelkey}{HTML}{0455BF}
\definecolor{refkey}{rgb}{0,0.6,0.0}
\definecolor{dblue}{HTML}{0455BF}
\definecolor{dgreen}{HTML}{02724A}
\definecolor{myellow}{HTML}{D97904}
\definecolor{dred}{HTML}{D90404}
\definecolor{dbrown}{rgb}{0.65 0.40 0.35}
\usepackage{upref} 
\usepackage{hyperref}
\hypersetup{colorlinks=true,linktocpage=true,linkcolor=dblue,%
citecolor=dgreen,urlcolor=dred}

\DeclareFontFamily{U}{BOONDOX-calo}{\skewchar\font=45 }
\DeclareFontShape{U}{BOONDOX-calo}{m}{n}{
  <-> s*[1.05] BOONDOX-r-calo}{}
\DeclareFontShape{U}{BOONDOX-calo}{b}{n}{
  <-> s*[1.05] BOONDOX-b-calo}{}
\DeclareMathAlphabet{\mathcalboondox}{U}{BOONDOX-calo}{m}{n}
\SetMathAlphabet{\mathcalboondox}{bold}{U}{BOONDOX-calo}{b}{n}
\DeclareMathAlphabet{\mathbcalboondox}{U}{BOONDOX-calo}{b}{n}

\newcommand{\MM}{\ensuremath{\mathcalboondox{m}}}
\newcommand{\LL}{\ensuremath{\mathcalboondox{l}}}
\newcommand{\CC}{\ensuremath{\mathcalboondox{c}}}
\renewcommand{\leq}{\ensuremath{\leqslant}}
\renewcommand{\geq}{\ensuremath{\geqslant}}

\newcommand{\Argmin}{\ensuremath{\text{\rm Argmin}\,}}
\newcommand{\Argmind}[2]{\ensuremath{%
\underset{\substack{#1}}{\text{\rm Argmin}}\;\;#2}}
\newcommand{\minmax}[3]{\ensuremath{%
\underset{\substack{#1}}{\text{\rm minimize}}\;\:
\underset{\substack{#2}}{\text{\rm maximize}}\;\;#3}}
\newcommand{\minimize}[2]{\ensuremath{\underset{\substack{{#1}}}%
{\text{\rm minimize}}\;\;#2}}
\newcommand{\Scal}[2]{\bigg\langle{#1}\;\bigg|\:{#2}\bigg\rangle} 
 
\newcommand{\scal}[2]{{\langle{{#1}\mid{#2}}\rangle}}

\newcommand{\menge}[2]{\big\{{#1}~|~{#2}\big\}} 
\newcommand{\Menge}[2]{\left\{{#1}~\middle|~{#2}\right\}} 

\newcommand{\GGG}{\ensuremath{\boldsymbol{\mathcal{G}}}}
\newcommand{\HHH}{\ensuremath{\boldsymbol{\mathcal{H}}}}

\newcommand{\UUU}{\ensuremath{\boldsymbol{\mathcal{U}}}}
\newcommand{\VVV}{\ensuremath{\boldsymbol{\mathcal{V}}}}

\newcommand{\Qq}{\ensuremath{{\mathsf{Q}}}}
\newcommand{\HH}{\ensuremath{{\mathcal{H}}}}
\newcommand{\VV}{\ensuremath{{\mathcal{V}}}}

\newcommand{\GG}{\ensuremath{{\mathcal{G}}}}
\newcommand{\KK}{\ensuremath{{\mathbb{K}}}}

\newcommand{\WC}{\ensuremath{\mathfrak{W}}}
\newcommand{\NN}{\ensuremath{\mathbb{N}}}

\newcommand{\sad}{\ensuremath{\boldsymbol{\EuScript{S}}}}
\newcommand{\kut}{\ensuremath{\boldsymbol{\EuScript{K}}}}
\newcommand{\MMM}{\ensuremath{\boldsymbol{\EuScript{M}}}}
\newcommand{\TTT}{\ensuremath{\boldsymbol{\EuScript{T}}}}
\newcommand{\proxc}[2]{{#1}{\ensuremath{\raisebox{0.2mm}%
{\mbox{\scriptsize\rotatebox[origin=c]{45}%
{\tiny$\,\square\,$}}}{#2}}}}
\newcommand{\proxcc}[2]{{#1}{\ensuremath{\raisebox{0.2mm}%
{\mbox{\scriptsize\rotatebox[origin=c]{45}%
{\tiny$\,\blacksquare\,$}}}{#2}}}}

\newcommand{\pnabla}[1]{\ensuremath{\nabla_{\!#1}}}
\newcommand{\Sum}{\ensuremath{\displaystyle\sum}}

\newcommand{\emp}{\ensuremath{\varnothing}}

\newcommand{\Id}{\ensuremath{\mathrm{Id}}}

\newcommand{\RR}{\ensuremath{\mathbb{R}}}
\newcommand{\RP}{\ensuremath{\left[0,{+}\infty\right[}}

\newcommand{\BL}{\ensuremath{\EuScript{B}}}

\newcommand{\RPP}{\ensuremath{\left]0,{+}\infty\right[}}

\newcommand{\RX}{\ensuremath{\left]{-}\infty,{+}\infty\right]}}
\newcommand{\RXX}{\ensuremath{\left[{-}\infty,{+}\infty\right]}}
\newcommand{\XXX}{\ensuremath{\boldsymbol{\mathsf{X}}}}
\newcommand{\xx}{\ensuremath{\boldsymbol{\mathsf{x}}}}
\newcommand{\intdom}{\ensuremath{\mathrm{int\,dom}\,}}

\newcommand{\weakly}{\ensuremath{\rightharpoonup}}
\newcommand{\exi}{\ensuremath{\exists\,}}

\newcommand{\pinf}{\ensuremath{{{+}\infty}}}
\newcommand{\minf}{\ensuremath{{{-}\infty}}}
\DeclareMathOperator{\dom}{dom}

\DeclareMathOperator{\prox}{prox}
\DeclareMathOperator{\proj}{proj}
\DeclareMathOperator{\gra}{gra}
\DeclareMathOperator{\inte}{int}
\DeclareMathOperator{\sri}{sri}
\DeclareMathOperator{\reli}{ri}
\DeclareMathOperator{\ran}{ran}
\DeclareMathOperator{\zer}{zer}
\DeclareMathOperator{\Fix}{Fix}
\DeclareMathOperator{\cone}{cone}

\newcommand{\bdry}{\ensuremath{\text{\rm bdry}\,}}

\newcommand{\infconv}{\ensuremath{\mbox{\small$\,\square\,$}}}
\newcommand{\pushfwd}{\ensuremath{\mbox{\Large$\,\triangleright\,$}}}
\newcommand{\zeroun}{\ensuremath{\left]0,1\right[}} 
\newcommand{\rzeroun}{\ensuremath{\left]0,1\right]}}

\newcommand{\moyo}[2]{\prescript{#2}{}{\!#1}}

\def\abstract{\noindent{\bfseries Abstract}. \ignorespaces}

\usepackage{theorem}
\newtheorem{theorem}{Theorem}[section]
\newtheorem{lemma}[theorem]{Lemma}

\newtheorem{proposition}[theorem]{Proposition}
\newtheorem{assumption}[theorem]{Assumption}
\theoremstyle{plain}{\theorembodyfont{\rmfamily}%
}

\theoremstyle{plain}{\theorembodyfont{\rmfamily}%
\newtheorem{example}[theorem]{Example}}
\theoremstyle{plain}{\theorembodyfont{\rmfamily}%
\newtheorem{remark}[theorem]{Remark}}
\theoremstyle{plain}{\theorembodyfont{\rmfamily}%
\newtheorem{framework}[theorem]{Framework}}
\theoremstyle{plain}{\theorembodyfont{\rmfamily}%
\newtheorem{algorithm}[theorem]{Algorithm}}
\theoremstyle{plain}{\theorembodyfont{\rmfamily}%
}
\theoremstyle{plain}{\theorembodyfont{\rmfamily}%
\newtheorem{definition}[theorem]{Definition}}
\theoremstyle{plain}{\theorembodyfont{\rmfamily}%
\newtheorem{problem}[theorem]{Problem}}
\theoremstyle{plain}{\theorembodyfont{\rmfamily}%
}
\usepackage{enumitem}
\setlist[enumerate]{itemsep=2pt}
\setlist[itemize]{itemsep=2pt}

\numberwithin{equation}{section}
\numberwithin{figure}{section}
\DeclareOldFontCommand{\rm}{\normalfont\rmfamily}{\mathrm}
\usepackage{authblk}

\newcommand{\email}[1]{\href{mailto:#1}{\nolinkurl{#1}}}
\author{{\large Patrick L. Combettes}\thanks{%
This work was supported by the National Science
Foundation under grant CCF-2211123.}}
\affil{North Carolina State University,
Department of Mathematics\\ 
Raleigh, NC 27695-8205, USA\\
\email{plc@math.ncsu.edu}
}

\begin{document}

\title{\bfseries\LARGE The geometry of monotone operator 
splitting methods}

\date{~}

\maketitle

\begin{abstract}
We propose a geometric framework to describe and analyze a wide
array of operator splitting methods for solving monotone inclusion
problems. The initial inclusion problem, which typically involves
several operators combined through monotonicity-preserving
operations, is seldom solvable in its original form. We embed it in
an auxiliary space, where it is associated with a surrogate 
monotone inclusion problem with a more tractable structure and
which allows for easy recovery of solutions to the initial problem.
The surrogate problem is solved by successive projections onto
half-spaces containing its solution set. The outer approximation
half-spaces are constructed by using the individual operators
present in the model separately. This geometric framework is shown
to encompass traditional methods as well as state-of-the-art
asynchronous block-iterative algorithms, and its flexible structure
provides a pattern to design new ones. 
\end{abstract}
\tableofcontents
\newpage

\vskip 12mm

\section{Introduction} 
\label{sec:1} 
Throughout, $\HH$ is a real Hilbert space with scalar product
$\scal{\cdot}{\cdot}$ and $2^{\HH}$ stands for the power set of
$\HH$. Our main focus is on the following monotone inclusion
problem.

\begin{problem}
\label{prob:1}
Let $M\colon\HH\to 2^{\HH}$ be a monotone operator, that is, 
\begin{equation}
\label{e:0}
(\forall x\in\HH)(\forall y\in\HH)(\forall x^*\in Mx)
(\forall y^*\in My)\quad\scal{x-y}{x^*-y^*}\geq 0.
\end{equation}
The task is to find $x\in\HH$ such that $0\in Mx$.
\end{problem}

Monotone inclusion problems are intimately linked to the birth of
nonlinear analysis. They first appeared as a powerful models to
establish existence, uniqueness, and stability results for various
nonlinear problems \cite{Brow68,Ghiz69,Kach60,Zara60,Zara71}. Over
the past six decades, monotone inclusion models have penetrated
almost all areas of mathematics and its applications. Nowadays,
Problem~\ref{prob:1} models a broad range of equilibria in areas
such as 
dynamical systems \cite{Adly17},
ill-posed problems \cite{Albe06},
domain decomposition methods \cite{Aldu23,Nume16,Atto11},
circuit theory \cite{Ande76,Cha23b,Chaf23,Sepu23,Goel17},
machine learning \cite{Argy12,MaPa18,Jena11,Vait18},
evolution equations \cite{Sico10,Brez73,Show97},
partial differential equations 
\cite{Barb10,Brez98,Clas17,Ghou09,Pasc78,Show97,Zei90B},
signal processing \cite{Beck10,Banf11,Smms05,Pott93},
image processing \cite{Bedn18,Cham16,Siim22,Glow16,Pesq21},
game theory 
\cite{Belg21,Boer21,Bric13,Joca22,Cohe87,Facc07,Facc03,Gaut21},
network flow problems \cite{Bert98,Buim22,Rock84,Rock95},
equilibrium theory \cite{Bric12,Jnca05,Moud99},
mean-field games \cite{Bri23b,Bri18b},
control theory \cite{Brog07,Brog20,Caml16,Dol79a,Sing22},
data science \cite{Chan17,Sign21,Wrig22},
optimization \cite{Bord18,Ecks92,Gols96,Tsen90,Tsen91},
statistics \cite{Ejst20,Bien21},
neural networks \cite{Svva20,Wins20,Yich20},
traffic equilibrium \cite{Dafe80,Fuku96},
systems theory \cite{Deso75,Dol79b},
mechanics \cite{Glow83,Merc80},
optimal transportation \cite{Papa14},
and minimax theory \cite{Roc70b}.

Early numerical solution methods to solve Problem~\ref{prob:1}
can be found in \cite{Anti76,Bruc73,Bruc74,Korp76,Lion67,%
Petr66,Sibo70,Vain60,Vain61,Zara60,Zara64}.
These methods are of the explicit Euler type, meaning that, at
iteration $n$, the update $x_{n+1}$ is determined by finding a
point in $Mx_n$. An alternative method, which first appeared
in \cite{Lieu68} and then in more detail in 
\cite{Roc76a}, is the proximal point algorithm, where the
update is obtained through the implicit relation 
$x_n-x_{n+1}\in Mx_{n+1}$. Such approaches have limited
potential since they can be directly implemented only in specific
situations. For instance, the Euler step methods of
\cite{Bruc73,Bruc74,Bruc75} impose certain properties on $M$ and
asymptotically vanishing step sizes, which is detrimental to
numerical stability and speed of convergence. On the other hand,
the proximal point algorithm requires explicit expressions for the
resolvent of $M$, which is seldom possible. In most problems,
however, $M$ has a complex structure and it is typically expressed
in terms of monotonicity-preserving operations involving simpler
operators. The principle governing \emph{splitting methods} is to
devise algorithms in which each of the elementary operators arising
in the decomposition of $M$ are used individually, hence breaking
up Problem~\ref{prob:1} into tasks that are more manageable.

The first monotone operator splitting methods arose in the late
1970s and were motivated by applications in mechanics and partial
differential equations \cite{Glow83,Glow89,Merc80}. 
The three main algorithms that dominated the field 
were designed for problems in which 
\begin{equation}
M=A+B, 
\end{equation}
where $A\colon\HH\to 2^{\HH}$ and $B\colon\HH\to 2^{\HH}$ are
maximally monotone: 
the forward-backward method \cite{Merc79},
the Douglas--Rachford method \cite{Lion79}, and Tseng's
forward-backward-forward method \cite{Tsen00}. 
In recent years, the field of monotone operator splitting
algorithms has benefited from a new impetus, fueled by the emerging
application areas mentioned above and their demand for solving
efficiently increasingly complex large-dimensional problems. Thus,
duality techniques have arisen to address composite models of the
form 
\begin{equation}
M=A+L^*\circ B\circ L, 
\end{equation}
where $L$ is a linear operator from $\HH$
to a Hilbert space $\GG$ and $A\colon\HH\to 2^{\HH}$ and
$B\colon\GG\to 2^{\GG}$ are maximally
monotone \cite{Siop11}. These techniques have been
further developed to devise splitting algorithms for the more
structured model \cite{Botr14,Svva12,Bang13}
\begin{equation}
\label{e:2012}
M=A+\sum_{k=1}^pL_k^*\circ\bigl(B_k^{-1}+D_k^{-1}\bigl)^{-1}
\circ L_k+C,
\end{equation}
where each linear operator $L_k$ maps $\HH$ to a Hilbert space
$\GG_k$, and the operators $A\colon\HH\to 2^{\HH}$,
$B_k\colon\GG_k\to 2^{\GG_k}$, $D_k\colon\GG_k\to 2^{\GG_k}$, and
$C\colon\HH\to\HH$ are maximally monotone. Splitting algorithms for
models which are more finely structured than \eqref{e:2012} have
also been proposed as well as multivariate versions that capture
coupled systems of monotone inclusions; see \cite{Moor22} and the
references therein. On a different front, block-iterative
algorithms, which allow for the activation of only a subgroup of
operators present in the model at a given iteration, have also been
developed \cite{Buin22,Moor22,MaPr18,John22}. At the same time,
a multitude of splitting algorithms tailored to specific models
have been elaborated. For instance, if $A\colon\HH\to 2^{\HH}$ and
$B\colon\HH\to 2^{\HH}$ are maximally monotone and
$C\colon\HH\to\HH$ is cocoercive, splitting algorithms have been
proposed in \cite{Davi17,Ragu19} for the decomposition $M=A+B+C$
and in particular in \cite{Bric18} if $B\colon\HH\to\HH$ is
Lipschitzian and in \cite{Lata17} if $B\colon\HH\to\HH$ is linear
and bounded. 

Given the abundance of activity in monotone operator splitting
techniques, it is important to identify general structures and
principles, as well as possible bonds between algorithm design
methodologies in order not only to simplify and clarify the state
of the art, but also to facilitate the developments of new methods
in the future. From the outset, fixed point theory has been a tool
of choice to achieve this goal. For instance, it has played an
important role in the analysis of the proximal point algorithm
\cite{Krya73,Mart72,Roc76a}. In \cite{Opti04}, fixed point
iterations of averaged operators were shown to provide a
convenient framework to investigate the asymptotic behavior of
classical splitting algorithms such as the forward-backward,
backward-backward, Douglas--Rachford, and Peaceman--Rachford
algorithms. Further applications of averaged operator iterations
to design and analyze splitting methods can be found in
\cite{Bric23,Cham16,Siop17,Sign21,Yama15,Cond23,Davi17,%
Ragu19,Ragu13,Ragu15,Ryut20,Xue23a}. Fixed point modeling is also
a central algorithmic development tool in recent works such as
\cite{Arag23,Bric18,Mali23}. In spite of these achievements, fixed
point methods seem less well suited to capture in simple terms the
most flexible splitting methods such as the block-iterative
asynchronous methods of \cite{Buin22,Moor22,MaPr18,John22}, which
were built using geometric arguments. The purpose of the present
paper is to provide a standardized pattern for building and
analyzing splitting methods around the following geometric
framework. It comprises an embedding step, where the initial
Problem~\ref{prob:1} is replaced by a more tractable surrogate
inclusion problem in an auxiliary space $\XXX$ from which the
solutions to the original problem can be easily recovered. The
second step is an iterative process in which the current iterate
is projected onto a closed half-space that serves as an outer
approximation to the surrogate solution set.

\begin{framework}
\label{f:1}
Geometric algorithmic template for solving Problem~\ref{prob:1}.
\begin{enumerate}
\item
{\bfseries Embedding:}
Find a real Hilbert space $\XXX$, a maximally monotone operator 
$\MMM\colon\XXX\to 2^{\XXX}$, and an operator
$\TTT\colon\XXX\to\HH$ such that $\TTT(\zer\MMM)\subset\zer M$. We
call $(\XXX,\MMM,\TTT)$ an \emph{embedding} of
Problem~\ref{prob:1}.
\item
{\bfseries Iterations:}
\begin{equation}
\label{e:99}
\begin{array}{l}
\text{for}\;n=0,1,\ldots\\
\left\lfloor
\begin{array}{l}
\boldsymbol{\mathsf{H}}_n\;
\text{is a closed half-space of $\XXX$ such that}\;
\zer\MMM\subset\boldsymbol{\mathsf{H}}_n\\
\xx_{n+1}\;\text{is a relaxed projection of $\xx_n$ onto\;}
{\boldsymbol{\mathsf{H}}_n}.
\end{array}
\right.\\
\end{array}
\end{equation}
\end{enumerate}
\end{framework}

In optimization, the use of half-spaces as outer approximations to
the solution set goes back to the cutting plane methods of
\cite{Che59b,Kell60,Levi66}; see also \cite{Laur70,Vein67,Zang69}.
In monotone inclusion problems, modeling iterations as successive
projections onto separating half-spaces occurs in several
papers \cite{Moor01,Eoop01,Sol99b,Sol99c}. We aim at showing that
Framework~\ref{f:1} is sufficiently broad and flexible to encompass
a wide array of existing methods while providing a template to
create new ones. It will allow us to derive in a unified fashion
simple proofs of existing convergence results. It will also make it
possible to establish seamlessly strongly convergent variants of
these algorithms. The proofs we provide are new, and so are some of
the results.

The remainder of the paper is organized as follows. To make our
presentation self-contained, Section~\ref{sec:2} covers the
necessary mathematical background on monotone operator theory. It
also contains various examples of maximally monotone operators and
a detailed history of the field. In Section~\ref{sec:3}, we present
several models for decomposing $M$ in Problem~\ref{prob:1}. These
decompositions will generate the embeddings required in
Framework~\ref{f:1} and form the backbone of the splitting methods
discussed in the paper. The geometric principles underlying our
approach are presented in Section~\ref{sec:4}, where the main
convergence theorems are laid out. In
Section~\ref{sec:ppa}, we study the proximal point algorithm and
explore several of its facets. In Sections~\ref{sec:dr},
\ref{sec:fbf}, and \ref{sec:fb}, we study, respectively, the
Douglas--Rachford, forward-backward-forward, and forward-backward
methods through the lens of Framework~\ref{f:1} and capture a broad
range of algorithms and applications by embedding them in bigger
spaces. Block-iterative Kuhn--Tucker and saddle projective
splitting methods are addressed in Sections~\ref{sec:ps} and
\ref{sec:sad}, respectively. Finally, several extensions and
variants of the results are discussed in Section~\ref{sec:var}.

\section{Monotone operators}
\label{sec:2}

\subsection{Notation and basic definitions}

The material of this section can be found in \cite{Livre1}.

\subsubsection{General notation} 
$\HH$ and $\GG$ are real Hilbert spaces, $\BL(\HH,\GG)$ is the
space of bounded linear operators from $\HH$ to $\GG$,
$\BL(\HH)=\BL(\HH,\HH)$, and $\HH\oplus\GG$ denotes the Hilbert
direct sum of $\HH$ and $\GG$. The identity operator of 
$\HH$ is denoted by $\Id_{\HH}$, its scalar product by
$\scal{\cdot}{\cdot}_{\HH}$, and the associated norm by
$\|\cdot\|_{\HH}$ (the subscripts will be omitted when the context
is clear). The weak convergence of a sequence $(x_n)_{n\in\NN}$ to
$x$ is denoted by $x_n\weakly x$, whereas $x_n\to x$ denotes its
strong convergence; the set of weak sequential cluster points of 
$(x_n)_{n\in\NN}$ is denoted by $\WC(x_n)_{n\in\NN}$.
\subsubsection{Sets}
\label{sec:sets}

Let $C$ be a subset of $\HH$. The interior of $C$ is $\inte C$,
the \emph{indicator function} of $C$ is 
\begin{equation}
\label{e:iota}
\iota_C\colon\HH\to\RX\colon x\mapsto 
\begin{cases}
0,&\text{if}\;\:x\in C;\\
\pinf,&\text{otherwise},
\end{cases}
\end{equation}
the \emph{support function} of $C$ is 
\begin{equation}
\label{e:si}
\sigma_C\colon\HH\to\RXX\colon x^*\mapsto
\sup_{x\in C}\scal{x}{x^*},
\end{equation}
and the \emph{distance function} to $C$ is 
\begin{equation}
d_C\colon\HH\to\RX\colon x\mapsto\inf_{y\in C}\|x-y\|.
\end{equation}
Suppose that $C$ is convex. We denote by $\cone C$ the smallest
cone that contains $C$ and by $\sri C$ the \emph{strong relative
interior} of $C$, i.e., 
\begin{equation}
\label{e:sri}
\sri C=\menge{x\in C}{\cone(-x+C)\;\text{is a closed vector 
subspace of}\;\HH}. 
\end{equation}
If $\HH$ is finite-dimensional, $\sri C$ coincides with the 
\emph{relative interior} $\reli C$ of $C$, i.e., 
the interior of $C$ relative to the smallest affine subspace of
$\HH$ containing $C$.
Suppose that $C$ is nonempty, closed, and convex. For every
$x\in\HH$, 
\begin{equation}
\label{e:pU1}
\proj_C\hspace{.2mm}x\;
\text{is the unique point in $C$ such that}\;
d_C(x)=\|x-\proj_Cx\|. 
\end{equation}
This process defines the \emph{projection operator}
$\proj_C\colon\HH\to\HH$ of $C$. The simple case of a closed
half-space is central to our approach.

\begin{example}[{\protect{\cite[Example~29.20]{Livre1}}}]
\label{ex:H}
Let $u^*\in\HH$, let $\eta\in\RR$, and suppose that
$H=\menge{z\in\HH}{\scal{z}{u^*}\leq\eta}\neq\emp$. 
Let $x\in\HH$ and set
\begin{equation}
d=
\begin{cases}
\dfrac{\scal{x}{u^*}-\eta}
{\|u^*\|^2}u^*,&\text{if}\:\:\scal{x}{u^*}>\eta;\\
0,&\text{otherwise.}\\
\end{cases}\\
\end{equation}
Then $\proj_Hx=x-d$.
\end{example}

\subsubsection{Functions}
The set of minimizers of a function $f\colon\HH\to\RX$ is denoted
by $\Argmin f$ and, if it is a singleton, its unique element is
denoted by $\text{argmin}_{x\in\HH}f(x)$.
The \emph{infimal convolution} of $f\colon\HH\to\RX$ and 
$h\colon\HH\to\RX$ is
\begin{equation}
\label{e:infconv1}
f\infconv h\colon\HH\to\RXX\colon x\mapsto
\inf_{y\in\HH}\big(f(y)+h(x-y)\big).
\end{equation}
We denote by $\Gamma_0(\HH)$ the class of functions 
$f\colon\HH\to\RX$ which are lower semicontinuous, convex, and 
such that $\dom f=\menge{x\in\HH}{f(x)<\pinf}\neq\emp$. Let
$f\in\Gamma_0(\HH)$. The \emph{conjugate} of $f$ is 
\begin{equation}
\label{e:conj}
\Gamma_0(\HH)\ni f^*\colon x^*\mapsto
\sup_{x\in\HH}\bigl(\scal{x}{x^*}-f(x)\bigr).
\end{equation}
For every $x\in\HH$,
\begin{equation}
\label{e:pU2}
\prox_f\hspace{.2mm}x\;\text{is the unique minimizer over $\HH$ of}
\;y\mapsto f(y)+\dfrac{1}{2}\|x-y\|^2.
\end{equation}
This process defines the \emph{proximity operator}
$\prox_f\colon\HH\to\HH$ of $f$. We have
\begin{equation}
\label{e:jjm7}
(\forall\gamma\in\RPP)(\forall x\in\HH)\quad 
x=\prox_{\gamma f}x+\gamma\,\prox_{f^*/\gamma}\big(x/\gamma\big).
\end{equation}
The \emph{Moreau envelope} of $f$ of parameter 
$\gamma\in\RPP$ is 
\begin{equation}
\label{e:jjm3}
\moyo{f}{\gamma}=f\infconv\bigg(\frac{1}{2\gamma}\|\cdot\|^2\bigg).
\end{equation}

\subsubsection{Set-valued operators}
Let $M\colon\HH\to 2^{\HH}$. The \emph{graph} of $M$ is
\begin{equation}
\label{e:d-1}
\gra M=\menge{(x,x^*)\in\HH\times\HH}{x^*\in Mx}.
\end{equation}
The \emph{inverse} of $M$ is the operator 
$M^{-1}\colon\HH\to 2^{\HH}$ defined through the relation
\begin{equation}
\bigl(\forall(x,x^*)\in\HH\times\HH\bigr)\quad x^*\in
Mx\quad\Leftrightarrow\quad x\in M^{-1}x^*.
\end{equation}
Thus,
\begin{equation}
\label{e:d-2}
\gra M^{-1}=\menge{(x^*,x)\in\HH\times\HH}{(x,x^*)\in\gra M}.
\end{equation}
The set of \emph{fixed points} of $M$ is
\begin{equation}
\label{e:f-2}
\Fix M=\menge{x\in\HH}{x\in Mx},
\end{equation}
the set of \emph{zeros} of $M$ is
\begin{equation}
\label{e:d6}
\zer M=M^{-1}0=\menge{x\in\HH}{0\in Mx},
\end{equation}
and the \emph{resolvent} of $M$ is the operator
\begin{equation}
\label{e:d7}
J_M=(\Id+M)^{-1}.
\end{equation}
In other words, 
\begin{equation}
\label{e:ee14}
(\forall x\in\HH)(\forall p\in\HH)\quad
p\in J_Mx\;\Leftrightarrow\;(p,x-p)\in\gra M
\end{equation}
and therefore
\begin{equation}
\label{e:z}
\zer M=\Fix J_M.
\end{equation}
We have
\begin{equation}
\label{e:jjm8}
(\forall\gamma\in\RPP)(\forall x\in\HH)\quad 
x-J_{\gamma M}x=\gamma\,J_{M^{-1}/\gamma}\big(x/\gamma\big).
\end{equation}
The \emph{Yosida approximation} of index $\gamma\in\RPP$ of $M$ is
\begin{equation}
\label{e:yosi1}
\moyo{M}{\gamma}=\frac{\Id-J_{\gamma M}}{\gamma}=
\bigl(\gamma\Id+M^{-1}\bigr)^{-1}
=\bigl(J_{\gamma^{-1}M^{-1}}\bigr)\circ\gamma^{-1}\Id
\end{equation}
and it satisfies
\begin{equation}
\label{e:yosi2}
\zer M=\zer\moyo{M}{\gamma}.
\end{equation}
The \emph{domain} of $M$ is
\begin{equation}
\label{e:d-3}
\dom M=\menge{x\in\HH}{Mx\neq\emp}
\end{equation}
and the \emph{range} of $M$ is
\begin{equation}
\label{e:d-4}
\ran M=\bigcup_{x\in\dom M}Mx
=\menge{x^*\in\HH}{(\exi x\in\dom M)\;x^*\in Mx}.
\end{equation}
We have 
\begin{equation}
\label{e:d5}
\dom M^{-1}=\ran M\;\;\text{and}\;\;\ran M^{-1}=\dom M.
\end{equation}
If, for some $x\in\HH$, $Mx$ is a singleton, we let $Mx$ denote its
single element. We say that $M$ is \emph{injective} if 
$(\forall x\in\HH)(\forall y\in\HH)$
$Mx\cap My\neq\emp$ $\Rightarrow$ $x=y$.
Finally, given $A\colon\HH\to 2^{\HH}$, $B\colon\GG\to 2^{\GG}$, 
$L\in\BL(\HH,\GG)$, and $\alpha\in\RR$, we set
\begin{equation}
\label{e:d8}
\begin{array}{ccll}
A+\alpha L^*\circ B\circ L\colon
&\!\!\HH&\!\!\to&\!\!2^{\HH}\\
&\!\!x&\!\!\mapsto&\!\!
\menge{x^*+\alpha L^*y^*}{x^*\in Ax\;\text{and}\;y^*\in B(Lx)}.
\end{array}
\end{equation}

\subsubsection{Monotone operators}
Let $M\colon\HH\to 2^{\HH}$. Then $M$ is \emph{monotone} if
\begin{equation}
\label{e:m1}
\big(\forall (x,x^*)\in\gra M\big)
\big(\forall (y,y^*)\in\gra M\big)
\quad\scal{x-y}{x^*-y^*}\geq 0
\end{equation}
and \emph{maximally monotone} if, further, there exists no
monotone operator $A\colon\HH\to 2^{\HH}$ such that 
$\gra M\subset\gra A\neq\gra M$, that is 
(see Figure~\ref{fig:6}),
\begin{multline} 
\label{e:m2}
\hskip -1mm\big(\forall (x,x^*)\in\HH\times\HH\big)\\
\big[\:(x,x^*)\in\gra M\;\Leftrightarrow\;
\bigl(\forall (y,y^*)\in\gra M\bigr)\;\;
\scal{x-y}{x^*-y^*}\geq 0\:\big].
\end{multline}
We have
\begin{equation}
\label{e:zero}
M\;\text{maximally monotone}\;\Rightarrow\;\zer M\;
\text{is closed and convex}.
\end{equation}
Let $\beta\in\RPP$. Then $M$ is $\beta$-\emph{strongly monotone} if
$M-\beta\Id$ is monotone, that is, 
\begin{equation}
\label{e:strmonotone}
\bigl(\forall (x,x^*)\in\gra M\bigr)
\bigl(\forall (y,y^*)\in\gra M\bigr)\quad
\scal{x-y}{x^*-y^*}\geq\beta\|x-y\|^2.
\end{equation}
Now let $D$ be a nonempty subset of $\HH$, let $\alpha\in\RPP$, 
and let $M\colon D\to\HH$. Then $M$ is \emph{nonexpansive} if
\begin{equation}
\label{e:nex}
(\forall x\in D)(\forall y\in D)\quad\|Mx-My\|\leq\|x-y\|,
\end{equation}
$\alpha$-\emph{averaged} if $\alpha\leq 1$ and 
$\Id+\alpha^{-1}(M-\Id)$ is nonexpansive,
$\alpha$-\emph{cocoercive} if $M^{-1}$ is $\alpha$-strongly
monotone, that is, 
\begin{equation}
\label{e:coco}
(\forall x\in D)(\forall y\in D)\quad
\scal{x-y}{Mx-My}\geq\alpha\|Mx-My\|^2,
\end{equation}
and \emph{firmly nonexpansive} if it is $1$-cocoercive.
Alternatively, 
\begin{equation}
\label{e:ne}
M\;\text{is firmly nonexpansive}\;\Leftrightarrow\;
2M-\Id\;\text{is nonexpansive}.
\end{equation}

\begin{figure}
\begin{center}
\includegraphics[width=110mm]{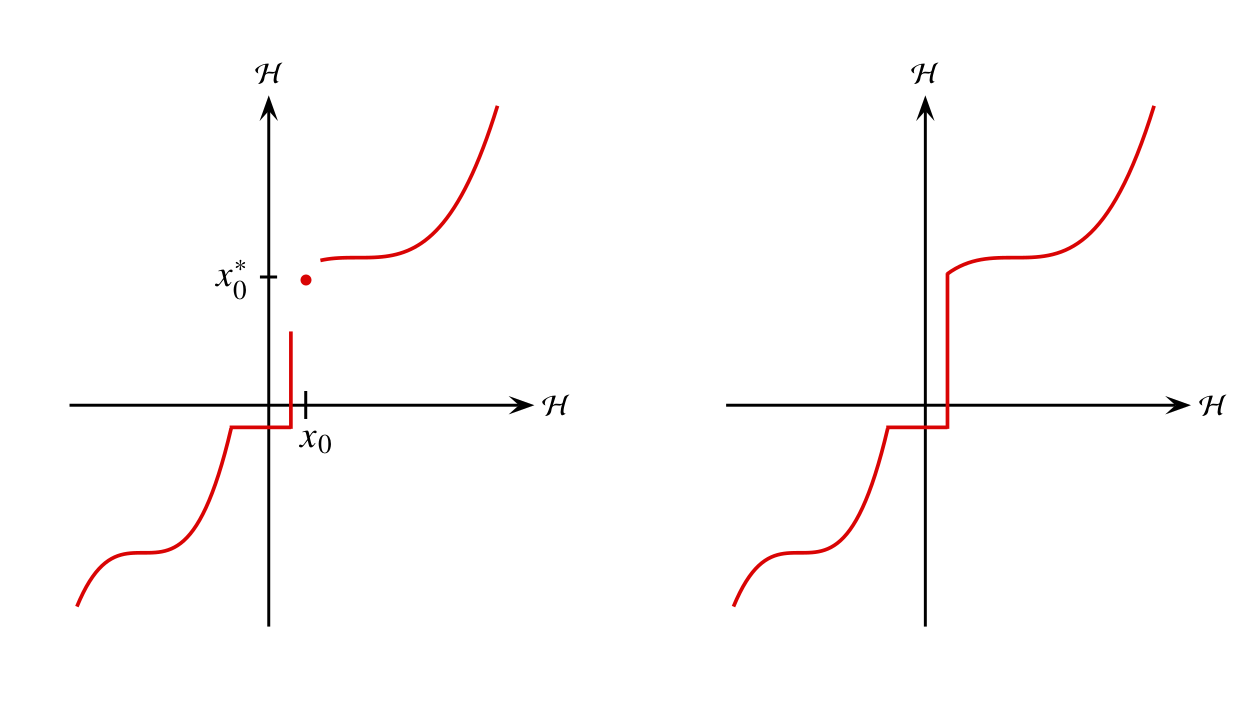}
\end{center}
\caption{Left: Graph of a monotone, but not maximally monotone, 
operator: the point $(x_0,x^*_0)$ can be added to the graph and the
resulting graph remains monotone. Right: Graph of a maximally
monotone operator: adding any point to the graph does not preserve
its monotonicity.
}
\label{fig:6}
\end{figure}

The following result is known as the Baillon--Haddad theorem.

\begin{lemma}[{\protect{\cite[Corollaire~10]{Bail77}}}]
\label{l:bh}
Let $\alpha\in\RPP$ and let $f\colon\HH\to\RR$ be convex, 
Fr\'echet differentiable, and such that $\nabla f$ is
$1/\alpha$-Lipschitzian. Then $\nabla f$ is $\alpha$-cocoercive.
\end{lemma}

\subsection{History}

Monotonicity goes back to classical calculus and the notion of an
increasing real-valued function defined on an interval
$D\subset\RR$, i.e., a function $f\colon D\to\RR$ that satisfies 
\begin{equation}
\label{e:d1}
(\forall x\in D)(\forall y\in D)\quad
\big(x-y\big)\big(f(x)-f(y)\big)\geq 0.
\end{equation}
The special properties enjoyed by such functions have long been 
recognized; see for instance \cite{Darb75,Frod29,Hahn21}. The
monotonicity condition \eqref{e:d1} is also tied to the infancy of
the theory of convex functions. Thus, it was shown in \cite{Jens06}
that, if $D$ is open and $g\colon D\to\RR$ is a twice
differentiable function with derivative $f$, then \eqref{e:d1}
implies that $g$ is convex. On the numerical side, \eqref{e:d1} is
an important property in connection with solving iteratively the 
root finding problem \cite{Papa13}
\begin{equation}
\label{e:d2}
\text{find}\;\;x\in D\;\;\text{such that}\;\;f(x)=0.
\end{equation}
Monotone operators on $\RR$ also appeared in nonlinear circuit
theory in the 1940s in the form of quasi-linear resistors
\cite{Duff46,Duff47,Duff48}. A quasi-linear resistor is a two-pole
circuit element characterized by the property that the current
going through it increases smoothly with the voltage across it. In
other words, the transformation underlying its current-voltage
characteristic is differentiable and increasing. Dipoles with
monotonic characteristics were further investigated in
\cite{Mill51}. To study networks involving a broader range of
devices, this concept was extended by Minty in \cite{Mint60,Mint61}
to maximally monotone set-valued transformations on $\RR$ (see
Figure~\ref{fig:10} and \cite{Cede62} for examples). Interestingly,
as will be discussed shortly, Minty turned out to be one of the
founders of monotone operator theory. For further relevant early
work on the connections between monotone operators and network
theory, see \cite{Berg62,Deso74} and, for more abstract 
ramifications, see \cite{Dol79b,Rock84}.

\begin{figure}
\begin{center}
\includegraphics[width=110mm]{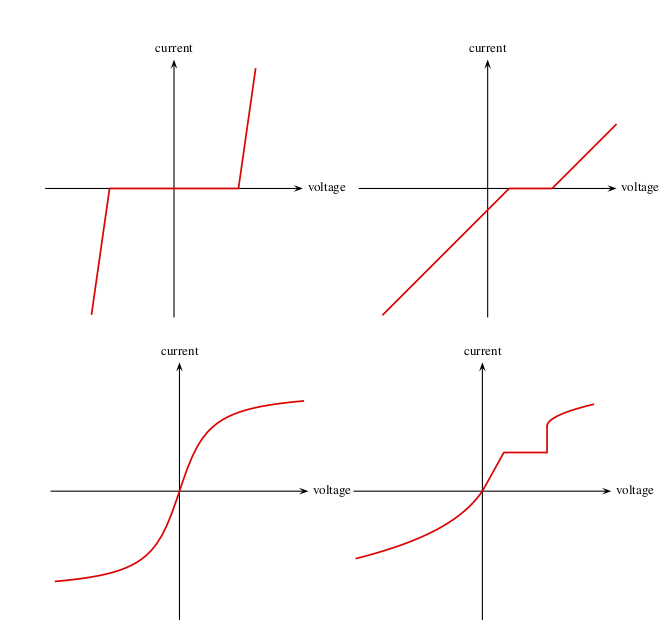}
\caption{Current-voltage characteristics of quasi-linear resistors
as monotone operators from $\RR$ to $2^\RR$.
Top left: breakdown diodes in series \cite{Reic61}.
Top right: breakdown diode and resistance in series \cite{Reic61}.
Bottom left: anode-dynode beam-deflection tube \cite{Reic61}.
Bottom right: the maximally monotone current-voltage characteristic
of \cite{Mint61}.}
\label{fig:10}
\end{center}
\end{figure}

Another precursor of monotonicity is found in linear functional
analysis, where a linear operator $M\colon\HH\supset D\to\HH$ is
declared accretive if \cite{Kato80} 
\begin{equation}
\label{e:po}
(\forall x\in D)\quad\scal{x}{Mx}\geq 0.
\end{equation}
In this context, the notion of a maximally accretive operator was
introduced in \cite{Phil59}. Accretive operators are also central
to passive linear network theory \cite{Belt72,Zame68}. One of the
first instances of \eqref{e:po} in electrical networks is the
current-voltage transformation of the four-pole circuit element
known as an ideal gyrator \cite{Tell48}.

The above notions of increasing functions and positive operators
can be brought together by considering an operator
$M\colon\HH\supset D\to\HH$ such that
\begin{equation}
\label{e:d4}
(\forall x\in D)(\forall x\in D)\quad\scal{x-y}{Mx-My}\geq 0.
\end{equation}
Instances of \eqref{e:d4} appear implicitly in \cite{Golo35} and,
more explicitly, in \cite{Vain56,Vai59a} in connection with the
existence of solutions to Hammerstein integral equations; see also
\cite{Golo36} for more general types of equations. Another
instance, which corresponds to what is now called strict
monotonicity, appears in \cite{Buck56}, where $\HH$ is the standard
Euclidean space. The systematic study of operators satisfying
\eqref{e:d4} started in 1960 an opened an important new chapter of
nonlinear functional analysis. Three independent papers submitted
that year are associated with the birth of monotone operator
theory.
\begin{itemize}
\item
In an article submitted in February 1960, Ka{\v{c}}urovski{\u\i}
\cite{Kach60} called \emph{monotone} an operator that satisfies
\eqref{e:d4}. This paper concerned the monotonicity of the gradient
of a differentiable convex function (see also \cite{Vai59b}) and
the existence of solutions to certain nonlinear equations. It also
introduced strongly monotone operators. 
\item
In a technical report completed in June 1960, Zarantonello called
\eqref{e:d4} an (isotonically) monotonicity property and discussed
supra-unitary (in modern language, strongly monotone) operators. In
connection with the solution of nonlinear equations, an important
result of \cite{Zara60} is that, if $M\colon\HH\to\HH$ is monotone
and Lipschitzian, then $\Id+M$ is surjective.
\item
In an article submitted in December 1960, Minty \cite{Mint62} also
called $M\colon D\to\HH$ monotone if it satisfies \eqref{e:d4}. In
addition, he introduced the fundamental concept of maximal
monotonicity and established key connections with nonexpansive
operators. Although, strictly speaking, his definitions dealt with
single-valued operators, he established results on monotone
relations that naturally suggest extensions to the set-valued case
\eqref{e:0}. According to Browder \cite{Brow65}, who initiated the
study of set-valued monotone operators in Banach spaces, the
Hilbertian setting was worked out by Minty in unpublished notes.
\end{itemize} 
Accounts of the history of the development of monotone operator
theory in the 1960s can be found in \cite{Borw10}, \cite{Brow68},
\cite{Kach68}, \cite[Section~2.12]{Lion69}, \cite{Mint68}, and
\cite[Chapter~VI]{Vain72}. In that period, the main mathematical
areas of applications were nonlinear equations, partial
differential equations, boundary-value problems,
nonexpansive semigroups, convex analysis,
evolution equations, and variational inequalities; see
\cite{Brez66,Brow63,Brow68,Ghiz69,Komu67,Lera65,More66,%
Vish61,Zara71} and
their bibliographies. At the same time, monotonicity continued to
be used in the analysis of networks and systems, for instance in
\cite{Zam66a,Zam66b}, where it is known as incremental
positiveness; see also \cite{Deso75} where monotonicity is called
incremental passivity. The main use of monotone operators was to
establish existence, uniqueness, or stability results in a variety
of nonlinear problems in analysis.

\subsection{Examples of maximally monotone operators}
\label{sec:23}

The following example concerns single-valued operators;
Examples~\ref{ex:12}--\ref{ex:12b} follow from it 
\cite[Chapter~20]{Livre1}.

\begin{example}[{\protect{\cite[Lemma~1]{Mint63}}}]
\label{ex:11}
Let $A\colon\HH\to\HH$ be monotone and \emph{hemicontinuous}
(in particular, continuous) in the sense that
\begin{equation}
\bigl(\forall (x,y,z)\in\HH^3\bigr)\quad
\lim_{0<\alpha\downarrow 0}\scal{z}{A(x+\alpha y)}=\scal{z}{Ax}.
\end{equation}
Then $A$ is maximally monotone.
\end{example}

\begin{example}
\label{ex:12}
Let $T\colon\HH\to\HH$ be nonexpansive and let $\alpha\in[-1,1]$. 
Then $\Id+\alpha T$ is maximally monotone. In particular, set
$A=\Id-T$. Then $A$ is maximally monotone and $\zer A=\Fix T$.
\end{example}

\begin{example} 
\label{ex:13}
Let $A\colon\HH\to\HH$ be cocoercive. Then $A$ is maximally 
monotone.
\end{example}

\begin{example} 
\label{ex:12f}
Let $M\colon\HH\to 2^{\HH}$ be maximally monotone and set $A=J_M$.
Then $A$ is maximally monotone and $\zer A=\zer M^{-1}$.
\end{example}

\begin{example} 
\label{ex:13y}
Let $M\colon\HH\to 2^{\HH}$ be maximally monotone, let
$\gamma\in\RPP$, and set $A=\moyo{M}{\gamma}$ (see
\eqref{e:yosi1}). Then $A$ is $\gamma$-cocoercive, hence
maximally monotone, and $\zer A=\zer M$.
\end{example}

\begin{example} 
\label{ex:12g}
Let $f\in\Gamma_0(\HH)$ and set $A=\prox_f$. Then $A$ is maximally
monotone.
\end{example}

\begin{example} 
\label{ex:12h}
Let $C$ be a nonempty closed convex subset of $\HH$ and set 
$A=\proj_C$. Then $A$ is maximally monotone.
\end{example}

\begin{example} 
\label{ex:12b}
Let $A\in\BL(\HH)$ be a \emph{skew} operator, i.e., $A^*=-A$. Then 
$A$ is maximally monotone.
\end{example}

Here is an elementary example of a maximally monotone set-valued
operator on the real line.

\begin{example}
\label{ex:0}
Let $a\in\RR$ and $b\in\RR$ be such that $a<b$, let
$f\colon[a,b]\to\RR$ be increasing (see \eqref{e:d1}), and define
\begin{equation}
(\forall x\in\RR)\quad Ax=
\begin{cases}
\emp,&\text{if}\;\;x\notin[a,b];\\
\left]\minf,f(a)\right],&\text{if}\;\;x=a;\\
\left[f(b),\pinf\right[,&\text{if}\;\;x=b;\\
\big[\sup f(\left[a,x\right[),\inf f(\left]x,b\right])
\big],&\text{if}\;\;x\in\left]a,b\right[.
\end{cases}
\end{equation}
Then $A$ is maximally monotone.
\end{example}

The following example is a central result in variational methods
(see \cite[Corollary p.~244]{Mint64} for a special case).

\begin{example}[{\protect{\cite{More65}}}]
\label{ex:1}
Let $f\colon\HH\to\RX$ be proper. Then the \emph{subdifferential}
\begin{equation}
\label{e:subdiff}
\partial f\colon\HH\to 2^{\HH}\colon x\mapsto\menge{x^*\in\HH}
{(\forall y\in\HH)\;\;\scal{y-x}{x^*}+f(x)\leq f(y)}
\end{equation}
of $f$ is monotone and (\emph{Fermat's rule}) $\zer\partial
f=\Argmin f$. If $f\in\Gamma_0(\HH)$, then $\partial f$ is
maximally monotone and $(\partial f)^{-1}=\partial f^*$.
\end{example}

\begin{example}[{\protect{\cite[Theorem~24.3]{Rock70}}}]
\label{ex:rocky}
Let $A\colon\RR\to 2^{\RR}$ be maximally monotone. Then there
exists $f\in\Gamma_0(\RR)$ such that $A=\partial f$. 
\end{example}

\begin{example}
\label{ex:2}
Let $C$ be a nonempty convex subset of $\HH$. Then, setting
$f=\iota_C$ in Example~\ref{ex:1}, we conclude that the
\emph{normal cone} operator 
\begin{equation}
\begin{array}{lll}
N_C=\partial\iota_C\colon
&\hspace{-2mm}\HH&\hspace{-2mm} \to 2^{\HH}\\
&\hspace{-1mm} x&\hspace{-2mm}\mapsto
\begin{cases}
\menge{x^*\in\HH}{(\forall y\in C)\;\scal{y-x}{x^*}\leq 0},
&\text{if}\;\;x\in C;\\
\emp,&\text{otherwise}
\end{cases}
\end{array}
\end{equation}
of $C$ is monotone and that it is maximally monotone if $C$ is
closed, in which case $(N_C)^{-1}=\partial\sigma_C$.
\end{example}

\begin{example}
\label{ex:2V}
Let $V$ be a closed vector subspace of $\HH$. Then it follows from
Example~\ref{ex:2} that 
\begin{equation}
N_V\colon\HH\to2^{\HH}\colon x\mapsto
\begin{cases}
V^\bot,&\text{if}\;\;x\in V;\\
\emp,&\text{otherwise}
\end{cases}
\end{equation}
is maximally monotone and $(N_V)^{-1}=N_{V^\bot}$.
\end{example}

The next two examples involve the Laplacian operator and are
central to partial differential equations
\cite{Atto14,Barb10,Brez71,Ghou09,Zei90B}.

\begin{example}[{\protect{\cite[Theorem~17.2.10]{Atto14}}}]
\label{ex:1Ld}
Let $\Omega$ be a nonempty bounded open subset of $\RR^N$, 
suppose that $\HH=L^2(\Omega)$, and set
\begin{equation}
A\colon\HH\to 2^{\HH}\colon x\mapsto
\begin{cases}
-\Delta x,&\text{if}\;x\in H_0^1(\Omega)\;
\text{and}\;\Delta x\in\HH;\\
\emp,&\text{otherwise.}
\end{cases}
\end{equation}
Then it follows from Example~\ref{ex:1} that $A$ is maximally 
monotone as the subdifferential of the function
\begin{equation}
f\colon\HH\to\RX\colon x\mapsto
\begin{cases}
\dfrac{1}{2}{\displaystyle\int_\Omega}
\|\nabla x(\omega)\|^2d\omega,&\text{if}\;\;
x\in H_0^1(\Omega);\\
\pinf,&\text{otherwise},
\end{cases}
\end{equation}
which is in $\Gamma_0(\HH)$. In addition, if $\bdry\Omega$ is of
class $\mathscr{C}^2$, then 
$\dom\partial f=H^2(\Omega)\cap H_0^1(\Omega)$.
\end{example}

\begin{example}[{\protect{\cite[Section~17.2.9]{Atto14}}}]
\label{ex:1Lv}
Let $\Omega$ be a nonempty bounded open subset of $\RR^N$ such that
$\bdry\Omega$ is of class $\mathscr{C}^2$, let 
$\partial/\partial\nu$ denote the outward normal derivative to
$\bdry\Omega$, suppose that $\HH=L^2(\Omega)$, let $h\in\HH$,
and set
\begin{equation}
\begin{array}{lll}
A\colon&\HH&\!\!\!\!\to 2^{\HH}\\
&x&\!\!\!\!\mapsto
\begin{cases}
-\Delta x-h,&\text{if}\;x\in H^2(\Omega)\;
\text{and}\;\partial x/\partial\nu=0\;\text{a.e.\ on}
\;\bdry\Omega;\\
\emp,&\text{otherwise.}
\end{cases}
\end{array}
\end{equation}
Then it follows from Example~\ref{ex:1} that $A$ is maximally 
monotone as the subdifferential of the function
\begin{equation}
\begin{array}{lll}
f\colon&\HH&\!\!\!\!\to\RX\\
&x&\!\!\!\!\mapsto
\begin{cases}
\dfrac{1}{2}{\displaystyle\int_\Omega}
\|\nabla x(\omega)\|^2d\omega-{\displaystyle\int_\Omega}
x(\omega)h(\omega)d\omega,&\text{if}\;\;x\in H^1(\Omega);\\
\pinf,&\text{otherwise},
\end{cases}
\end{array}
\end{equation}
which is in $\Gamma_0(\HH)$.
\end{example}

The next scenario arises in the study of evolution equations by
monotonicity methods \cite{Brez71,Brez73,Show97,Zei90B}.

\begin{example}[{\protect{\cite[Example~4]{Brez71},
\cite[Chapter~IV]{Show97}, \cite[Chapter~32]{Zei90B}}}]\
\label{ex:7} 
Let $\mathsf{H}$ be a separable real Hilbert space, let 
$T\in\RPP$, and suppose that $\HH=L^2([0,T];{\mathsf H})$.
For every $y\in\HH$, the function
$x\colon[0,T]\to{\mathsf H}\colon t\mapsto\int_0^ty(s)ds$ is
differentiable a.e.\ on $\left]0,T\right[$ with 
$x'=y$ a.e. Define 
\begin{equation}
H^1\bigl([0,T];{\mathsf H}\bigr)=
\menge{x\in\HH}{x'\in L^2\bigl([0,T];{\mathsf H}\bigr)},
\end{equation}
let $\mathsf{x}_0\in\mathsf{H}$, and set
\begin{equation}
\label{e:ex7}
A\colon\HH\to 2^{\HH}\colon x\mapsto
\begin{cases}
\{x'\},&\text{if}\;\;x\in H^1\bigl([0,T];{\mathsf H}\bigr)\;
\text{and}\;x(0)=\mathsf{x}_0;\\
\emp,&\text{otherwise}
\end{cases}
\end{equation}
and
\begin{equation}
\label{e:ex7p}
B\colon\HH\to 2^{\HH}\colon x\mapsto
\begin{cases}
\{x'\},&\text{if}\;\;x\in H^1\bigl([0,T];{\mathsf H}\bigr)\;
\text{and}\;x(0)=x(T);\\
\emp,&\text{otherwise}.
\end{cases}
\end{equation}
Then $A$ and $B$ are maximally monotone.
\end{example}

\begin{example}[{\protect{\cite[Exemple~2.3.3]{Brez73}}}]
\label{ex:44}
Let $(\Omega,\mathcal{F},\mu)$ be a measure space, let 
$\mathsf{H}$ be a separable real Hilbert space, let
$\mathsf{A}\colon\mathsf{H}\to 2^\mathsf{H}$ be maximally monotone,
and set $\HH=L^2((\Omega,\mathcal{F},\mu);\mathsf{H})$. Define
$A\colon\HH\to 2^{\HH}$ via
\begin{multline}
\label{e:ex44}
(\forall x\in\HH)(\forall x^*\in\HH)\quad (x,x^*)\in\gra A\;\;
\Leftrightarrow\\
\text{for $\mu$-almost every}\;\omega\in\Omega,\;\;
\bigl(x(\omega),x^*(\omega)\bigr)\in\gra\mathsf{A}
\end{multline}
and suppose that one of the following holds:
\begin{enumerate}
\item
$\mu(\Omega)<\pinf$.
\item
$\mathsf{0}\in\mathsf{A}\mathsf{0}$.
\end{enumerate}
Then $A$ is maximally monotone.
\end{example}

We now turn to an equilibrium problem in the sense of
\cite{Blum94}.

\begin{example}[{\protect{\cite[Theorem~3.5]{Aoya08}}}]
\label{ex:40}
Let $C$ be a nonempty closed convex subset of $\HH$ and suppose
that $F\colon C\times C\to\RR$ satisfies the following:
\begin{enumerate}
\item
$(\forall x\in C)$ $F(x,x)=0$.
\item
$(\forall x\in C)(\forall y\in C)$ $F(x,y)+F(y,x)\leq 0$.
\item
For every $x\in C$, $F(x,\cdot)\colon C\to\RR$ is lower 
semicontinuous and convex.
\item
$(\forall x\in C)(\forall y\in C)(\forall z\in C)$
$\underset{0<\varepsilon\to 0}{\varlimsup}
F\big((1-\varepsilon)x+\varepsilon z,y\big)\leq F(x,y)$.
\end{enumerate}
Set 
\begin{equation}
\hspace{-1mm}
\begin{array}{lll}
\!\!A\colon\!\!\!&\!\!\!\HH&\!\!\!\!\to 2^{\HH}\\
       &x&\!\!\!\!\mapsto
\begin{cases}
\menge{x^*\in\HH}{(\forall y\in C)\;F(x,y)+
\scal{x-y}{x^*}\geq 0},&\text{if}\;x\in C;\\
\emp,&\text{otherwise}.
\end{cases}
\end{array}
\end{equation}
Then $A$ is maximally monotone and
$\zer A=\menge{x\in C}{(\forall y\in C)\;F(x,y)\geq 0}$ is the set
of \emph{equilibria} of $F$.
\end{example}

We conclude with an example in the theory of saddle functions.

\begin{example}[{\protect{\cite[Theorem~3]{Roc70b}}}]
\label{ex:3} 
Let $F\colon\HH\oplus\GG\to\RXX$ be a \emph{saddle function},
i.e., a convex-concave function which
is proper and closed in the sense of \cite{Roc70b,Roc71d} (for
instance, for every $x\in\HH$ and every $y\in\GG$, 
$-F(x,\cdot)\in\Gamma_0(\GG)$ and $F(\cdot,y)\in\Gamma_0(\HH)$). 
Set
\begin{equation}
\label{e:rocky70}
(\forall x\in\HH)(\forall y\in\GG)\quad 
A(x,y)=\partial F(\cdot,y)(x)\times
\partial\bigl(-F(x,\cdot)\bigr)(y).
\end{equation}
Then $A$ is maximally monotone and
\begin{equation}
\zer{A}=\menge{(x,y)\in\HH\oplus\GG}
{F(x,y)=\inf F(\HH,y)=\sup F(x,\GG)}
\end{equation}
is the set of \emph{saddle points} of $F$.
\end{example}

The following illustration is set in the powerful perturbation
framework of Rockafellar \cite{Rock69,Roc70b,Rock74} (see also 
\cite{Joly71}), which provides a systematic tool to construct
duality frameworks in minimization problems.

\begin{example}
\label{ex:pjl}
Let $\VV$ be a real Hilbert space, let $f\colon\HH\to\RX$ be a
proper function, and consider the primal problem
\begin{equation}
\label{e:71p}
\minimize{x\in\HH}{f(x)}.
\end{equation}
Let $F\colon\HH\oplus\VV\to\RX$ be a \emph{perturbation} of $f$,
i.e., $(\forall x\in\HH)$ $f(x)=F(x,0)$. The associated
\emph{Lagrangian} is
\begin{equation}
\label{e:L}
\mathscr{L}_{F}\colon\HH\oplus\VV\mapsto\RXX\colon 
(x,v^*)\mapsto\inf_{v\in\VV}\big(F(x,v)-\scal{v}{v^*}\big),
\end{equation}
the associated \emph{dual problem} is
\begin{equation}
\label{e:71d}
\minimize{v^*\in\VV}
{\sup_{x\in\HH}\bigl(-\mathscr{L}_{F}(x,v^*)\bigr)},
\end{equation}
and the associated \emph{saddle operator} is
\begin{equation}
\label{e:K}
\sad_{F}\colon\HH\oplus\VV\to 2^{\HH\oplus\,\VV}
\colon (x,v^*)\mapsto\partial
\big(\mathscr{L}_{F}(\cdot,v^*)\big)(x)\times
\partial\bigl(-\mathscr{L}_{F}(x,\cdot)\bigr)(v^*).
\end{equation}
It follows from Example~\ref{ex:3} that $\sad_F$ is maximally
monotone. In addition, if $(x,v^*)\in\zer\sad_F$, then $x$ solves
\eqref{e:71p} and $v^*$ solves \eqref{e:71d}.
\end{example}

\subsection{Basic theory}

\subsubsection{Operations preserving maximal monotonicity}

The examples of Section~\ref{sec:23} can be combined in various
fashions to create maximally monotone operators.

\begin{lemma}[{\protect{\cite[Proposition~20.22]{Livre1}}}]
\label{l:0602}
Let $A\colon\HH\to 2^{\HH}$ be maximally monotone, 
let $z\in\HH$, let $u\in\HH$, and let $\gamma\in\RPP$. Then 
$A^{-1}$ and $x\mapsto u+\gamma A(x+z)$ are maximally monotone.
\end{lemma}

\begin{lemma}[{\protect{\cite[Proposition~23.18]{Livre1}}}]
\label{l:0616}
Let $(\HH_i)_{i\in I}$ be a finite family of 
real Hilbert spaces, set 
\begin{equation}
\HHH=\bigoplus_{i\in I}\HH_i, 
\end{equation}
and, for every $i\in I$, let $A_i\colon\HH_i\to 2^{\HH_i}$ be 
maximally monotone. Set 
\begin{equation}
\boldsymbol{A}\colon\HHH\to 2^{\HHH}\colon
(x_i)_{i\in I}\mapsto\mbox{\LARGE{$\times$}}_{i\in I}A_ix_i. 
\end{equation}
Then $\boldsymbol{A}$ is maximally monotone.
\end{lemma}

\begin{lemma}
\label{l:9}
Let $\beta\in\RPP$, let $A\colon\HH\to 2^{\HH}$, let 
$U\in\BL(\HH)$ be self-adjoint and $\beta$-strongly monotone, 
and let $\mathcal{X}$ be the real Hilbert space 
obtained by endowing $\HH$ with the scalar product
$(x,y)\mapsto\scal{Ux}{y}$. Then the following hold:
\begin{enumerate}
\item
\label{l:9i-}
$\zer(U^{-1}\circ A)=\zer A$.
\item
\label{l:9i}
Suppose that $A\colon\HH\to 2^{\HH}$ is maximally monotone. Then
$U^{-1}\circ A\colon\mathcal{X}\to 2^{\mathcal{X}}$ is maximally
monotone.
\item
\label{l:9ii}
Let $\alpha\in\RPP$ and suppose that $A\colon\HH\to\HH$ is
$\alpha$-cocoercive. Then 
$U^{-1}\circ A\colon\mathcal{X}\to 2^{\mathcal{X}}$ is 
$\alpha\beta$-cocoercive.
\end{enumerate}
\end{lemma}
\begin{proof}
\ref{l:9i-} is clear and \ref{l:9i} is proved in 
\cite[Lemma~3.7(i)]{Opti14}.

\ref{l:9ii}: Take $(x,y)\in\HH\times\HH$. Then 
\begin{align}
\scal{x-y}{(U^{-1}\circ A)x-(U^{-1}\circ A)y}_{\mathcal{X}}
&=\scal{x-y}{Ax-Ay}_{\HH}\nonumber\\
&\geq\alpha\|Ax-Ay\|^2_{\HH}. 
\end{align}
However, 
$\|U^{-1}x\|^2_{\mathcal{X}}=\scal{x}{U^{-1}x}_{\HH}\leq
\|U\|^{-1}\,\|x\|^2_{\HH}$ and $\|U\|^{-1}\leq\beta^{-1}$
\cite[Section~VI.2.6]{Kato80}.
\end{proof}

\begin{lemma}[{\protect{\cite[Theorem~25.3]{Livre1},
\cite[Section~24]{Botr10}, \cite[Corollary~4.2(a)]{Penn00}}}]
\label{l:10}
Let $A\colon\HH\to 2^\HH$ and $B\colon\GG\to 2^\GG$ 
be maximally monotone, let $L\in\BL(\HH,\GG)$, and suppose that 
\begin{equation}
\cone\bigl(L(\dom A)-\dom B\bigr)\;
\text{is a closed vector subspace of}\;\GG.
\end{equation}
Then $A+L^*\circ B\circ L$ is maximally monotone.
\end{lemma}

\begin{lemma}[{\protect{\cite[Corollary~25.5]{Livre1}}}]
\label{l:11}
Let $A\colon\HH\to 2^{\HH}$ and $B\colon\HH\to 2^{\HH}$ be 
maximally monotone and such that one of the following holds:
\begin{enumerate}
\item
\label{l:11i-}
$\cone\,(\dom A-\dom B)\;
\text{is a closed vector subspace of}\;\HH$.
\item 
\label{l:11i}
$\dom B=\HH$.
\item 
\label{l:11ii}
$\dom A\cap\intdom B\neq\emp$.
\end{enumerate}
Then $A+B$ is maximally monotone.
\end{lemma}

\begin{lemma}[{\protect{\cite[Theorem~2.1]{Alim14}}}]
\label{l:92}
Let $A\colon\HH\to 2^{\HH}$ be maximally monotone and
let $B\colon\HH\to 2^{\HH}$ be monotone and such that 
$\dom B=\HH$ and $A-B$ is monotone. Then $A-B$ is maximally
monotone.
\end{lemma}

\begin{lemma}
\label{l:15}
Let $A\colon\HH\to 2^{\HH}$ and $B\colon\HH\to 2^{\HH}$ be
maximally monotone. Define the \emph{parallel sum} of $A$ and $B$
as
\begin{equation}
\label{e:20115n}
A\infconv B=\bigl(A^{-1}+B^{-1}\bigr)^{-1}
\end{equation}
and suppose that $\cone\,(\ran A-\ran B)$ 
is a closed vector subspace of $\HH$.
Then $A\infconv B$ is maximally monotone.
\end{lemma}
\begin{proof}
This follows from \eqref{e:d5}, Lemma~\ref{l:0602}, and
Lemma~\ref{l:11}\ref{l:11i-}. 
\end{proof}

\begin{lemma}[{\protect{\cite[Lemma~2.2]{Beck14}}}]
Let $A\colon\HH\to 2^{\HH}$ and $B\colon\GG\to 2^{\GG}$ be
maximally monotone, and let $L\in\BL(\HH,\GG)$. Define
the \emph{parallel composition} of $A$ with $L$ as
\begin{equation}
\label{e:20115a}
L\pushfwd A=\big(L\circ A^{-1}\circ L^*\bigr)^{-1}.
\end{equation}
Suppose that
\begin{equation}
\label{e:JMay}
\cone\big(\ran A-L^*(\ran B)\big)\;
\text{is a closed vector subspace of}\;\HH.
\end{equation}
Then $(L\pushfwd A)\infconv B$ is a maximally monotone operator 
from $\GG$ to $2^{\GG}$.
\end{lemma}

\begin{example}[{\protect{\cite[Proposition~4.5(i)--(ii)]%
{Svva23}}}]
\label{ex:22}
Let $L\in\BL(\HH,\GG)$ be such that $\|L\|\leq 1$ and let 
$B\colon\GG\to 2^{\GG}$ be maximally monotone. Define the
\emph{resolvent composition} of $B$ with $L$ as 
\begin{equation}
\label{e:xc}
\proxc{L}{B}=L^*\pushfwd(B+\Id_{\GG})-\Id_{\HH}
\end{equation}
and the \emph{resolvent cocomposition} of $B$ with $L$ as 
$\proxcc{L}{B}=(\proxc{L}{B^{-1}})^{-1}$.
Then $\proxc{L}{B}$ and $\proxcc{L}{B}$ are maximally monotone
operators from $\HH$ to $2^{\HH}$.
\end{example}

\begin{example}
\label{ex:23}
Let $0<p\in\NN$, let $(\omega_k)_{1\leq k\leq p}$ be a family in 
$\rzeroun$ such that $\sum_{k=1}^p\omega_k=1$, and let 
$(A_k)_{1\leq k\leq p}$ be maximally monotone operators from $\HH$
to $2^{\HH}$. Then the \emph{resolvent average}
\begin{equation}
\label{e:ra}
\Bigg(\sum_{k=1}^p\omega_kJ_{A_k}\Bigg)^{-1}-\Id_{\HH}
\end{equation}
is maximally monotone. This result was originally established in
\cite[Proposition~2.7]{Baus16} and derived from Example~\ref{ex:22}
in \cite[Remark~4.10(ii)]{Svva23}.
\end{example}

\begin{example}
\label{ex:21}
Let $A\colon\HH\to 2^\HH$ be a maximally monotone operator and let
$V$ be a closed vector subspace of $\HH$. The \emph{partial
inverse} of $A$ with respect to $V$ is the operator
$A_V\colon\HH\to 2^{\HH}$ with graph
\begin{equation}
\label{e:pi}
\gra A_V=\menge{(\proj_Vx+\proj_{V^\bot}x^*,
\proj_Vx^*+\proj_{V^\bot}x)}{(x,x^*)\in\gra A}.
\end{equation}
This construction was introduced in \cite{Spin83}, which
contains the following (see \cite[Section~2]{Spin83}): 
\begin{enumerate}
\item
\label{ex:21i}
$A_V$ is maximally monotone.
\item
\label{ex:21ii}
Let $x\in\HH$. Then $x\in\zer A_V$ $\Leftrightarrow$
$(\proj_Vx,\proj_{V^\bot}x)\in\gra A$.
\end{enumerate}
\end{example}

\subsubsection{Resolvent}

In terms of solving inclusion problems, the resolvent of
\eqref{e:d7} is the most important operator attached to a monotone
operator $A$. First, as seen in \eqref{e:ee14}, it can be employed
as a device to generate points in the graph of $A$. Second, as seen
in \eqref{e:z}, its fixed point set coincides with the set of zeros
of $A$. Third, resolvents provide an effective bridge between the
theory of nonexpansive operators and that of monotone operators.
This connection goes back to the theory of semigroups of linear
nonexpansive operators. The following result, essentially due to
Minty \cite{Mint62}, establishes such a connection in the nonlinear
case. It states in particular that the resolvent of a maximally
monotone operator is a firmly nonexpansive operator which is
defined everywhere.

\begin{figure}
\begin{center}
\includegraphics[width=110mm]{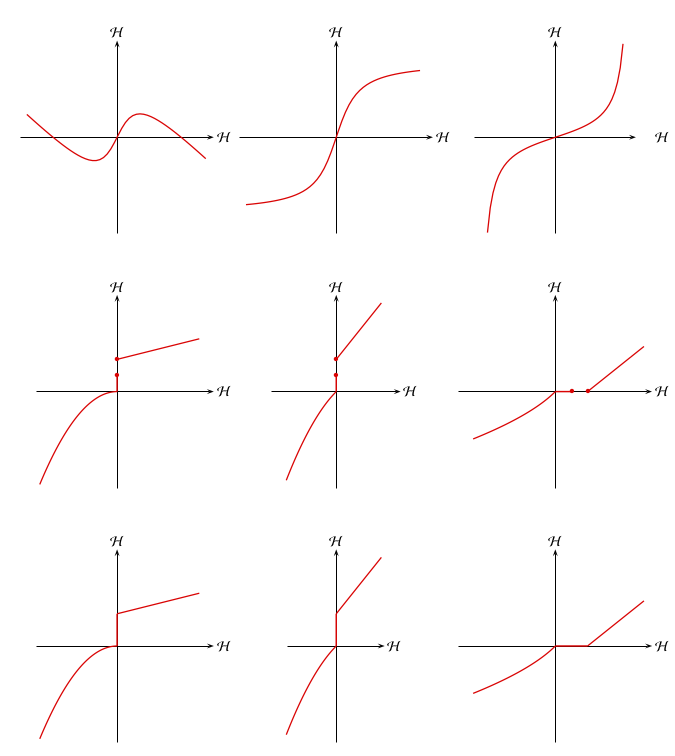}
\caption{Illustration of Minty's theorem (Lemma~\ref{l:0}).
From left to right on each row: graph of $A$, graph of $\Id+A$, and
graph of $J_A$.
Top: $A$ is not monotone: $\ran(\Id+A)=\dom J_A\neq\HH$ and
$J_A$ is not firmly nonexpansive.
Middle: $A$ is monotone but not maximally monotone: $J_A$ is 
firmly nonexpansive but $\ran(\Id+A)=\dom J_A\neq\HH$.
Bottom: $A$ is maximally monotone: $J_A$ is firmly nonexpansive
with $\ran(\Id+A)=\dom J_A=\HH$.}
\label{fig:11}
\end{center}
\end{figure}

\begin{lemma}[{\protect{\cite[Proposition~23.8]{Livre1}}}]
\label{l:0}
Let $D$ be a nonempty subset of $\HH$, let $T\colon D\to\HH$,
and set $A=T^{-1}-\Id$. Then the following hold
(see Figure~\ref{fig:11}):
\begin{enumerate}
\item
\label{l:0i}
$D=\ran(\Id+A)$ and $T=J_A$. 
\item 
\label{l:0ii}
$T$ is firmly nonexpansive if and only if $A$ is monotone.
\item 
\label{l:0iii}
$T$ is firmly nonexpansive and $D=\HH$ if and only if $A$ 
is maximally monotone. 
\end{enumerate}
\end{lemma}

Here are a few examples of resolvents that will be explicitly 
needed; see \cite{Livre1,Chie18,Banf11} for additional examples
with closed form expressions and, in particular, instances of 
proximity operators.

\begin{example}[{\protect{\cite[Proposition~6.a]{More65}}}]
\label{ex:r5}
Let $f\in\Gamma_0(\HH)$. Then $J_{\partial f}=\prox_f$.
\end{example}

\begin{example}[{\protect{\cite[Exemple~p.~2897]{Mor62b}}}]
\label{ex:r6}
Let $C$ be a nonempty closed convex subset of $\HH$. Then
$J_{N_C}=\prox_{\iota_C}=\proj_C$.
\end{example}

\begin{example}[{\protect{\cite[Proposition~23.18]{Livre1}}}]
\label{ex:r1}
Let $0<m\in\NN$, let $(\HH_i)_{1\leq i\leq m}$ be real Hilbert 
spaces, set 
\begin{equation}
\HHH=\bigoplus_{i=1}^m\HH_i, 
\end{equation}
and, for every 
$i\in\{1,\ldots,m\}$, let $A_i\colon\HH_i\to 2^{\HH_i}$ be 
maximally monotone. Set 
\begin{equation}
\boldsymbol{A}\colon\HHH\to 2^{\HHH}\colon
(x_i)_{1\leq i\leq m}\mapsto
\underset{1\leq i\leq m}{\mbox{\LARGE{$\times$}}}A_ix_i. 
\end{equation}
Then
$\boldsymbol{A}$ is maximally monotone (Lemma~\ref{l:0616}) and 
\begin{equation}
J_{\boldsymbol{A}}\colon\HHH\to\HHH\colon (x_i)_{1\leq i\leq m}
\mapsto\big(J_{A_i}x_i\big)_{1\leq i\leq m}.
\end{equation}
\end{example}

\begin{example}
\label{ex:21b}
Let $A\colon\HH\to 2^{\HH}$ be maximally monotone, let $V$ be a
closed vector subspace of $\HH$, and let $A_V$ be the partial
inverse of Example~\ref{ex:21}. In addition, let $x\in\HH$ and 
$p\in\HH$. Then 
\begin{equation}
p=J_{A_V}x\quad\Leftrightarrow\quad 
\proj_Vp+\proj_{V^\bot}(x-p)=J_Ax.
\end{equation}
\end{example}
\begin{proof}
This is implicitly in \cite[Section~4]{Spin83}; see
\cite[Lemma~2.2]{Optl14} for a proof.
\end{proof}

\begin{example}[{\protect{\cite[Lemmas~3.7(iii) and~3.1]{Opti14}}}]
\label{ex:b12}
As in Lemma~\ref{l:9}\ref{l:9i}, $A\colon\HH\to 2^{\HH}$ is
maximally monotone, $U\in\BL(\HH)$ is self-adjoint and strongly
monotone, and $\mathcal{X}$ is the real Hilbert space obtained by
endowing $\HH$ with the scalar product $(x,y)\mapsto\scal{Ux}{y}$.
Then $J_{U^{-1}\circ A}=(U+A)^{-1}\circ U$.
\end{example}

\begin{example}[{\protect{\cite[Propositions~1.2 and 
4.1(v)]{Svva23}}}]
\label{ex:r3}
Let $L\in\BL(\HH,\GG)$ be such that $\|L\|\leq 1$, let 
$B\colon\GG\to 2^{\GG}$ be maximally monotone, and consider the
resolvent compositions of Example~\ref{ex:22}. Then 
\begin{equation}
J_{\proxc{L}{B}}=L^*\circ J_B\circ L\quad\text{and}\quad
J_{\proxcc{L}{B}}=\Id_\HH-L^*\circ L+L^*\circ J_B\circ L.
\end{equation}
\end{example}

\subsubsection{Warped resolvents}
\label{sec:wr}

A generalization of the notion of a resolvent is the following.

\begin{definition}[{\protect{\cite[Definition~1.1]{Jmaa20}}}]
\label{d:wr}
Let $D$ be a nonempty subset of $\HH$, let $U\colon D\to\HH$, and
let $M\colon\HH\to 2^{\HH}$ be such that $\ran U\subset\ran(U+M)$
and $U+M$ is injective. The \emph{warped resolvent} of $M$ with
kernel $U$ is $J_M^U=(U+M)^{-1}\circ U\colon D\to D$. 
\end{definition}

The properties of warped resolvent generalize those of classical
ones. In this respect, here is an extension of 
\eqref{e:ee14}--\eqref{e:z}.

\begin{lemma}
\label{l:warp}
Let $D$ and $E$ be nonempty subsets of $\HH$, 
let $U\colon D\to\HH$, let 
$C\colon E\to\HH$, and let $W\colon\HH\to 2^{\HH}$ be such that
$\ran U\subset\ran(U+W+C)$ and $U+W+C$ is injective. Then 
the following hold: 
\begin{enumerate}
\item
\label{l:warpi}
Let $x\in D$ and $p\in D$. Then $p=J_{W+C}^Ux$ $\Leftrightarrow$ 
$(p,Ux-Up-Cp)\in\gra W$.
\item
\label{l:warpii}
$\Fix J_{W+C}^U=D\cap\zer(W+C)$.
\end{enumerate}
\end{lemma}
\begin{proof}
Note that $J_{W+C}^U\colon D\to D$ is well defined.

\ref{l:warpi}: 
$p=J_{W+C}^Ux$ 
$\Leftrightarrow$ $p=(U+W+C)^{-1}(Ux)$
$\Leftrightarrow$ $Ux\in Up+Wp+Cp$
$\Leftrightarrow$ $Ux-Up-Cp\in Wp$.

\ref{l:warpii}: Let $x\in\HH$. Then \ref{l:warpi} yields
$x=J_{W+C}^Ux$ 
$\Leftrightarrow$ [$x\in D$ and $(x,-Cx)\in\gra W$] 
$\Leftrightarrow$ [$x\in D$ and $x\in\zer(W+C)$]. 
\end{proof}

An instance of a warped resolvent with a linear kernel appears in
Example~\ref{ex:b12}, where $D=\HH$ and $U\in\BL(\HH)$ is a
self-adjoint strongly monotone operator. Self-adjoint monotone
operators which are not strongly monotone have also been used as
kernels; see \cite{Bred22,Xue23a}. The next example features a
monotone kernel in $\BL(\HH)$ which is not self-adjoint.

\begin{example}
\label{ex:59}
Let $A\colon\HH\to 2^{\HH}$ and $B\colon\GG\to 2^{\GG}$ be
maximally monotone, and suppose that $0\neq L\in\BL(\HH,\GG)$. Set
$\XXX=\HH\oplus\GG$ and 
\begin{equation}
\label{e:ex59}
\begin{cases}
\kut\colon\XXX\to 2^{\XXX}\colon(x,y^*)\mapsto
(Ax+L^*y^*)\times(B^{-1}y^*-Lx)\\
\boldsymbol{U}\colon\XXX\to\XXX\colon(x,y^*)
\mapsto(x-L^*y^*,Lx+y^*).
\end{cases}
\end{equation}
As will be seen in Lemma~\ref{l:31z}, $\kut$ is the
Kuhn--Tucker operator associated with the problem of finding a 
zero of $A+L^*\circ B\circ L$. It follows from \eqref{e:ex59} that
\begin{equation}
J_{\kut}^{\boldsymbol{U}}\colon\XXX\to\XXX\colon
(x,y^*)\mapsto\bigl(J_A(x-L^*y^*),J_{B^{-1}}(Lx+y^*)\bigr),
\end{equation}
whereas $J_{\kut}$ is typically intractable.
\end{example}

The next examples employ nonlinear kernels.

\begin{example}
\label{ex:90}
Let $M\colon\HH\to 2^{\HH}$ be maximally monotone and such that 
$\zer M\neq\emp$, let $f\colon\HH\to\RX$ be a Legendre function
such that $\dom M\subset\intdom f$, and set $D=\intdom f$ and
$U=\nabla f$. Then it follows from
\cite[Corollary~3.14(ii)]{Sico03} that $J_M^U\colon D\to D$ is a
well-defined warped resolvent, called the $D$-resolvent of $M$. It
is an essential tool in the study of algorithms based on Bregman
distances which goes back to \cite{Breg67,Cens92,Ecks93,Tebo92}.
\end{example}

\begin{example}
\label{ex:91}
Let $A\colon\HH\to 2^{\HH}$ and $B\colon\HH\to 2^{\HH}$ be
maximally monotone, and let $f\in\Gamma_0(\HH)$ be essentially 
smooth \cite{Sico03}.
Suppose that $D=(\intdom f)\cap\dom A$ is a nonempty subset
of $\intdom B$, that $B$ is single-valued on $\intdom B$,
that $\nabla f$ is strictly monotone on $D$,
and that $(\nabla f-B)(D)\subset\ran(\nabla f+A)$. Set $M=A+B$ and
$U\colon D\to\HH\colon x\mapsto\nabla f(x)-Bx$. Then the warped 
resolvent coincides with the Bregman forward-backward operator 
$J_{M}^U=(\nabla f+A)^{-1}\circ(\nabla f-B)$ investigated in 
\cite{Svva21}, where it is shown to capture a construction found in
\cite{Rena97} and known as the \emph{auxiliary principle}.
In the case when $A$ and $B$ are subdifferentials, $J_M^U$ is the
operator studied in \cite{Nguy17} and, in Euclidean spaces, in
\cite{Bolt17}. Scenarios in which $J_M^U$ is more manageable than
$J_M$ are discussed in 
\cite{Bolt17,Svva21,Lufe18,Nguy17,Rena97,Tebo18}.
\end{example}

\begin{example}
\label{ex:92}
Let $A\colon\HH\to 2^{\HH}$, let $C\colon\HH\to\HH$ be cocoercive,
let $Q\colon\HH\to\HH$ be monotone and Lipschitzian, and let
$\gamma\in\RPP$. The underlying problem is to find a point in 
$\zer(A+C+Q)$ and we recover the 
\emph{nonlinear forward-backward operator} of \cite{Gise21}
as a warped resolvent as follows. Set $M=\gamma(A+C+Q)$, let
$K\colon\HH\to\HH$ be strongly monotone and Lipschitzian, and set
$U=K-\gamma(C+Q)$. Then 
$J_M^U=(K+\gamma A)^{-1}\circ(K-\gamma(C+Q))$, which is the
operator driving the algorithms of \cite{Gise21}.
\end{example}

\begin{remark}
\label{r:ex93}
If $B$ is cocoercive and $f=\|\cdot\|^2/2$ in 
Example~\ref{ex:91}, or if $K=\Id$ and $Q=0$ and $C=B$ in 
Example~\ref{ex:92}, then 
$J_M^U=J_{\gamma A}\circ(\Id-\gamma B)$. This operator 
will arise in the forward-backward algorithm of 
Section~\ref{sec:fb}. 
\end{remark}

\begin{lemma}
\label{l:20}
Let $Q\colon\HH\to\HH$ be Lipschitzian with constant 
$\beta\in\RPP$, let $K\colon\HH\to\HH$ be strongly monotone
with constant $\alpha\in\RPP$,
let $\varepsilon\in\left]0,\alpha\right[$,
and set $U=K-\gamma Q$. Then the following hold:
\begin{enumerate}
\item
\label{l:20i}
Let $\gamma\in\left]0,(\alpha-\varepsilon)/\beta\right]$. Then 
$U$ is $\varepsilon$-strongly monotone.
{\rm(\cite[Lemma~5.1(i)]{Jmaa20})}
\item
\label{l:20ii}
Suppose that $\alpha=1$ and $K=\Id$, and let
$\gamma\in\left]0,(1-\varepsilon)/\beta\right]$,
Then $U$ is cocoercive with constant $1/(2-\varepsilon)$.
{\rm(\cite[Lemma~5.1(ii)]{Jmaa20})}
\item
\label{l:20iii}
Suppose that $\alpha=1$, $K=\Id$, and $Q$ is $1/\beta$-cocoercive,
and let $\gamma\in\left]0,2/\beta\right[$. Then 
$U$ is $\gamma\beta/2$-averaged, hence nonexpansive.
{\rm(\cite[Lemma~2.3]{Opti04})}
\end{enumerate}
\end{lemma}

\subsubsection{Topological properties}

We record key properties of the graphs of monotone operators.

\begin{lemma}[{\protect{\cite[Proposition~20.38(ii)]{Livre1}}}]
\label{l:12}
Let $M\colon\HH\to 2^{\HH}$ be maximally monotone. Then $\gra M$ is
sequentially closed in 
$\HH^{\operatorname{weak}}\times\HH^{\operatorname{strong}}$, i.e.,
for every sequence $(x_n,x^*_n)_{n\in\NN}$ in $\gra M$ and every
$(x,x^*)\in\HH\times\HH$, if $x_n\weakly x$ and $x^*_n\to x^*$,
then $(x,x^*)\in\gra M$.
\end{lemma}

\begin{lemma}[{\protect{\cite[Corollary~26.6]{Livre1}}}]
\label{l:30}
Let $A\colon\HH\to 2^{\HH}$ and $B\colon\HH\to 2^{\HH}$ be 
maximally monotone, let $(x_n,x_n^*)_{n\in\NN}$ be a sequence in
$\gra A$, let $(y_n,y^*_n)_{n\in\NN}$ be a sequence in $\gra B$,
let $x\in\HH$, and let $x^*\in\HH$. Suppose that
\begin{equation}
\label{e:2013}
x_n\weakly x,\;\;
x_n^*\weakly x^*,\;\;
x_n-y_n\to 0,\;\;
\text{and}
\;\;x_n^*+y^*_n\to 0.
\end{equation}
Then $x\in\zer(A+B)$, $-x^*\in\zer(-A^{-1}\circ(-\Id)+B^{-1})$,
$(x,x^*)\in\gra A$, and $(x,-x^*)\in\gra B$.
\end{lemma}

\subsubsection{Subdifferentials}

The subdifferential operator of Example~\ref{ex:1} is an essential 
tool in variational analysis.

\begin{lemma}[{\protect{\cite[Proposition~16.6 and 
Theorem~16.47(i)]{Livre1}}}]
\label{l:22}
Let $f\in\Gamma_0(\HH)$, $g\in\Gamma_0(\GG)$, and 
$L\in\BL(\HH,\GG)$ be such that $(L(\dom f))\cap\dom g\neq\emp$. 
Then the following hold:
\begin{enumerate}
\item 
\label{l:22i}
$\zer(\partial f+L^*\circ(\partial g)\circ L)\subset
\zer\partial(f+g\circ L)=\Argmin(f+g\circ L)$.
\item 
\label{l:22ii}
Suppose that one of the following is satisfied:
\begin{enumerate}
\item 
\label{l:22iia}
$0\in\sri(L(\dom f)-\dom g)$.
\item 
$L(\dom f)-\dom g$ is a closed vector subspace of $\GG$.
\item 
$\dom g=\GG$.
\item 
$\GG$ is finite-dimensional and 
$(\reli L(\dom f))\cap(\reli\dom g)\neq\emp$.
\end{enumerate}
Then $\partial(f+g\circ L)=\partial f+L^*\circ(\partial g)\circ L$.
\end{enumerate}
\end{lemma}

\section{Structured monotone inclusions}
\label{sec:3}

Our master problem is the following two-operator inclusion.

\begin{problem}
\label{prob:0}
Let $A\colon\HH\to 2^{\HH}$ and $B\colon\HH\to 2^{\HH}$ be 
maximally monotone. The objective is to
\begin{equation}
\label{e:p0}
\text{find}\;\;x\in\HH\;\;\text{such that}\;\;0\in Ax+Bx.
\end{equation}
\end{problem}

\subsection{Two-operator formulations}

We provide problem formulations which correspond to specific
choices of the operators $A$ and $B$ in Problem~\ref{prob:0}
from the examples of Section~\ref{sec:23}.

\begin{problem}
\label{prob:21}
In Problem~\ref{prob:0}, let $f\in\Gamma_0(\HH)$, set 
$A=\partial f$, and suppose that $B$ is at most single-valued. Then
\eqref{e:p0} reduces to the variational inequality problem
\cite{Lion69}
\begin{equation}
\label{e:p21}
\text{find}\;\;x\in\HH\;\;\text{such that}\;\;
(\forall y\in\HH)\;\;\scal{x-y}{Bx}+f(x)\leq f(y).
\end{equation}
\end{problem}

\begin{problem}
\label{prob:22}
In Problem~\ref{prob:21}, let $C$ be a nonempty closed convex
subset of $\HH$ and set $f=\iota_C$.
Then \eqref{e:p21} reduces to the standard variational
inequality problem \cite{Fich63,Kind80}
\begin{equation}
\label{e:p22}
\text{find}\;\;x\in C\;\;\text{such that}\;\;
(\forall y\in C)\;\;\scal{x-y}{Bx}\leq 0.
\end{equation}
\end{problem}

\begin{problem}
\label{prob:24}
In Problem~\ref{prob:22}, suppose that $C$ is a cone with dual
cone $C^\oplus$. Then \eqref{e:p22} reduces to the
\emph{complementarity problem} \cite{Facc03}
\begin{equation}
\label{e:p24}
\text{find}\;\;x\in C\;\;\text{such that}\;\;x\perp Bx
\;\;\text{and}\;\;Bx\in C^\oplus.
\end{equation}
\end{problem}

\begin{problem}
\label{prob:27}
In Problem~\ref{prob:0}, let $f\in\Gamma_0(\HH)$ and 
$g\in\Gamma_0(\HH)$, and set $A=\partial f$ and $B=\partial g$. 
Suppose that one of the following holds:
\begin{enumerate}
\item
\label{prob:27i}
$0\in\sri(\dom f-\dom g)$.
\item
\label{prob:27ii}
$g\colon\HH\to\RR$ is differentiable.
\end{enumerate}
Then the objective is to
\begin{equation}
\label{e:p27}
\minimize{x\in\HH}{f(x)+g(x)}. 
\end{equation}
\end{problem}

\begin{problem}
\label{prob:50}
In Problem~\ref{prob:27}, let $C$ be a nonempty closed convex
subset of $\HH$ and set $f=\iota_C$. Suppose that one of the
following holds:
\begin{enumerate}
\item
\label{prob:27iC}
$0\in\sri(C-\dom g)$.
\item
\label{prob:27iiC}
$g\colon\HH\to\RR$ is differentiable.
\end{enumerate}
Then the objective is to
\begin{equation}
\label{e:p50}
\minimize{x\in C}{g(x)}. 
\end{equation}
\end{problem}

\subsection{Composite problems}
\label{sec:cp}

We start by presenting a duality framework for monotone inclusions 
introduced in \cite{Penn00,Robi99,Robi01} (see 
\cite{Aldu05,Atto96,Ecks99,Fuku96,Gaba83,Merc80,Mosc72,Robi98} for 
special cases).

\begin{problem}
\label{prob:3}
Let $A\colon\HH\to 2^{\HH}$ and $B\colon\GG\to 2^{\GG}$ be
maximally monotone, and let $L\in\BL(\HH,\GG)$. The objective is to
solve the primal inclusion
\begin{equation}
\label{e:p3}
\text{find}\;\;x\in\HH\;\;\text{such that}\;\;
0\in Ax+L^*\big(B(Lx)\big)
\end{equation}
together with the dual inclusion
\begin{equation}
\label{e:d3}
\text{find}\;\;y^*\in\GG\;\;\text{such that}\;\;
0\in -L\bigl(A^{-1}(-L^*y^*)\bigr)+B^{-1}y^*.
\end{equation}
\end{problem}

\begin{lemma}[{\protect{\cite[Propositions~2.7 and 2.8]%
{Siop11}}}]
\label{l:31z}
In the setting of Problem~\ref{prob:3}, 
let $\XXX=\HH\oplus\GG$, let $Z$ and $Z^*$ be the sets of solutions
to \eqref{e:p3} and \eqref{e:d3}, respectively, and set
\begin{equation}
\label{e:31z}
\begin{cases}
\boldsymbol{M}\colon\XXX\to 2^{\XXX}\colon(x,y^*)\mapsto
Ax\,\times\,B^{-1}y^*\\
\boldsymbol{S}\colon\XXX\to\XXX\colon(x,y^*)\mapsto(L^*y^*,-Lx).
\end{cases}
\end{equation}
Define the \emph{Kuhn--Tucker operator} of Problem~\ref{prob:3} as 
\begin{equation}
\label{e:kt3}
\kut=\boldsymbol{M}+\boldsymbol{S} 
\end{equation}
and the set of \emph{Kuhn--Tucker points} as $\zer\kut$.
Then the following hold:
\begin{enumerate}
\item 
\label{l:31zi--} 
$\boldsymbol{M}$ is maximally monotone.
\item 
\label{l:31zi-} 
$\boldsymbol{S}\in\BL(\XXX)$ is skew and maximally monotone, with
$\|\boldsymbol{S}\|=\|L\|$.
\item 
\label{l:31zi} 
$\kut$ is maximally monotone.
\item 
\label{l:31zii} 
$\zer\kut$ is a closed convex subset of $Z\times Z^*$
in $\XXX$.
\item 
\label{l:31ziv} 
(see also \cite{Ecks99,Penn00,Robi99}) 
$Z\neq\emp$ $\Leftrightarrow$ $\zer\kut\neq\emp$
$\Leftrightarrow$ $Z^*\neq\emp$.
\end{enumerate}
\end{lemma}

The best known instance for Problem~\ref{prob:3} is the classical
Fenchel--Rockafellar duality framework \cite{Rock67}.

\begin{problem}
\label{prob:35}
Let $f\in\Gamma_0(\HH)$, $g\in\Gamma_0(\GG)$, and 
$L\in\BL(\HH,\GG)$ be such that
\begin{equation}
\label{e:p31}
0\in\sri\bigl(L(\dom f)-\dom g\bigr).
\end{equation}
Set $A=\partial f$ and $B=\partial g$ in Problem~\ref{prob:3}. Then
it follows from Lemma~\ref{l:22} that \eqref{e:p3} is the primal 
problem
\begin{equation}
\label{e:p32}
\minimize{x\in\HH}{f(x)+g(Lx)},
\end{equation}
\eqref{e:d3} is the \emph{Fenchel--Rockafellar dual} problem 
\begin{equation}
\label{e:d33}
\minimize{y^*\in\GG}{f^*(-L^*y^*)+g^*(y^*)},
\end{equation}
and \eqref{e:kt3} yields the Kuhn--Tucker operator
\begin{equation}
\label{e:kt32}
\kut\colon(x,y^*)\mapsto
\bigl(\partial f(x)+L^*y^*\bigr)\times
\bigl(-Lx+\partial g^*(y^*)\bigr).
\end{equation}
\end{problem}

\begin{problem}
\label{prob:2}
Let $V$ be a closed vector subspace of $\HH$ and let
$A\colon\HH\to 2^{\HH}$ be maximally monotone. Then, in the case
when $\GG=\HH$ and $L=\Id$, the Kuhn--Tucker operator 
\eqref{e:kt3} associated with the operators $N_V$ and $A$ is
\begin{equation}
\label{e:kut5}
\kut\colon\HH\oplus\HH\to 2^{\HH\oplus\HH}\colon
(x,x^*)\mapsto\bigl(N_Vx+x^*\bigr)\times\big(A^{-1}x^*-x\bigr).
\end{equation}
In view of Example~\ref{ex:2V}, the problem of finding a zero of
the maximally monotone operator $\kut$ reduces to
\begin{equation}
\label{e:27f}
\text{find}\;\;x\in V\;\;\text{and}\;\;x^*\in V^\bot\;\;
\text{such that}\;\;x^*\in Ax.
\end{equation}
This formulation was first considered by Spingarn in \cite{Spin83}.
\end{problem}

An extension of Problem~\ref{prob:3} involving several linearly 
composed terms is the following.

\begin{problem}
\label{prob:9}
Let $0<p\in\NN$, let $A\colon\HH\to 2^{\HH}$ be maximally 
monotone, and, for every $k\in\{1\ldots,p\}$, let $\GG_k$ be a real
Hilbert space, let $B_k\colon\GG_k\to 2^{\GG_k}$ be maximally 
monotone, and let $L_k\in\BL(\HH,\GG_k)$. The objective is to 
solve the primal inclusion
\begin{equation}
\label{e:p9}
\text{find}\;\;x\in\HH\;\;\text{such that}\;\;
0\in Ax+\sum_{k=1}^pL_k^*\big(B_k(L_kx)\big)
\end{equation}
together with the dual inclusion
\begin{multline}
\label{e:d9}
\text{find}\;\;y^*_1\in\GG_1,\ldots,y^*_p\in\GG_p\;\;
\text{such that}\\
\biggl(\exi x\in A^{-1}\biggl(-\Sum_{k=1}^pL_k^*y_k^*\biggr)
\biggr)\bigl(\forall k\in\{1,\ldots,p\}\bigr)\;\;
L_kx\in B_k^{-1}y^*_k.
\end{multline}
\end{problem}

\begin{lemma}
\label{l:37z}
In the setting of Problem~\ref{prob:9}, set 
$\XXX=\HH\oplus\GG_1\oplus\cdots\oplus\GG_p$ and let $Z$ and 
$Z^*$ be the sets of solutions to \eqref{e:p9} and \eqref{e:d9}, 
respectively. Define the \emph{Kuhn--Tucker operator} of
Problem~\ref{prob:9} as 
\begin{multline}
\label{e:kt37}
\kut\colon\XXX\to 2^{\XXX}\colon
(x,y^*_1,\ldots,y^*_p)\mapsto\\
\biggl(Ax+\sum_{k=1}^pL_k^*y_k^*\biggr)\times
\bigl(-L_1x+B_1^{-1}y^*_1\bigr)\times\cdots\times
\bigl(-L_px+B_p^{-1}y^*_p\bigr)
\end{multline}
and the set of \emph{Kuhn--Tucker points} as $\zer\kut$.
Then the following hold:
\begin{enumerate}
\item 
\label{l:37zi} 
$\kut$ is maximally monotone.
\item 
\label{l:37zii} 
$\zer\kut$ is a closed convex subset of $Z\times Z^*$
in $\XXX$.
\item 
\label{l:37ziii} 
$Z\neq\emp$ $\Leftrightarrow$ 
$\zer\kut\neq\emp$ $\Leftrightarrow$ $Z^*\neq\emp$.
\end{enumerate}
\end{lemma}
\begin{proof}
Similar to that of Lemma~\ref{l:31z}.
\end{proof}

An alternative angle on Problem~\ref{prob:35} is provided by the 
Lagrangian approach of Example~\ref{ex:pjl}.
Set $\boldsymbol{f}\colon\HH\oplus\GG\to\RX\colon
\boldsymbol{x}=(x,y)
\mapsto f(x)+g(y)$, $\boldsymbol{L}\colon\HH\oplus\GG
\to\GG\colon(x,y)\mapsto Lx-y$, and
$\XXX=\HH\oplus\GG\oplus\GG$.
Then the primal problem \eqref{e:p32} is equivalent to 
\begin{equation}
\minimize{\boldsymbol{x}\in\ker\boldsymbol{L}}
{\boldsymbol{f}(\boldsymbol{x})}
\end{equation}
and a standard perturbation function for it is 
\cite[Example~4']{Rock74} (see also 
\cite[Proposition~19.21]{Livre1}) 
\begin{equation}
\boldsymbol{F}\colon\XXX\to\RX
\colon(\boldsymbol{x},v)\mapsto
\boldsymbol{f}(\boldsymbol{x})+\iota_{\{0\}}(\boldsymbol{Lx}+v).
\end{equation}
We derive from \eqref{e:L} that the associated Lagrangian is
\begin{equation}
\label{e:Lr}
\mathscr{L}_{\boldsymbol{F}}\colon\XXX\to\RX\colon 
(\boldsymbol{x},v^*)\mapsto\boldsymbol{f}(\boldsymbol{x})
+\scal{\boldsymbol{Lx}}{v^*},
\end{equation}
from \eqref{e:71d} that the associated dual problem is
\eqref{e:d33}, and from \eqref{e:K} that the associated
saddle operator is
\begin{equation}
\sad_{\boldsymbol{F}}\colon\XXX\to 2^{\XXX}
\colon (\boldsymbol{x},v^*)\mapsto
(\partial\boldsymbol{f}(\boldsymbol{x})+\boldsymbol{L}^*v^*)
\times\{-\boldsymbol{Lx}\}, 
\end{equation}
i.e., 
\begin{equation}
\label{e:Kr}
\begin{array}{ccll}
\sad_{\boldsymbol{F}}\colon&\XXX&\to&2^{\XXX}\\
&(x,y,v^*)&\mapsto&
\bigl(\partial f(x)+L^*v^*\bigr)\times
\bigl(\partial g(y)-v^*\bigr)\times\{-Lx+y\}.
\end{array}
\end{equation}
We saw in Example~\ref{ex:pjl} that, if 
$(x,y,v^*)\in\zer\sad_{\boldsymbol{F}}$, then $x$ solves the primal
problem \eqref{e:p32} and $v^*$ solves the dual problem
\eqref{e:d33}. A version of this result for Problem~\ref{prob:3}
is the following where, although there is no notion of a
Lagrangian, we can introduce a saddle operator.

\begin{lemma}
\label{l:lim1}
In the setting of Problem~\ref{prob:3}, set
$\XXX=\HH\oplus\GG\oplus\GG$ and let $Z$ and 
$Z^*$ be the sets of solutions to \eqref{e:p3} and \eqref{e:d3}, 
respectively. Define the Kuhn--Tucker operator $\kut$ as in 
\eqref{e:kt3} and define the \emph{saddle operator} of
Problem~\ref{prob:3} as 
\begin{equation}
\label{e:lim1}
\begin{array}{ccll}
\sad\colon&\XXX&\to&2^{\XXX}\\
&(x,y,v^*)&\mapsto&(Ax+L^*v^*)\times(By-v^*)\times\{-Lx+y\}.
\end{array}
\end{equation}
Then the following hold:
\begin{enumerate}
\item
\label{l:lim1i}
$\sad$ is maximally monotone.
\item
\label{l:lim1i+}
$\zer\sad$ is closed and convex.
\item
\label{l:lim1ii}
Suppose that $
(x,y,v^*)\in\zer\sad$. Then 
$(x,v^*)\in\zer\kut\subset Z\times Z^*$.
\item
\label{l:lim1iv}
$Z^*\neq\emp$ $\Leftrightarrow$
$\zer\sad\neq\emp$ $\Leftrightarrow$ $\zer\kut\neq\emp$
$\Leftrightarrow$ $Z\neq\emp$.
\end{enumerate}
\end{lemma}
\begin{proof}
A special case of \cite[Proposition~1(i)--(v)(a)]{Moor22}.
\end{proof}

\subsection{Examples of embeddings in Framework~\ref{f:1}}

\begin{example}
\label{ex:f0}
Suppose that it is computationally feasible solve 
Problem~\ref{prob:1} directly in the original space $\HH$. Then 
an embedding of Problem~\ref{prob:1} is just $(\HH,M,\Id)$.
\end{example}

\begin{example}
\label{ex:f9}
Let $M\colon\HH\to 2^{\HH}$ be a maximally monotone operator,
let $U\in\BL(\HH)$ be a self-adjoint strongly monotone operator, 
let $\XXX$ be the real Hilbert space obtained by endowing $\HH$ 
with the scalar product $(x,y)\mapsto\scal{Ux}{y}$, 
let $\MMM=U^{-1}\circ M$, and set $\TTT=\Id$. Then it follows from
Lemma~\ref{l:9}\ref{l:9i-}--\ref{l:9i} that $(\XXX,\MMM,\TTT)$ is
an embedding of Problem~\ref{prob:1}.
\end{example}

\begin{example}
\label{ex:f12}
Let $\alpha\in\rzeroun$ and let $T\colon\HH\to\HH$ be
$\alpha$-averaged. In Problem~\ref{prob:1}, suppose that
$M=\Id-T$ (see Example~\ref{ex:12}) and set 
\begin{equation}
\label{e:M1}
\XXX=\HH,\;\;
\MMM=\bigg(\Id+\dfrac{1}{2\alpha}(T-\Id)\bigg)^{-1}-\Id,\;\;
\text{and}\;\;\TTT=\Id.
\end{equation}
Then $(\XXX,\MMM,\TTT)$ is an embedding of Problem~\ref{prob:1}.
Indeed, since $\Id+\alpha^{-1}(T-\Id)$ is nonexpansive, we derive
from \cite[Proposition~4.4]{Livre1} that 
$\Id+(2\alpha)^{-1}(T-\Id)$ is firmly nonexpansive and hence 
from Lemma~\ref{l:0}\ref{l:0iii} that $\MMM$ is maximally 
monotone, with $\zer\MMM=\zer M=\Fix T$.
\end{example}

\begin{example}
\label{ex:f11}
Let $A\colon\HH\to 2^{\HH}$ and $B\colon\HH\to 2^{\HH}$ be 
maximally monotone, and let $\gamma\in\RPP$. Let 
\begin{equation}
\label{e:1968}
\XXX=\HH,\;\;
\MMM=\big(J_{\gamma A}\circ(2J_{\gamma B}-\Id)+\Id-J_{\gamma B}
\big)^{-1}-\Id,\;\;\text{and}\;\;\TTT=J_{\gamma B}.
\end{equation}
Then it follows from \cite[Section~4]{Ecks92} that 
$(\XXX,\MMM,\TTT)$ is an embedding of Problem~\ref{prob:0}.
In this setting, we actually have $\TTT(\zer\MMM)=\zer M$
\cite[Lemma~2.6(iii)]{Opti04}.
\end{example}

\begin{example}
\label{ex:f1}
Let $A\colon\HH\to 2^{\HH}$ and $B\colon\HH\to 2^{\HH}$ be 
maximally monotone. Let $\XXX=\HH\oplus\HH$,
$\MMM\colon\XXX\to 2^{\XXX}\colon (x,x^*)\mapsto
(Ax+x^*)\times(-x+B^{-1}x^*)$, and 
$\TTT\colon\XXX\to\HH\colon(x,x^*)\mapsto x$. Then applying
Lemma~\ref{l:31z} with $\GG=\HH$ and $L=\Id$ shows that
$(\XXX,\MMM,\TTT)$ is an embedding of Problem~\ref{prob:0}. 
This embedding is implicitly present in the projective splitting
algorithm of \cite{Svai08}, which is therefore an instance of
Framework~\ref{f:1}. 
\end{example}

We now discuss structured inclusion problems that offer greater
modeling flexibility by involving three or more operators. The
principle of a splitting algorithm, which is to involve each
operator individually, faces a serious challenge in the presence of
such formulations. Indeed, since inclusion is a binary relation,
for reasons discussed in \cite{Siop11,Play13} and analyzed in more
depth in \cite{Ryue20}, it is not possible to split problems that
involve more than two set-valued operators. A purpose of
Framework~\ref{f:1} is to circumvent this fundamental limitation by
seeking more tractable reformulations in bigger spaces. 

\begin{example}
\label{ex:f2}
Let $0<p\in\NN$ and, for every $k\in\{1,\ldots,p\}$, let 
$A_k\colon\HH\to 2^{\HH}$ be maximally monotone. The problem is to 
\begin{equation}
\label{e:f1}
\text{find}\;\;x\in\HH\;\;\text{such that}\;\;
0\in\sum_{k=1}^pA_kx.
\end{equation}
Let $\XXX$ be the $p$-fold Hilbert direct sum $\HH^p$ and set
\begin{equation}
\begin{cases}
\boldsymbol{V}=\menge{(x_1,\ldots,x_p)\in\XXX}{x_1=\cdots=x_p}\\
\boldsymbol{A}\colon\XXX\to 2^{\XXX}\colon
(x_1,\ldots,x_p)\mapsto A_1x_1\times\cdots\times A_px_p\\
\MMM=\boldsymbol{A}+N_{\boldsymbol{V}}\\
\TTT\colon\XXX\to\HH\colon(x_1,\ldots,x_p)\mapsto x_1.
\end{cases}
\end{equation}
Then 
\begin{equation}
\boldsymbol{V}^\bot=\Menge{(x^*_1,\ldots,x^*_p)\in\XXX}
{\sum_{k=1}^px^*_k=0} 
\end{equation}
and it follows from Example~\ref{ex:2V} that
$(\XXX,\MMM,\TTT)$ is an embedding of \eqref{e:f1}. This setting to
split the sum of $p>2$ monotone operators was introduced by
Spingarn in \cite[Section~5]{Spin83} (see also \cite{Gols87}). 
It reduces the $p$-operator problem \eqref{e:f1} to the
two-operator inclusion $\boldsymbol{0}\in\boldsymbol{A}
\boldsymbol{x}+N_{\boldsymbol{V}}\boldsymbol{x}$.
The idea of rephrasing multi-operator problems in product spaces 
finds its roots in convex feasibility problems 
\cite{Pier76,Pier84}, where the problem of finding a point in the
intersection $\bigcap_{k=1}^pC_k$ of closed convex subsets 
$(C_k)_{1\leq k\leq p}$ of $\HH$ is associated with that of 
finding a point in 
$\boldsymbol{C}\cap \boldsymbol{V}$ in $\XXX$, where
$\boldsymbol{C}=C_1\times\cdots\times C_p$.
\end{example}

\begin{example}
\label{ex:f3}
In the setting of Problem~\ref{prob:3}, set $\XXX=\HH\oplus\GG$, 
define $\boldsymbol{M}$ and $\boldsymbol{S}$ as in \eqref{e:31z},
let $\kut=\boldsymbol{M}+\boldsymbol{S}$ be the Kuhn--Tucker
operator of \eqref{e:kt3}, and let
$\TTT\colon\XXX\to\HH\colon(x,y^*)\mapsto x$. Then, in view of
Lemma~\ref{l:31z}\ref{l:31zii}, $(\XXX,\kut,\TTT)$ is an embedding
of \eqref{e:p3}. This embedding, which underlies the
\emph{monotone+skew} framework of \cite{Siop11}, reduces
Problem~\ref{prob:3}, which involves three operators in the primal
space $\HH$ (namely, $A$, $B$, and $L$), to a problem in $\XXX$
that involves the two operators $\boldsymbol{M}$ and
$\boldsymbol{S}$.
\end{example}

\begin{example}
\label{ex:f4}
In the setting of Problem~\ref{prob:9}, set 
$\XXX=\HH\oplus\GG_1\oplus\cdots\oplus\GG_p$, let $\kut$ be the
Kuhn--Tucker operator of \eqref{e:kt37}, and let
\begin{equation}
\TTT\colon\XXX\to\HH\colon(x,y^*_1,\ldots,y^*_p)\mapsto x. 
\end{equation}
Then
it follows from Lemma~\ref{l:37z}\ref{l:37zii} that
$(\XXX,\kut,\TTT)$ is an embedding of \eqref{e:p9}.
\end{example}

Next, we consider an embedding for strongly monotone problems.

\begin{example}
\label{ex:f13}
Let $\rho\in\RPP$, let $0<p\in\NN$, let $z\in\HH$, and let 
$A\colon\HH\to 2^{\HH}$ be maximally monotone. For every
$k\in\{1,\ldots,p\}$, let $B_k\colon\GG_k\to 2^{\GG_k}$ 
and $D_k\colon\GG_k\to 2^{\GG_k}$ be maximally monotone, and
suppose that $0\neq L_k\in\BL(\HH,\GG_k)$. The problem is to
\begin{equation}
\label{e:96p}
\text{find}\;\;x\in\HH\;\;\text{such that}\;\;
z\in Ax+\sum_{k=1}^pL_k^*\big((B_k\infconv D_k)
(L_kx)\big)+\rho x.
\end{equation}
Let $\XXX=\GG_1\oplus\cdots\oplus\GG_p$, let 
\[
\begin{array}{ccll}
\!\!\!\!\!\MMM\colon
&\!\!\XXX&\!\!\to&\!\!2^{\XXX}\\
&\!\!(y^*_1,\ldots,y^*_p)&\!\!\mapsto&\!\!
\Biggl(-L_1\bigg(J_{A/\rho}\bigg(\dfrac{1}{\rho}\bigg(z-
\displaystyle{\sum_{k=1}^p}L_k^*{y^*_k}\bigg)\bigg)\bigg)+
B_1^{-1}{y^*_1}+D_1^{-1}{y^*_1}\Biggr)\\
\end{array}
\]
\begin{equation}
\label{e:96d}
\hspace{17mm}\times\cdots\times
\Biggl(-L_p\bigg(J_{A/\rho}\bigg(\dfrac{1}{\rho}\bigg(z-
\sum_{k=1}^pL_k^*{y^*_k}\bigg)\bigg)\bigg)+
B_p^{-1}{y^*_p}+D_p^{-1}{y^*_p}\Biggr),
\end{equation}
and let
\begin{equation}
\TTT\colon\XXX\to\HH\colon(y^*_1,\ldots,y^*_p)\mapsto 
J_{A/\rho}\bigg(\dfrac{1}{\rho}\bigg(z-
\sum_{k=1}^pL_k^*{y^*_k}\bigg)\bigg). 
\end{equation} 
Then it follows from \cite[Proposition~5.2(iii)]{Opti14} that
$(\XXX,\MMM,\TTT)$ is an embedding of \eqref{e:96p}.
\end{example}

Our last example concerns an embedding based on a saddle operator.

\begin{example}
\label{ex:f5}
In the setting of Problem~\ref{prob:3}, set 
$\XXX=\HH\oplus\GG\oplus\GG$, let $\sad$ be the saddle 
operator of \eqref{e:lim1}, and let
$\TTT\colon\XXX\to\HH\colon(x,y,v^*)\mapsto x$. Then it follows
from Lemma~\ref{l:lim1}\ref{l:lim1ii} that $(\XXX,\sad,\TTT)$ is an
embedding of \eqref{e:p3}.
\end{example}

Additional examples of embeddings will be provided by
Examples~\ref{ex:f6}, \ref{ex:f7}, and \ref{ex:f8}.

\section{Two geometric convergence principles}
\label{sec:4}

\subsection{Overview}

The methodology of Framework~\ref{f:1} is to identify a
target set $Z$ in a suitable Hilbert space in such a way that every
point in $Z$ yields a solution to the original problem of interest.
The algorithms we shall consider are Fej\'erian in the sense that
every iteration brings the current iterate closer to every point in
$Z$.

\subsection{Fej\'er monotone scheme}

Let us first recall some basic facts about weak and strong
convergence in Hilbert spaces.

\begin{lemma}{\rm\cite[Section~2.5]{Livre1}}
\label{l:1}
Let $(x_n)_{n\in\NN}$ be a sequence in $\HH$ and let $x\in\HH$. 
Then the following hold:
\begin{enumerate}
\item
\label{l:1i}
Let $Z$ be a nonempty subset of $\HH$. Suppose that
$\WC(x_n)_{n\in\NN}\subset Z$ and that, for every $z\in Z$, 
$(\|x_n-z\|)_{n\in\NN}$ converges. Then $(x_n)_{n\in\NN}$ converges
weakly to a point in $Z$.
\item
\label{l:1ii}
$x_n\weakly x$ $\;\Leftrightarrow\;$ 
$\big[\,(x_n)_{n\in\NN}$ is bounded and
$\WC(x_n)_{n\in\NN}=\{x\}\,\big]$.
\item
\label{l:1iii}
$x_n\to x$ $\;\Leftrightarrow\;$ $\big[\,x_n\weakly x$ and 
$\;\varlimsup\|x_n\|\leq\|x\|\,\big]$.
\end{enumerate}
\end{lemma}

\begin{theorem}
\label{t:1}
Let $Z$ be a nonempty closed convex subset of $\HH$, let 
$(\lambda_n)_{n\in\NN}$ be a sequence of relaxation parameters
in $\left]0,2\right[$, and let $x_0\in\HH$. Iterate
(see Figure~\ref{fig:1})
\begin{equation}
\label{e:5}
\begin{array}{l}
\text{for}\;n=0,1,\ldots\\
\left\lfloor
\begin{array}{l}
H_n\;\text{is a closed half-space such that}\;Z\subset H_n\\
p_n=\proj_{H_n}x_n\\
x_{n+1}=x_n+\lambda_n(p_n-x_n).
\end{array}
\right.\\
\end{array}
\end{equation}
Then the following hold:
\begin{enumerate}
\item
\label{t:1i}
Fej\'er monotonicity:
$(\forall z\in Z)(\forall n\in\NN)$ $\|x_{n+1}-z\|\leq\|x_n-z\|$.
\item
\label{t:1ii}
$\sum_{n\in\NN}\lambda_n(2-\lambda_n)\|p_n-x_n\|^2<\pinf$.
\item
\label{t:1iii}
Suppose that $\sup_{n\in\NN}\lambda_n<2$. Then
$\sum_{n\in\NN}\|x_{n+1}-x_n\|^2<\pinf$.
\item
\label{t:1iv}
Suppose that $\WC(x_n)_{n\in\NN}\subset Z$. Then $(x_n)_{n\in\NN}$
converges weakly to a point in $Z$.
\end{enumerate}
\end{theorem}
\begin{proof}
Let $z\in Z$. Then, for every $n\in\NN$,
$H_n=\menge{u\in\HH}{\scal{u-p_n}{x_n-p_n}\leq 0}$ and, since 
$z\in H_n$, \eqref{e:5} yields
\begin{align}
\hskip -6mm \|x_{n+1}-z\|^2
&=\|x_n-z\|^2+2\lambda_n\scal{x_n-z}{p_n-x_n}+
\lambda_n^2\|p_n-x_n\|^2
\nonumber\\
&=\|x_n-z\|^2-\lambda_n(2-\lambda_n)\|p_n-x_n\|^2
+2\lambda_n\scal{z-p_n}{x_n-p_n}
\nonumber\\
&\leq\|x_n-z\|^2-\lambda_n(2-\lambda_n)\|p_n-x_n\|^2
\label{e:65}\\
&=\|x_n-z\|^2-\dfrac{2-\lambda_n}{\lambda_n}\|x_{n+1}-x_n\|^2
\label{e:66}\\
&\leq\|x_n-z\|^2.
\label{e:64}
\end{align}

\ref{t:1i}: See \eqref{e:64}.

\ref{t:1ii}: Fix $N\in\NN$. Then \eqref{e:65} yields
\begin{equation}
\sum_{n=0}^N\lambda_n(2-\lambda_n)\|p_n-x_n\|^2
\leq\|x_0-z\|^2
\end{equation}
and we conclude by letting $N\to\pinf$.

\ref{t:1ii}$\Rightarrow$\ref{t:1iii}: This follows from
\eqref{e:66}.

\ref{t:1iv}: In view of \ref{t:1i}, $(\|x_n-z\|)_{n\in\NN}$
converges. The claim therefore follows from 
Lemma~\ref{l:1}\ref{l:1i}.
\end{proof}

\begin{remark}
\label{r:lf}
In 1922, Fej\'er \cite{Feje22} studied the following problem: given
a nonempty closed set $Z\subset\RR^N$ and a point $y\notin Z$, can
one find a point $x\in\RR^N$ such that
\begin{equation}
\label{e:lf}
(\forall z\in Z)\quad\|x- z\|<\|y- z\|.
\end{equation}
This led Motzkin and Schoenberg to adopt in \cite{Motz54} the
terminology \emph{Fej\'er monotone} to describe sequences
satisfying property \ref{t:1i} in Theorem~\ref{t:1}. In their paper
(see also \cite{Agmo54}), an algorithm was developed to solve
systems of linear inequalities in $\RR^N$ by successive projections
onto the half-spaces defining the polyhedral solution set $Z$, and
Fej\'er monotonicity was shown to be an adequate tool to study
the convergence of this algorithm. Further analysis of Fej\'er
monotonicity was proposed in
\cite{Breg65,Ere68a,Ere68b,Raik67,Raik69} and nowadays it
constitutes a central tool to analyze the asymptotic behavior of
various algorithms \cite{Livre1}. 
\end{remark}

\begin{remark}
\label{r:42}
In general, the convergence of $(x_n)_{n\in\NN}$ to $x\in Z$ in
Theorem~\ref{t:1}\ref{t:1iv} is only weak and, even if it were 
strong, there exists no rate of convergence on
$(\|x_n-x\|)_{n\in\NN}$, even in Euclidean spaces 
\cite{Baus09,Gubi67,Youl87}. In particular, achieving a linear rate
of convergence, that is, securing the existence of $\kappa\in\RPP$
and $\rho\in\zeroun$ such that
\begin{equation}
\label{e:4t}
(\forall n\in\NN)\quad\|x_n-x\|\leq\kappa\rho^n,
\end{equation}
requires stringent additional assumptions on the problem. In our
inclusion context, a typical assumption is strong monotonicity; 
see \cite[Proposition~26.16]{Livre1} for an example. In the broader
context of Theorem~\ref{t:1}\ref{t:1i}, it is clear that 
$(d_C(x_n))_{n\in\NN}$ decreases and that, for every $n\in\NN$ and
$m\in\NN$, $\|x_n-x_{n+m}\|\leq\|x_n-\proj_Cx_n\|+
\|x_{n+m}-\proj_Cx_n\|\leq 2d_C(x_n)$. Hence, 
\eqref{e:4t} will hold with $\kappa=2d_C(x_0)$ if
the decreasing property can be strengthened to
$(\forall n\in\NN) $ $d_C(x_{n+1})\leq\rho d_C(x_n)$.
\end{remark}

\begin{remark}
The implementation of \eqref{e:5} is said to be \emph{unrelaxed}
if $(\forall n\in\NN)$ $\lambda_n=1$.
\end{remark}

\begin{figure}
\begin{center}
\includegraphics[width=110mm]{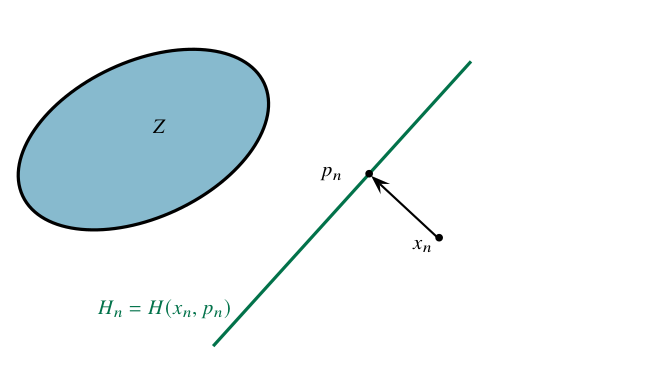}
\caption{Iteration $n$ of the Fej\'erian algorithm \eqref{e:5}.
}
\label{fig:1}
\end{center}
\end{figure}

\subsection{Haugazeau-like scheme}

Theorem~\ref{t:1} guarantees only weak convergence to an
unspecified point in $Z$ and, as will be seen on several occasions
later, strong convergence fails in general (many of these examples
will be based on a scenario of \cite{Hund04} concerning the method
of alternating projections). However, in some infinite-dimensional
applications in areas such as inverse problems, control, mechanics,
PDEs, optics, and analog computing, weak convergence does not offer
sufficient guarantees and strong convergence is required. The
geometric approach described in this section emanates from ideas
found in the work of Haugazeau on the convex feasibility problem
\cite{Haug67,Haug68}. It will provide strong convergence to a
specific point in $Z$, namely the projection of the initial point
onto $Z$. This means that the resulting algorithm is also of
interest, even in Euclidean spaces, as a best approximation method.

The following technical fact will be employed repeatedly.

\begin{lemma}{\rm(\cite[Th\'eor\`eme~3-1]{Haug68}; see also
\cite[Corollary~29.25]{Livre1})}
\label{l:2}
Let $(x,y,z)\in\HH^3$. Define
\begin{equation}
\label{e:n1}
H(x,y)=\menge{z\in\HH}{\scal{z-y}{x-y}\leq 0},
\end{equation}
$C=H(x,y)\cap H(y,z)$, and, if $C\neq\emp$, 
\begin{equation}
\label{e:01-05}
\Qq(x,y,z)=\proj_Cx.
\end{equation}
Set $\chi=\scal{x-y}{y-z}$, $\mu=\|x-y\|^2$,
$\nu=\|y-z\|^2$, and $\rho=\mu\nu-\chi^2$.
Then exactly one of the following holds:
\begin{enumerate}
\item
\label{l:2i}
$\rho=0$ and $\chi<0$, in which case $C=\emp$.
\item 
\label{l:2ii}
\rm{[}$\,\rho=0$ and $\chi\geq 0\,$\rm{]} or 
$\rho>0$, in which case $C\neq\emp$ and 
\begin{equation}
\label{e:Q}
\hskip -12mm \Qq(x,y,z)=
\begin{cases}
z,&\text{if}\;\rho=0\;\text{and}\;
\chi\geq 0;\\[+0mm]
\displaystyle
x+(1+\chi/\nu)(z-y), 
&\text{if}\;\rho>0\;\text{and}\;
\chi\nu\geq\rho;\\
\displaystyle y+(\nu/\rho)
\big(\chi(x-y)+\mu(z-y)\big), 
&\text{if}\;\rho>0\;\text{and}\;\chi\nu<\rho.
\end{cases}
\end{equation}
\end{enumerate}
\end{lemma}

The essential components of the following theorem are found in the
unpublished thesis of Haugazeau \cite{Haug68} (see \cite{Haug67}
for a preliminary variant), where he considered the specific
problem of projecting a point onto the intersection of finitely
many sets using their individual projection operators cyclically.

\begin{theorem}
\label{t:2}
Let $Z$ be a nonempty closed convex subset of $\HH$, let 
$(\lambda_n)_{n\in\NN}$ be a sequence of relaxation parameters
in $\rzeroun$, and let $x_0\in\HH$. Iterate 
(see Figure~\ref{fig:2})
\begin{equation}
\label{e:6}
\begin{array}{l}
\text{for}\;n=0,1,\ldots\\
\left\lfloor
\begin{array}{l}
H_n\;\text{is a closed half-space such that}\;Z\subset H_n\\
p_n=\proj_{H_n}x_n\\
r_n=x_n+\lambda_n(p_n-x_n)\\
x_{n+1}=\Qq(x_0,x_n,r_n).
\end{array}
\right.\\
\end{array}
\end{equation}
Then the sequence $(x_n)_{n\in\NN}$ is well defined and the
following hold:
\begin{enumerate}
\item
\label{t:2i}
$(\forall n\in\NN)$ $Z\subset H(x_0,x_n)\cap H(x_n,r_n)$.
\item
\label{t:2ii}
$(\exi\ell\in\RP)$ $\|x_n-x_0\|\uparrow\ell\leq d_Z(x_0)$.
\item
\label{t:2iii}
$\sum_{n\in\NN}\|x_{n+1}-x_n\|^2<\pinf$.
\item
\label{t:2iv}
$\sum_{n\in\NN}\lambda_n^2\|p_n-x_n\|^2<\pinf$.
\item
\label{t:2v}
Suppose that $\WC(x_n)_{n\in\NN}\subset Z$. Then $(x_n)_{n\in\NN}$
converges strongly to $\proj_Zx_0$.
\end{enumerate}
\end{theorem}
\begin{proof}
First, recall that the projector onto a nonempty closed convex 
subset $D$ of $\HH$ is characterized by \cite[Theorem~3.16]{Livre1}
\begin{equation}
\label{e:2001}
(\forall x\in\HH)\quad\proj_Dx\in D\quad\text{and}\quad 
D\subset H(x,\proj_Dx).
\end{equation}
We also observe that \eqref{e:6} implies that 
\begin{align}
\label{e:2000}
&\hskip -5mm (\forall n\in\NN)\quad H(x_n,p_n)\nonumber\\
&\hskip 15mm=\menge{z\in\HH}{\scal{z-p_n}{x_n-r_n}\leq 0}
\nonumber\\
&\hskip 15mm=\menge{z\in\HH}{\scal{z-r_n}{x_n-r_n}\leq
\scal{p_n-r_n}{x_n-r_n}}\nonumber\\
&\hskip 15mm=\menge{z\in\HH}{\scal{z-r_n}{x_n-r_n}\leq
-\lambda_n(1-\lambda_n)\|x_n-p_n\|^2}\nonumber\\
&\hskip 15mm\subset H(x_n,r_n).
\end{align}

\ref{t:2i}:
Let $n\in\NN$ be such that $x_n$ exists. It follows from 
\eqref{e:6} and \eqref{e:2000} that 
$Z\subset H_n=H(x_n,p_n)\subset H(x_n,r_n)$. It is therefore enough
to show that $Z\subset H(x_0,x_n)$. This inclusion certainly holds
for $n=0$ since $H(x_0,x_0)=\HH$. Furthermore, it follows from
\eqref{e:2001} and \eqref{e:6} that
\begin{eqnarray}
\label{e:xenon}
Z\subset H(x_0,x_n)
&\Rightarrow&Z\subset 
H(x_0,x_n)\cap 
H(x_n,r_n)\nonumber\\
&\Rightarrow&Z\subset 
H\bigl(x_0,\Qq(x_0,x_n,r_n)\bigr)\nonumber\\
&\Leftrightarrow&Z\subset 
H(x_0,x_{n+1}),
\end{eqnarray}
which establishes the assertion by induction. This also shows that
$H(x_0,x_n)\cap H(x_n,r_n)\neq\emp$ and hence that $x_{n+1}$ is
well defined.

\ref{t:2ii}--\ref{t:2iii}:
Let $n\in\NN$. By construction, 
$x_{n+1}=\Qq(x_0,x_n,r_n)\in H(x_0,x_n)\cap H(x_n,r_n)$. 
Consequently, since $x_n$ is the projection of $x_0$ onto
$H(x_0,x_n)$ and $x_{n+1}\in H(x_0,x_n)$, we have 
$\|x_0-x_n\|\leq\|x_0-x_{n+1}\|$.
On the other hand, since $\proj_{Z}x_0\in Z\subset H(x_0,x_n)$, 
we have $\|x_0-x_n\|\leq\|x_0-\proj_{Z}x_0\|$.
It follows that $(\|x_0-x_k\|)_{k\in\NN}$ converges to some
$\ell\in[0,\|x_0-\proj_{Z}x_0\|]$, which establishes
\ref{t:2ii}, and that
\begin{equation}
\label{e:07g}
\lim\|x_0-x_k\|\leq\|x_0-\proj_{Z}x_0\|.
\end{equation}
However, since $x_{n+1}\in H(x_0,x_n)$, we have
\begin{align}
\|x_{n+1}-x_n\|^2
&\leq\|x_{n+1}-x_n\|^2+2\scal{x_{n+1}-x_n}{x_n-x_0}\nonumber\\
&=\|x_0-x_{n+1}\|^2-\|x_0-x_n\|^2.
\end{align}
Hence, 
\begin{equation}
\sum_{k=0}^n\|x_{k+1}-x_k\|^2
\leq\|x_0-x_{n+1}\|^2\leq\|x_0-\proj_{Z}x_0\|^2
\end{equation}
and therefore 
\begin{equation}
\sum_{k\in\NN}\|x_{k+1}-x_k\|^2<\pinf.
\end{equation}

\ref{t:2iv}:
For every $n\in\NN$, we derive from the inclusion 
$x_{n+1}\in H(x_n,r_n)$ that
\begin{align}
\label{e:07m}
\|r_n-x_n\|^2
&\leq\|x_{n+1}-r_n\|^2+
\|x_n-r_n\|^2\nonumber\\
&\leq\|x_{n+1}-r_n\|^2+
2\scal{x_{n+1}-r_n}{r_n-x_n}+\|x_n-r_n\|^2\nonumber\\
&=\|x_{n+1}-x_n\|^2.
\end{align}
Hence, by \ref{t:2iii} and \eqref{e:6}, 
\begin{equation}
\sum_{n\in\NN}\lambda_n^2\|p_n-x_n\|^2=
\sum_{n\in\NN}\|r_n-x_n\|^2<\pinf. 
\end{equation}

\ref{t:2v}:
Let us note that \ref{t:2ii} implies that
$(x_n)_{n\in\NN}$ is bounded. Now let $x\in\WC(x_n)_{n\in\NN}$,
say $x_{k_n}\weakly x$. Then, by weak lower semicontinuity of 
$\|\cdot\|$ \cite[Lemma~2.42]{Livre1} and \ref{t:2ii},
\begin{equation}
\|x_0-x\|\leq\varliminf\|x_0-x_{k_n}\|\leq\|x_0-\proj_{Z}x_0\|=
\inf_{z\in Z}\|x_0-z\|. 
\end{equation}
Hence, since $x\in Z$, $x=\proj_{Z}x_0$
is the only weak sequential cluster point of 
$(x_n)_{n\in\NN}$ and it follows from Lemma~\ref{l:1}\ref{l:1ii}
that $x_n\weakly\proj_Zx_0$. In turn, \ref{t:2ii} yields 
\begin{equation}
\|x_0-\proj_Zx_0\|\leq\varliminf\|x_0-x_n\|=
\lim\|x_0-x_n\|\leq\|x_0-\proj_Zx_0\|.
\end{equation}
Thus, $x_0-x_n\weakly x_0-\proj_Zx_0$ and 
$\|x_0-x_n\|\to\|x_0-\proj_{Z}x_0\|$.
We therefore derive from Lemma~\ref{l:1}\ref{l:1iii}
that $x_0-x_n\to x_0-\proj_Zx_0$, i.e., $x_n\to\proj_Zx_0$.
\end{proof}

\begin{figure}
\begin{center}
\includegraphics[width=110mm]{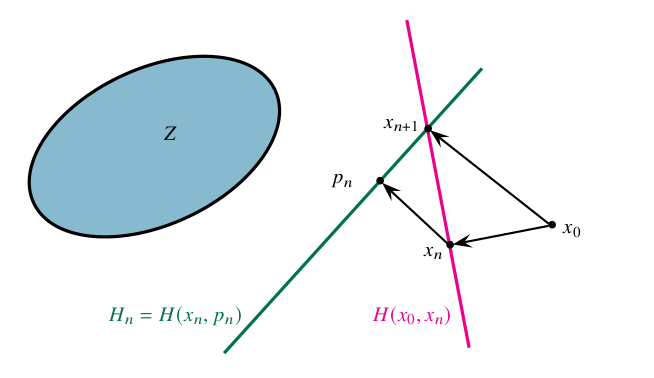}
\caption{Iteration $n$ of the Haugazeau-like algorithm 
\eqref{e:6} with $\lambda_n=1$.
}
\label{fig:2}
\end{center}
\end{figure}

\subsection{Graph-based cuts}
\label{sec:gr}

We consider the problem of finding a zero of a maximally
monotone operator $M\colon\HH\to 2^{\HH}$ decomposed as
$M=W+C$, where $W\colon\HH\to 2^{\HH}$ is maximally monotone and
$C\colon\HH\to\HH$ is cocoercive, using the geometric principles of
Theorems~\ref{t:1} and \ref{t:2}. To this end, we shall construct
half-spaces by selecting points in the graph of $W$. Let us start
with a weak convergence result.

\begin{theorem}
\label{t:1c}
Let $\alpha\in\RPP$, let $W\colon\HH\to 2^{\HH}$ be maximally 
monotone, let $C\colon\HH\to\HH$ be $\alpha$-cocoercive and
such that $Z=\zer(W+C)\neq\emp$, let $x_0\in\HH$, and let
$(\lambda_n)_{n\in\NN}$ be a sequence in $\left]0,2\right[$.
Iterate
\begin{equation}
\label{e:fejer19}
\begin{array}{l}
\text{for}\;n=0,1,\ldots\\
\left\lfloor
\begin{array}{l}
(w_n,w_n^*)\in\gra W,\;q_n\in\HH\\
t_n^*=w_n^*+Cq_n\\
\delta_n=\scal{x_n-w_n}{t_n^*}-\|w_n-q_n\|^2/(4\alpha)\\[2mm]
d_n=
\begin{cases}
\dfrac{\delta_n}
{\|t_n^*\|^2}t_n^*,&\text{if}\:\:\delta_n>0;\\
0,&\text{otherwise}\\
\end{cases}\\
x_{n+1}=x_n-\lambda_nd_n.
\end{array} 
\right. 
\end{array} 
\end{equation}
Then the following hold:
\begin{enumerate}
\item
\label{t:1ci}
$(x_n)_{n\in\NN}$ is bounded.
\item
\label{t:1cii}
$\sum_{n\in\NN}\lambda_n(2-\lambda_n)\|d_n\|^2<\pinf$.
\item
\label{t:1ciii}
Suppose that $w_n-x_n\weakly 0$, $w_n-q_n\to 0$, and $t_n^*\to 0$.
Then $(x_n)_{n\in\NN}$ converges weakly to a point in $Z$.
\end{enumerate}
\end{theorem}
\begin{proof}
We first observe that \eqref{e:fejer19} is well defined since
$(\forall n\in\NN)$ $\delta_n>0$ $\Rightarrow$ $t^*_n\neq 0$.
It follows from Example~\ref{ex:13} and
Lemma~\ref{l:11}\ref{l:11i} that
\begin{equation}
\label{e:z7}
W+C\;\text{is maximally monotone},
\end{equation}
and hence from \eqref{e:zero} that $Z$ is a nonempty closed convex
subset of $\HH$. Set 
\begin{equation}
\label{e:h8}
(\forall n\in\NN)\quad
H_n=\Menge{z\in\HH}{\scal{z-w_n}{t_n^*}\leq
\dfrac{\|w_n-q_n\|^2}{4\alpha}}
\end{equation}
and let $z\in Z$. For every $n\in\NN$, since $(z,-Cz)\in\gra W$ and
$(w_n,w_n^*)\in\gra W$, it results from the monotonicity of $W$ 
that $\scal{w_n-z}{w_n^*+Cz}\geq 0$. Hence, since $C$ is 
$\alpha$-cocoercive,
\begin{align}
&\hskip -3mm
(\forall n\in\NN)\quad
\scal{z-w_n}{t_n^*}\nonumber\\
&\hskip 22mm=\scal{z-w_n}{w_n^*+Cq_n}\nonumber\\
&\hskip 22mm\leq\scal{z-w_n}{Cq_n-Cz}\label{e:c35}\\
&\hskip 22mm=\scal{q_n-w_n}{Cq_n-Cz}+\scal{z-q_n}{Cq_n-Cz}
\nonumber\\
&\hskip 22mm\leq\scal{q_n-w_n}{Cq_n-Cz}-\alpha\|Cq_n-Cz\|^2
\label{e:c33}\\
&\hskip 22mm=2\Scal{\dfrac{q_n-w_n}{\sqrt{4\alpha}}}
{\sqrt{\alpha}(Cq_n-Cz)}-\big\|\sqrt{\alpha}(Cq_n-Cz)\big\|^2
\nonumber\\
&\hskip 22mm=\dfrac{\|w_n-q_n\|^2}{4\alpha}-
\bigg\|\sqrt{\alpha}(Cq_n-Cz)+\dfrac{w_n-q_n}
{\sqrt{4\alpha}}\bigg\|^2
\nonumber\\
&\hskip 22mm\leq\dfrac{\|w_n-q_n\|^2}{4\alpha}.
\label{e:c36}
\end{align}
This shows that $(\forall n\in\NN)$ $Z\subset H_n$. 
In addition, it results from \eqref{e:fejer19} and 
Example~\ref{ex:H} that
\begin{equation}
\label{e:fejer23}
(\forall n\in\NN)\quad
x_{n+1}=x_n+\lambda_n\bigl(\proj_{H_n}x_n-x_n\bigr),
\end{equation}
which corresponds to the setting of Theorem~\ref{t:1}.

\ref{t:1ci}:
This follows from Theorem~\ref{t:1}\ref{t:1i}.

\ref{t:1cii}:
This follows from Theorem~\ref{t:1}\ref{t:1ii}.

\ref{t:1ciii}:
Let $x\in\WC(x_n)_{n\in\NN}$, say $x_{k_n}\weakly x$. Then
$w_{k_n}=x_{k_n}+(w_{k_n}-x_{k_n})\weakly x$. On the other hand,
since $C$ is $1/\alpha$-Lipschitzian, 
\begin{equation}
\|w_n^*+Cw_n\|=\|t_n^*+Cw_n-Cq_n\|
\leq\|t_n^*\|+\dfrac{\|w_n-q_n\|}{\alpha}\to 0. 
\end{equation}
In addition, since $(w_n,w^*_n)_{n\in\NN}$ is in $\gra W$, 
$(w_n,w^*_n+Cw_n)_{n\in\NN}$ is in $\gra(W+C)$. 
It then follows from \eqref{e:z7} and Lemma~\ref{l:12} that 
$x\in Z$. We conclude by invoking Theorem~\ref{t:1}\ref{t:1iv}.
\end{proof}

We now turn to strong convergence.

\begin{theorem}
\label{t:2c}
Let $\alpha\in\RPP$, let $W\colon\HH\to 2^{\HH}$ be maximally 
monotone, let $C\colon\HH\to\HH$ be $\alpha$-cocoercive and
such that $Z=\zer(W+C)\neq\emp$, let $x_0\in\HH$, and let
$(\lambda_n)_{n\in\NN}$ be a sequence in $\left]0,1\right]$.
Iterate
\begin{equation}
\label{e:haug19}
\begin{array}{l}
\text{for}\;n=0,1,\ldots\\
\left\lfloor
\begin{array}{l}
(w_n,w_n^*)\in\gra W,\;q_n\in\HH\\
t_n^*=w_n^*+Cq_n\\
\delta_n=\scal{x_n-w_n}{t_n^*}-\|w_n-q_n\|^2/(4\alpha)\\[2mm]
d_n=
\begin{cases}
\dfrac{\delta_n}
{\|t_n^*\|^2}t_n^*,&\text{if}\:\:\delta_n>0;\\
0,&\text{otherwise}\\
\end{cases}\\
r_n=x_n-\lambda_nd_n\\
x_{n+1}=\Qq(x_0,x_n,r_n),
\end{array} 
\right. 
\end{array} 
\end{equation}
where $\Qq$ is defined in Lemma~\ref{l:2}. Then the following hold:
\begin{enumerate}
\item
\label{t:2ci}
$(x_n)_{n\in\NN}$ is bounded.
\item
\label{t:2cii}
$\sum_{n\in\NN}\|x_{n+1}-x_n\|^2<\pinf$.
\item
\label{t:2cii+}
$\sum_{n\in\NN}\lambda_n^2\|d_n\|^2<\pinf$.
\item
\label{t:2ciii}
Suppose that $w_n-x_n\weakly 0$, $w_n-q_n\to 0$, and $t_n^*\to 0$.
Then $(x_n)_{n\in\NN}$ converges strongly to $\proj_Zx_0$. 
\end{enumerate}
\end{theorem}
\begin{proof}
Define $(H_n)_{n\in\NN}$ as in \eqref{e:h8} and note that
\eqref{e:c36} yields $Z\subset\bigcap_{n\in\NN}H_n$. Furthermore,
we derive from \eqref{e:haug19} and Example~\ref{ex:H} that 
$(\forall n\in\NN)$ $r_n=x_n+\lambda_n(\proj_{H_n}x_n-x_n)$.
This places us in the setting of Theorem~\ref{t:2}.

\ref{t:2ci}: This follows from Theorem~\ref{t:2}\ref{t:2ii}.

\ref{t:2cii}:
See Theorem~\ref{t:2}\ref{t:2iii}.

\ref{t:2cii+}:
This follows from Theorem~\ref{t:2}\ref{t:2iv}.

\ref{t:2ciii}:
As in the proof of Theorem~\ref{t:1c}\ref{t:1ciii}, 
$\WC(x_n)_{n\in\NN}\subset Z$. The claim follows from
Theorem~\ref{t:2}\ref{t:2v}.
\end{proof}

In the absence of the cocoercive operator $C$, we can choose
$(q_n)_{n\in\NN}=(w_n)_{n\in\NN}$ in \eqref{e:fejer19} and 
\eqref{e:haug19}, and Theorems~\ref{t:1c} and \ref{t:2c} 
simplify as follows. 

\begin{proposition}
\label{p:1}
Let $M\colon\HH\to 2^{\HH}$ be a maximally monotone operator
such that $Z=\zer M\neq\emp$, let $x_0\in\HH$, and let
$(\lambda_n)_{n\in\NN}$ be a sequence in $\left]0,2\right[$.
Iterate
\begin{equation}
\label{e:fejer14}
\begin{array}{l}
\text{for}\;n=0,1,\ldots\\
\left\lfloor
\begin{array}{l}
(m_n,m_n^*)\in\gra M\\[2mm]
d_n=
\begin{cases}
\dfrac{\scal{x_n-m_n}{m_n^*}}
{\|m_n^*\|^2}m_n^*,&\text{if}\:\:\scal{x_n-m_n}{m_n^*}>0;\\
0,&\text{otherwise}\\
\end{cases}\\
x_{n+1}=x_n-\lambda_nd_n.
\end{array} 
\right. 
\end{array} 
\end{equation}
Then the following hold:
\begin{enumerate}
\item
\label{p:1i}
$\sum_{n\in\NN}\lambda_n(2-\lambda_n)\|d_n\|^2<\pinf$.
\item
\label{p:1ii}
Suppose that $m_n-x_n\weakly 0$ and $m_n^*\to 0$. Then
$(x_n)_{n\in\NN}$ converges weakly to a point in $Z$.
\end{enumerate}
\end{proposition}

\begin{proposition}
\label{p:2}
Let $M\colon\HH\to 2^{\HH}$ be a maximally monotone operator
such that $Z=\zer M\neq\emp$, let $x_0\in\HH$, and let
$(\lambda_n)_{n\in\NN}$ be a sequence in $\left]0,1\right]$.
Iterate
\begin{equation}
\label{e:haug14}
\begin{array}{l}
\text{for}\;n=0,1,\ldots\\
\left\lfloor
\begin{array}{l}
(m_n,m_n^*)\in\gra M\\[2mm]
d_n=
\begin{cases}
\dfrac{\scal{x_n-m_n}{m_n^*}}
{\|m_n^*\|^2}m_n^*,&\text{if}\:\:\scal{x_n-m_n}{m_n^*}>0;\\
0,&\text{otherwise}\\
\end{cases}\\
r_n=x_n-\lambda_nd_n\\
x_{n+1}=\Qq(x_0,x_n,r_n),
\end{array} 
\right. 
\end{array} 
\end{equation}
where $\Qq$ is defined in Lemma~\ref{l:2}. Then the following hold:
\begin{enumerate}
\item
\label{p:2i}
$\sum_{n\in\NN}\lambda_n^2\|d_n\|^2<\pinf$.
\item
\label{p:2ii}
Suppose that $m_n-x_n\weakly 0$ and $m_n^*\to 0$. Then
$(x_n)_{n\in\NN}$ converges to strongly to $\proj_Zx_0$. 
\end{enumerate}
\end{proposition}

\subsection{Warped resolvent cuts}
\label{sec:m+c}

Algorithms~\eqref{e:fejer19} and \eqref{e:haug19} are conceptual in
the sense that they do not provide an explicit mechanism to find
points in the graph of $W$. In this section, we propose
implementable versions that pick points in $\gra W$ using the
warped resolvents of Lemma~\ref{l:warp}. 

\begin{theorem}
\label{t:8}
Let $\alpha\in\RPP$, 
let $W\colon\HH\to 2^{\HH}$ be maximally monotone, 
let $C\colon\HH\to\HH$ be $\alpha$-cocoercive and
such that $Z=\zer(W+C)\neq\emp$, let $x_0\in\HH$, and let
$(\lambda_n)_{n\in\NN}$ be a sequence in $\left]0,2\right[$.
Further, for every $n\in\NN$, let $U_n\colon\HH\to\HH$ be an
operator such that $\ran U_n\subset\ran(U_n+W+C)$ and 
$U_n+W+C$ is injective. Iterate
\begin{equation}
\label{e:fejer8}
\begin{array}{l}
\text{for}\;n=0,1,\ldots\\
\left\lfloor
\begin{array}{l}
w_n=J_{W+C}^{U_n}x_n\\
w_n^*=U_nx_n-U_nw_n-Cw_n\\
q_n\in\HH\\
t_n^*=w_n^*+Cq_n\\
\delta_n=\scal{x_n-w_n}{t_n^*}-\|w_n-q_n\|^2/(4\alpha)\\[2mm]
d_n=
\begin{cases}
\dfrac{\delta_n}
{\|t_n^*\|^2}t_n^*,&\text{if}\:\:\delta_n>0;\\
0,&\text{otherwise}\\
\end{cases}\\
x_{n+1}=x_n-\lambda_nd_n.
\end{array} 
\right. 
\end{array} 
\end{equation}
Then the following hold:
\begin{enumerate}
\item
\label{t:8i}
$\sum_{n\in\NN}\lambda_n(2-\lambda_n)\|d_n\|^2<\pinf$.
\item
\label{t:8ii}
Suppose that one of the following is satisfied:
\begin{enumerate}
\item
\label{t:8iic}
$\sum_{n\in\NN}\lambda_n(2-\lambda_n)=\pinf$ and
$(\|d_n\|)_{n\in\NN}$ converges;
\item
\label{t:8iid}
$\inf_{n\in\NN}\lambda_n>0$ and $\sup\lambda_n<2$;
\end{enumerate}
together with one of the following:
\begin{enumerate}[resume]
\item
\label{t:8iia}
$w_n-x_n\weakly 0$, $U_nw_n-U_nx_n\to 0$, and
$w_n-q_n\to 0$;
\item
\label{t:8iib}
$q_n-x_n\to 0$ and there exist
$\beta_1\in\left]1/(4\alpha),\pinf\right[$ and $\beta_2\in\RPP$
such that the kernels $(U_n)_{n\in\NN}$ are $\beta_1$-strongly
monotone and $\beta_2$-Lipschitzian.
\end{enumerate}
Then $(x_n)_{n\in\NN}$ converges weakly to a point in $Z$.
\end{enumerate}
\end{theorem}
\begin{proof}
Lemma~\ref{l:warp}\ref{l:warpi} indicates that \eqref{e:fejer8} is
governed by the scenario of Theorem~\ref{t:1c}.

\ref{t:8i}: See Theorem~\ref{t:1c}\ref{t:1cii}.

\ref{t:8ii}: A consequence of \ref{t:8i} under
\ref{t:8iic} or \ref{t:8iid} is that 
\begin{equation}
\label{e:aqp1}
\|d_n\|\to 0.
\end{equation}
Indeed, the claim is clear under \ref{t:8iid} whereas, under 
\ref{t:8iic}, we have $\varliminf\|d_n\|=0$ and therefore
$\lim\|d_n\|=0$. Next, let us assume that \ref{t:8iia} holds.
Then it follows from \eqref{e:fejer8} and \eqref{e:coco} that
\begin{align}
(\forall n\in\NN)\quad
\|t_n^*\|
&=\|U_nw_n-U_nx_n+Cw_n-Cq_n\|\nonumber\\
&\leq\|U_nw_n-U_nx_n\|+\|Cw_n-Cq_n\|\label{e:no}\\
&\leq\|U_nw_n-U_nx_n\|+\dfrac{\|w_n-q_n\|}{\alpha}\nonumber\\
&\to 0.
\label{e:mo}
\end{align}
In view of Theorem~\ref{t:1c}\ref{t:1ciii}, the claim is
established. It remains to show that
\ref{t:8iib}$\Rightarrow$\ref{t:8iia}.
Because the operators $(U_n+W+C)_{n\in\NN}$ are $\beta_1$-strongly
monotone, the operators $(U_n+W+C)^{-1}_{n\in\NN}$ are
$\beta_1$-cocoercive, hence $1/\beta_1$-Lipschitzian. Consequently,
since the operators $(U_n)_{n\in\NN}$ are $\beta_2$-Lipschitzian, 
the operators $(J^{U_n}_{W+C})_{n\in\NN}$ are
$\beta_2/\beta_1$-Lipschitzian. Now let $z\in Z$. Then we derive
from \eqref{e:fejer8} and Lemma~\ref{l:warp}\ref{l:warpii} that 
\begin{equation}
(\forall n\in\NN)\;\;
\|w_n-z\|=\Bigl\|J^{U_n}_{W+C}x_n-J^{U_n}_{W+C}z\Bigr\|
\leq\dfrac{\beta_2}{\beta_1}\|x_n-z\|.
\end{equation}
Appealing to Theorem~\ref{t:1c}\ref{t:1ci}, we infer that
$(w_n)_{n\in\NN}$ is bounded. Thus, since $q_n-x_n\to 0$ and $C$ is
$1/\alpha$-Lipschitzian, the sequences 
\begin{equation}
\label{e:aqp7}
(\|w_n-x_n\|)_{n\in\NN},\;(\|w_n-q_n\|)_{n\in\NN},\;\text{and}\;
(\|Cw_n-Cq_n\|)_{n\in\NN}\;\text{are bounded.}
\end{equation}
However, \eqref{e:no} entails that
\begin{equation}
(\forall n\in\NN)\quad
\|t_n^*\|\leq\beta_2\|w_n-x_n\|+\dfrac{\|w_n-q_n\|}{\alpha},
\end{equation}
which verifies that $(\|t^*_n\|)_{n\in\NN}$ is bounded. In turn,
\eqref{e:fejer8} and \eqref{e:aqp1} imply that
\begin{equation}
\label{e:aqp5}
\varlimsup\delta_n\leq\lim\|t_n^*\|\,\|d_n\|=0.
\end{equation}
Moreover, for every $n\in\NN$, \eqref{e:fejer8} yields
\begin{align}
\delta_n
&=\scal{w_n-x_n}{U_nw_n-U_nx_n}+\scal{w_n-x_n}{Cw_n-Cq_n}
-\dfrac{\|w_n-q_n\|^2}{4\alpha}
\nonumber\\
&\geq\beta_1\|w_n-x_n\|^2
+\scal{w_n-q_n}{Cw_n-Cq_n}
+\scal{q_n-x_n}{Cw_n-Cq_n}\nonumber\\
&\quad-\dfrac{\|w_n-q_n\|^2}{4\alpha}
\nonumber\\
&\geq\beta_1\Bigl(\|w_n-q_n\|^2+2\scal{w_n-q_n}{q_n-x_n}
+\|q_n-x_n\|^2\Bigr)\nonumber\\
&\quad
+\alpha\|Cw_n-Cq_n\|^2
+\scal{q_n-x_n}{Cw_n-Cq_n}-\dfrac{\|w_n-q_n\|^2}{4\alpha}
\nonumber\\
&\geq\biggl(\beta_1-\dfrac{1}{4\alpha}\biggl)\|w_n-q_n\|^2
+\beta_1\Bigl(2\scal{w_n-q_n}{q_n-x_n}
+\|q_n-x_n\|^2\Bigr)\nonumber\\
&\quad
+\scal{q_n-x_n}{Cw_n-Cq_n}
\nonumber\\
&\geq\biggl(\beta_1-\dfrac{1}{4\alpha}\biggl)\|w_n-q_n\|^2
\nonumber\\
&\quad
+\|q_n-x_n\|\bigl(\beta_1\|q_n-x_n\|-2\beta_1\|w_n-q_n\|
+\|Cw_n-Cq_n\|\bigr).
\end{align}
Therefore, since $\|q_n-x_n\|\to 0$, it follows from 
\eqref{e:aqp7} and \eqref{e:aqp5} that $w_n-q_n\to 0$ and hence
that $w_n-x_n\to 0$. Since
\begin{equation}
\|U_nw_n-U_nx_n\|\leq\beta_2\|w_n-x_n\|\leq
\beta_2(\|w_n-q_n\|+\|q_n-x_n\|)\to 0, 
\end{equation}
the proof is complete.
\end{proof}

\begin{remark}
\label{r:99}
In the special case when $C=0$, $(q_n)_{n\in\NN}=(w_n)_{n\in\NN}$,
and conditions \ref{t:8iid} and \ref{t:8iia} are satisfied,
Theorem~\ref{t:8}\ref{t:8ii} is closely related to
\cite[Theorem~4.2(ii)]{Jmaa20}.
\end{remark}

We conclude this section with the strongly convergent best
approximation companion algorithm resulting from
Theorem~\ref{t:2c}.

\begin{theorem}
\label{t:8s}
Let $\alpha\in\RPP$, 
let $W\colon\HH\to 2^{\HH}$ be maximally monotone, 
let $C\colon\HH\to\HH$ be $\alpha$-cocoercive and
such that $Z=\zer(W+C)\neq\emp$, let $x_0\in\HH$, and let
$(\lambda_n)_{n\in\NN}$ be a sequence in $\rzeroun$.
Further, for every $n\in\NN$, let $U_n\colon\HH\to\HH$ be 
an operator such that $\ran U_n\subset\ran(U_n+W+C)$ and 
$U_n+W+C$ is injective. Iterate
\begin{equation}
\label{e:haug8}
\begin{array}{l}
\text{for}\;n=0,1,\ldots\\
\left\lfloor
\begin{array}{l}
w_n=J_{W+C}^{U_n}x_n\\
w_n^*=U_nx_n-U_nw_n-Cw_n\\
q_n\in\HH\\
t_n^*=w_n^*+Cq_n\\
\delta_n=\scal{x_n-w_n}{t_n^*}-\|w_n-q_n\|^2/(4\alpha)\\[2mm]
d_n=
\begin{cases}
\dfrac{\delta_n}
{\|t_n^*\|^2}t_n^*,&\text{if}\:\:\delta_n>0;\\
0,&\text{otherwise}\\
\end{cases}\\
r_n=x_n-\lambda_nd_n\\
x_{n+1}=\Qq(x_0,x_n,r_n),
\end{array} 
\right. 
\end{array} 
\end{equation}
where $\Qq$ is defined in Lemma~\ref{l:2}. Then the following hold:
\begin{enumerate}
\item
\label{t:8si+}
$\sum_{n\in\NN}\lambda_n^2\|d_n\|^2<\pinf$.
\item
\label{t:8sii}
Suppose that one of the following is satisfied:
\begin{enumerate}
\item
\label{t:8siic}
$\sum_{n\in\NN}\lambda_n^2=\pinf$ and
$(\|d_n\|)_{n\in\NN}$ converges;
\item
\label{t:8siid}
$\inf_{n\in\NN}\lambda_n>0$;
\end{enumerate}
together with one of the following:
\begin{enumerate}[resume]
\item
\label{t:8siia}
$w_n-x_n\weakly 0$, $U_nw_n-U_nx_n\to 0$, and
$w_n-q_n\to 0$;
\item
\label{t:8siib}
$q_n-x_n\to 0$ and there exist
$\beta_1\in\left]1/(4\alpha),\pinf\right[$ and $\beta_2\in\RPP$
such that the kernels $(U_n)_{n\in\NN}$ are $\beta_1$-strongly
monotone and $\beta_2$-Lipschitzian.
\end{enumerate}
Then $(x_n)_{n\in\NN}$ converges strongly to $\proj_Zx_0$. 
\end{enumerate}
\end{theorem}
\begin{proof}
In view of Lemma~\ref{l:warp}\ref{l:warpi}, \eqref{e:haug8} 
is an instance of \eqref{e:haug19} and we shall therefore employ 
Theorem~\ref{t:2c}.


\ref{t:8si+}: See Theorem~\ref{t:2c}\ref{t:2cii+}.

\ref{t:8sii}: 
It follows from \ref{t:8si+} and \eqref{e:haug8} that $d_n\to 0$.
Indeed, this is evident under \ref{t:8siid}
whereas, under \ref{t:8iic}, we have $\varliminf\|d_n\|=0$ and
therefore $\lim\|d_n\|=0$. Let us now assume that \ref{t:8siia}
holds. Then \eqref{e:mo} is satisfied and we obtain the assertion
by invoking Theorem~\ref{t:2c}\ref{t:2ciii}. Finally, to
show that \ref{t:8siib}$\Rightarrow$\ref{t:8siia}, we remark that
Theorem~\ref{t:2c}\ref{t:2ci} asserts
that $(x_n)_{n\in\NN}$ is bounded. Hence, we follow the same
pattern as in the proof of Theorem~\ref{t:8}\ref{t:8iib} to
conclude.
\end{proof}

\section{The proximal point algorithm}
\label{sec:ppa}

\subsection{Preview} 
The proximal point algorithm is an implicit
method to construct a zero of a maximally monotone operator which
goes back to a quadratic programming method proposed in
\cite[Section~5.8]{Bell66}. In the nonlinear case, it first
appeared in Lieutaud's work \cite{Lieu68} (this fact seems to have
been overlooked in the literature, see Remark~\ref{r:2}), then in
\cite{Mart70,Mart72} for subdifferentials and in \cite{Roc76a} for
the general case. Iteration $n$ of the unrelaxed form of the
algorithm can be interpreted as a backward Euler discretization 
of the Cauchy problem \cite[Section~3.2]{Aub84a}
(see Example~\ref{ex:7})
\begin{equation}
\label{e:ev}
\begin{cases}
x(0)=x_0\\\
-x'(t)\in Mx(t),\;\text{for a.e.}\;t\in\RPP
\end{cases}
\end{equation}
with time step $\gamma_n\in\RPP$, that is,
\begin{equation}
\label{e:evd}
\dfrac{x_n-x_{n+1}}{\gamma_n}\in Mx_{n+1}
\end{equation}
or, equivalently, $x_{n+1}=J_{\gamma_n M}x_n$. 

\subsection{Fej\'erian algorithm}

The following theorem, which brings together results from 
\cite{Brez78,Ecks92,Gaba83,Gols79,Lema89,Mart70,Mart72,Roc76a},
will be derived from Theorem~\ref{t:8}.

\begin{theorem}
\label{t:70}
Let $M\colon\HH\to 2^{\HH}$ be a maximally monotone operator such
that $Z=\zer M\neq\emp$, let $x_0\in\HH$, let 
$(\lambda_n)_{n\in\NN}$ be a sequence in $\left]0,2\right[$, and
let $(\gamma_n)_{n\in\NN}$ be a sequence in $\RPP$. Iterate
\begin{equation}
\label{e:1}
(\forall n\in\NN)\quad x_{n+1}=
x_n+\lambda_n\big(J_{\gamma_n M}x_n-x_n\big)
\end{equation}
and suppose that one of the following holds:
\begin{enumerate}
\item
\label{t:70i}
$\sum_{n\in\NN}\lambda_n(2-\lambda_n)=\pinf$ and 
$(\forall n\in\NN)$ $\gamma_n=1$.
\item
\label{t:70ii}
$\sum_{n\in\NN}\gamma_n^2=\pinf$ and 
$(\forall n\in\NN)$ $\lambda_n=1$.
\item
\label{t:70iii}
$\inf_{n\in\NN}\lambda_n>0$, $\sup_{n\in\NN}\lambda_n<2$, and
$\inf_{n\in\NN}\gamma_n>0$.
\end{enumerate}
Then $\|J_{\gamma_n M}x_n-x_n\|/\gamma_n\to 0$ and 
$(x_n)_{n\in\NN}$ converges weakly to a point in $Z$.
\end{theorem}
\begin{proof}
Let us apply Theorem~\ref{t:8} with 
\begin{equation}
\label{e:aqp57}
C=0\;\;\text{and}\;\;(\forall n\in\NN)\;\;
U_n=\gamma_n^{-1}\Id\;\;\text{and}\;\;q_n=w_n. 
\end{equation}
We derive from \eqref{e:z} that the variables of the iterations
\eqref{e:fejer8} satisfy
\begin{equation}
\label{e:aqp10}
(\forall n\in\NN)\quad 
t_n^*=\dfrac{x_n-w_n}{\gamma_n},\;
\delta_n=\gamma_n\|t_n^*\|^2,
\;\text{and}\;d_n=x_n-w_n.
\end{equation}
Thus, the sequence $(x_n)_{n\in\NN}$ produced by \eqref{e:1} 
coincides with that of \eqref{e:fejer8}. In turn, 
Theorem~\ref{t:8}\ref{t:8i} yields 
\begin{equation}
\label{e:f20}
\sum_{n\in\NN}\lambda_n(2-\lambda_n)\|d_n\|^2<\pinf.
\end{equation}
We now show that one of conditions \ref{t:8iic}--\ref{t:8iid} and
one of conditions \ref{t:8iia}--\ref{t:8iib} of
Theorem~\ref{t:8}\ref{t:8ii} are fulfilled in each scenario. We
also recall from \eqref{e:aqp1} that \ref{t:8iic} and \ref{t:8iid}
in Theorem~\ref{t:8} each imply that 
\begin{equation}
\label{e:aqp9}
d_n\to 0.
\end{equation}

\ref{t:70i}: 
Let us check that conditions \ref{t:8iic} and \ref{t:8iib} are
fulfilled. For
\ref{t:8iic}, it is enough to show that $(\|d_n\|)_{n\in\NN}$
decreases. To this end, set $T=2J_M-\Id$. Then
Lemma~\ref{l:0}\ref{l:0iii} and \eqref{e:ne} assert that $T$ is
nonexpansive. Therefore, \eqref{e:aqp10} yields
\begin{align}
\label{e:f23}
(\forall n\in\NN)\quad 2\|d_{n+1}\|
&=\|Tx_{n+1}-x_{n+1}\|\nonumber\\
&=\|Tx_{n+1}-Tx_n+(1-\lambda_n/2)(Tx_n-x_n)\|\nonumber\\
&\leq\|x_{n+1}-x_n\|+(1-\lambda_n/2)\|Tx_n-x_n\|\nonumber\\
&=(\lambda_n/2)\|Tx_n-x_n\|+
(1-\lambda_n/2)\|Tx_n-x_n\|\nonumber\\
&=2\|d_n\|,
\end{align}
as desired.
For \ref{t:8iib}, note that \eqref{e:aqp9} and \eqref{e:aqp10}
imply that $q_n-x_n=w_n-x_n=-d_n\to 0$. In addition, it is clear
from \eqref{e:aqp57} that $(U_n)_{n\in\NN}$ satisfies the required
conditions with $\beta_1=\beta_2=1$.

\ref{t:70ii}:
Condition \ref{t:8iid} holds. To show that \ref{t:8iia} holds as
well, we first infer from \eqref{e:aqp10} and \eqref{e:f20}
that $\sum_{n\in\NN}\gamma_n^2\|t_n^*\|^2<\pinf$ and hence that
$w_n-x_n=-\gamma_n t_n^*\to 0$.
Furthermore, since $\sum_{n\in\NN}\gamma_n^2=\pinf$, 
$\varliminf\|t_n^*\|=0$. On the other hand, $(\forall n\in\NN)$ 
$t_n^*=\gamma_n^{-1}(x_n-w_n)=\gamma_n^{-1}(x_n-x_{n+1})$. Hence,
using \eqref{e:ee14}, the monotonicity of $M$, and the
Cauchy--Schwarz inequality, we obtain
\begin{align}
\label{e:jk41}
(\forall n\in\NN)\quad 0
&\leq\scal{w_n-w_{n+1}}{t^*_n-t^*_{n+1}}/\gamma_{n+1}
\nonumber\\
&=\scal{x_{n+1}-x_{n+2}}{t^*_n-t^*_{n+1}}/\gamma_{n+1}
\nonumber\\
&=\scal{t^*_{n+1}}{t^*_n-t^*_{n+1}}\nonumber\\
&=\scal{t^*_{n+1}}{t^*_n}-\|t^*_{n+1}\|^2\nonumber\\
&\leq\|t^*_{n+1}\|\big(\|t^*_n\|-\|t^*_{n+1}\|\big),
\end{align}
which shows that $(\|t_n^*\|)_{n\in\NN}$ decreases. Altogether, 
$U_nx_n-U_nw_n=t_n^*\to 0$.

\ref{t:70iii}:
Condition \ref{t:8iid} is assumed. Let us check 
\ref{t:8iia}. Since \eqref{e:aqp10} and \eqref{e:f20} yield
$\sum_{n\in\NN}\gamma_n^2\|t_n^*\|^2<\pinf$, we have
$x_n-w_n=\gamma_nt_n^*\to 0$. Finally, since
$\inf_{n\in\NN}\gamma_n>0$, $U_nx_n-U_nw_n=t_n^*\to 0$. 

We conclude the proof by noting that in all three cases above
we have $\|J_{\gamma_n M}x_n-x_n\|/\gamma_n=\|t^*_n\|\to 0$.
\end{proof}

\begin{remark}
\label{r:1}
Let $f\in\Gamma_0(\HH)$ and suppose that $M=\partial f$ in
Theorem~\ref{t:70}. Then, as seen in Example~\ref{ex:1}, $M$ is
maximally monotone and $Z=\Argmin f$. In this case, the condition
on $(\gamma_n)_{n\in\NN}$ in Theorem~\ref{t:70}\ref{t:70ii} can be
improved to $\sum_{n\in\NN}\gamma_n=\pinf$
\cite[Th\'eor\`eme~9]{Brez78}.
\end{remark}

\subsection{Haugazeau-like algorithm}
We employ Theorem~\ref{t:8s} to obtain a strongly convergent
variant of the proximal point algorithm; see \cite{Moor01,Solo00}
for related results. Examples of proximal point iterations that
fail to converge strongly are constructed in 
\cite{Baus04,Bord18,Gule91}.

\begin{theorem}
\label{t:70s}
Let $M\colon\HH\to 2^{\HH}$ be a maximally monotone operator such
that $Z=\zer M\neq\emp$, let $x_0\in\HH$, let
$(\lambda_n)_{n\in\NN}$ be a sequence in $\left]0,1\right]$ such
that $\inf_{n\in\NN}\lambda_n>0$, and let 
$(\gamma_n)_{n\in\NN}$ be a sequence in $\RPP$ such that
$\inf_{n\in\NN}\gamma_n>0$. Iterate
\begin{equation}
\label{e:2}
(\forall n\in\NN)\quad x_{n+1}=
\Qq\bigl(x_0,x_n,x_n+\lambda_n(J_{\gamma_n M}x_n-x_n)\bigr),
\end{equation}
where $\Qq$ is defined in Lemma~\ref{l:2}.
Then $(x_n)_{n\in\NN}$ converges strongly to $\proj_Zx_0$.
\end{theorem}
\begin{proof}
In Theorem~\ref{t:8s}, set $C=0$ and $(\forall n\in\NN)$
$U_n=\gamma_n^{-1}\Id$ and $q_n=w_n$. Then 
\eqref{e:aqp10} holds and the sequence $(x_n)_{n\in\NN}$ produced 
by \eqref{e:2} coincides with that of \eqref{e:haug8}. In turn, 
Theorem~\ref{t:8s}\ref{t:8si+} yields 
$\sum_{n\in\NN}\lambda_n^2\|d_n\|^2<\pinf$. Therefore,
$x_n-w_n=d_n\to 0$ and $U_nx_n-U_nw_n=\gamma_n^{-1}d_n\to 0$. This
confirms that condition \ref{t:8siia} in 
Theorem~\ref{t:8s}\ref{t:8sii} is fulfilled.  
Since condition \ref{t:8siid} holds by assumption, the proof is
complete. 
\end{proof}

\subsection{Special cases and variants} 
As mentioned in Section~\ref{sec:1}, direct implementations of the
proximal point algorithm are limited due to the potential
difficulty of evaluating the resolvents in \eqref{e:1} and
\eqref{e:2}. As we shall see in this section, the proximal point
framework can nonetheless be an effective device to establish
indirectly the convergence of algorithms that can be identified,
possibly in a different space, as an instance of \eqref{e:1}. Early
examples in the context of inequality-constrained minimization
problems are found in \cite{Roc76b}, where a dual application of an
approximate proximal point algorithm was shown to yield a method of
multipliers (also called the augmented Lagrangian method) that
extends some classical ones from \cite{Hest69} and \cite{Powe69}
(see also \cite{Roc73p}). A primal-dual quadratically perturbed
variant of this algorithm, known as the proximal method of
multipliers, was also introduced in \cite{Roc76b} as an application
of an approximate proximal point algorithm to find saddle points of
the Lagrangian (see also \cite{Rock24,Shef14} and their
bibliographies for recent work along these lines). The applications
described below reduce to implementations of the proximal point
algorithm that feature full operator splitting when several linear
and nonlinear operators are present in the original problem.

\subsubsection{The Euler method}
\label{sec:eul}

We derive from the proximal point algorithm a (forward)
Euler method to find a zero of a cocoercive operator.

\begin{proposition}
\label{p:19c}
Let $\alpha\in\RPP$ and let $B\colon\HH\to\HH$ be 
$\alpha$-cocoercive, with $\zer B\neq\emp$. Let 
$(\gamma_n)_{n\in\NN}$ be a sequence in $\left]0,2\alpha\right[$ 
such that $\sum_{n\in\NN}\gamma_n(2\alpha-\gamma_n)=\pinf$ and 
let $x_0\in\HH$. Iterate
\begin{equation}
\label{e:c1}
(\forall n\in\NN)\quad x_{n+1}=x_n-\gamma_nBx_n.
\end{equation}
Then $(x_n)_{n\in\NN}$ converges weakly to a point in $\zer B$. 
\end{proposition}
\begin{proof}
Set $M=(\Id-\alpha B)^{-1}-\Id$. Since $\alpha B$ is firmly
nonexpansive with domain $\HH$, $\Id-\alpha B$ is likewise and
Lemma~\ref{l:0}\ref{l:0iii} asserts that $M$ is maximally monotone.
On the other hand, $\zer M=\zer B$, $J_M=\Id-\alpha B$, and hence
\eqref{e:c1} becomes
\begin{equation}
\label{e:groetsch3}
(\forall n\in\NN)\quad x_{n+1}=x_n+\lambda_n(J_Mx_n-x_n),
\quad\text{where}\quad\lambda_n=\gamma_n/\alpha\in\left]0,2\right[.
\end{equation}
Thus, since $\sum_{n\in\NN}\lambda_n(2-\lambda_n)=\pinf$, the claim
follows from Theorem~\ref{t:70}\ref{t:70i}.
\end{proof}

\begin{remark}
\label{r:19c}
As just shown, the Euler method \eqref{e:c1} is an instance of
the proximal point algorithm \eqref{e:1}. Conversely, we can
interpret the proximal point iterations in the format
\begin{equation}
\label{e:g4}
(\forall n\in\NN)\quad x_{n+1}=x_n+\lambda_n(J_Mx_n-x_n),
\quad\text{where}\quad\lambda_n\in\left]0,2\right[
\end{equation}
as an instance of \eqref{e:c1}. Indeed, let $M\colon\HH\to 2^{\HH}$
be maximally monotone and set $B=\moyo{M}{1}$ and
$(\forall n\in\NN)$ $\gamma_n=\lambda_n$. Then, as seen in
Example~\ref{ex:13y}, $\zer M=\zer B$ and $B$ is $1$-cocoercive,
while \eqref{e:yosi1} implies that \eqref{e:g4} reduces to
\eqref{e:c1}.
\end{remark}

The following example is about the gradient method (see
\cite{Cauc47,Curr44} for the premises of this algorithm).
 
\begin{example}
\label{ex:19c}
Let $\alpha\in\RPP$ and let $g\colon\HH\to\RR$ be convex,
differentiable, and such that $\nabla g$ is 
$1/\alpha$-Lipschitzian,
with $\Argmin g\neq\emp$. Let $(\gamma_n)_{n\in\NN}$ be a sequence 
in $\left]0,2\alpha\right[$ such that 
$\sum_{n\in\NN}\gamma_n(2\alpha-\gamma_n)=\pinf$ and let 
$x_0\in\HH$. Iterate
\begin{equation}
\label{e:c2}
(\forall n\in\NN)\quad x_{n+1}=x_n-\gamma_n\nabla g(x_n).
\end{equation}
Then $(x_n)_{n\in\NN}$ converges weakly to a point in $\Argmin g$. 
\end{example}
\begin{proof}
Combine Lemma~\ref{l:bh} and Proposition~\ref{p:19c}.
\end{proof}

As noted in \cite[Remark~4.8(ii)]{Nona05} in the context of
Example~\ref{ex:19c}, the convergence in Proposition~\ref{p:19c}
can fail to be strong. The next result, which guarantees strong
convergence, is obtained by defining $M$ and
$(\lambda_n)_{n\in\NN}$ as in the proof of Proposition~\ref{p:19c}
and using Theorem~\ref{t:70s}.

\begin{proposition}
\label{p:19cs}
Let $\alpha\in\RPP$ and let $B\colon\HH\to\HH$ be 
$\alpha$-cocoercive, with $\zer B\neq\emp$. Let 
$(\gamma_n)_{n\in\NN}$ be a sequence in $\left]0,\alpha\right]$ 
such that $\inf_{n\in\NN}\gamma_n>0$ and let $x_0\in\HH$. Iterate
\begin{equation}
\label{e:c1s}
(\forall n\in\NN)\quad
x_{n+1}=\Qq\bigl(x_0,x_n,x_n-\gamma_nBx_n\bigr),
\end{equation}
where $\Qq$ is defined in Lemma~\ref{l:2}. Then $(x_n)_{n\in\NN}$
converges strongly to $\proj_{\zer B}\,x_0$. 
\end{proposition}

\subsubsection{Fixed point problem}

We address the basic problem of constructing a fixed point of a
nonexpansive operator $T\colon\HH\to\HH$. The following result is
derived as an instance of the proximal point algorithm of
Theorem~\ref{t:70} via the embedding of Example~\ref{ex:f12}. 

\begin{proposition}
\label{p:19}
Let $\alpha\in\rzeroun$ and let $T\colon\HH\to\HH$ be
$\alpha$-\emph{averaged}. Suppose that $\Fix T\neq\emp$, let
$(\lambda_n)_{n\in\NN}$ be a sequence in $\left]0,1/\alpha\right[$
such that $\sum_{n\in\NN}\lambda_n(1-\alpha\lambda_n)=\pinf$, and
let $x_0\in\HH$. Iterate
\begin{equation}
\label{e:groetsch1}
(\forall n\in\NN)\quad x_{n+1}=x_n+\lambda_n(Tx_n-x_n).
\end{equation}
Then $(x_n)_{n\in\NN}$ converges weakly to a point in $\Fix T$. 
\end{proposition}
\begin{proof}
We use the embedding of Example~\ref{ex:f12}. 
Define $\MMM$ as in \eqref{e:M1} and note that 
$J_{\MMM}=\Id+(2\alpha)^{-1}(T-\Id)$. We therefore rewrite
\eqref{e:groetsch1} as 
\begin{equation}
\label{e:groetsch2}
(\forall n\in\NN)\quad x_{n+1}=x_n+\mu_n(J_{\MMM}x_n-x_n),
\quad\text{where}\quad\mu_n=2\alpha\lambda_n\in\left]0,2\right[.
\end{equation}
Then $\sum_{n\in\NN}\mu_n(2-\mu_n)=\pinf$ and, appealing to
Theorem~\ref{t:70}\ref{t:70i}, we conclude that $(x_n)_{n\in\NN}$
converges weakly to a point in $\zer\MMM=\Fix T$.
\end{proof}

In the case when $\alpha=1$, Proposition~\ref{p:19} is due to
Groetsch \cite{Groe72} and \eqref{e:groetsch1} is known as the
\emph{Krasnosel'ski\u\i--Mann iteration}, owing to its connection
with iterative schemes proposed in \cite{Kras55} and \cite{Mann53},
and it is a pillar of nonlinear numerical functional analysis
\cite{Livre1,Cegi12,Dong22}. Here is a strongly convergent variant
derived from Theorem~\ref{t:70s} (see \cite{Gene75} for an
example of the failure of strong convergence in
Proposition~\ref{p:19}).

\begin{proposition}
\label{p:19s}
Let $\alpha\in\rzeroun$ and let $T\colon\HH\to\HH$ be
$\alpha$-averaged. Suppose that $\Fix T\neq\emp$, let
$(\lambda_n)_{n\in\NN}$ be a sequence in 
$\left]0,1/(2\alpha)\right]$ such that $\inf_{n\in\NN}\lambda_n>0$,
and let $x_0\in\HH$. Iterate
\begin{equation}
\label{e:groetsch4}
(\forall n\in\NN)\quad x_{n+1}=
\Qq\bigl(x_0,x_n,x_n+\lambda_n(Tx_n-x_n)\bigr),
\end{equation}
where $\Qq$ is defined in Lemma~\ref{l:2}. Then $(x_n)_{n\in\NN}$
converges strongly to $\proj_{\Fix T}\,x_0$. 
\end{proposition}
\begin{proof}
Define $\MMM$ as in \eqref{e:M1}, argue as in the proof of
Proposition~\ref{p:19} to observe that \eqref{e:groetsch4} is an
instance of \eqref{e:2}, and conclude by invoking
Theorem~\ref{t:70s}.
\end{proof}

\subsubsection{Resolvent compositions}

We focus on the inclusion problem of \cite[Section~6]{Svva23},
which is modeled by resolvent compositions (see 
Example~\ref{ex:r3}) and solvable via the proximal point algorithm.

\begin{proposition}
\label{p:23}
Suppose that $L\in\BL(\HH,\GG)$ satisfies $0<\|L\|\leq 1$,
let $B\colon\GG\to 2^{\GG}$ be maximally monotone, let $V\neq\{0\}$
be a closed vector subspace of $\HH$, and let $\gamma\in\RPP$.
Let $S$ be the set of solutions to the problem
\begin{equation}
\label{e:zp1}
\text{find}\;\:x\in V\;\:\text{such that}\;\:0\in B(Lx)
\end{equation}
and let $Z$ be the set of solutions to the problem 
\begin{equation}
\label{e:p23}
\text{find}\;\:x\in\HH\;\:\text{such that}\;\:
0\in\bigl(\proxc{\proj_V}{\big(\proxcc{L}{(\gamma B)}\big)}\bigr)x.
\end{equation}
Then \eqref{e:p23} is an exact relaxation of \eqref{e:zp1} in the
sense that $S\neq\emp$ $\Rightarrow$ $Z=S$. Now assume that
$Z\neq\emp$, let $(\lambda_n)_{n\in\NN}$ be a sequence in
$\left]0,2\right[$ such that 
$\sum_{n\in\NN}\lambda_n(2-\lambda_n)=\pinf$, and let $x_0\in V$.
Iterate
\begin{equation}
\label{e:z9}
\begin{array}{l}
\text{for}\;n=0,1,\ldots\\
\left\lfloor
\begin{array}{l}
y_n=Lx_n\\
q_n=J_{\gamma B}y_n-y_n\\
z_n=L^*q_n\\
x_{n+1}=x_n+\lambda_n\proj_Vz_n.
\end{array}
\right.
\end{array}
\end{equation}
Then $(x_n)_{n\in\NN}$ converges weakly to a point in $Z$.
\end{proposition}
\begin{proof}
The exact relaxation claim is established in
\cite[Theorem~6.3(v)]{Svva23}. Now
set $M=\proxc{\proj_V}{(\proxcc{L}{(\gamma B)})}$ and note that
$\|\proj_V\|=1$ and $\proj_V^*=\proj_V$. Hence, it follows from
Example~\ref{ex:22} that $M$ is maximally monotone and from 
Example~\ref{ex:r3} that $J_M=\proj_V\circ(\Id_\HH-L^*\circ L
+L^*\circ J_{\gamma B}\circ L)\circ\proj_V$. Altogether, the
convergence result follows from Theorem~\ref{t:70}\ref{t:70i}
\end{proof}

Here is a strongly convergent algorithm based on the Haugazeau
variant.

\begin{proposition}
\label{p:23s}
Suppose that $L\in\BL(\HH,\GG)$ satisfies $0<\|L\|\leq 1$,
let $B\colon\GG\to 2^{\GG}$ be maximally monotone, let $V\neq\{0\}$
be a closed vector subspace of $\HH$, and let $\gamma\in\RPP$.
Suppose that the set $Z$ of solutions to the problem 
\begin{equation}
\label{e:p23s}
\text{find}\;\:x\in\HH\;\:\text{such that}\;\:
0\in\big(\proxc{\proj_V}{\big(\proxcc{L}{(\gamma B)}\big)}\big)x
\end{equation}
is not empty. Let $(\lambda_n)_{n\in\NN}$ be a sequence in
$\left]0,1\right]$ such that $\inf_{n\in\NN}\lambda_n>0$, and let
$x_0\in V$. Iterate
\begin{equation}
\label{e:z9s}
\begin{array}{l}
\text{for}\;n=0,1,\ldots\\
\left\lfloor
\begin{array}{l}
y_n=Lx_n\\
q_n=J_{\gamma B}y_n-y_n\\
z_n=L^*q_n\\
x_{n+1}=\Qq\bigl(x_0,x_n,x_n+\lambda_n\proj_Vz_n\bigr),
\end{array}
\right.
\end{array}
\end{equation}
where $\Qq$ is defined in Lemma~\ref{l:2}. Then $(x_n)_{n\in\NN}$
converges strongly to $\proj_Zx_0$.
\end{proposition}
\begin{proof}
Arguing as in the proof of Proposition~\ref{p:23}, this is an
application of Theorem~\ref{t:70s} with 
$M=\proxc{\proj_V}{(\proxcc{L}{(\gamma B)})}$ and 
$(\forall n\in\NN)$ $\gamma_n=1$.
\end{proof}

Below we recover the relaxation framework of \cite{Siim22} for
signal reconstruction in the presence of possibly inconsistent
nonlinear observations.

\begin{example}
\label{ex:z23}
Let $0<p\in\NN$, let $\gamma\in\RPP$, and let $V\neq\{0\}$ be a
closed vector subspace of $\HH$. For every $k\in\{1,\ldots,p\}$,
let $\GG_k$ be a real Hilbert space, let $L_k\in\BL(\HH,\GG_k)$,
let $\omega_k\in\RPP$, let $F_k\colon\GG_k\to\GG_k$ be firmly
nonexpansive, and let $r_k\in\GG_k$. Consider the nonlinear
reconstruction problem \cite[Problem~1.1]{Siim22}
\begin{equation}
\label{e:z44}
\text{find}\;\:x\in V\;\:\text{such that}\;\:
\bigl(\forall k\in\{1,\ldots,p\}\bigr)\quad F_k(L_kx)=r_k
\end{equation}
and the relaxed variational inequality problem 
\cite[Problem~1.3]{Siim22}
\begin{equation}
\label{e:z45}
\text{find}\:\;x\in V\:\;\text{such that}\:\;
\sum_{k=1}^p\omega_kL_k^*\bigl(F_k(L_kx)-r_k\bigr)\in V^\bot.
\end{equation}
Suppose that $0<\sum_{k=1}^p\omega_k\|L_k\|^2\leq 1$ and that
\eqref{e:z45} admits solutions. Let $x_0\in V$, let
$(\lambda_n)_{n\in\NN}$ be a sequence in $\left]0,2\right[$ such
that $\sum_{n\in\NN}\lambda_n(2-\lambda_n)=\pinf$, and iterate
\begin{equation}
\label{e:z94}
\begin{array}{l}
\text{for}\;n=0,1,\ldots\\
\left\lfloor
\begin{array}{l}
\text{for}\;k=1,\ldots,p\\
\left\lfloor
\begin{array}{l}
y_{k,n}=L_kx_n\\
q_{k,n}=r_k-F_ky_{k,n}\\
\end{array}
\right.\\
z_n=\sum_{k=1}^p\omega_kL_k^*q_{k,n}\\
x_{n+1}=x_n+\lambda_n\proj_Vz_n.
\end{array}
\right.
\end{array}
\end{equation}
Then $(x_n)_{n\in\NN}$ converges weakly to a solution to 
\eqref{e:z45}.
\end{example}
\begin{proof}
Let $\GG$ be the standard product vector space
$\GG_1\times\cdots\times\GG_p$, with generic element
$\boldsymbol{y}=(y_k)_{1\leq k\leq p}$, and equipped with the
scalar product $(\boldsymbol{y},\boldsymbol{y}')\mapsto
\sum_{k=1}^p\omega_k\scal{y_k}{y'_k}$. Further, set
$L\colon\HH\to\GG\colon x\mapsto(L_1x,\ldots,L_px)$ and 
\begin{equation}
\label{e:z61}
B\colon\GG\to 2^{\GG}\colon\boldsymbol{y}\mapsto
\bigl((\Id-F_1+r_1)^{-1}y_1-y_1\bigr)\times\cdots\times 
\bigl((\Id-F_p+r_p)^{-1}y_p-y_p\bigr).
\end{equation}
In this setting, \eqref{e:z44} is a realization of \eqref{e:zp1},
\eqref{e:z45} of \eqref{e:p23}, and
\eqref{e:z94} of \eqref{e:z9} (see
\cite[Example~6.10]{Svva23} for details). The claim
therefore results from Proposition~\ref{p:23}.
\end{proof}

\subsubsection{The method of partial inverses}
\label{sec:pi}

We go back to a formulation already touched upon in
Problem~\ref{prob:2}.
Given a maximally monotone operator $A\colon\HH\to 2^{\HH}$ and 
a closed vector subspace $V$ of $\HH$, Spingarn considered in
\cite{Spin83} the problem
\begin{equation}
\label{e:27a}
\text{find}\;\;x\in V\;\;\text{and}\;\;x^*\in V^\bot\;\;
\text{such that}\;\;x^*\in Ax
\end{equation}
and solved it by applying the proximal point algorithm to the 
partial inverse $A_V$ (see Example~\ref{ex:21}). The resulting
algorithm is called the \emph{method of partial inverses}. The 
following is a relaxed version of the convergence result of 
\cite[Theorem~4.1(i)]{Spin83} (see \cite[Theorem~2.4]{Optl14}).

\begin{theorem}
\label{t:13}
Let $A\colon\HH\to 2^{\HH}$ be a maximally monotone operator, let
$V$ be a closed vector subspace of $\HH$, and let
$(\lambda_n)_{n\in\NN}$ be a sequence in $\left]0,2\right[$ such
that $\sum_{n\in\NN}\lambda_n(2-\lambda_n)=\pinf$. Suppose that
\eqref{e:27a} has solutions, let $x_0\in V$, let
$x^*_0\in V^\bot$, and iterate
\begin{equation}
\label{e:27b}
\begin{array}{l}
\text{for}\;n=0,1,\ldots\\
\begin{array}{l}
\left\lfloor
\begin{array}{l}
p_n=J_A(x_n+x^*_n)\\
p^*_n=x_n+x^*_n-p_n\\
x_{n+1}=x_n-\lambda_n \proj_Vp^*_n\\
x^*_{n+1}=x^*_n-\lambda_n \proj_{V^\bot}p_n.
\end{array}
\right.\\[2mm]
\end{array}
\end{array} 
\end{equation}
Then the following hold:
\begin{enumerate}
\item
\label{t:13i}
$\proj_Vp_n-x_n\to 0$ and $\proj_{V^\bot}p^*_n-x^*_n\to 0$.
\item
\label{t:13ii}
There exists a solution $(x,x^*)$ to \eqref{e:27a} such that
$x_n\weakly x$ and $x^*_n\weakly x^*$.
\end{enumerate}
\end{theorem}
\begin{proof}
Set 
\begin{equation}
\label{e:19k}
(\forall n\in\NN)\quad z_n=x_n+x^*_n
\end{equation}
and note that, since $(x_n)_{n\in\NN}$ lies in $V$ and
$(x^*_n)_{n\in\NN}$ lies in $V^\bot$, \eqref{e:27b} can be
rewritten as
\begin{equation}
\label{e:27bb}
\begin{array}{l}
\text{for}\;n=0,1,\ldots\\
\begin{array}{l}
\left\lfloor
\begin{array}{l}
p_n=J_{A}(x_n+x^*_n)\\
p^*_n=x_n+x^*_n-p_n\\
x_{n+1}=x_n+\lambda_n(\proj_{V}p_n-x_n)\\
x^*_{n+1}=x^*_n+\lambda_n(\proj_{V^\bot}p^*_n-x^*_n).
\end{array}
\right.\\[2mm]
\end{array}
\end{array}
\end{equation}
Thus,
\begin{align}
\label{e:30b}
(\forall n\in\NN)\quad&
\proj_V\bigg(\frac{z_{n+1}-z_n}{\lambda_n}+z_n\bigg)
+\proj_{V^\bot}\bigg(z_n-\bigg(
\frac{z_{n+1}-z_n}
{\lambda_n}+z_n\bigg)\bigg)\nonumber\\
&=\proj_V\bigg(\frac{z_{n+1}-z_n}{\lambda_n}+z_n\bigg)
+\proj_{V^\bot}\bigg(\frac{z_n-z_{n+1}}
{\lambda_n}\bigg)\nonumber\\
&=\proj_V\bigg(\frac{x_{n+1}-x_n}{\lambda_n}+x_n
\bigg)+\proj_{V^\bot}
\bigg(\frac{x^*_n-x^*_{n+1}}
{\lambda_n}\bigg)
\nonumber\\
&=\proj_Vp_n+\proj_{V^\bot}(x^*_n-p^*_n)
\nonumber\\
&=\proj_V p_n+\proj_{V^\bot}(p_n-x_n)
\nonumber\\
&=p_n\nonumber\\
&=J_Az_n.
\end{align}
Hence, it follows from \eqref{e:19k}, \eqref{e:27bb}, and 
Example~\ref{ex:21b} that 
\begin{equation}
\label{e:20x}
(\forall n\in\NN)\quad z_{n+1}=
z_n+\lambda_n\big(J_{A_V}z_n-z_n\big). 
\end{equation}
Altogether, we derive from Theorem~\ref{t:70}\ref{t:70i} that 
\begin{equation}
\label{e:22d}
J_{A_V}z_n-z_n \to 0
\end{equation}
and that there exists 
$z\in\zer A_V$ such that
\begin{equation}
\label{e:22c}
z_n\weakly z.
\end{equation}

\ref{t:13i}: 
In view of \eqref{e:27bb}, \eqref{e:19k}, Example~\ref{ex:21b}, 
and \eqref{e:22d}, we have
\begin{equation}
\label{e:21m}
\proj_Vp_n-x_n
=\proj_V(J_{A_V}z_n)-x_n=\proj_V\big(J_{A_V}z_n-z_n\big)\to 0
\end{equation}
and 
\begin{equation}
\label{e:21n}
x^*_n-\proj_{V^\bot}p^*_n
=\proj_{V^\bot}(p_n-x_n)
=\proj_{V^\bot}J_Az_n
=\proj_{V^\bot}\big(z_n-J_{A_V}z_n\big)
\to 0.
\end{equation}

\ref{t:13ii}: 
As seen above $z\in\zer A_V$. Now set 
$(x,x^*)=(\proj_Vz,\proj_{V^\bot}z)$.
Then Example~\ref{ex:21}\ref{ex:21ii} guarantees that $(x,x^*)$
solves \eqref{e:27a}. In addition, since $\proj_V$ and 
$\proj_{V^\bot}$ are linear and continuous, they are 
weakly continuous. We conclude that 
$x_n=\proj_Vz_n\weakly\proj_Vz=x$ and 
$x^*_n=\proj_{V^\bot}z_n\weakly \proj_{V^\bot}z=x^*$. 
\end{proof}

\begin{example}
\label{ex:88}
In Theorem~\ref{t:13}, let $f\in\Gamma_0(\HH)$ be such that
$0\in\sri(\dom f-V)$, set $A=\partial f$, and suppose that $f$
admits minimizers over $V$. Then \eqref{e:27a} amounts to finding
a solution to the Fenchel dual pair
\begin{equation}
\label{e:cauq3}
\minimize{x\in V}{f(x)}\quad\text{and}\quad
\minimize{x^*\in V^\bot}{f^*(x^*)}.
\end{equation}
In this case, given $x_0\in V$ and $x^*_0\in V^\bot$, 
the method of partial inverses \eqref{e:27b} iterates
\begin{equation}
\label{e:cauq2}
\begin{array}{l}
\text{for}\;n=0,1,\ldots\\
\begin{array}{l}
\left\lfloor
\begin{array}{l}
p_n=\prox_f(x_n+x^*_n)\\
p^*_n=x_n+x^*_n-p_n\\
x_{n+1}=x_n-\lambda_n \proj_Vp^*_n\\
x^*_{n+1}=x^*_n-\lambda_n \proj_{V^\bot}p_n
\end{array}
\right.\\[2mm]
\end{array}
\end{array} 
\end{equation}
and Theorem~\ref{t:13}\ref{t:13ii} guarantees that there exists a
primal-dual solution $(x,x^*)$ of \eqref{e:cauq3} such that
$x_n\weakly x$ and $x^*_n\weakly x^*$.
\end{example}

Algorithm \eqref{e:27b} has many applications in convex
optimization, e.g.,
\cite{Idri89,Lema89,Leno17,Penn02,Spin83,Spin85,Spin87}. As shown
in \cite{Rock19}, it also constitutes the basic building block of
the \emph{progressive hedging algorithm} in stochastic programming
\cite{Rock91}.

Although the method of partial inverses \eqref{e:27b} is presented
in the context of the simple problem \eqref{e:27a}, it has far
reaching ramifications. We present below an application proposed in
\cite{Optl14}, where it is applied to Problem~\ref{prob:9}. In
terms of Framework~\ref{f:1}, this approach can be seen as
a rephrasing of Problem~\ref{prob:9} as an instance of
\eqref{e:27a} in $\XXX=\HH\oplus\GG_1\oplus\cdots\oplus\GG_p$.

\begin{proposition}
\label{p:2013}
Let $0<p\in\NN$, let $A\colon\HH\to 2^{\HH}$ be maximally 
monotone, and, for every $k\in\{1\ldots,p\}$, let $\GG_k$ be a real
Hilbert space, let $B_k\colon\GG_k\to 2^{\GG_k}$ be maximally 
monotone, and let $L_k\in\BL(\HH,\GG_k)$. Suppose that the set 
$Z$ of solutions to the inclusion
\begin{equation}
\label{e:24p}
\text{find}\;\;x\in\HH\;\;\text{such that}\;\;
0\in Ax+\sum_{k=1}^pL_k^*\big(B_k(L_kx)\big)
\end{equation}
is not empty and let $Z^*$ be the set of solutions to the dual
inclusion
\begin{multline}
\label{e:24d}
\text{find}\;\;y^*_1\in\GG_1,\ldots,y^*_p\in\GG_p\;\;
\text{such that}\\
\biggl(\exi x\in A^{-1}\biggl(-\Sum_{k=1}^pL_k^*y_k^*\biggr)
\biggr)\bigl(\forall k\in\{1,\ldots,p\}\bigr)\;\;
L_kx\in B_k^{-1}y^*_k.
\end{multline}
Let $x_0\in\HH$ and let $(\lambda_n)_{n\in\NN}$ be a sequence in
$\left]0,2\right[$ such that 
$\sum_{n\in\NN}\lambda_n(2-\lambda_n)=\pinf$. Set
\begin{equation}
U=\biggl(\Id+\sum_{k=1}^pL_k^*\circ L_k\biggr)^{-1} 
\end{equation}
and, for every
$k\in\{1,\ldots,p\}$, let $y^*_{k,0}\in\GG_k$ and set
$y_{k,0}=L_kx_0$. Additionally, set
\begin{equation}
x^*_0=-\sum_{k=1}^pL_k^*y^*_{k,0}, 
\end{equation}
and iterate
\begin{equation}
\label{e:2013j}
\begin{array}{l}
\text{for}\;n=0,1,\ldots\\
\left\lfloor
\begin{array}{l}
p_n=J_A(x_n+x^*_n)\\
p^*_n=x_n+x^*_n-p_n\\
\text{for}\;k=1,\ldots,p\\
\left\lfloor
\begin{array}{l}
q_{k,n}=J_{B_k}(y_{k,n}+y^*_{k,n})\\
q^*_{k,n}=y_{k,n}+y^*_{k,n}-q_{k,n}\\
\end{array}
\right.\\[1mm]
t_n=U\big(p^*_n+\sum_{k=1}^pL_k^*q^*_{k,n}\big)\\
w_n=U\big(p_n+\sum_{k=1}^pL_k^*q_{k,n}\big)\\
x_{n+1}=x_n-\lambda_n t_n\\
x^*_{n+1}=x^*_n+\lambda_n(w_n-p_n)\\
\text{for}\;k=1,\ldots,p\\
\left\lfloor
\begin{array}{l}
y_{k,n+1}=y_{k,n}-\lambda_n L_kt_n\\
y^*_{k,n+1}=y^*_{k,n}+\lambda_n(L_kw_n-q_{k,n}).
\end{array}
\right.\\[1mm]
\end{array}
\right.\\[2mm]
\end{array}
\end{equation}
Then there exist $x\in Z$ and $(y^*_k)_{1\leq k\leq p}\in Z^*$  
such that $x_n\weakly{x}$ ~and, 
for every $k\in\{1,\ldots,p\}$, $y^*_{k,n}\weakly y^*_k$.
\end{proposition}
\begin{proof}
Define 
\begin{equation}
\label{e:GG}
\begin{cases}
\GG=\GG_1\oplus\cdots\oplus\GG_p\\
B\colon\GG\to 2^{\GG}\colon(y_1,\ldots,y_p)\mapsto
B_1y_1\times\cdots\times B_py_p\\
L\colon\HH\to\GG\colon x\mapsto(L_1x,\ldots,L_px)
\end{cases}
\end{equation}
and note that $L^*\colon\GG\to\HH\colon(y^*_1,\ldots,y^*_p)\mapsto 
L^*_1y^*_1+\cdots+L^*_py^*_p$. Moreover set, for every
$n\in\NN$, $q_n=(q_{k,n})_{1\leq k\leq p}$,
$q^*_n=(q^*_{k,n})_{1\leq k\leq p}$, 
$y_n=(y_{k,n})_{1\leq k\leq p}$, 
and $y^*_n=(y^*_{k,n})_{1\leq k\leq p}$.
In this setting, $B$ is maximally monotone and
$J_B\colon (y_k)_{1\leq k\leq p}\mapsto 
(J_{B_k}y_k)_{1\leq k\leq p}$
(Example~\ref{ex:r1}), so that
\eqref{e:2013j} can be rewritten as
\begin{equation}
\label{e:2013jj}
\begin{array}{l}
\text{for}\;n=0,1,\ldots\\
\left\lfloor
\begin{array}{l}
p_n=J_A(x_n+x^*_n)\\
q_n=J_B(y_n+y^*_n)\\
p^*_n=x_n+x^*_n-p_n\\
q^*_n=y_{n}+y^*_n-q_n\\
t_n=U\big(p^*_n+L^*q^*_n\big)\\
w_n=U\big(p_n+L^*q_n\big)\\
x_{n+1}=x_n-\lambda_n t_n\\
y_{n+1}=y_n-\lambda_n Lt_n\\
x^*_{n+1}=x^*_n+\lambda_n(w_n-p_n)\\
y^*_{n+1}=y^*_n+\lambda_n(Lw_n-q_n).
\end{array}
\right.\\[1mm]
\end{array}
\end{equation}
Let us introduce 
\begin{equation}
\label{e:15A}
\begin{cases}
\XXX=\HH\oplus\GG\\
\boldsymbol{V}=\menge{(x,y)\in\XXX}{Lx=y}\\
\boldsymbol{Z}=\menge{(x,y^*)\in\XXX}{-L^*y^*\in Ax\;\;
\text{and}\;\;y^*\in B(Lx)}\\
\boldsymbol{A}\colon\XXX\to 2^{\XXX}\colon(x,y)\mapsto 
Ax\times By\\
\boldsymbol{S}=\menge{(\boldsymbol{x},\boldsymbol{x}^*)\in
\boldsymbol{V}\times\boldsymbol{V}^\bot}{
\boldsymbol{x}^*\in\boldsymbol{A}\boldsymbol{x}}
\end{cases}
\end{equation}
and observe that
\begin{equation}
\label{e:15b}
\begin{cases}
\boldsymbol{V}^\bot=\menge{(x^*,y^*)\in\XXX}{x^*=-L^*y^*}\\
\boldsymbol{S}=
\menge{\big((x,Lx),(-L^*y^*,y^*)\big)\in\XXX\times\XXX}
{(x,y^*)\in\boldsymbol{Z}}.
\end{cases}
\end{equation}
Then Lemma~\ref{l:37z}\ref{l:37ziii} implies that
\begin{equation}
\label{e:2013-09-21}
\text{\eqref{e:24p} admits solutions}
\;\Leftrightarrow\;\boldsymbol{Z}\neq\emp
\;\Leftrightarrow\;\boldsymbol{S}\neq\emp.
\end{equation}
Now define $(\forall n\in\NN)$ 
$\boldsymbol{p}_n=(p_n,q_n)$,
$\boldsymbol{p}^*_n=(p^*_n,q^*_n)$, 
$\boldsymbol{x}_n=(x_n,y_n)$, and
$\boldsymbol{x}^*_n=(x^*_n,y^*_n)$.
Then $\boldsymbol{x}_0\in\boldsymbol{V}$ and 
$\boldsymbol{x}^*_0\in\boldsymbol{V}^\bot$.
Moreover, by Lemma~\ref{l:0616} and Example~\ref{ex:r1},
$\boldsymbol{A}$ is maximally monotone and 
\begin{equation}
\label{e:15R}
(\forall n\in\NN)\quad J_{\boldsymbol{A}}
(\boldsymbol{x}_n+\boldsymbol{x}^*_n)=
\big(J_A(x_n+x^*_n),J_B(y_n+y^*_n)\big).
\end{equation}
Furthermore, since $U=(\Id+L^*\circ L)^{-1}$, it follows from
\eqref{e:15A} and \cite[Example~29.19]{Livre1} that
\begin{equation}
\label{e:15S}
(\forall n\in\NN)\quad\proj_{\boldsymbol{V}^\bot}\boldsymbol{p}_n
=\Big(p_n-U(p_n+L^*q_n),q_n-L\big(U(p_n+L^*q_n)\big)\Big)
\end{equation}
and 
\begin{equation}
\label{e:15T}
(\forall n\in\NN)\quad\proj_{\boldsymbol{V}}
\boldsymbol{p}^*_n
=\Big(U(p^*_n+L^*q^*_n),L\big(U(p^*_n+L^*q^*_n)\big)\Big).
\end{equation}
Combining \eqref{e:15R}, \eqref{e:15S}, and \eqref{e:15T}, we
rewrite \eqref{e:2013jj} as 
\begin{equation}
\label{e:2013jjj}
\begin{array}{l}
\text{for}\;n=0,1,\ldots\\
\left\lfloor
\begin{array}{l}
\boldsymbol{p}_n=J_{\boldsymbol{A}}
(\boldsymbol{x}_n+\boldsymbol{x}^*_n)\\
\boldsymbol{p}^*_n=\boldsymbol{x}_n+\boldsymbol{x}^*_n-
\boldsymbol{p}_n\\
\boldsymbol{x}_{n+1}=\boldsymbol{x}_n-\lambda_n
\proj_{\boldsymbol{V}}\boldsymbol{p}^*_n\\
\boldsymbol{x}^*_{n+1}=\boldsymbol{x}^*_n-\lambda_n
\proj_{\boldsymbol{V}^\bot}\boldsymbol{p}_n.
\end{array}
\right.\\[1mm]
\end{array}
\end{equation}
In turn, Theorem~\ref{t:13}\ref{t:13ii} implies that 
there exists $({\boldsymbol{x}},{\boldsymbol{x}^*})
\in\boldsymbol{S}$ such that 
$\boldsymbol{x}_n\weakly{\boldsymbol{x}}$ and
$\boldsymbol{x}^*_n\weakly{\boldsymbol{x}^*}$.
We then derive from \eqref{e:15b} that there exists
$(x,y^*)\in\boldsymbol{Z}$ such that
$(x_n,y^*_n)\weakly(x,y^*)$. We complete the proof by invoking
Lemma~\ref{l:37z}\ref{l:37zii}.
\end{proof}

\subsubsection{Renorming}
\label{sec:ren1}

The potency of the proximal point algorithm can be further extended
by setting it up in a renormed space. In terms of
Framework~\ref{f:1}, the guiding principle lies in the embedding of
Example~\ref{ex:f9}. Here is a weak convergence result.

\begin{proposition}
\label{p:3}
Let $M\colon\HH\to 2^{\HH}$ be a maximally monotone operator such
that $Z=\zer M\neq\emp$, let $U\in\BL(\HH)$ be a self-adjoint 
strongly monotone operator, and let $\mathcal{X}$ be the real 
Hilbert space obtained by endowing $\HH$ with the scalar product 
$(x,y)\mapsto\scal{Ux}{y}$. Let $x_0\in\HH$, let 
$(\lambda_n)_{n\in\NN}$ be a sequence in $\left]0,2\right[$, and
let $(\gamma_n)_{n\in\NN}$ be a sequence in $\RPP$. Iterate
\begin{equation}
\label{e:1U}
\begin{array}{l}
\text{for}\;n=0,1,\ldots\\
\left\lfloor
\begin{array}{l}
u_n=\gamma_n^{-1}Ux_n\\
p_n=\big(\gamma_n^{-1}U+M\big)^{-1}u_n\\
x_{n+1}=x_n+\lambda_n(p_n-x_n)
\end{array}
\right.
\end{array}
\end{equation}
and suppose that one of the following holds:
\begin{enumerate}
\item
\label{p:3i}
$\sum_{n\in\NN}\lambda_n(2-\lambda_n)=\pinf$ and 
$(\forall n\in\NN)$ $\gamma_n=1$.
\item
\label{p:3ii}
$\sum_{n\in\NN}\gamma_n^2=\pinf$ and 
$(\forall n\in\NN)$ $\lambda_n=1$.
\item
\label{p:3iii}
$\inf_{n\in\NN}\lambda_n>0$, $\sup_{n\in\NN}\lambda_n<2$, and
$\inf_{n\in\NN}\gamma_n>0$.
\end{enumerate}
Then $(x_n)_{n\in\NN}$ converges weakly to a point in $Z$.
\end{proposition}
\begin{proof}
In view of Lemma~\ref{l:9}\ref{l:9i} and Example~\ref{ex:b12},
\eqref{e:1U} is just the proximal point algorithm \eqref{e:1}
applied to the maximally monotone operator $U^{-1}\circ M$ in
$\mathcal{X}$. Since weak convergences in $\HH$ and $\mathcal{X}$
coincide, the claims follow from Lemma~\ref{l:9}\ref{l:9i-} and
Theorem~\ref{t:70}.
\end{proof}

\begin{remark}
\label{r:wr1}
In terms of the warped resolvent of Section~\ref{sec:wr}, the
update in \eqref{e:1U} can be written as 
$x_{n+1}=x_n+\lambda_n(J_{\gamma_n M}^Ux_n-x_n)$.
\end{remark}

Likewise, Theorem~\ref{t:70s} leads to a strongly convergent
algorithm.

\begin{proposition}
\label{p:3s}
Let $M\colon\HH\to 2^{\HH}$ be a maximally monotone operator such
that $Z=\zer M\neq\emp$, let $U\in\BL(\HH)$ be a self-adjoint 
strongly monotone operator, and let $\mathcal{X}$ be the real 
Hilbert space obtained by endowing $\HH$ with the scalar product 
$(x,y)\mapsto\scal{Ux}{y}$. Let $x_0\in\HH$, let 
$(\lambda_n)_{n\in\NN}$ be a sequence in $\left]0,1\right]$ such
that $\inf_{n\in\NN}\lambda_n>0$, and let 
$(\gamma_n)_{n\in\NN}$ be a sequence in $\RPP$ such that
$\inf_{n\in\NN}\gamma_n>0$. Iterate
\begin{equation}
\label{e:1Us}
\begin{array}{l}
\text{for}\;n=0,1,\ldots\\
\left\lfloor
\begin{array}{l}
u_n=\gamma_n^{-1} Ux_n\\
p_n=\big(\gamma_n^{-1} U+M\big)^{-1}u_n\\
x_{n+1}=\Qq\bigl(x_0,x_n,x_n+\lambda_n(p_n-x_n)\bigr),
\end{array}
\right.
\end{array}
\end{equation}
where $\Qq$ is defined in Lemma~\ref{l:2}. Then $(x_n)_{n\in\NN}$
converges strongly to $\proj_Zx_0$.
\end{proposition}
\begin{proof}
It follows from Lemma~\ref{l:9}\ref{l:9i} and 
Example~\ref{ex:b12} that applying the algorithm \eqref{e:2} to 
the maximally monotone operator $U^{-1}\circ M$ in 
$\mathcal{X}$ yields \eqref{e:1Us}.
Since strong convergences in $\HH$ and $\mathcal{X}$ coincide, the
assertion follows from Lemma~\ref{l:9}\ref{l:9i-} and
Theorem~\ref{t:70s}.
\end{proof}

Although the inversion of the operators
$(\gamma_n^{-1}U+M)_{n\in\NN}$ in \eqref{e:1U} and \eqref{e:1Us}
may be intimidating, we show below that the renormed proximal point
algorithm leads to important instances of fully executable
splitting algorithms. First, we revisit a classical minimization
problem and recover an algorithm known as the \emph{proximal
Landweber method}.

\begin{example}
\label{ex:1Ua}
Let $\varphi\in\Gamma_0(\HH)$, let $\mu\in\RPP$, and let $y\in\GG$.
Suppose that $0\neq L\in\BL(\HH,\GG)$ and that the set $Z$ of
solutions to the optimization problem
\begin{equation}
\label{e:1Ua}
\minimize{x\in\HH}{\varphi(x)+\dfrac{\mu}{2}\|Lx-y\|^2}
\end{equation}
is not empty. Without loss of generality (rescale), assume that
$\mu\|L\|^2<1$. Let $x_0\in\HH$, let $(\lambda_n)_{n\in\NN}$ be a
sequence in $\left]0,2\right[$ such that
$\sum_{n\in\NN}\lambda_n(2-\lambda_n)=\pinf$, and iterate
\begin{equation}
\label{e:1U1}
\begin{array}{l}
\text{for}\;n=0,1,\ldots\\
\left\lfloor
\begin{array}{l}
u_n=x_n-\mu L^*(Lx_n)\\
p_n=\prox_\varphi(u_n+\mu L^*y)\\
x_{n+1}=x_n+\lambda_n(p_n-x_n).
\end{array}
\right.
\end{array}
\end{equation}
Then $(x_n)_{n\in\NN}$ converges weakly to a point in $Z$.
\end{example}
\begin{proof}
Set $f=\varphi-\mu\scal{\cdot}{L^*y}$, 
$M=\partial(\varphi+\mu\|L\cdot-y\|^2/2)=
\partial f+\mu L^*\circ L$, and
$U=\Id-\mu L^*\circ L$. Then $f\in\Gamma_0(\HH)$, $M$ is
maximally monotone with $\zer M=Z$ by virtue of Example~\ref{ex:1},
$U\in\BL(\HH)$ is self-adjoint and strongly monotone, and 
$(U+M)^{-1}=\prox_f=\prox_\varphi(\cdot+\mu L^*y)$.
Consequently, \eqref{e:1U1} is the implementation of \eqref{e:1U} 
with, for every $n\in\NN$, $\gamma_n=1$, and 
Proposition~\ref{p:3}\ref{p:3i} brings the conclusion.
\end{proof}

Next, we return to the primal-dual composite inclusion framework of
Problem~\ref{prob:3} and approach it via Framework~\ref{f:1} where,
as discussed in Example~\ref{ex:f3}, the embedding is based on
$\XXX=\HH\oplus\GG$ and the Kuhn--Tucker operator $\kut$ of
Lemma~\ref{l:31z}.

\begin{example}
\label{ex:1Ub}
Let $A\colon\HH\to 2^{\HH}$ and $B\colon\GG\to 2^{\GG}$ be
maximally monotone, and let $L\in\BL(\HH,\GG)$. Suppose that the 
set $Z$ of solutions to the primal inclusion 
\begin{equation}
\label{e:1Up}
\text{find}\;\;{x}\in\HH\;\;\text{such that}\;\;
0\in Ax+L^*\big(B(Lx)\big)
\end{equation}
is not empty and let $Z^*$ be the set of solutions to 
the dual inclusion
\begin{equation}
\label{e:1Ud}
\text{find}\;\;y^*\in\GG\;\;\text{such that}\;\;
0\in-L\bigl(A^{-1}(-L^*y^*)\bigr)+B^{-1}y^*.
\end{equation}
Let $(\lambda_n)_{n\in\NN}$ be a sequence in $\left]0,2\right[$
such that $\sum_{n\in\NN}\lambda_n(2-\lambda_n)=\pinf$, let
$x_0\in\HH$, let $y^*_0\in\GG$, and let $\sigma\in\RPP$ and
$\tau\in\RPP$ be such that $\tau\sigma\|L\|^2<1$. Iterate
\begin{equation}
\label{e:2013v}
\begin{array}{l}
\text{for}\;n=0,1,\ldots\\
\left\lfloor
\begin{array}{l}
x^*_n=\tau L^*y_n^*\\
p_n=J_{\tau A}(x_n-x^*_n)\\
y_n=\sigma L(2p_n-x_n)\\
q^*_n=J_{\sigma B^{-1}}(y^*_n+y_n)\\
x_{n+1}=x_n+\lambda_n(p_n-x_n)\\
y^*_{n+1}=y^*_n+\lambda_n(q^*_n-y^*_n).
\end{array}
\right.\\[1mm]
\end{array}
\end{equation}
Then there exist $x\in Z$ and $y^*\in Z^*$ such that $x_n\weakly x$
and $y^*_n\weakly y^*$.
\end{example}
\begin{proof}
Set $\XXX=\HH\oplus\GG$ and 
\begin{equation}
\label{e:UU}
\begin{cases}
\kut\colon\XXX\to 2^{\XXX}\colon(x,y^*)\mapsto
\bigl(Ax+L^*y^*\bigr)\times\bigl(-Lx+B^{-1}y^*\bigr)\\
\boldsymbol{U}\colon\XXX\to\XXX\colon(x,y^*)\mapsto
\bigl(\tau^{-1}x-L^*y^*,-Lx+\sigma^{-1}y^*\bigr).
\end{cases}
\end{equation}
As seen in Lemma~\ref{l:31z}\ref{l:31zi}--\ref{l:31zii},
$\kut$ is the maximally monotone Kuhn--Tucker operator
associated with \eqref{e:1Up}--\eqref{e:1Ud} and to prove the claim
it is enough to show that $(x_n,y^*_n)_{n\in\NN}$ converges weakly
to a point in $\zer\kut$, which we shall derive from
Proposition~\ref{p:3}\ref{p:3i}. It is clear that 
$\boldsymbol{U}\in\BL(\XXX)$ is self-adjoint. Now set
$\beta=1-\sqrt{\sigma\tau}\|L\|$. Then, since
$\tau\sigma\|L\|^2<1$, $\beta\in\zeroun$ and, for every
$(x,y^*)\in\XXX$, the Cauchy--Schwarz inequality yields
\begin{align}
\scal{\boldsymbol{U}(x,y^*)}{(x,y^*)}_{\XXX}
&=\tau^{-1}\|x\|^2-2\scal{Lx}{y^*}+\sigma^{-1}\|y^*\|^2
\nonumber\\
&\geq\tau^{-1}\|x\|^2-2\sqrt{\tau\sigma}\|L\|\,
\bigg\|\dfrac{x}{\sqrt{\tau}}\bigg\|\,
\bigg\|\dfrac{y^*}{\sqrt{\sigma}}\bigg\|
+\sigma^{-1}\|y^*\|^2
\nonumber\\
&=\tau^{-1}\|x\|^2-2(1-\beta)\,
\bigg\|\dfrac{x}{\sqrt{\tau}}\bigg\|\,
\bigg\|\dfrac{y^*}{\sqrt{\sigma}}\bigg\|
+\sigma^{-1}\|y^*\|^2
\nonumber\\
&=\bigg(\bigg\|\dfrac{x}{\sqrt{\tau}}\bigg\|-
\bigg\|\dfrac{y^*}{\sqrt{\sigma}}\bigg\|\bigg)^2
+2\beta\bigg\|\dfrac{x}{\sqrt{\tau}}\bigg\|\,
\bigg\|\dfrac{y^*}{\sqrt{\sigma}}\bigg\|
\nonumber\\
&=(1-\beta)\bigg(\bigg\|\dfrac{x}{\sqrt{\tau}}\bigg\|-
\bigg\|\dfrac{y^*}{\sqrt{\sigma}}\bigg\|\bigg)^2
+\beta\bigg(\bigg\|\dfrac{x}{\sqrt{\tau}}\bigg\|^2+
\bigg\|\dfrac{y^*}{\sqrt{\sigma}}\bigg\|^2\bigg)
\nonumber\\
&\geq\beta\big(\tau^{-1}\|x\|^2+\sigma^{-1}\|y^*\|^2\big)
\nonumber\\
&\geq\beta\min\{\tau^{-1},\sigma^{-1}\}\|(x,y^*)\|_{\XXX}^2,
\end{align}
which confirms that $\boldsymbol{U}$ is strongly monotone. It
remains to show that \eqref{e:2013v} is a realization of
\eqref{e:1U} with the above operators $\kut$ and
$\boldsymbol{U}$. Define $(\forall n\in\NN)$
$\boldsymbol{x}_n=(x_n,y^*_n)$,
$\boldsymbol{p}_n=(p_n,q^*_n)$, and
$\boldsymbol{u}_n=\boldsymbol{U}\boldsymbol{x}_n$.
Then we derive from \eqref{e:2013v} and \eqref{e:ee14} that
\begin{equation}
\label{e:dN1}
(\forall n\in\NN)\quad
\begin{cases}
x_n-p_n-\tau L^*y^*_n\in\tau Ap_n\\
y_n^*-q_n^*+\sigma L(2p_n-x_n)\in\sigma B^{-1}q_n^*
\end{cases}
\end{equation}
This yields 
$(\forall n\in\NN)$
$\boldsymbol{u}_n-\boldsymbol{U}\boldsymbol{p}_n\in
\kut\boldsymbol{p}_n$, i.e.,
$\boldsymbol{p}_n=(\boldsymbol{U}+
\kut)^{-1}\boldsymbol{u}_n$.
Altogether, \eqref{e:2013v} corresponds to the iteration
\begin{equation}
\begin{array}{l}
\text{for}\;n=0,1,\ldots\\
\left\lfloor
\begin{array}{l}
\boldsymbol{u}_n=\boldsymbol{U}\boldsymbol{x}_n\\
\boldsymbol{p}_n=\big(\boldsymbol{U}+\kut\big)^{-1}
\boldsymbol{u}_n\\
\boldsymbol{x}_{n+1}=\boldsymbol{x}_n+\lambda_n(\boldsymbol{p}_n
-\boldsymbol{x}_n),
\end{array}
\right.
\end{array}
\end{equation}
which is precisely \eqref{e:1U} with $(\forall n\in\NN)$
$\gamma_n=1$. 
\end{proof}

\begin{remark}
\label{r:U}
Here are a few observations regarding Example~\ref{ex:1Ub}. 
\begin{enumerate}
\item
\label{r:Ui}
We have derived weak convergence from 
Proposition~\ref{p:3}\ref{p:3i}. Using items 
\ref{p:3ii} or \ref{p:3iii} in Proposition~\ref{p:3} leads to
alternative forms of \eqref{e:2013v} involving proximal parameters
$(\gamma_n)_{n\in\NN}$.
\item
\label{r:Uii}
It is straightforward to derive a strongly convergent best
approximation variant of \eqref{e:2013v} from
Proposition~\ref{p:3s} by following the same pattern as in the
proof of Example~\ref{ex:1Ub}, i.e., applying \eqref{e:1Us} to the
operators $\kut$ and $\boldsymbol{U}$ of \eqref{e:UU}.
\item
Algorithm \eqref{e:2013v} can be adapted to Problem~\ref{prob:9} by
applying it to the setting of \eqref{e:GG} and using 
Example~\ref{ex:r1}.
\item
\label{r:Uiii}
Let $f\in\Gamma_0(\HH)$ and $g\in\Gamma_0(\GG)$, and set
$A=\partial f$ and $B=\partial g$ in Example~\ref{ex:1Ub}, which
corresponds to the primal-dual minimization setting of
Problem~\ref{prob:35}. The specialization of Example~\ref{ex:1Ub}
to this minimization problem appears in 
\cite[Theorem~3.2]{Cond13}, where \eqref{e:2013v} is called the
\emph{Chambolle--Pock algorithm} because it collapses to the
algorithm proposed in \cite[Algorithm~I]{Cham11} in Euclidean
spaces when $(\forall n\in\NN)$ $\lambda_n=1$ (see \cite{Cond23}
for variations on this algorithm). The fact that the
Chambolle--Pock algorithm is a renormed proximal point algorithm
was first observed in \cite{Heyu12}.
\end{enumerate}
\end{remark}

\section{Douglas--Rachford splitting}
\label{sec:dr}

\subsection{Preview}
\label{sec:odr}

The Douglas--Rachford splitting algorithm is an implicit
alternating direction method designed in \cite{Doug56} to solve the
matrix equation $Ax+Bx=f$, where $A$ and $B$ are 
positive-definite matrices arising from the discretization of
partial differentiation operators. It is described by the iteration
process
\begin{equation}
\label{e:dr1}
\begin{array}{l}
\text{for}\;n=0,1,\ldots\\
\left\lfloor
\begin{array}{l}
x_{n+1/2}-x_n+Ax_{n+1/2}+Bx_n=f\\
x_{n+1}-x_n+Ax_{n+1/2}+Bx_{n+1}=f.
\end{array}
\right.\\[2mm]
\end{array}
\end{equation}
In 1968, Lieutaud \cite{Lieu68} (see also \cite{Lieu69}) proposed
an infinite-dimensional nonlinear generalization of the method by
showing that \eqref{e:dr1} can be extended to single-valued
hemicontinuous monotone operators with $\dom A=\dom B=\HH$. In
particular, he established in \cite{Lieu68} that, with the
additional assumption that $A$ or $B$ is strongly monotone,
$(x_n)_{n\in\NN}$ converges strongly to some $x\in\HH$ which
satisfies $Ax+Bx=f$. The investigation of the method for general
set-valued maximally monotone operators was initiated in
\cite{Lion79}, with subsequent improvements in
\cite{Livre1,Baus17,Joca09,Ecks92,Svai11}. See also \cite{Xue23b}
for further analysis.

To chart the path from the original Douglas--Rachford algorithm to
its modern version for monotone set-valued operators, let us
go back to the matrix setting. Upon eliminating the
intermediate variables $(x_{n+1/2})_{n\in\NN}$ in \eqref{e:dr1} and
noting that $AJ_A=\Id-J_A$, we obtain
\begin{align}
\label{e:dr2}
(\forall n\in\NN)\quad
x_{n+1}&=J_B\big(x_n-AJ_A(x_n-Bx_n+f)+f\big)
\nonumber\\
&=J_B\big(Bx_n+J_A(x_n-Bx_n+f)\big).
\end{align}
Now set $(\forall n\in\NN)$ $x_n=J_By_n$. Then we derive from 
\eqref{e:dr2} that 
\begin{align}
\label{e:dr3}
(\forall n\in\NN)\quad y_{n+1}
&=BJ_By_n+J_A(J_By_n-BJ_By_n+f)
\nonumber\\
&=y_n-J_By_n+J_A(2J_By_n-y_n+f),
\end{align}
which leads to the recursion
\begin{equation}
\label{e:dr4}
\begin{array}{l}
\text{for}\;n=0,1,\ldots\\
\left\lfloor
\begin{array}{l}
x_n=J_By_n\\
z_n=J_A(2x_n-y_n+f)\\
y_{n+1}=y_n+z_n-x_n.
\end{array}
\right.\\[2mm]
\end{array}
\end{equation}
As noted in \cite{Lion79}, unlike \eqref{e:dr1}, this algorithm is
well defined for arbitrary maximally monotone set-valued operators
and is now referred to as the Douglas--Rachford
splitting algorithm in this context.

\begin{remark}
\label{r:2}
In particular, upon setting $B=0$ and $f=0$ in \eqref{e:dr4} and
assuming that $A\colon\HH\to\HH$ is hemicontinuous and strongly
monotone, it follows from Lieutaud's result \cite{Lieu68} that the
sequence $(x_n)_{n\in\NN}$ generated by the recursion
\begin{equation}
(\forall n\in\NN)\quad x_{n+1}=J_Ax_n
\end{equation}
converges strongly to a zero of $A$. This is actually the first 
instance of convergence of the proximal point algorithm, which has
been attributed to later work in the literature. The case when
$A$ and $B$ are gradients of convex functions was also considered
in \cite{Lieu68} in connection with the minimization of the sum of
two differentiable convex functions.
\end{remark}

\subsection{Weak convergence}

We present results for a form of the Douglas--Rachford algorithm
\eqref{e:dr4} which includes relaxation parameters and a dual
inclusion problem.

\begin{theorem}
\label{t:4}
Let $A\colon\HH\to 2^{\HH}$ and $B\colon\HH\to 2^{\HH}$ be 
maximally monotone, let $(\lambda_n)_{n\in\NN}$ be a sequence in 
$\left]0,2\right[$ such that 
$\sum_{n\in\NN}\lambda_n(2-\lambda_n)=\pinf$,
and let $\gamma\in\RPP$. Suppose that the 
set $Z$ of solutions to the inclusion 
\begin{equation}
\label{e:29p}
\text{find}\;\;x\in\HH\;\;\text{such that}\;\;0\in Ax+Bx
\end{equation}
is not empty and let $Z^*$ be the set of solutions to the dual
problem 
\begin{equation}
\label{e:29d}
\text{find}\;\;x^*\in\HH\;\;\text{such that}\;\;
0\in -A^{-1}(-x^*)+B^{-1}x^*.
\end{equation}
Let $y_0\in\HH$ and iterate
\begin{equation}
\label{e:dr6}
\begin{array}{l}
\text{for}\;n=0,1,\ldots\\
\left\lfloor
\begin{array}{l}
x_n=J_{\gamma B}y_n\\
x_n^*=\gamma^{-1}(y_n-x_n)\\
z_n=J_{\gamma A}(2x_n-y_n)\\
y_{n+1}=y_n+\lambda_n(z_n-x_n).
\end{array}
\right.\\[2mm]
\end{array}
\end{equation}
Then there exists $y\in\HH$ such that $y_n\weakly y$. 
Now set $x=J_{\gamma B}y$ and $x^*=\moyo{B}{\gamma}y$.
Then the following hold:
\begin{enumerate}
\item
\label{t:4i}
$x_n\weakly x\in Z$.
\item
\label{t:4ii}
$x_n^*\weakly x^*\in Z^*$.
\end{enumerate}
\end{theorem}
\begin{proof}
We rely on the embedding of Example~\ref{ex:f11}. Set 
\begin{equation}
\label{e:1969}
R_{\gamma A}=2J_{\gamma A}-\Id,\;\;
R_{\gamma B}=2J_{\gamma B}-\Id,\;\;\text{and}\;\;
\MMM=\Bigg(\dfrac{R_{\gamma A}\circ R_{\gamma B}+\Id}{2}
\Bigg)^{-1}-\Id.
\end{equation}
Then it follows from \eqref{e:ne} and Lemma~\ref{l:0}\ref{l:0iii}
that $(R_{\gamma A}\circ R_{\gamma B}+\Id)/2$ is firmly
nonexpansive and that $\MMM$ is maximally monotone. In addition, 
\cite[Proposition~26.1(iii)(b)]{Livre1} asserts that
\begin{equation}
\label{e:dr12}
\emp\neq Z=J_{\gamma B}(\zer\MMM),
\end{equation}
while \cite[Proposition~26.1(iii)(c)]{Livre1} asserts that
\begin{equation}
\label{e:dr19}
\emp\neq Z^*=\moyo{B}{\gamma}(\zer\MMM).
\end{equation}
Furthermore, we derive from
\eqref{e:dr6} and \eqref{e:1969} that
\begin{equation}
\label{e:dr11}
(\forall n\in\NN)\quad y_{n+1}
=y_n+\dfrac{\lambda_n}{2}
\big(R_{\gamma A}(R_{\gamma B}y_n)-y_n\big)
=y_n+\lambda_n\big(J_{\MMM}y_n-y_n\big),
\end{equation}
i.e., $(y_n)_{n\in\NN}$ is constructed by the proximal point
algorithm \eqref{e:1} for $\MMM$. Since \eqref{e:dr12} implies that
$\zer\MMM\neq\emp$, Theorem~\ref{t:70}\ref{t:70i} asserts that
\begin{equation}
\label{e:dr20}
J_{\MMM}y_n-y_n\to 0\quad\text{and}\quad
(\exi y\in\zer\MMM)\quad y_n\weakly y. 
\end{equation}
In turn, \eqref{e:dr12} yields $x=J_{\gamma B}y\in Z$,
while \eqref{e:dr6} yields
\begin{equation}
\label{e:dr13}
z_n-x_n=J_{\gamma A}(2x_n-y_n)-x_n=J_{\MMM}y_n-y_n\to 0.
\end{equation}

\ref{t:4i}: Let us set
\begin{equation}
\label{e:go9}
(\forall n\in\NN)\quad z^*_n=\gamma^{-1}(2x_n-y_n-z_n).
\end{equation}
Then \eqref{e:dr6} and \eqref{e:ee14} yield
\begin{equation}
\label{e:10-03a}
(\forall n\in\NN)\quad 
\begin{cases}
(z_n,z^*_n)\in\gra A\\
(x_n,x^*_n)\in\gra B\\
x_n-z_n=\gamma(x^*_n+z^*_n).
\end{cases}
\end{equation}
Since Lemma~\ref{l:0}\ref{l:0iii} asserts that $J_{\gamma B}$ is 
nonexpansive, 
\begin{equation}
(\forall n\in\NN)\quad
\|x_n-x_0\|=\|J_{\gamma B}y_n-J_{\gamma B}y_0\|\leq\|y_n-y_0\|.
\end{equation}
Hence, since $(y_n)_{n\in\NN}$ is bounded, so is 
$(x_n)_{n\in\NN}$. Now take $z\in\WC(x_n)_{n\in\NN}$, say
$x_{k_n}\weakly z$. Then it follows from \eqref{e:dr13},
\eqref{e:dr20}, \eqref{e:go9}, and \eqref{e:10-03a} that
\begin{equation}
z_{k_n}\weakly z,\;
z^*_{k_n}\weakly\gamma^{-1}(z-y),\;
z_n-x_n\to 0,\;\text{and}\;
z^*_n+x^*_n=\gamma^{-1}(x_n-z_n)\to 0.
\end{equation}
In turn, Lemma~\ref{l:30} yields $z\in\zer(A+B)=Z$, 
\begin{equation}
\label{e:31c}
\big(z,\gamma^{-1}(z-y)\big)\in\gra A,\quad\text{and}\quad
\big(z,\gamma^{-1}(y-z)\big)\in\gra B. 
\end{equation}
Hence, \eqref{e:ee14} implies that 
\begin{equation}
\label{e:31x}
z=J_{\gamma B}y.
\end{equation}
Thus, $x=J_{\gamma B}y$ is the unique weak sequential
cluster point of the bounded sequence $(x_n)_{n\in\NN}$ and
therefore, by Lemma~\ref{l:1}\ref{l:1ii}, $x_n\weakly x$. 

\ref{t:4ii}: 
We have $y_n\weakly y\in\zer\MMM$ and, by \ref{t:4i}, 
$x_n\weakly x$.
Hence, $x^*_n=\gamma^{-1}(y_n-x_n)\weakly\gamma^{-1}(y-x)=
\moyo{B}{\gamma}y=x^*$. In view of \eqref{e:dr19}, the
proof is complete.
\end{proof}

\begin{remark}
\label{r:dr}
The convergence result of \cite{Lion79} is that, for the unrelaxed
scheme \eqref{e:dr4}, $(y_n)_{n\in\NN}$ converges weakly to a point
$y\in\HH$ such that $J_{\gamma B}y\in Z$ (see
\cite{Opti04,Ecks92} for the relaxed case). In the special case
when $J_{\gamma B}$ is weakly sequentially continuous, as is the 
case when $\HH$ is finite-dimensional, 
$x_n=J_{\gamma B}y_n\weakly J_{\gamma B}y\in Z$. The key fact
that $(x_n)_{n\in\NN}$ converges weakly to a point in $\zer(A+B)$
without any further assumption was first proved in \cite{Svai11} in
the unrelaxed case. Theorem~\ref{t:4} was established in
\cite[Theorem~26.11]{Livre1}. The component of the proof given
above up to \eqref{e:dr20} exploits an idea from \cite{Ecks92},
that identifies the core iteration of \eqref{e:dr6} as an
instantiation of the proximal point algorithm.
\end{remark}

\begin{remark}
\label{r:dr2}
Connections between the Douglas--Rachford algorithms and the method
of partial inverses of Section~\ref{sec:pi} are discussed in 
\cite[Section~1]{Lawr87}; see also \cite[Section~5]{Ecks92} and 
\cite{Mahe95}. Let us show that we can actually derive 
Theorem~\ref{t:13}\ref{t:13ii} from Theorem~\ref{t:4}. Let
$(x_n)_{n\in\NN}$, $(x^*_n)_{n\in\NN}$, $(p_n)_{n\in\NN}$ and
$(p^*_n)_{n\in\NN}$ be the sequence generated by \eqref{e:27b}
and set $(\forall n\in\NN)$ $y_n=x_n+x_n^*$ and 
$z_n=\proj_V(2p_n-y_n)$. Then \eqref{e:27b} yields
\begin{equation}
(\forall n\in\NN)\;\;\proj_Vp_n^*+\proj_{V^{\perp}}p_n
=\proj_V(y_n-p_n)+p_n-\proj_Vp_n=p_n-z_n.
\end{equation}
Altogether,
\begin{equation}
\label{e:g1}
(\forall n\in\NN)\;\;
p_n=J_Ay_n,\;z_n=\proj_V(2p_n-y_n),\;\text{and}\;
y_{n+1}=y_n+\lambda_n(z_n-p_n).
\end{equation}
In view of Example~\ref{ex:r6}, this recursion is precisely that 
of \eqref{e:dr6} for the operators $(N_V,A)$ with $\gamma=1$. We
therefore derive the following from Theorem~\ref{t:4}:
$(y_n)_{n\in\NN}$ converges weakly to a point $y\in\HH$ and, if we
set $x=J_Ay$ and $x^*=y-J_Ay$, then 
$p_n\weakly x\in\zer(N_V+A)$ and, by Example~\ref{ex:2V},
$p^*_n\weakly x^*\in\zer(N_{V^\bot}+A^{-1})$. Furthermore, 
\eqref{e:31c}--\eqref{e:31x} implies that
$(x,-x^*)=(x,x-y)\in\gra N_V$
and $(x,x^*)=(x,y-x)\in\gra A$. Thus, 
Example~\ref{ex:2V} yields $(x,x^*)\in\gra N_V\cap\gra A$ and
$(x,x^*)$ therefore solves \eqref{e:27a}. Finally, since 
\cite[Equation~(11)]{Joca09} asserts that $J_Ay=\proj_Vy$ and
since $\proj_V$ is weakly continuous, we have
$x_n=\proj_V(x_n+x^*_n)=\proj_Vy_n\weakly\proj_Vy=x$ and
$x^*_n=\proj_{V^\bot}y_n\weakly\proj_{V^\bot}y=y-\proj_Vy=x^*$.
Let us add that, in this setting, the operator $\MMM$ of
\eqref{e:1969} is just the partial inverse $A_V$.
\end{remark}

\begin{remark}
The many application areas of the Douglas--Rachford algorithm
(in its original two-operator form or transposed in product spaces)
include 
road design \cite{Bau16r},
equilibrium problems \cite{Bric12},
biostatistics \cite{Stat21},
signal recovery \cite{Comb07},
traffic theory \cite{Fuku96},
noise removal \cite{Stei10},
and compressive sensing \cite{Yuyu17}
(see also \cite{Lind21} for additional references).
\end{remark}

\subsection{Strong convergence}
As shown in \cite[Counterexample~2]{Sico20}, the convergence of
$(x_n)_{n\in\NN}$ in Theorem~\ref{t:4}\ref{t:4i} is only weak. The
following version based on Theorem~\ref{t:70s} furnishes strong
convergence.

\begin{theorem}
\label{t:4s}
Let $A\colon\HH\to 2^{\HH}$ and $B\colon\HH\to 2^{\HH}$ be 
maximally monotone, suppose that $\zer(A+B)\neq\emp$, let
$y_0\in\HH$, let $(\lambda_n)_{n\in\NN}$ be a sequence in 
$\left]0,1\right]$ such that $\inf_{n\in\NN}\lambda_n>0$, and
let $\gamma\in\RPP$. Iterate
\begin{equation}
\label{e:42s}
\begin{array}{l}
\text{for}\;n=0,1,\ldots\\
\left\lfloor
\begin{array}{l}
x_n=J_{\gamma B}y_n\\
x_n^*=\gamma^{-1}(y_n-x_n)\\
z_n=J_{\gamma A}(2x_n-y_n)\\
y_{n+1}=\Qq\bigl(y_0,y_n,y_n+\lambda_n(z_n-x_n)\bigr),
\end{array}
\right.\\[2mm]
\end{array}
\end{equation}
where $\Qq$ is defined in Lemma~\ref{l:2}.
Let $Z$ and $Z^*$ be the sets of
solutions to \eqref{e:29p} and \eqref{e:29d}, respectively. Then
the following hold:
\begin{enumerate}
\item
$(x_n)_{n\in\NN}$ converges strongly to a point in $Z$.
\item
$(x_n^*)_{n\in\NN}$ converges strongly to a point in $Z^*$.
\end{enumerate}
\end{theorem}
\begin{proof}
Define $\MMM$ as in \eqref{e:1969} and set $y=\proj_{\zer\MMM}y_0$,
$x=J_{\gamma B}y$, and $x^*=\gamma^{-1}(y-x)$. Then it follows
from \eqref{e:dr12} that $x\in Z$ and from \eqref{e:dr19} that
$x^*\in Z^*$. Additionally, we derive from \eqref{e:42s} that
\begin{equation}
\label{e:dr21}
(\forall n\in\NN)\quad y_{n+1}
=\Qq\bigl(y_0,y_n,y_n+\lambda_n(J_{\MMM}y_n-y_n)\bigr).
\end{equation}
Hence, Theorem~\ref{t:70s} yields $y_n\to y$ and, by continuity of
$J_{\gamma B}$, $x_n=J_{\gamma B}y_n\to J_{\gamma B}y=x$. Finally,
$x_n^*=\gamma^{-1}(y_n-x_n)\to\gamma^{-1}(y-x)=x^*$.
\end{proof}

\begin{remark}
\label{r:dr7}
The method of partial inverses of Theorem~\ref{t:13} may converge 
only weakly \cite[Counterexample~4]{Sico20}. A strongly convergent
version can be designed using Remark~\ref{r:dr2} and
Theorem~\ref{t:4s}.
\end{remark}

\subsection{Special cases and variants}

\subsubsection{Minimization setting}

We illustrate an application of the Douglas--Rachford algorithm
to primal-dual minimization. 

\begin{example}
\label{ex:dr1}
Let $f\in\Gamma_0(\HH)$ and $g\in\Gamma_0(\HH)$ be such that 
$Z=\Argmin(f+g)\neq\emp$ and $0\in\sri(\dom f-\dom g)$. Set
$Z^*=\Argmin(f^*\circ(-\Id)+g^*)$, let $(\lambda_n)_{n\in\NN}$ be a
sequence in $\left]0,2\right[$ such that
$\sum_{n\in\NN}\lambda_n(2-\lambda_n)=\pinf$, let $\gamma\in\RPP$,
let $y_0\in\HH$, and iterate
\begin{equation}
\label{e:dr6m}
\begin{array}{l}
\text{for}\;n=0,1,\ldots\\
\left\lfloor
\begin{array}{l}
x_n=\prox_{\gamma g}y_n\\
x_n^*=\gamma^{-1}(y_n-x_n)\\
z_n=\prox_{\gamma f}(2x_n-y_n)\\
y_{n+1}=y_n+\lambda_n(z_n-x_n).
\end{array}
\right.\\[2mm]
\end{array}
\end{equation}
Then it follows from Problem~\ref{prob:35}, Example~\ref{ex:r5},
and Theorem~\ref{t:4} that there exists $(x,x^*)\in Z\times Z^*$
such that $x_n\weakly x$ and $x^*_n\weakly x^*$.
\end{example}

\begin{remark}
\label{r:dr3}
Relations between the Douglas--Rachford algorithm \eqref{e:dr6m}
and other methods have been noted in the literature. 
\begin{enumerate}
\item
It is observed in \cite[Section~3.1.1]{Cond13} that the
Douglas--Rachford algorithm \eqref{e:dr6m} can be viewed as a
limiting case of the Chambolle--Pock algorithm (see
Remark~\ref{r:U}\ref{r:Uiii}) by implementing it in the case when
$\GG=\HH$, $L=\Id$, and $\sigma=1/\tau=\gamma$. Note, however, that
this setting violates the condition $\tau\sigma\|L\|^2<1$ used to
prove weak convergence of \eqref{e:2013v} in Example~\ref{ex:1Ub}.
\item
\label{r:dr3ii}
Consider the setting of Problem~\ref{prob:35} and note that the
primal minimization problem \eqref{e:p32} is equivalent to 
\begin{equation}
\label{e:rio-30b}
\minimize{(x,y)\in\gra L}{f(x)+g(y)}.
\end{equation}
The (unscaled) \emph{augmented Lagrangian} associated with
\eqref{e:rio-30b} is the saddle function (see Example~\ref{ex:3}) 
on $(\HH\oplus\GG)\oplus\GG$ defined as
\begin{align}
\label{e:rio-30c}
\hskip -16mm
F\colon\HH\oplus\GG\oplus\GG&\to\,\RX\nonumber\\
(x,y,v^*)~~~&\mapsto f(x)+g(y)+
\scal{Lx-y}{v^*}+\frac{1}{2}\|Lx-y\|^2.
\end{align}
Iteration $n$ of the alternating-direction method of multipliers
(ADMM) consists in minimizing $F$ over $x$ for $y_n$ and $v^*_n$
fixed to get $x_n$, then over $y$ for $x_n$ and $v^*_n$ fixed
to get $y_{n+1}$, and then applying a proximal maximization step
with respect to the Lagrange multiplier $v^*$ for $x_n$ and
$y_{n+1}$ fixed to get $v^*_{n+1}$. It was originally proposed in
\cite{Glow74}, refined in \cite{Gaba76}, and further developed in
\cite{Boyd10,Ecks92,Gaba83,Glow89}. Given $y_0\in\GG$ and
$v^*_0\in\GG$, ADMM iterates
\begin{equation}
\label{e:bp5}
\begin{array}{l}
\text{for}\;n=0,1,\ldots\\
\left\lfloor
\begin{array}{l}
x_n\in\Argmind{x\in\HH}{\Big(f(x)+
\scal{Lx}{v^*_n}+\dfrac{1}{2}\|Lx-y_n\|^2}\Big)\\[3mm]
d_n=Lx_n\\
y_{n+1}=\underset{y\in\GG}{\text{argmin}}\;
\Big(g(y)-\scal{y}{v^*_n}
+\dfrac{1}{2}\|d_n-y\|^2\Big)\\[3mm]
v^*_{n+1}=v^*_n+d_n-y_{n+1}.
\end{array}
\right.\\
\end{array}
\end{equation}
It should be emphasized that ADMM is not a splitting algorithm in
our sense since the computation of $x_n$ involves a minimization
step which does not separate $f$ and $L$, and can therefore be hard
to execute. This step is also set-valued in general. Nonetheless,
\eqref{e:bp5} can be interpreted as an application of the
Douglas--Rachford algorithm \eqref{e:dr6m} to the functions
$f^*\circ(-L^*)$ (here again, note that $f$ and $L$ are not
separated and that the typically non-explicit operator
$\prox_{f^*\circ(-L^*)}$ intervenes) and $g^*$ present in the dual
problem \eqref{e:d33} \cite{Gaba83} (see also \cite{Ecks92}). This
is merely an algorithmic identification and not a claim that ADMM
converges. Convergence requires more restrictions on the problem,
for instance finite-dimensionality of $\HH$ and $\GG$ and
invertibility of $L^*\circ L$ in \cite[Section~5]{Ecks92}. For
further analysis, see \cite{Bane21,Botr19,Ryue19}.
\end{enumerate}
\end{remark}

\subsubsection{Peaceman--Rachford splitting}
The first implicit alternating direction method \cite{Birk59} 
to solve the positive-definite matrix equation $Ax+Bx=f$ is the
Peaceman--Rachford algorithm \cite{Peac55} (see also
\cite{Doug55}). It is described by the iterative process
\begin{equation}
\label{e:pr1}
\begin{array}{l}
\text{for}\;n=0,1,\ldots\\
\left\lfloor
\begin{array}{l}
x_{n+1/2}-x_n+Ax_{n+1/2}+Bx_n=f\\
x_{n+1}-x_{n+1/2}+Ax_{n+1/2}+Bx_{n+1}=f.
\end{array}
\right.\\[2mm]
\end{array}
\end{equation}
Using the same arguments used to transition from \eqref{e:dr1} to
\eqref{e:dr4}, we rewrite \eqref{e:pr1} as
\begin{equation}
\label{e:pr4}
\begin{array}{l}
\text{for}\;n=0,1,\ldots\\
\left\lfloor
\begin{array}{l}
x_n=J_By_n\\
z_n=J_A(2x_n-y_n+f)\\
y_{n+1}=y_n+2(z_n-x_n).
\end{array}
\right.\\[2mm]
\end{array}
\end{equation}
The strong convergence of $(x_n)_{n\in\NN}$ to a solution to the
equation $Ax+Bx=f$, where $A$ and $B$ are single-valued
hemicontinuous monotone operators such that $\dom A=\dom B=\HH$ 
and $B$ is strongly monotone, was established in \cite{Lieu68} and,
with the additional assumption that $\HH$ is finite-dimensional and
the operators are continuous, in \cite{Kell69}.

Algorithm \eqref{e:pr4} was first considered for general maximally
monotone set-valued operators $A$ and $B$ in \cite{Lion79}. In the
presence of a scaling parameter $\gamma\in\RPP$ and taking $f=0$
without loss of generality, the Peaceman--Rachford algorithm
becomes 
\begin{equation} 
\label{e:pr5} 
\begin{array}{l}
\text{for}\;n=0,1,\ldots\\ 
\left\lfloor 
\begin{array}{l}
x_n=J_{\gamma B}y_n\\ 
z_n=J_{\gamma A}(2x_n-y_n)\\ 
y_{n+1}=y_n+2(z_n-x_n),
\end{array} 
\right.\\[2mm] 
\end{array} 
\end{equation} 
Upon defining $\MMM$ as in \eqref{e:1969}, we derive from
\eqref{e:pr5} that 
\begin{equation} 
\label{e:pr11} 
(\forall n\in\NN)\quad y_{n+1}=(2J_{\MMM}-\Id)y_n.
\end{equation} 
We can view \eqref{e:pr5} as a limiting case of the
Douglas--Rachford algorithm \eqref{e:dr6} in which the relaxation
parameters $(\lambda_n)_{n\in\NN}$ are allowed to be $2$. This, of
course, means that \eqref{e:pr5} operates outside of the setting of
Theorem~\ref{t:70} and hence of the geometric framework of
Theorem~\ref{t:1}. As a result, the weak convergence of
$(y_n)_{n\in\NN}$ cannot be guaranteed without additional
assumptions since \eqref{e:pr11} amounts to iterating a merely
nonexpansive operator (see \cite[Remark~6]{Lion79} for a
counterexample). Strong convergence of $(x_n)_{n\in\NN}$ to a point
in $\zer(A+B)$ takes place when $B$ is strongly monotone
\cite[Remark~2]{Lion79}. More generally, strong convergence 
occurs when $B$ is uniformly monotone on bounded sets or when 
$\inte\Fix(2J_{\gamma A}-\Id)(2J_{\gamma B}-\Id)\neq\emp$
\cite[Remark~2.2(iv)]{Joca09}.

\subsubsection{A three-operator splitting algorithm}
\label{sec:three}

An extension of the Douglas--Rachford algorithm \eqref{e:dr6}
was proposed in \cite{Davi17} by adding a cocoercive operator to
the inclusion \eqref{e:29p}.

\begin{proposition}
\label{p:dy}
Let $\tau\in\RPP$, let $A\colon\HH\to 2^{\HH}$ and 
$B\colon\HH\to 2^{\HH}$ be maximally monotone, 
and let $C\colon\HH\to\HH$ be $\tau$-cocoercive. Suppose that the
set $Z$ of solutions to the inclusion 
\begin{equation}
\label{e:r29q}
\text{find}\;\;x\in\HH\;\;\text{such that}\;\;0\in Ax+Bx+Cx
\end{equation}
is not empty and let $Z^*$ be the set of solutions to the dual
problem 
\begin{equation}
\label{e:79d}
\text{find}\;\;x^*\in\HH\;\;\text{such that}\;\;
0\in -(A+C)^{-1}(-x^*)+B^{-1}x^*.
\end{equation}
Let $\gamma\in\left]0,2\tau\right[$, set 
$\delta=2-\gamma/(2\tau)$, let $(\lambda_n)_{n\in\NN}$ be a
sequence in $\left]0,\delta\right[$ such
that $\sum_{n\in\NN}\lambda_n(\delta-\lambda_n)=\pinf$, and
let $y_0\in\HH$. Iterate
\begin{equation}
\label{e:3ops2}
\begin{array}{l}
\text{for}\;n=0,1,\ldots\\
\left\lfloor
\begin{array}{l}
x_n=J_{\gamma B}\,y_n\\
x_n^*=\gamma^{-1}(y_n-x_n)\\
r_n=y_n+\gamma C x_n\\
z_n=J_{\gamma A}(2x_n-r_n)\\
y_{n+1}=y_n+\lambda_n(z_n-x_n).
\end{array}
\right.\\[2mm]
\end{array}
\end{equation}
Then there exists $y\in\HH$ such that $y_n\weakly y$. 
Now set $x=J_{\gamma B}y$ and $x^*=\moyo{B}{\gamma}y$.
Then the following hold:
\begin{enumerate}
\item
\label{p:dyi}
$x_n\weakly x\in Z$.
\item
\label{p:dyii}
$x_n^*\weakly x^*\in Z^*$.
\end{enumerate}
\end{proposition}
\begin{proof}
Remarkably, we can closely follow the proof of Theorem~\ref{t:4}.
The key additional facts established in 
\cite[Proposition~2.1 and Lemma~2.2]{Davi17} are that, for 
$\alpha=1/\delta$,
\begin{equation}
\label{e:T3ops}
T=J_{\gamma A}\circ \big(2J_{\gamma B}-\Id-\gamma 
C\circ J_{\gamma B}\big)+\Id-J_{\gamma B}
\;\text{is $\alpha$-averaged and}\;Z=J_{\gamma B}(\Fix T). 
\end{equation}
We write the maximally monotone operator $\MMM$ of \eqref{e:M1} as
\begin{equation}
\label{e:M2}
\MMM=\bigg(\Id+\dfrac{1}{2\alpha}\Big(
J_{\gamma A}\circ\big(2J_{\gamma B}-\Id-\gamma 
C\circ J_{\gamma B}\big)-J_{\gamma B}\Big)\bigg)^{-1}-\Id
\end{equation}
and, in view of Example~\ref{ex:f12} and \eqref{e:T3ops}, work 
with the embedding $(\HH,\MMM,J_{\gamma B})$ of \eqref{e:r29q}.
Then $\emp\neq Z=J_{\gamma B}(\zer\MMM)$ and $(y_n)_{n\in\NN}$ is
produced by the proximal point algorithm $(\forall n\in\NN)$ 
$y_{n+1}=y_n+\mu_n(J_{\MMM}y_n-y_n)$, where 
$\mu_n=2\alpha\lambda_n\in\left]0,2\right[$. Using
Theorem~\ref{t:70}\ref{t:70i}, we infer that $(y_n)_{n\in\NN}$
converges weakly to a point $y\in\zer\MMM$ and that 
$J_{\MMM}y_n-y_n\to 0$. Hence, we derive from \eqref{e:T3ops},
\eqref{e:3ops2}, and \eqref{e:M2} that
\begin{equation}
\label{e:da11}
x=J_{\gamma B}y\in Z\;\;{and}\;\;
z_n-x_n=2\alpha(J_{\MMM}y_n-y_n)\to 0,
\end{equation}
and hence that
\begin{equation}
\label{e:da12}
\|Cz_n-Cx_n\|\leq\alpha^{-1}\|z_n-x_n\|\to 0.
\end{equation}

\ref{p:dyi}:
Set $(\forall n\in\NN)$ $z^*_n=\gamma^{-1}(2x_n-z_n-r_n)+Cz_n$.
In view of \eqref{e:3ops2} and \eqref{e:ee14},
\begin{equation}
\label{e:0-03a}
(\forall n\in\NN)\quad 
\begin{cases}
(z_n,z^*_n)\in\gra(A+C)\\
(x_n,x^*_n)\in\gra B\\
z^*_n+x^*_n=\gamma^{-1}(x_n-z_n)+Cz_n-Cx_n.
\end{cases}
\end{equation}
Next, fix $z\in\WC(x_n)_{n\in\NN}$, say
$x_{k_n}\weakly z$. Since $y_{k_n}\weakly y$, it follows from 
\eqref{e:da11}, \eqref{e:da12}, \eqref{e:0-03a}, and
\eqref{e:3ops2} that
\begin{equation}
\label{e:da0}
z_{k_n}\weakly z,\;
z^*_{k_n}\weakly\gamma^{-1}(z-y),\;
z_n-x_n\to 0,\;\;\text{and}\;\;
z^*_n+x^*_n\to 0.
\end{equation}
By applying Lemma~\ref{l:30} to the maximally monotone operators 
$A+C$ (see Example~\ref{ex:13} and Lemma~\ref{l:11}\ref{l:11i}) and
$B$, we deduce from \eqref{e:0-03a} and \eqref{e:da0} that
$z\in\zer(A+C+B)=Z$,
\begin{equation}
\big(z,\gamma^{-1}(z-y)\big)\in\gra(A+C),\quad\text{and}\quad
\big(z,\gamma^{-1}(y-z)\big)\in\gra B. 
\end{equation}
In turn, \eqref{e:ee14} asserts that $z=J_{\gamma B}y$, making
$x=J_{\gamma B}y$ the unique weak sequential cluster point of 
$(x_n)_{n\in\NN}$ which is bounded since $(y_n)_{n\in\NN}$ is. By
Lemma~\ref{l:1}\ref{l:1ii}, $x_n\weakly x$. 

\ref{p:dyii}:
Since $y_n\weakly y$ and $x_n\weakly x$, we have
$x^*_n=\gamma^{-1}(y_n-x_n)\weakly\gamma^{-1}(y-x)=
\moyo{B}{\gamma}y=x^*\in Z^*$ by \eqref{e:dr19}.
\end{proof}

\begin{remark}
\label{r:dy}
Here are a few comments on Proposition~\ref{p:dy}.
\begin{enumerate}
\item
The conclusion of Proposition~\ref{p:dy}\ref{p:dyi} was first
established in \cite[Theorem~2.1.1(b)]{Davi17} with a different
proof. See also \cite{Ragu19} for a discussion and connections with
\cite{Ragu13}.
\item
The duality result of Proposition~\ref{p:dy}\ref{p:dyii}
is new. 
\item
A strongly convergent version of Proposition~\ref{p:dy}
can be obtained by adapting the proof of Theorem~\ref{t:4s} to the
presence of $C$, as was done above. 
\item
When $C=0$, Proposition~\ref{p:dy} produces the Douglas--Rachford
setting of Theorem~\ref{t:4}. When $B=0$, \eqref{e:3ops2} yields a
special case of the forward-backward method of
\cite[Proposition~4.4(iii)]{Yama15} in which the proximal
parameters are all equal to $\gamma$.
\end{enumerate}
\end{remark}

\section{Tseng's forward-backward-forward splitting}
\label{sec:fbf}

\subsection{Preview}
\label{sec:korp}

In Section~\ref{sec:eul}, we have discussed a Euler method for
finding a zero of a single-valued operator $B\colon\HH\to\HH$ under
a cocoercivity condition. Under the more general assumption that
$B$ is monotone and $\beta$-Lipschitzian, the Euler method is no
longer appropriate, and we can use a scheme proposed by Antipin
\cite{Anti76} and Korpelevi\v{c} \cite{Korp76} that involves a
double activation of the operator $B$. Specifically, in this 
method, $\gamma\in\left]0,1/\beta\right[$ and $x_0\in\HH$ are fixed
and we iterate
\begin{equation}
\label{e:algo6ak}
\begin{array}{l}
\text{for}\;n=0,1,\ldots\\
\left\lfloor
\begin{array}{l}
b_n^*=\gamma Bx_n\\
m_n=x_n-b_n^*\\
m_n^*=Bm_n\\
x_{n+1}=x_n-\gamma m_n^*.
\end{array}
\right.\\
\end{array}
\end{equation}
Clearly, the sequence $(m_n,m_n^*)_{n\in\NN}$ lies $\gra B$ and
it is straightforward to see that, by choosing
$(\lambda_n)_{n\in\NN}$ suitably in \eqref{e:fejer14}, we obtain
\eqref{e:algo6ak}. The convergence properties of the 
Antipin--Korpelevi\v{c} method can therefore be deduced from the
results of Section~\ref{sec:gr} applied to $B$. 

Tseng's algorithm can be viewed as a generalization of
\eqref{e:algo6ak} for the problem of finding a zero of $A+B$, where
$A\colon\HH\to 2^{\HH}$ is maximally monotone and $B$ is as above.
It is called the forward-backward-forward algorithm because it
performs a forward step on $B$, then a backward step on $A$, and
finally another forward step on $B$. We are going to derive the
convergence of Tseng's forward-backward-forward splitting algorithm
from the principles of Section~\ref{sec:gr} and, more precisely,
from the warped resolvent algorithm of Section~\ref{sec:m+c}. 

\subsection{Fej\'erian algorithm}
\label{sec:fbfw}

We cast the forward-backward-forward algorithm as an instance
of \eqref{e:fejer8} and then prove its weak convergence via
Theorem~\ref{t:8}. This result was originally established in
\cite[Theorem~3.4(b)]{Tsen00}, where different arguments were used.

\begin{theorem}
\label{t:6}
Let $\beta\in\RPP$, let $A\colon\HH\to 2^{\HH}$ be maximally
monotone, let $B\colon\HH\to\HH$ be monotone and 
$\beta$-Lipschitzian, and suppose that $Z=\zer(A+B)\neq\emp$. Let
$x_0\in\HH$, let $\varepsilon\in\left]0,1/(\beta+1)\right[$,
and let $(\gamma_n)_{n\in\NN}$ be a sequence in
$[\varepsilon,(1-\varepsilon)/\beta]$.
Iterate
\begin{equation}
\label{e:fbf}
\begin{array}{l}
\text{for}\;n=0,1,\ldots\\
\left\lfloor
\begin{array}{l}
b_n^*=\gamma_nBx_n\\
m_n=J_{\gamma_nA}(x_n-b_n^*)\\
x_{n+1}=m_n-\gamma_nBm_n+b_n^*.
\end{array}
\right.\\
\end{array}
\end{equation}
Then $(x_n)_{n\in\NN}$ converges weakly to a point in $Z$.
\end{theorem}
\begin{proof}
Our objective is to apply Theorem~\ref{t:8} with 
\begin{equation}
\label{e:aqp28}
W=A+B,\;C=0,\;\text{and}\;(\forall n\in\NN)\;\;
U_n=\gamma_n^{-1}\Id-B\;\text{and}\;q_n=w_n. 
\end{equation}
Since $C=0$, let us rename $(w_n)_{n\in\NN}$ as $(m_n)_{n\in\NN}$.
Example~\ref{ex:11} and Lemma~\ref{l:11}\ref{l:11i} entail that $W$
is maximally monotone. Moreover, a consequence of 
Lemma~\ref{l:20}\ref{l:20i}--\ref{l:20ii} is that
\begin{equation}
\label{e:aqp23}
(\forall n\in\NN)\;\;\gamma_nU_n\;\text{is $\varepsilon$-strongly
monotone and $1/(2-\varepsilon)$-cocoercive.}
\end{equation}
Additionally, we derive from \cite[Proposition~3.9]{Jmaa20} that 
\begin{equation}
\label{e:aqp27}
(\forall n\in\NN)\quad
\ran U_n\subset\ran(U_n+W+C)\;\;\text{and}\;\;U_n+W+C\;
\text{is injective}.
\end{equation}
We also observe that
\begin{equation}
(\forall n\in\NN)\quad J_{W+C}^{U_n}=J_{A+B}^{U_n}=
\bigl(\gamma_n^{-1}\Id+A\bigr)\circ\bigl(\gamma_n^{-1}\Id-B\bigr)
=J_{\gamma_nA}\circ(\Id-\gamma_n B).
\end{equation}
Hence, the variables of \eqref{e:fejer8} in this setting become
\begin{equation}
\label{e:aqp20}
(\forall n\in\NN)\quad 
\begin{cases}
m_n=J_{\gamma_nA}(x_n-\gamma_n Bx_n)\\
t_n^*=U_nx_n-U_nm_n\\
\delta_n=\scal{m_n-x_n}{U_nm_n-U_nx_n}.
\end{cases}
\end{equation}
Now set
\begin{equation}
\label{e:aqp21}
(\forall n\in\NN)\quad 
\lambda_n=
\begin{cases}
\dfrac{\gamma_n\|t_n^*\|^2}{\delta_n},
&\text{if}\:\:\delta_n>0;\\
\varepsilon,&\text{otherwise.}
\end{cases}
\end{equation}
We derive from \eqref{e:aqp23} that
\begin{equation}
\label{e:aqp62}
(\forall n\in\NN)\quad\delta_n=
\scal{m_n-x_n}{U_nm_n-U_nx_n}\geq
\beta\varepsilon\|m_n-x_n\|^2,
\end{equation}
which implies that
\begin{equation}
\label{e:aqp16}
(\forall n\in\NN)\quad\delta_n\leq 0\;\Leftrightarrow\;m_n=x_n
\;\Leftrightarrow\;t^*_n=0.
\end{equation}
A consequence of \eqref{e:aqp23} is that, if
$\delta_n>0$,
\begin{equation}
\dfrac{\varepsilon}{\gamma_n}\leq
\dfrac{\|U_nm_n-U_nx_n\|}{\|m_n-x_n\|}
\leq\dfrac{\|U_nm_n-U_nx_n\|^2}{\scal{m_n-x_n}{U_nm_n-U_nx_n}}
\leq\dfrac{2-\varepsilon}{\gamma_n}
\end{equation}
and we therefore obtain from \eqref{e:aqp21} that
\begin{equation}
\label{e:fejer-15}
\lambda_n=\dfrac{\gamma_n\|U_nm_n-U_nx_n\|^2}
{\scal{m_n-x_n}{U_nm_n-U_nx_n}}\in[\varepsilon,2-\varepsilon].
\end{equation}
Hence, \eqref{e:fejer8} and \eqref{e:aqp16} yield
\begin{equation}
(\forall n\in\NN)\quad d_n=\dfrac{\gamma_n}{\lambda_n}t_n^*.
\end{equation}
Consequently, the sequence $(x_n)_{n\in\NN}$ produced by
\eqref{e:fbf} coincides with that of \eqref{e:fejer8}. 
We therefore appeal to Theorem~\ref{t:8}\ref{t:8ii} to conclude
since its condition \ref{t:8iid} holds thanks to 
\eqref{e:fejer-15}, whereas its condition \ref{t:8iib} holds 
thanks to \eqref{e:aqp23} and the fact that $(\gamma_n)_{n\in\NN}$
lies in $[\varepsilon,(1-\varepsilon)/\beta]$.
\end{proof}

\subsection{Haugazeau-like algorithm}
\label{sec:fbfs}

We present a strongly convergent best approximation version of the
forward-backward-forward method based on Theorem~\ref{t:8s}.

\begin{theorem}
\label{t:6s}
Let $\beta\in\RPP$, let $A\colon\HH\to 2^{\HH}$ be maximally
monotone, let $B\colon\HH\to\HH$ be monotone and 
$\beta$-Lipschitzian, and suppose that $Z=\zer(A+B)\neq\emp$. Let 
$x_0\in\HH$, let $\varepsilon\in\left]0,1/(\beta+1)\right[$,
and let $(\gamma_n)_{n\in\NN}$ be a sequence in
$\left[\varepsilon,(1-\varepsilon)/\beta\right]$. Iterate
\begin{equation}
\label{e:algo6s}
\begin{array}{l}
\text{for}\;n=0,1,\ldots\\
\left\lfloor
\begin{array}{l}
b_n^*=\gamma_nBx_n\\
m_n=J_{\gamma_nA}(x_n-b_n^*)\\
r_n=\dfrac{1}{2}\bigl(x_n+m_n-\gamma_nBm_n+b_n^*\bigr)\\
x_{n+1}=\Qq(x_0,x_n,r_n),
\end{array}
\right.\\
\end{array}
\end{equation}
where $\Qq$ is defined in Lemma~\ref{l:2}. Then $(x_n)_{n\in\NN}$
converges strongly to $\proj_Zx_0$.
\end{theorem}
\begin{proof}
We prove the claim as an application of Theorem~\ref{t:8s} in the
setting of \eqref{e:aqp28}. Let us use the same variables as in
\eqref{e:aqp20} and 
\begin{equation}
\label{e:aqp24}
(\forall n\in\NN)\quad 
\lambda_n=
\begin{cases}
\dfrac{\gamma_n\|t_n^*\|^2}{2\delta_n},
&\text{if}\:\:\delta_n>0;\\
\varepsilon/2,&\text{otherwise.}
\end{cases}
\end{equation}
Then, using the same arguments as in the proof of
Theorem~\ref{t:8}, we see that $(\lambda_n)_{n\in\NN}$ lies in 
$[\varepsilon/2,1]$ and that the sequence $(x_n)_{n\in\NN}$ 
produced by \eqref{e:algo6s} coincides with that of 
\eqref{e:haug8}. Since conditions \ref{t:8siid} and \ref{t:8siib}
in Theorem~\ref{t:8s}\ref{t:8sii} are fulfilled, we obtain the
claim.
\end{proof}

\subsection{Special cases and variants}

\subsubsection{The monotone+skew algorithm}
\label{sec:m+s}

The approach presented here was proposed in
\cite{Siop11} to solve the monotone inclusion \eqref{e:p3} and it
was the first algorithm to fully split the operators $A$, $B$, and
$L$. Its methodology conforms to the program of
Framework~\ref{f:1}: we use the embedding of Example~\ref{ex:f3} to
transfer the initial 3-operator problem \eqref{e:p3} in the primal
space $\HH$ to one involving the Kuhn--Tucker operator
$\kut=\boldsymbol{M}+\boldsymbol{S}$ of \eqref{e:kt3} in the larger
primal-dual space $\XXX=\HH\oplus\GG$. The algorithmic strategy
\emph{per se} is then straightforward: since $\boldsymbol{M}$ is
maximally monotone and $\boldsymbol{S}$ is monotone and
Lipschitzian, we can apply Tseng's forward-backward-forward
algorithm (Theorem~\ref{t:6}) in $\XXX$ to find a Kuhn--Tucker
point and hence a primal-dual solution. 

We derive from Theorem~\ref{t:6} the weak convergence of the
monotone+skew algorithm of \cite[Theorem~3.1(ii)]{Siop11} (we can
derive a strongly convergent version from Theorem~\ref{t:6s} using
the same arguments).

\begin{proposition}
\label{p:31}
Let $A\colon\HH\to 2^{\HH}$ and $B\colon\GG\to 2^{\GG}$ be
maximally monotone, and assume that $0\neq L\in\BL(\HH,\GG)$.
Suppose that the set $Z$ of solutions to the primal inclusion
\begin{equation}
\label{e:31p}
\text{find}\;\;x\in\HH\;\;\text{such that}\;\;
0\in Ax+L^*\big(B(Lx)\big)
\end{equation}
is not empty and let $Z^*$ be the set of solutions to the dual
inclusion 
\begin{equation}
\label{e:31d}
\text{find}\;\;y^*\in\GG\;\;\text{such that}\;\;
0\in -L\bigl(A^{-1}(-L^*y^*)\bigr)+B^{-1}y^*.
\end{equation}
Let $x_0\in\HH$, let $y^*_0\in\GG$, let 
$\varepsilon\in\left]0,1/(\|L\|+1)\right[$, 
let $(\gamma_n)_{n\in\NN}$ be a sequence in 
$[\varepsilon,(1-\varepsilon)/\|L\|\,]$, and set
\begin{equation}
\label{e:ny27b}
\begin{array}{l}
\text{for}\;n=0,1,\ldots\\
\left\lfloor
\begin{array}{l}
y_{1,n}=x_n-\gamma_nL^*y^*_n\\
y^*_{2,n}=y^*_n+\gamma_nLx_n\\
m_{1,n}=J_{\gamma_n A}y_{1,n}\\
m^*_{2,n}=J_{\gamma_n B^{-1}}y^*_{2,n}\\
q_{1,n}=m_{1,n}-\gamma_nL^*m^*_{2,n}\\
q^*_{2,n}=m^*_{2,n}+\gamma_nLm_{1,n}\\
x_{n+1}=x_n-y_{1,n}+q_{1,n}\\
y^*_{n+1}=y^*_n-y^*_{2,n}+q^*_{2,n}.
\end{array}
\right.\\
\end{array}
\end{equation}
Then there exist $x\in Z$ and $y^*\in Z^*$ such that 
$-L^*{y^*}\in Ax$, $y^*\in B(Lx)$, $x_n\weakly{x}$, and
$y^*_n\weakly{y^*}$.
\end{proposition}
\begin{proof}
Set $\XXX=\HH\oplus\GG$, define $\boldsymbol{M}$ and
$\boldsymbol{S}$ as in \eqref{e:31z}, and set
$(\forall n\in\NN)$ $\boldsymbol{x}_n=(x_n,y_n^*)$,
$\boldsymbol{y}_n=(y_{1,n},y^*_{2,n})$,
$\boldsymbol{m}_n=(m_{1,n},m^*_{2,n})$, and
$\boldsymbol{q}_n=(q_{1,n},q^*_{2,n})$,
Then, in view of Example~\ref{ex:r1}, \eqref{e:ny27b} becomes
\begin{equation}
\label{e:al6}
\begin{array}{l}
\text{for}\;n=0,1,\ldots\\
\left\lfloor
\begin{array}{l}
\boldsymbol{y}_n=\boldsymbol{x}_n-\gamma_n\boldsymbol{S}
\boldsymbol{x}_n\\
\boldsymbol{m}_n=J_{\gamma_n\boldsymbol{M}}\boldsymbol{y}_n\\
\boldsymbol{q}_n=\boldsymbol{m}_n-\gamma_n\boldsymbol{S}
\boldsymbol{m}_n\\
\boldsymbol{x}_{n+1}=\boldsymbol{x}_n-\boldsymbol{y}_n
+\boldsymbol{q}_n,
\end{array}
\right.\\
\end{array}
\end{equation}
which we rewrite as an instance of \eqref{e:fbf}, namely,
\begin{equation}
\begin{array}{l}
\text{for}\;n=0,1,\ldots\\
\left\lfloor
\begin{array}{l}
\boldsymbol{b}_n^*=\gamma_n\boldsymbol{S}\boldsymbol{x}_n\\
\boldsymbol{m}_n=J_{\gamma_n\boldsymbol{M}}
(\boldsymbol{x}_n-\boldsymbol{b}_n^*)\\
\boldsymbol{x}_{n+1}=\boldsymbol{m}_n-\gamma_n\boldsymbol{S}
\boldsymbol{m}_n+\boldsymbol{b}_n^*.
\end{array}
\right.\\
\end{array}
\end{equation}
It then follows from Theorem~\ref{t:6} and Lemma~\ref{l:31z}
that $(\boldsymbol{x}_n)_{n\in\NN}$ converges weakly to a point in 
$\zer(\boldsymbol{M}+\boldsymbol{S})\subset Z\times Z^*$,
as claimed.
\end{proof}

\begin{remark}
\label{r:c4}
The methodology of Theorem~\ref{t:6} is to find a Kuhn--Tucker
point, i.e., a zero of $\boldsymbol{M}+\boldsymbol{S}$. As noted in
\cite[Remark~2.9]{Siop11}, this can also be achieved by using the
Douglas--Rachford algorithm \eqref{e:dr6} which, upon setting
$U=(\Id+\,\gamma^2L^*\circ L)^{-1}$ and 
$V=(\Id+\,\gamma^2L\circ L^*)^{-1}$, and taking $\gamma\in\RPP$
and a sequence $(\lambda_n)_{n\in\NN}$ in $\left]0,2\right[$ such
that $\sum_{n\in\NN}\lambda_n(2-\lambda_n)=\pinf$, assumes the form
\begin{equation}
\label{e:main3}
\begin{array}{l}
\text{for}\;n=0,1,\ldots\\
\left\lfloor
\begin{array}{l}
x_n=U(y_{1,n}-\gamma L^*y^*_{2,n})\\
y^*_n=V(y^*_{2,n}+\gamma Ly_{1,n})\\
y_{1,n+1}=y_{1,n}+\lambda_n\big(J_{\gamma A}
(2x_n-y_{1,n})-x_n\big)\\
y^*_{2,n+1}=y^*_{2,n}+\lambda_n
\big(J_{\gamma B^{-1}}(2y^*_n-y^*_{2,n})-y^*_n\big).
\end{array}
\right.\\[2mm]
\end{array}
\end{equation}
Weak convergence of $(x_n,y^*_n)_{n\in\NN}$ to a point in 
$Z\times Z^*$ follows from Theorem~\ref{t:4}\ref{t:4i}.
The numerical effectiveness of \eqref{e:main3} depends on the ease
of implementation of the operators $U$ and $V$. This approach was
rediscovered in \cite{Ocon14} in an image restoration application. 
\end{remark}

\subsubsection{A Lagrangian approach to composite minimization}
\label{sec:Lagr}

We revisit the setting of Problem~\ref{prob:35}, which was
identified as an instance of Problem~\ref{prob:3} and can therefore
be solved using \eqref{e:ny27b} with $A=\partial f$ and 
$B=\partial g$. Following \cite[Section~4.5]{Bord18}, we explore a
different route which amounts to employing the embedding
$(\XXX,\sad_{\boldsymbol{F}},\TTT)$, where 
$\XXX=\HH\oplus\GG\oplus\GG$, 
\begin{equation}
\label{e:Kr2}
\begin{array}{ccll}
\sad_{\boldsymbol{F}}\colon&\XXX&\to&2^{\XXX}\\
&(x,y,v^*)&\mapsto&
\bigl(\partial f(x)+L^*v^*\bigr)\times
\bigl(\partial g(y)-v^*\bigr)\times\{-Lx+y\}
\end{array}
\end{equation}
is the saddle operator of \eqref{e:Kr}, and 
$\TTT\colon\XXX\to\HH\colon(x,y,v^*)\mapsto x$. 
Let us write
$\sad_{\boldsymbol{F}}=\boldsymbol{M}+\boldsymbol{S}$, where
\begin{equation}
\label{e:r76}
\begin{cases}
\boldsymbol{M}\colon (x,y,v^*)\mapsto
\partial f(x)\times\partial g(y)\times\{0\}\\
\boldsymbol{S}\colon (x,y,v^*)\mapsto
\big(L^*v^*,-v^*,-Lx+y\big).
\end{cases}
\end{equation}
Then $\|\boldsymbol{S}\|=\sqrt{1+\|L\|^2}$ and
$(\forall n\in\NN)$ $J_{\gamma_n\boldsymbol{M}}=
\prox_{\gamma_n f}\times\prox_{\gamma_n g}\times\,\Id$. Hence, 
applying Theorem~\ref{t:6} to this decomposition in $\XXX$, we
obtain the following realization of Framework~\ref{f:1}.

\begin{proposition}
\label{p:i}
Let $f\in\Gamma_0(\HH)$, $g\in\Gamma_0(\GG)$, and 
$L\in\BL(\HH,\GG)$ be such that 
$0\in\sri(L(\dom f)-\dom g)$. Suppose that the 
primal problem
\begin{equation}
\label{e:p321}
\minimize{x\in\HH}{f(x)+g(Lx)}
\end{equation}
admits solutions and consider the dual problem
\begin{equation}
\label{e:d321}
\minimize{v^*\in\GG}{f^*(-L^*v^*)+g^*(v^*)}.
\end{equation}
Let $(x_0,y_0,v^*_0)\in\HH\oplus\GG\oplus\GG$, let 
$\varepsilon\in\,]0,1/(1+\sqrt{1+\|L\|^2})[$, and let 
$(\gamma_n)_{n\in\NN}$ be a sequence in
$[\varepsilon,(1-\varepsilon)/\sqrt{1+\|L\|^2}]$. Iterate
\begin{equation}
\label{e:mp5}
\begin{array}{l}
\text{for}\;n=0,1,\ldots\\
\left\lfloor
\begin{array}{l}
r_n=\gamma_n(Lx_n-y_n)\\
m_{1,n}=\prox_{\gamma_n f}\big(x_n-\gamma_nL^*v^*_n\big)\\[1mm]
m_{2,n}=\prox_{\gamma_n g}\big(y_n+\gamma_nv^*_n\big)\\[1mm]
x_{n+1}=m_{1,n}-\gamma_nL^*r_n\\[1mm]
y_{n+1}=m_{2,n}+\gamma_nr_n\\[1mm]
v^*_{n+1}=v^*_n+\gamma_n\big(Lm_{1,n}-m_{2,n}\big).
\end{array}
\right.\\
\end{array}
\end{equation}
Then $(x_n)_{n\in\NN}$ and $(v^*_n)_{n\in\NN}$ converge weakly to 
solutions to \eqref{e:p321} and \eqref{e:d321},
respectively.
\end{proposition}

\begin{remark}
Let $(\mu_n)_{n\in\NN}$ be a sequence in  
$[\varepsilon,(1-\varepsilon)\text{min}\{1,1/\|L\|\}/2]$.
Algorithm \eqref{e:mp5} bears a certain resemblance with the 
iterative scheme
\begin{equation}
\label{e:ct94}
\begin{array}{l}
\text{for}\;n=0,1,\ldots\\
\left\lfloor
\begin{array}{l}
p_{n}=v^*_n+\mu_n(Lx_n-y_n)\\[1mm]
x_{n+1}=\prox_{\mu_n f}\big(x_n-\mu_nL^*p_{n}\big)\\[1mm]
y_{n+1}=\prox_{\mu_n g}\big(y_{n}+\mu_np_n\big)\\[1mm]
v^*_{n+1}=v^*_n+\mu_n\big(Lx_{n+1}-y_{n+1}\big)
\end{array}
\right.\\
\end{array}
\end{equation}
proposed in \cite{Chen94} to solve \eqref{e:p321}--\eqref{e:d321}
in a finite-dimensional setting.
\end{remark}

\begin{remark}
\label{e:jon}
In the finite-dimensional context of \cite{Ecks94}, the saddle
operator \eqref{e:Kr2} was split as
$\sad_{\boldsymbol{F}}=\boldsymbol{M}_1+\boldsymbol{M}_2$, where 
\begin{equation}
\label{e:r75}
\begin{cases}
\boldsymbol{M}_1\colon (x,y,v^*)\mapsto
\bigl(\partial f(x)+L^*v^*\bigr)\times\{0\}\times\{-Lx\}\\
\boldsymbol{M}_2\colon (x,y,v^*)\mapsto
\{0\}\times\bigl(\partial g(y)-v^*\bigr)\times\{y\}.
\end{cases}
\end{equation}
Given $\gamma\in\RPP$, $\mu_1\in\RR$, $\mu_2\in\RR$, and 
$(x_0,y_0,v^*_0)\in\HH\oplus\GG\oplus\GG$, applying the 
Douglas--Rachford algorithm \eqref{e:dr6} to find a zero of
$\boldsymbol{M}_1+\boldsymbol{M}_2$ leads to the 
algorithm \cite{Ecks94}
\begin{equation}
\label{e:7y}
\hspace{-2mm}
\begin{array}{l}
\text{for}\;n=0,1,\ldots\\
\left\lfloor
\begin{array}{l}
x_{n+1}\in\Argmind{x\in\HH}{\bigg(f(x)+\scal{Lx}{v^*_n}+
\dfrac{1}{2\gamma}\|Lx-y_n\|^2+\dfrac{\gamma\mu_1^2}{2}
\|x-x_n\|^2\bigg)}\\[3mm]
y_{n+1}=\underset{y\in\GG}{\text{argmin}}\;
\bigg(g(y)-\scal{y}{v^*_n}
+\dfrac{1}{2\gamma}\|Lx_{n+1}-y\|^2+\dfrac{\gamma\mu_2^2}{2}
\|y-y_n\|^2\bigg)\\[3mm]
v^*_{n+1}=v^*_n+\gamma^{-1}\big(Lx_{n+1}-y_{n+1}\big).
\end{array}
\right.\\
\end{array}
\end{equation}
When $\mu_1=\mu_2=0$, we recover the alternating direction method
of multipliers (ADMM) discussed in Remark~\ref{r:dr3}\ref{r:dr3ii}.
Just like ADMM, \eqref{e:7y} necessitates a potentially complex
minimization involving $f$ and $L$ jointly to construct $x_{n+1}$.
By contrast, \eqref{e:mp5} achieves full splitting of $f$, $g$, and
$L$.
\end{remark}

\begin{remark}
\label{r:dj1}
In view of Example~\ref{ex:f5}, the above saddle operator 
formalism can be extended to the more general primal-dual inclusion
pair of Problem~\ref{prob:3}. As in Proposition~\ref{p:i}, a zero
$(x,y,v^*)$ of the saddle operator $\sad$ of \eqref{e:lim1} can be
constructed by executing \eqref{e:mp5}, where $\prox_{\gamma_n f}$ 
is replaced with $J_{\gamma_n A}$ and $\prox_{\gamma_n g}$ with
$J_{\gamma_n B}$. In this setting, the weak limits $x$ and $v^*$
solve, respectively, the primal inclusion \eqref{e:p3} and the 
dual inclusion \eqref{e:d3}.  
\end{remark}

\subsubsection{Mixtures of composite, Lipschitzian, and 
parallel-sum operators} 
The Kuhn--Tucker operator of Lemma~\ref{l:31z} employed in
Section~\ref{sec:m+s} can be expressed in block format as
\begin{equation}
\label{e:kt1}
\kut=\boldsymbol{M}+\boldsymbol{S}=
\underbrace{
\begin{bmatrix}
A&0\\
0&B^{-1}
\end{bmatrix}
}_{\text{monotone}}
+
\underbrace{
\begin{bmatrix}
0&L^*\\
-L&0
\end{bmatrix}
}_{\text{skew}}.
\end{equation}
A Kuhn--Tucker point was obtained in Proposition~\ref{p:31} by
applying the forward-backward-forward algorithm \eqref{e:fbf} to
$\boldsymbol{M}$ and $\boldsymbol{S}$. In doing so, we did 
not exploit the linearity and skewness of $\boldsymbol{S}$, but
just the fact that it is monotone and Lipschitzian. Let us observe
that, if we fill the diagonal of $\boldsymbol{S}$ with monotone
Lipschitzian operators $Q\colon\HH\to\HH$ and
$D^{-1}\colon\GG\to\GG$, we obtain a new monotone and Lipschitzian
operator $\boldsymbol{Q}\colon\XXX\to\XXX$. In lieu of
\eqref{e:kt1}, we then consider the decomposition
\begin{equation}
\label{e:kt2}
\kut=\boldsymbol{M}+\boldsymbol{Q}=
\underbrace{
\begin{bmatrix}
A&0\\
0&B^{-1}
\end{bmatrix}
}_{\text{monotone}}
\;\;+
\underbrace{
\begin{bmatrix}
Q&L^*\\
-L&D^{-1}
\end{bmatrix}
}_{\text{monotone and Lipschitzian}}.
\end{equation}
Using \eqref{e:20115n}, we write 
\begin{equation}
\label{e:kt9}
\kut=
\begin{bmatrix}
A+Q&L^*\\
-L&(B\infconv D)^{-1}
\end{bmatrix}
\end{equation}
and interpret it as a variant of the Kuhn--Tucker operator
\eqref{e:kt3} associated with Problem~\ref{prob:3} in which $A$ is
replaced with $A+Q$ and $B$ with $B\infconv D$. In other words, the
primal inclusion is to
\begin{equation}
\label{e:32p}
\text{find}\;\;x\in\HH\;\;\text{such that}\;\;
0\in Ax+L^*\big((B\infconv D)(Lx)\big)+Qx
\end{equation}
and the dual inclusion is to
\begin{equation}
\label{e:32d}
\text{find}\;\;y^*\in\GG\;\;\text{such that}\;\;
0\in -L\big((A+Q)^{-1}(-L^*y^*)\big)+B^{-1}y^*+D^{-1}y^*
\end{equation}
or, equivalently, 
\begin{equation}
\label{e:32d'}
\text{find}\;\;y^*\in\GG\;\;\text{such that}\;\;(\exi x\in\HH)\;\;
\begin{cases}
-L^*y^*\in Ax+Qx\\
Lx\in B^{-1}y^*+D^{-1}y^*.
\end{cases}
\end{equation}
As in Lemma~\ref{l:31z}, for every $(x,y^*)\in\XXX$,
\begin{equation}
\label{e:as}
(x,y^*)\in\zer\kut\quad\Rightarrow\quad
\begin{cases}
x\;\text{solves \eqref{e:32p}}\\
y^*\;\text{solves \eqref{e:32d'}}
\end{cases}
\end{equation}
and we therefore recover the embedding principle of
Framework~\ref{f:1}. 

\begin{example}
\label{ex:f6}
In the above setting, set $\XXX=\HH\oplus\GG$, let $\kut$ be
the Kuhn--Tucker operator of \eqref{e:kt9}, and let
$\TTT\colon\XXX\to\HH\colon(x,y^*)\mapsto x$. Then 
$(\XXX,\kut,\TTT)$ is an embedding of \eqref{e:32p}.
\end{example}

The primal-dual inclusion problem \eqref{e:32p}--\eqref{e:32d} was
first investigated in \cite{Svva12}, where it was solved via
Tseng's forward-backward-forward algorithm. Here is
\cite[Theorem~3.1(ii)(c)--(d)]{Svva12}, which describes this
approach when the operators $L$, $B$, and $D$ above are deployed in
a product space $\GG=\GG_1\oplus\cdots\oplus\GG_p$ in the spirit of
Problem~\ref{prob:9} (further analysis of the asymptotic behavior
of the method in special cases can be found in \cite{Botr14}).

\begin{proposition}
\label{p:18}
Let $0<p\in\NN$, let $\mu\in\RPP$, let $A\colon\HH\to 2^{\HH}$ be 
maximally monotone, let $Q\colon\HH\to\HH$ be monotone and 
$\mu$-Lipschitzian. For every $k\in\{1,\ldots,p\}$, 
let $\nu_k\in\RPP$, let $\GG_k$ be a real Hilbert space, 
let $B_k\colon\GG_k\to 2^{\GG_k}$ be maximally monotone,
let $D_k\colon\GG_k\to 2^{\GG_k}$ be maximally monotone and such
that $D_k^{-1}\colon\GG_k\to\GG_k$ is $\nu_k$-Lipschitzian, and
assume that $0\neq L_k\in\BL(\HH,\GG_k)$. Suppose that the set $Z$
of solutions to the primal inclusion 
\begin{equation}
\label{e:p18}
\text{find}\;\;x\in\HH\;\;\text{such that}\;\;
0\in Ax+\sum_{k=1}^pL_k^*\big((B_k\infconv D_k)(L_kx)\big)+Qx
\end{equation}
is not empty and let $Z^*$ be the set of solutions to 
the dual inclusion
\begin{multline}
\label{e:d18}
\text{find}\;\;y^*_1\in\GG_1,\ldots,y^*_p\in\GG_p
\;\:\text{such that}\\
(\exi x\in\HH)\;
\begin{cases}
-\sum_{k=1}^pL_k^*{y^*_k}\in Ax+Qx\\
\bigl(\forall k\in\{1,\ldots,p\}\bigr)\;L_kx
\in B^{-1}_ky^*_k+D^{-1}_ky^*_k.
\end{cases}
\end{multline}
Set 
\begin{equation}
\label{e:06:25b}
\beta=\max\{\mu,\nu_1,\ldots,\nu_p\}+\sqrt{\sum_{k=1}^p\|L_k\|^2},
\end{equation}
let $x_0\in\HH$, let
$(y^*_{1,0},\ldots,y^*_{p,0})\in\GG_1\oplus\cdots\oplus\GG_p$, let 
$\varepsilon\in\left]0,1/(\beta+1)\right[$, and
let $(\gamma_n)_{n\in\NN}$ be a sequence in 
$[\varepsilon,(1-\varepsilon)/\beta]$. Iterate
\begin{equation}
\label{e:blackpage}
\begin{array}{l}
\text{for}\;n=0,1,\ldots\\
\left\lfloor
\begin{array}{l}
y_{1,n}=x_n-\gamma_n\big(Qx_n+\sum_{k=1}^pL_k^*y^*_{k,n}\big)\\
m_{1,n}=J_{\gamma_n A}\,y_{1,n}\\
\operatorname{for}\;k=1,\ldots,p\\
\left\lfloor
\begin{array}{l}
y^*_{2,k,n}=y^*_{k,n}+\gamma_n\big(L_kx_n-D_k^{-1}y^*_{k,n}\big)\\
m^*_{2,k,n}=J_{\gamma_n B_k^{-1}}y^*_{2,k,n}\\
q^*_{2,k,n}=m^*_{2,k,n}+\gamma_n\big(L_km_{1,n}-D_k^{-1}
m^*_{2,k,n}\big)\\
y^*_{k,n+1}=y^*_{k,n}-y^*_{2,k,n}+q^*_{2,k,n}
\end{array}
\right.\\[1mm]
q_{1,n}=m_{1,n}-\gamma_n\big(Qm_{1,n}+
\sum_{k=1}^pL_k^*m^*_{2,k,n}\big)\\
x_{n+1}=x_n-y_{1,n}+q_{1,n}.
\end{array}
\right.\\
\end{array}
\end{equation}
Then there exist $x\in Z$ and $(y^*_1,\ldots,y^*_p)\in Z^*$ such 
that $x_n\weakly{x}$, and, for every $k\in\{1,\ldots,p\}$,
$y^*_{k,n}\weakly{y^*_k}$.
\end{proposition}
\begin{proof}
The duality between \eqref{e:p18} and \eqref{e:d18} follows as in
Problem~\ref{prob:9}, by replacing $A$ with $A+Q$ and 
$(B^{-1}_k)_{1\leq k\leq p}$ with
$(B^{-1}_k+D^{-1}_k)_{1\leq k\leq p}$. Now set 
\begin{equation}
\label{e:GGD}
\begin{cases}
\GG=\GG_1\oplus\cdots\oplus\GG_p\\
B\colon\GG\to 2^{\GG}\colon(y_1,\ldots,y_p)\mapsto
B_1y_1\times\cdots\times B_py_p\\
D\colon\GG\to 2^{\GG}\colon(y_1,\ldots,y_p)\mapsto
D_1y_1\times\cdots\times D_py_p\\
L\colon\HH\to\GG\colon x\mapsto(L_1x,\ldots,L_px),
\end{cases}
\end{equation}
define $\boldsymbol{M}$ and $\boldsymbol{Q}$ as in \eqref{e:kt2},
and set
\begin{equation}
\label{e:hor}
(\forall n\in\NN)\quad 
\begin{cases}
\boldsymbol{x}_n=\bigl(x_n,y^*_{1,n},\ldots,y^*_{p,n}\bigr)\\
\boldsymbol{m}_n=\bigl(m_{1,n},m^*_{2,1,n},\ldots,
m^*_{2,p,n}\bigr).
\end{cases}
\end{equation}
Then $\boldsymbol{M}$ is maximally monotone and $\boldsymbol{Q}$ is
monotone and $\beta$-Lipschitzian \cite[Equation~(3.11)]{Svva12}
and, following the same steps as in the proof of
Proposition~\ref{p:31}, we rewrite \eqref{e:blackpage} as
\begin{equation}
\label{e:fbf3}
\begin{array}{l}
\text{for}\;n=0,1,\ldots\\
\left\lfloor
\begin{array}{l}
\boldsymbol{b}_n^*=\gamma_n\boldsymbol{Q}\boldsymbol{x}_n\\
\boldsymbol{m}_n=J_{\gamma_n\boldsymbol{M}}
(\boldsymbol{x}_n-\boldsymbol{b}_n^*)\\
\boldsymbol{x}_{n+1}=\boldsymbol{m}_n-\gamma_n\boldsymbol{Q}
\boldsymbol{m}_n+\boldsymbol{b}_n^*
\end{array}
\right.\\
\end{array}
\end{equation}
and conclude by invoking Theorem~\ref{t:6} and \eqref{e:as}.
\end{proof}

\begin{remark}
\label{r:i3}
In \eqref{e:p18}, suppose that $p=1$, $\GG_1=\HH$, $L_1=\Id$, 
$B_1=B$, $D_1=\{0\}^{-1}$, and $\zer(A+B+Q)\neq\emp$. Let
$x_0\in\HH$, let $y^*_0\in\HH$, let 
$\varepsilon\in\left]0,1/(\mu+2)\right[$, and let 
$(\gamma_n)_{n\in\NN}$ be a sequence in 
$[\varepsilon,(1-\varepsilon)/(\mu+1)]$.
Then we deduce from Proposition~\ref{p:18} that the sequence
$(x_n)_{n\in\NN}$ generated by the iterations 
\begin{equation}
\label{e:blackpage2}
\begin{array}{l}
\text{for}\;n=0,1,\ldots\\
\left\lfloor
\begin{array}{l}
y_{n}=x_n-\gamma_n\big(Qx_n+y^*_{n}\big)\\
p_n=J_{\gamma_n A}\,y_n\\
q^*_n=J_{\gamma_n B^{-1}}(y^*_{n}+\gamma_nx_n)\\
x_{n+1}=x_n-y_{n}+p_n-\gamma_n\big(Qp_n+q^*_n\big)\\
y^*_{n+1}=q^*_n+\gamma_n(p_n-x_n).
\end{array}
\right.\\
\end{array}
\end{equation}
converges weakly to a zero of $A+B+Q$. An alternative method to
solve this inclusion is proposed in \cite{Bang20}, with constant
proximal parameters $(\gamma_n)_{n\in\NN}$ and the feature that it
coincides with the unrelaxed version of the Douglas--Rachford
algorithm when $Q=0$ (in the spirit of the method of 
Section~\ref{sec:three} where $Q$ is cocoercive).
\end{remark}

\begin{example}
\label{ex:33}
In Proposition~\ref{p:18}, make the additional assumptions that 
$Q=0$ and, for every $k\in\{1,\ldots,p\}$, $\GG_k=\HH$, 
$L_k=\Id$, and $D_k^{-1}$ is strictly monotone. Then \eqref{e:p18}
collapses to 
\begin{equation}
\label{e:p1833}
\text{find}\;\;x\in\HH\;\;\text{such that}\;\;
0\in Ax+\sum_{k=1}^p(B_k\infconv D_k)(x).
\end{equation}
It is shown in \cite[Proposition~4.2]{Siop13} that \eqref{e:p1833} 
is an exact relaxation of the (possibly inconsistent) instance of 
the problem
\begin{equation}
\label{e:p1834}
\text{find}\;\;x\in\HH\;\;\text{such that}\;\;
0\in Ax\;\;\text{and}\;\;
\bigl(\forall k\in\{1,\ldots,p\}\bigr)\;\;0\in B_kx
\end{equation}
in the sense that the solutions to \eqref{e:p1833} are the same 
as those to \eqref{e:p1834} when the latter happen to exist.
\end{example}

The specialization of Proposition~\ref{p:18} to minimization is as
follows. It features the ability to split infimal convolutions
(see \eqref{e:infconv1}) together with linearly composed functions.

\begin{example}[{\protect{\cite[Theorem~4.2(ii)(b)--(c)]{Svva12}}}]
\label{ex:18}
Let $0<p\in\NN$, let $\mu\in\RPP$, let $f\in\Gamma_0(\HH)$, 
and let $h\colon\HH\to\RR$ be convex, differentiable, and such 
that $\nabla h$ is $\mu$-Lipschitzian. For every
$k\in\{1,\ldots,p\}$, let $\nu_k\in\RPP$, let $\GG_k$ be a real
Hilbert space, let $g_k\in\Gamma_0(\GG_k)$, let
$\ell_k\in\Gamma_0(\GG_k)$ be $1/\nu_k$-strongly convex, 
and suppose that $0\neq L_k\in\BL(\HH,\GG_k)$. Let $Z$ be the set
of solutions to the primal problem
\begin{equation}
\label{e:ex18p}
\minimize{x\in\HH}
{f(x)+\sum_{k=1}^p(g_k\infconv\ell_k)(L_kx)+h(x)},
\end{equation}
let $Z^*$ be the set of solutions to the dual problem
\begin{equation}
\label{e:ex18d}
\minimize{y^*_1\in\GG_1,\ldots,y^*_p\in\GG_p}
{(f^*\infconv h^*)\biggl(-\sum_{k=1}^pL_k^*{y^*_k}\biggr)
+\sum_{k=1}^p\bigl(g_k^*(y^*_k)+\ell_k^*(y^*_k)\bigr)},
\end{equation}
and suppose that
\begin{equation}
\label{e:svv12}
\zer\biggl(\partial f+\sum_{k=1}^pL_k^*\circ\bigl(
\partial g_k\infconv\partial\ell_k\bigr)\circ L_k
+\nabla h\biggr)\neq\emp.
\end{equation}
Set 
\begin{equation}
\beta=\max\{\mu,\nu_1,\ldots,\nu_p\}+
\sqrt{\sum_{k=1}^p\|L_k\|^2},
\end{equation}
let $x_0\in\HH$, let
$(y^*_{1,0},\ldots,y^*_{p,0})\in\GG_1\oplus\cdots\oplus\GG_p$, let 
$\varepsilon\in\left]0,1/(\beta+1)\right[$, and
let $(\gamma_n)_{n\in\NN}$ be a sequence in 
$[\varepsilon,(1-\varepsilon)/\beta]$. Iterate
\begin{equation}
\label{e:blackpage3}
\begin{array}{l}
\text{for}\;n=0,1,\ldots\\
\left\lfloor
\begin{array}{l}
y_{1,n}=x_n-\gamma_n\big(\nabla h(x_n)+
\sum_{k=1}^pL_k^*y^*_{k,n}\big)\\
m_{1,n}=\prox_{\gamma_n f}\,y_{1,n}\\
\operatorname{for}\;k=1,\ldots,p\\
\left\lfloor
\begin{array}{l}
y^*_{2,k,n}=y^*_{k,n}+\gamma_n\big(L_kx_n-
\nabla\ell_k^*(y^*_{k,n})\big)\\
m^*_{2,k,n}=\prox_{\gamma_n g_k^*}y^*_{2,k,n}\\
q^*_{2,k,n}=m^*_{2,k,n}+\gamma_n\big(L_km_{1,n}-\nabla\ell_k^*
(m^*_{2,k,n})\big)\\
y^*_{k,n+1}=y^*_{k,n}-y^*_{2,k,n}+q^*_{2,k,n}
\end{array}
\right.\\[1mm]
q_{1,n}=m_{1,n}-\gamma_n\big(\nabla h(m_{1,n})+
\sum_{k=1}^pL_k^*m^*_{2,k,n}\big)\\
x_{n+1}=x_n-y_{1,n}+q_{1,n}.
\end{array}
\right.\\
\end{array}
\end{equation}
Then there exist $x\in Z$ and $(y^*_1,\ldots,y^*_p)\in Z^*$ such 
that $x_n\weakly{x}$, and, for every $k\in\{1,\ldots,p\}$,
$y^*_{k,n}\weakly{y^*_k}$.
\end{example}

\begin{remark}
\label{r:31}
Conditions under which \eqref{e:svv12} holds are provided in 
\cite[Proposition~4.3]{Svva12}.
\end{remark}

\section{Forward-backward splitting}
\label{sec:fb}

\subsection{Preview}
The forward-backward splitting method is a basic algorithm for
solving Problem~\ref{prob:0} when $B$ is cocoercive. At iteration
$n$, given a step size $\gamma_n\in\RPP$, a discrete dynamics 
associated with the Cauchy problem \eqref{e:ev} with $M=A+B$ is
\begin{equation}
\label{e:fwb}
\dfrac{x_n-x_{n+1}}{\gamma_n}\in Ax_{n+1}+Bx_n.
\end{equation}
It amounts to performing a forward Euler step relative to the
operator $B$ and a backward Euler step relative to the operator
$A$. In view of \eqref{e:ee14}, this means that
$x_{n+1}=J_{\gamma_n A}(x_n-\gamma_nBx_n)$. This iteration scheme
goes back to the gradient-projection method \cite{Gold64,Levi66}
for the constrained minimization of a smooth function (see
Example~\ref{ex:53} below) and its extension to variational
inequalities \cite{Baku74,Merc79}.

\subsection{Fej\'erian algorithm}
\label{sec:fbw}
We establish a new, geometric proof of the convergence of a
relaxed primal-dual version of the forward-backward algorithm found
in \cite[Proposition~4.4(iii)]{Yama15} for the primal result and in
\cite[Theorem~26.14(ii)]{Livre1} for the dual result, where the
proximal parameters $(\gamma_n)_{n\in\NN}$ are constant. Related
primal results and special cases can be found in
\cite{Gaba83,Lema96,Lema97,Merc80,Tsen91}. The importance of
cocoercivity in establishing weak convergence was first identified
by Mercier \cite{Merc79} in the context of variational inequalities
and, more generally, in \cite{Merc80}.

\begin{theorem}
\label{t:5}
Let $\alpha\in\RPP$, let $A\colon\HH\to 2^{\HH}$ be maximally
monotone, and let $B\colon\HH\to\HH$ be $\alpha$-cocoercive. Let 
$\varepsilon\in\left]0,\alpha/(\alpha+1)\right[$, let
$(\gamma_n)_{n\in\NN}$ be a sequence in
$\left[\varepsilon,(2-\varepsilon)\alpha\right]$, and let
\begin{equation}
\label{e:aqp63}
(\forall n\in\NN)\quad\varepsilon\leq\mu_n\leq
(1-\varepsilon)\dfrac{4\alpha-\gamma_n}{2\alpha}.
\end{equation}
Suppose that the set $Z$ of solutions to the problem 
\begin{equation}
\label{e:5p}
\text{find}\;\;x\in\HH\;\;\text{such that}\;\;0\in Ax+Bx
\end{equation}
is not empty and let $Z^*$ be the set of solutions to the dual
problem 
\begin{equation}
\label{e:5d}
\text{find}\;\;x^*\in\HH\;\;\text{such that}\;\;
0\in -A^{-1}(-x^*)+B^{-1}x^*.
\end{equation}
Let $x_0\in\HH$ and iterate
\begin{equation}
\label{e:fb}
\begin{array}{l}
\text{for}\;n=0,1,\ldots\\
\left\lfloor
\begin{array}{l}
b_n^*=\gamma_nBx_n\\
w_n=J_{\gamma_nA}(x_n-b_n^*)\\
x_{n+1}=x_n+\mu_n(w_n-x_n).
\end{array}
\right.\\
\end{array}
\end{equation}
Then the following hold:
\begin{enumerate}
\item
\label{t:5i}
$(x_n)_{n\in\NN}$ converges weakly to a point in $Z$.
\item
\label{t:5i+}
$Z^*$ contains a single point $\overline{x}^*$ and
$(\forall z\in Z)$ $Bz=\overline{x}^*$.
\item
\label{t:5ii}
$(Bx_n)_{n\in\NN}$ converges strongly to $\overline{x}^*$.
\end{enumerate}
\end{theorem}
\begin{proof}
The proof hinges on an application of Theorem~\ref{t:8} with
\begin{equation}
\label{e:aqp29}
W=A,\;C=B,\;\text{and}\;(\forall n\in\NN)\;\;
U_n=\gamma_n^{-1}\Id-B\;\text{and}\;q_n=x_n. 
\end{equation}
In this setting
\begin{equation}
(\forall n\in\NN)\quad J_{W+C}^{U_n}=J_{A+B}^{U_n}=
\bigl(\gamma_n^{-1}\Id+A\bigr)\circ\bigl(\gamma_n^{-1}\Id-B\bigr)
=J_{\gamma_nA}\circ(\Id-\gamma_n B)
\end{equation}
and the variables of \eqref{e:fejer8} become
\begin{equation}
\label{e:aqp30}
(\forall n\in\NN)\quad 
\begin{cases}
w_n=J_{\gamma_nA}(x_n-\gamma_n Bx_n)\\[2mm]
t_n^*=\dfrac{x_n-w_n}{\gamma_n}\\[4mm]
\delta_n=\biggl(\dfrac{1}{\gamma_n}
-\dfrac{1}{4\alpha}\biggr)\|w_n-x_n\|^2.
\end{cases}
\end{equation}
Furthermore, we derive from \cite[Proposition~3.9]{Jmaa20} that 
\eqref{e:aqp27} holds. Now set
\begin{equation}
\label{e:aqp34}
(\forall n\in\NN)\quad\lambda_n=
\dfrac{4\alpha\mu_n}{4\alpha-\gamma_n}.
\end{equation}
Then \eqref{e:aqp63} yields
\begin{equation}
\label{e:ke}
(\forall n\in\NN)\quad 
\varepsilon\leq\dfrac{4\alpha\varepsilon}{4\alpha-\varepsilon}
\leq\lambda_n\leq
\dfrac{4\alpha(1-\varepsilon)(4\alpha-\gamma_n)}
{(4\alpha-\gamma_n)2\alpha}
\leq 2-\varepsilon.
\end{equation}
We also deduce from \eqref{e:aqp30} that
\begin{equation}
\label{e:aqp37}
(\forall n\in\NN)\quad\delta_n\leq 0\;\Leftrightarrow\;w_n=x_n
\;\Leftrightarrow\;t^*_n=0.
\end{equation}
Hence, \eqref{e:fejer8} yields
\begin{equation}
\label{e:aqp36}
(\forall n\in\NN)\quad d_n=\dfrac{\mu_n}{\lambda_n}(x_n-w_n).
\end{equation}
Altogether, we arrive at the conclusion that the sequence
$(x_n)_{n\in\NN}$ produced by \eqref{e:fb} coincides with that of
\eqref{e:fejer8}. Hence, by Theorem~\ref{t:8}\ref{t:8i} and
\eqref{e:ke}, 
\begin{equation}
\sum_{n\in\NN}\|d_n\|^2<\pinf. 
\end{equation}
In turn, upon invoking \eqref{e:aqp36}, we obtain
\begin{equation}
\label{e:aqp35}
w_n-x_n\to 0.
\end{equation}

\ref{t:5i}:
In view of \eqref{e:ke}, condition \ref{t:8iid} in 
Theorem~\ref{t:8}\ref{t:8ii} is fulfilled. On the other hand,
since Lemma~\ref{l:20}\ref{l:20iii} asserts that the operators
$(\gamma_nU_n)_{n\in\NN}$ are nonexpansive, \eqref{e:aqp35} implies
that $\|U_nw_n-U_nx_n\|\leq\|w_n-x_n\|/\varepsilon\to 0$, 
so that condition \ref{t:8iia} is also fulfilled.
Thus, the assertion follows from Theorem~\ref{t:8}\ref{t:8ii}.

\ref{t:5i+}:
The strong monotonicity of $B^{-1}$ implies that of 
$-A^{-1}\circ(-\Id)+B^{-1}$. Hence, 
\cite[Corollary~23.37(ii)]{Livre1} asserts that \eqref{e:5d} admits
a unique solution $\overline{x}^*$. Now let $z\in Z$. Then 
$-Bz\in Az$ and therefore $-z\in-A^{-1}(-Bz)$. Thus,
$0=-z+z\in-A^{-1}(-Bz)+B^{-1}(Bz)$, i.e., 
$Bz\in Z^*=\{\overline{x}^*\}$.

\ref{t:5ii}:
It follows from \ref{t:5i} and \eqref{e:aqp35} that 
$(x_n)_{n\in\NN}$ and $(w_n)_{n\in\NN}$ are bounded. 
Now let $z\in Z$. We retrieve from \eqref{e:c33} that
\begin{equation}
\label{e:sc33}
(\forall n\in\NN)\quad\scal{z-w_n}{t_n^*}
\leq\scal{x_n-w_n}{Bx_n-Bz}-\alpha\|Bx_n-Bz\|^2.
\end{equation}
Hence, the Cauchy--Schwarz inequality, 
\eqref{e:coco}, \eqref{e:aqp30}, and \eqref{e:aqp35} imply that
\begin{align}
\label{e:s34}
\alpha\|Bx_n-Bz\|^2
&\leq\|w_n-x_n\|\,\|Bx_n-Bz\|+\|w_n-z\|\,\|t^*_n\|
\nonumber\\
&\leq\dfrac{1}{\alpha}\|w_n-x_n\|\,\|x_n-z\|+
\dfrac{1}{\gamma_n}\|w_n-z\|\,\|w_n-x_n\|
\nonumber\\
&\to 0.
\end{align}
In view of \ref{t:5i+}, $Bx_n\to Bz=\overline{x}^*$.
\end{proof}

The following examples address Example~\ref{prob:21} and
Example~\ref{prob:22}, respectively.

\begin{example}
\label{ex:51'}
Let $\alpha\in\RPP$, let $f\in\Gamma_0(\HH)$,
let $B\colon\HH\to\HH$ be $\alpha$-cocoercive, suppose that 
the set $Z$ of solutions to the variational inequality 
\begin{equation}
\label{e:51'}
\text{find}\;\;x\in\HH\;\;\text{such that}\;\;
(\forall y\in\HH)\;\;\scal{x-y}{Bx}+f(x)\leq f(y)
\end{equation}
is not empty, and let $Z^*$ be the set of solutions to the dual 
problem 
\begin{equation}
\label{e:51'd}
\text{find}\;\;x^*\in\HH\;\;\text{such that}\;\;
0\in -\partial f^*(-x^*)+B^{-1}x^*.
\end{equation}
Let $x_0\in\HH$, let 
$\varepsilon\in\left]0,\alpha/(\alpha+1)\right[$,
let $(\gamma_n)_{n\in\NN}$ be a sequence in
$\left[\varepsilon,(2-\varepsilon)\alpha\right]$, and suppose 
that $(\mu_n)_{n\in\NN}$ satisfies \eqref{e:aqp63}. Iterate
\begin{equation}
\label{e:algo511}
\begin{array}{l}
\text{for}\;n=0,1,\ldots\\
\left\lfloor
\begin{array}{l}
b_n^*=\gamma_nBx_n\\
w_n=\prox_{\gamma_nf}(x_n-b_n^*)\\
x_{n+1}=x_n+\mu_n(w_n-x_n).
\end{array}
\right.\\
\end{array}
\end{equation}
Then $(x_n)_{n\in\NN}$ converges weakly to a point in $Z$ and 
$(Bx_n)_{n\in\NN}$ converges strongly to the unique point in 
$Z^*$.
\end{example}
\begin{proof}
Use Example~\ref{ex:1} and Example~\ref{ex:r5} and set
$A=\partial f$ in Theorem~\ref{t:5}.
\end{proof}

\begin{example}
\label{ex:51}
Let $\alpha\in\RPP$, let $C$ be a nonempty closed convex subset of 
$\HH$, let $B\colon\HH\to\HH$ be $\alpha$-cocoercive, suppose
that the set $Z$ of solutions to the variational inequality 
\begin{equation}
\label{e:p22bis}
\text{find}\;\;x\in C\;\;\text{such that}\;\;
(\forall y\in C)\;\;\scal{x-y}{Bx}\leq 0
\end{equation}
is not empty, and let $Z^*$ be the set of solutions to the dual 
problem 
\begin{equation}
\label{e:22bisd}
\text{find}\;\;x^*\in\HH\;\;\text{such that}\;\;
0\in -\partial\sigma_C(-x^*)+B^{-1}x^*.
\end{equation}
Let $x_0\in\HH$, let 
$\varepsilon\in\left]0,\alpha/(\alpha+1)\right[$,
let $(\gamma_n)_{n\in\NN}$ be a sequence in
$\left[\varepsilon,(2-\varepsilon)\alpha\right]$, and suppose 
that $(\mu_n)_{n\in\NN}$ satisfies \eqref{e:aqp63}. Iterate
\begin{equation}
\label{e:algo512}
\begin{array}{l}
\text{for}\;n=0,1,\ldots\\
\left\lfloor
\begin{array}{l}
b_n^*=\gamma_nBx_n\\
w_n=\proj_C(x_n-b_n^*)\\
x_{n+1}=x_n+\mu_n(w_n-x_n).
\end{array}
\right.\\
\end{array}
\end{equation}
Then $(x_n)_{n\in\NN}$ converges weakly to a point in $Z$ and 
$(Bx_n)_{n\in\NN}$ converges strongly to the unique point in 
$Z^*$.
\end{example}
\begin{proof}
Use Example~\ref{ex:r6} and \eqref{e:si}, and set $f=\iota_C$ in
Example~\ref{ex:51'}.
\end{proof}

The following example focuses on the minimization in the setting of
Problem~\ref{prob:27}\ref{prob:27ii}. This framework has found a 
multitude of applications, especially in the areas of signal
processing and machine learning 
\cite{Argy12,Beck10,Chan08,MaPa18,Smms05,Dext22,Jena11,Vait18}.

\begin{example}
\label{ex:52}
Let $\beta\in\RPP$, let $f\in\Gamma_0(\HH)$ and let 
$g\colon\HH\to\RR$ be convex and differentiable. Suppose that
$\nabla g$ is $\beta$-Lipschitzian and that the set $Z$ of 
solutions to the problem
\begin{equation}
\label{e:p27'}
\minimize{x\in\HH}{f(x)+g(x)}
\end{equation}
is not empty, and let $Z^*$ be the set of solutions to the dual 
problem 
\begin{equation}
\label{e:p27'd}
\minimize{x^*\in\HH}{f^*(-x^*)+g^*(x^*)}.
\end{equation}
Let $x_0\in\HH$, let 
$\varepsilon\in\left]0,1/(\beta+1)\right[$,
let $(\gamma_n)_{n\in\NN}$ be a sequence in
$\left[\varepsilon,(2-\varepsilon)/\beta\right]$, and suppose 
that 
\begin{equation}
\label{e:aqp65}
(\forall n\in\NN)\quad\varepsilon\leq\mu_n\leq
(1-\varepsilon)\dfrac{4-\beta\gamma_n}{2}. 
\end{equation}
Iterate
\begin{equation}
\label{e:27'}
\begin{array}{l}
\text{for}\;n=0,1,\ldots\\
\left\lfloor
\begin{array}{l}
b_n^*=\gamma_n\nabla g(x_n)\\
w_n=\prox_{\gamma_nf}(x_n-b_n^*)\\
x_{n+1}=x_n+\mu_n(w_n-x_n).
\end{array}
\right.\\
\end{array}
\end{equation}
Then $(x_n)_{n\in\NN}$ converges weakly to a point in $Z$ and 
$(\nabla g(x_n))_{n\in\NN}$ converges strongly to the unique point
in $Z^*$.
\end{example}
\begin{proof}
The claim is established by applying Theorem~\ref{t:5}\ref{t:5i}
with $A=\partial f$ (see Example~\ref{ex:1}) and $B=\nabla g$ (see
Lemma~\ref{l:bh}).
\end{proof}

\begin{remark}
\label{r:71}
In some applications, it may be of interest to quantify the
asymptotic behavior of the function values
$(f(x_n)+g(x_n))_{n\in\NN}$ produced by \eqref{e:27'}. This topic
has been the focus of a lot of interest since the publication of
the influential papers \cite{Bec09a,Bec09b,Cham15}; see
\cite{Garr23} and its bibliography for recent results on the
unrelaxed implementation of \eqref{e:27'} with constant proximal
parameters.
\end{remark}

The following example, taken from \cite{Smms05}, models linear
inverse problems in which the prior knowledge is modeled by
penalizing the coefficients of the decomposition of the ideal
solution in an orthonormal basis (see \cite{MaPa18,Daub04,Figu03}
for special cases).

\begin{example}
\label{ex:52'}
Suppose that $\HH$ is separable, let $(e_k)_{k\in\KK\subset\NN}$ be
an orthonormal basis of $\HH$, let $y\in\GG$, suppose that $0\neq
L\in\BL(\HH,\GG)$, and let $(\phi_k)_{k\in\KK}$ be functions in
$\Gamma_0(\RR)$ such that $(\forall k\in\KK)$ $\phi_k\geq
0=\phi_k(0)$. Suppose that the set $Z$ of solutions to the problem
\begin{equation}
\label{e:p52'}
\minimize{x\in\HH}{\sum_{k\in\KK}\phi_k\bigl(\scal{x}{e_k}\bigr)+
\dfrac{1}{2}\|Lx-y\|^2}
\end{equation}
is not empty. Let $x_0\in\HH$, let 
$\varepsilon\in\left]0,1/(\|L\|^2+1)\right[$,
let $(\gamma_n)_{n\in\NN}$ be a sequence in
$\left[\varepsilon,(2-\varepsilon)/\|L\|^2\right]$, and suppose 
that 
\begin{equation}
\label{e:aqp66}
(\forall n\in\NN)\quad\varepsilon\leq\mu_n\leq
(1-\varepsilon)\dfrac{4-\|L\|^2\gamma_n}{2}. 
\end{equation}
Iterate
\begin{equation}
\label{e:52'}
\begin{array}{l}
\text{for}\;n=0,1,\ldots\\
\left\lfloor
\begin{array}{l}
b_n^*=\gamma_nL^*(Lx_n-y)\\
w_n=\sum_{k\in\KK}\bigl(\prox_{\gamma_n\phi_k}
\scal{x_n-b_n^*}{e_k}\bigr)e_k\\
x_{n+1}=x_n+\mu_n(w_n-x_n).
\end{array}
\right.\\
\end{array}
\end{equation}
Then $(x_n)_{n\in\NN}$ converges weakly to a point in $Z$.
\end{example}
\begin{proof}
Set $f\colon x\mapsto\sum_{k\in\KK}\phi_k(\scal{x}{e_k})$ and 
$g\colon x\mapsto \|Lx-y\|^2/2$. Then, as shown in 
\cite[Example~2.19]{Smms05}, $f\in\Gamma_0(\HH)$
and $\prox_{\gamma f}\colon x\mapsto
\sum_{k\in\KK}(\prox_{\gamma_n\phi_k}\scal{x}{e_k})e_k$.
On the other hand, $g$ is convex and differentiable and 
$\nabla g\colon x\mapsto L^*(Lx-y)$ is $\|L\|^2$-Lipschitzian.
Altogether, the conclusion follows from Example~\ref{ex:52}.
\end{proof}

Next, we specialize Example~\ref{ex:52} to the gradient-projection
method, which minimizes a smooth function over a convex set 
(see Example~\ref{prob:50}) and goes back to \cite{Gold64,Levi66}.

\begin{example}
\label{ex:53}
Let $\beta\in\RPP$, let $C$ be a nonempty closed convex subset of
$\HH$, and let $g\colon\HH\to\RR$ be convex and differentiable.
Suppose that $\nabla g$ is $\beta$-Lipschitzian and that the 
set $Z$ of solutions to the problem 
\begin{equation}
\minimize{x\in C}{g(x)} 
\end{equation}
is not empty, and let $Z^*$ be the set of solutions to the dual 
problem 
\begin{equation}
\minimize{x^*\in\HH}{\sigma(-x^*)+g^*(x^*)}.
\end{equation}
Let $x_0\in\HH$, let $\varepsilon\in\left]0,1/(\beta+1)\right[$,
let $(\gamma_n)_{n\in\NN}$ be a sequence in
$\left[\varepsilon,(2-\varepsilon)/\beta\right]$, and suppose 
that $(\mu_n)_{n\in\NN}$ satisfies \eqref{e:aqp65}. Iterate
\begin{equation}
\label{e:algo543}
\begin{array}{l}
\text{for}\;n=0,1,\ldots\\
\left\lfloor
\begin{array}{l}
b_n^*=\gamma_n\nabla g(x_n)\\
w_n=\proj_C(x_n-b_n^*)\\
x_{n+1}=x_n+\mu_n(w_n-x_n).
\end{array}
\right.\\
\end{array}
\end{equation}
Then $(x_n)_{n\in\NN}$ converges weakly to a point in $Z$ and 
$(\nabla g(x_n))_{n\in\NN}$ converges strongly to the unique point
in $Z^*$.
\end{example}
\begin{proof}
Set $f=\iota_C$ in Example~\ref{ex:52}. Alternatively, set
$B=\nabla g$ in Example~\ref{ex:51}.
\end{proof}

\begin{remark}
\label{e:}
In \cite{Atto18}, the backward-forward iterations  
\begin{equation}
\label{e:bf}
\begin{array}{l}
\text{for}\;n=0,1,\ldots\\
\left\lfloor
\begin{array}{l}
p_n=J_{\gamma A}x_n\\
q_n=p_n-\gamma Bp_n\\
x_{n+1}=x_n+\mu_n(q_n-x_n)
\end{array}
\right.\\
\end{array}
\end{equation}
are studied and shown to be related to the forward-backward
iterations applied to Yosida envelopes of $B$ and $A$.
\end{remark}

\subsection{Haugazeau-like algorithm}
\label{sec:fbs}
As seen in \cite[Remark~5.12]{Smms05}, the strong convergence of 
$(x_n)_{n\in\NN}$ in Theorem~\ref{t:5}\ref{t:5i} may fail. Item
\ref{t:5si} below on the strong convergence of a best approximation
forward-backward algorithm extends \cite[Theorem~5.6(i) and
Remark~5.5]{Jnca05}, where $(\forall n\in\NN)$
$\gamma_n=\gamma\in\left]0,2\alpha\right[$ and $\mu_n\leq 1$.

\begin{theorem}
\label{t:5s}
Let $\alpha\in\RPP$, let $A\colon\HH\to 2^{\HH}$ be maximally
monotone, and let $B\colon\HH\to\HH$ be $\alpha$-cocoercive. Let 
$\varepsilon\in\left]0,\min\{1/2,2\alpha\}\right[$,
let $(\gamma_n)_{n\in\NN}$ be a sequence in
$\left[\varepsilon,2\alpha\right]$, and let 
\begin{equation}
\label{e:aqp64}
(\forall n\in\NN)\quad\varepsilon\leq\mu_n\leq
\dfrac{4\alpha-\gamma_n}{4\alpha}.
\end{equation}
Suppose that the set $Z$ of solutions to the problem 
\begin{equation}
\label{e:5ps}
\text{find}\;\;x\in\HH\;\;\text{such that}\;\;0\in Ax+Bx
\end{equation}
is not empty and let $Z^*$ be the set of solutions to the dual
\begin{equation}
\label{e:5ds}
\text{find}\;\;x^*\in\HH\;\;\text{such that}\;\;
0\in -A^{-1}(-x^*)+B^{-1}x^*.
\end{equation}
Let $x_0\in\HH$ and iterate
\begin{equation}
\label{e:algo5s}
\begin{array}{l}
\text{for}\;n=0,1,\ldots\\
\left\lfloor
\begin{array}{l}
b_n^*=\gamma_nBx_n\\
w_n=J_{\gamma_nA}(x_n-b_n^*)\\
x_{n+1}=\Qq\bigl(x_0,x_n,x_n+\mu_n(w_n-x_n)\bigr),
\end{array}
\right.\\
\end{array}
\end{equation}
where $\Qq$ is defined in Lemma~\ref{l:2}. 
Then the following hold:
\begin{enumerate}
\item
\label{t:5si}
$(x_n)_{n\in\NN}$ converges strongly to $\proj_Zx_0$.
\item
\label{t:5si+}
$Z^*$ contains a single point $\overline{x}^*$ and
$(Bx_n)_{n\in\NN}$ converges strongly to $\overline{x}^*$.
\end{enumerate}
\end{theorem}
\begin{proof}
We apply Theorem~\ref{t:8s} in the setting of \eqref{e:aqp29},
using the same variables as in \eqref{e:aqp30} and
$(\lambda_n)_{n\in\NN}$ defined as in \eqref{e:aqp34}. Then
\eqref{e:aqp37} holds and
\begin{equation}
\label{e:aqp40}
(\forall n\in\NN)\quad 
\varepsilon\leq\dfrac{4\alpha\varepsilon}{4\alpha-\varepsilon}
\leq\lambda_n\leq 1.
\end{equation}
Therefore the sequence $(x_n)_{n\in\NN}$ produced by
\eqref{e:algo5s} coincides with that of \eqref{e:haug8}. Hence, by
Theorem~\ref{t:8s}\ref{t:8si+},
\begin{equation}
\label{e:aqp41}
w_n-x_n\to 0.
\end{equation}

\ref{t:5si}:
This follows from Theorem~\ref{t:8s}\ref{t:8sii} since,
as in the proof of Theorem~\ref{t:5}\ref{t:5i}, its
conditions \ref{t:8siid} and \ref{t:8siib} are fulfilled.

\ref{t:5si+}:
Since $B$ is continuous, \ref{t:5si} and
Theorem~\ref{t:5}\ref{t:5i+} imply that 
$Bx_n\to B(\proj_Zx_0)\in Z^*$, where $Z^*$ is a singleton. 
\end{proof}

\subsection{Special cases and variants}

\subsubsection{Projected Landweber method}

In inverse problems, constrained least-squares estimation has a
long history \cite{Benn18,Bert97,Eick92,Neub88,Phil62}. We address
the numerical solution of this problem from the viewpoint of the
forward-backward algorithm to obtain a relaxed version of the
projected Landweber method with iteration-dependent parameters.

\begin{proposition}
\label{p:541}
Let $\GG$ be a real Hilbert space, suppose that
$0\neq L\in\BL(\HH,\GG)$, let $y\in\GG$, and let $C$ be a closed
convex subset of $\HH$ such that the set $Z$ of solutions to the
problem 
\begin{equation}
\label{e:541}
\minimize{x\in C}{\dfrac{1}{2}\|Lx-y\|^2}
\end{equation}
is not empty. Let $x_0\in\HH$, let 
$\varepsilon\in\left]0,1/(\|L\|^2+1)\right[$,
let $(\gamma_n)_{n\in\NN}$ be a sequence in
$\left[\varepsilon,(2-\varepsilon)/\|L\|^2\right]$,
and suppose that $(\mu_n)_{n\in\NN}$ satisfies \eqref{e:aqp66}.
Iterate
\begin{equation}
\label{e:algo541}
\begin{array}{l}
\text{for}\;n=0,1,\ldots\\
\left\lfloor
\begin{array}{l}
b_n^*=\gamma_nL^*(Lx_n-y)\\
w_n=\proj_C(x_n-b_n^*)\\
x_{n+1}=x_n+\mu_n(w_n-x_n).
\end{array}
\right.\\
\end{array}
\end{equation}
Then $(x_n)_{n\in\NN}$ converges weakly to a point in $Z$.
\end{proposition}
\begin{proof}
Apply Example~\ref{ex:53} with $g\colon x\mapsto\|Lx-y\|^2/2$.
\end{proof}

Proposition~\ref{p:541} was established in
\cite[Section~3.1]{Eick92} with $(\forall n\in\NN)$ $\lambda_n=1$
and $\gamma_n=\gamma\in\left]0,2/\|L\|^2\right[$. There, it was
also conjectured that the convergence was strong, which was
disproved in \cite[Remark~5.12]{Smms05}. This motivates the
following result.

\begin{proposition}
\label{p:541s}
Let $\GG$ be a real Hilbert space, suppose that
$0\neq L\in\BL(\HH,\GG)$, let $y\in\GG$, let $C$ be a closed
convex subset of $\HH$, and suppose that the set $Z$ of solutions
to \eqref{e:541} is not empty. Let $x_0\in\HH$, let 
$\varepsilon\in\left]0,\min\{1/2,2/\|L\|^2\})\right[$,
let $(\gamma_n)_{n\in\NN}$ be a sequence in
$\left[\varepsilon,2/\|L\|^2\right]$, and suppose that
$(\forall n\in\NN)$ $\varepsilon\leq\mu_n\leq 1-\|L\|^2\gamma_n/4$.
Iterate
\begin{equation}
\label{e:algo542s}
\begin{array}{l}
\text{for}\;n=0,1,\ldots\\
\left\lfloor
\begin{array}{l}
b_n^*=\gamma_nL^*(Lx_n-y)\\
w_n=\proj_C(x_n-b_n^*)\\
x_{n+1}=\Qq\bigl(x_0,x_n,x_n+\mu_n(w_n-x_n)\bigr),
\end{array}
\right.\\
\end{array}
\end{equation}
where $\Qq$ is defined in Lemma~\ref{l:2}\ref{l:2ii}. Then
$(x_n)_{n\in\NN}$ converges strongly to $\proj_Zx_0$.
\end{proposition}
\begin{proof}
Follow the pattern of the proof of Proposition~\ref{p:541} and
use Example~\ref{ex:r6} to apply Theorem~\ref{t:5s}\ref{t:5si} with
$A=N_C$ and $B\colon x\mapsto L^*(Lx-y)$.
\end{proof}

Here is an application of Proposition~\ref{p:541} to the problem of
finding the best approximation to a point from a linearly
transformed convex set.

\begin{example}
\label{ex:39}
Consider the setting of Proposition~\ref{p:541} with the 
assumption that $L(C)$ is closed, which guarantees that
\eqref{e:541} admits solutions. Then $x_n\weakly x$, where $x$
solves \eqref{e:541}. Furthermore, if we set $p=Lx$, then
$p=\proj_{L(C)}y$. Hence, upon rewriting \eqref{e:algo541} as
\begin{equation}
\label{e:algo39}
\begin{array}{l}
\text{for}\;n=0,1,\ldots\\
\left\lfloor
\begin{array}{l}
q_n=Lx_n\\
b_n^*=\gamma_nL^*(q_n-y)\\
w_n=\proj_C(x_n-b_n^*)\\
x_{n+1}=x_n+\mu_n(w_n-x_n)
\end{array}
\right.\\
\end{array}
\end{equation}
and invoking the weak continuity of $L$, we conclude that
$q_n\weakly\proj_{L(C)}y$. 
\end{example}

\begin{example}
\label{ex:39'}
Let $\GG$ be a real Hilbert space, and suppose that
$0\neq L\in\BL(\HH,\GG)$ and that $\ran L$ is closed. Additionally,
let $x_0\in\HH$, let 
$\varepsilon\in\left]0,1/(\|L\|^2+1)\right[$,
and let $(\nu_n)_{n\in\NN}$ be a sequence in
$\left[\varepsilon,(2-\varepsilon)/\|L\|^2\right]$. Iterate
\begin{equation}
\label{e:algo39'}
\begin{array}{l}
\text{for}\;n=0,1,\ldots\\
\left\lfloor
\begin{array}{l}
q_n=Lx_n\\
x_{n+1}=x_n-\nu_nL^*q_n
\end{array}
\right.\\
\end{array}
\end{equation}
and let $q$ be the minimal-norm element of $\ran L$. Then
$q_n\weakly q$. 
\end{example}
\begin{proof}
Apply Example~\ref{ex:39} with $C=\HH$ and $y=0$.
\end{proof}

The next example is about a composite best approximation problem.

\begin{example}
\label{ex:542}
Let $\GG$ be a real Hilbert space, let $y\in\GG$, and let
$0<p\in\NN$. For every 
$k\in\{1,\ldots,p\}$, let $\HH_k$ be a real Hilbert space, let 
$C_k$ be a nonempty closed convex subset of $\HH_k$, let 
$0\neq L_k\in\BL(\HH_k,\GG)$, and let $x_{k,0}\in\HH_k$. Suppose
that $\sum_{k=1}^pL_k(C_k)$ is closed and set
$\beta=\sum_{k=1}^p\|L_k\|^2$. Furthermore, let 
$\varepsilon\in\left]0,1/(\beta+1)\right[$,
let $(\gamma_n)_{n\in\NN}$ be a sequence in
$\left[\varepsilon,(2-\varepsilon)/\beta\right]$, 
and suppose that $(\mu_n)_{n\in\NN}$ satisfies \eqref{e:aqp65}.
Iterate
\begin{equation}
\label{e:algo542}
\begin{array}{l}
\text{for}\;n=0,1,\ldots\\
\left\lfloor
\begin{array}{l}
q_n=\sum_{k=1}^pL_kx_{k,n}\\
\text{for}\;k=1,\ldots,p\\
\left\lfloor
\begin{array}{l}
b_{k,n}^*=\gamma_nL_k^*(q_n-y)\\
w_{k,n}=\proj_{C_k}(x_{k,n}-b_{k,n}^*)\\
x_{k,n+1}=x_{k,n}+\mu_n(w_{k,n}-x_{k,n}).
\end{array}
\right.\\
\end{array}
\right.\\
\end{array}
\end{equation}
Then $q_n\weakly\proj_{\sum_{k=1}^pL_k(C_k)}y$. 
\end{example}
\begin{proof}
Set $\HHH=\HH_1\oplus\cdots\oplus\HH_p$, 
$\boldsymbol{C}=C_1\times\cdots\times C_p$, and 
\begin{equation}
\boldsymbol{L}\colon\HHH\to\GG\colon(x_k)_{1\leq k\leq p}
\mapsto\sum_{k=1}^pL_kx_k.
\end{equation}
Then $\proj_{\boldsymbol{C}}\colon(x_k)_{1\leq k\leq p}\mapsto
(\proj_{C_k}x_k)_{1\leq k\leq p}$ (see Examples~\ref{ex:r6} and
\ref{ex:r1}), $\|L\|^2=\beta$, and 
$\boldsymbol{L}^*\colon\GG\to\HHH\colon
y^*\mapsto(L_1^*y^*,\ldots,L_p^*y^*)$.
Altogether, the result is an application of Example~\ref{ex:39} to 
$\boldsymbol{C}$ and $\boldsymbol{L}$ in $\HHH$.
\end{proof}

As an application of Example~\ref{ex:542}, we address the problem
of computing the best approximation from the Minkowski sum of
closed convex sets; see 
\cite{Minh19,Eave84,Mart94,Qinx19,Seeg98,Wang20,Wonj19}
for instances of decompositions with respect to such sums.

\begin{example}
\label{ex:546}
Let $z\in\HH$ and $0<p\in\NN$. For every $k\in\{1,\ldots,p\}$, let
$C_k$ be a nonempty closed convex subset of $\HH$ and let 
$x_{k,0}\in\HH$. Suppose that $\sum_{k=1}^pC_k$ is closed,
let $\varepsilon\in\left]0,1/(p+1)\right[$,
let $(\gamma_n)_{n\in\NN}$ be a sequence in
$\left[\varepsilon,(2-\varepsilon)/p\right]$,
and suppose that $(\forall n\in\NN)$
$\varepsilon\leq\mu_n\leq (1-\varepsilon)(2-p\gamma_n/2)$.
Iterate
\begin{equation}
\label{e:algo545}
\begin{array}{l}
\text{for}\;n=0,1,\ldots\\
\left\lfloor
\begin{array}{l}
q_n=\sum_{k=1}^px_{k,n}\\
b_n^*=\gamma_n(q_n-z)\\
\text{for}\;k=1,\ldots,p\\
\left\lfloor
\begin{array}{l}
w_{k,n}=\proj_{C_k}(x_{k,n}-b_n^*)\\
x_{k,n+1}=x_{k,n}+\mu_n(w_{k,n}-x_{k,n}).
\end{array}
\right.\\
\end{array}
\right.\\
\end{array}
\end{equation}
Then $q_n\weakly\proj_{\sum_{k=1}^pC_k}z$. 
\end{example}
\begin{proof}
Apply Example~\ref{ex:542} with $\GG=\HH$, $y=z$, and 
$(\forall k\in\{1,\ldots,p\})$ $\HH_k=\HH$ and $L_k=\Id$.
\end{proof}

\subsubsection{Partial Yosida approximation to inconsistent common
zero problems}

We extend a framework proposed in \cite[Section~6.3]{Opti04}, where
no linear transformations were present. We start with the following
composite common zero problem (see \cite{Byrn12} for a
special case).

\begin{problem}
\label{prob:71}
Let $A\colon\HH\to 2^{\HH}$ be maximally monotone and let
$0<p\in\NN$. For every $k\in\{1,\ldots,p\}$, let $\GG_k$ be a real 
Hilbert space, let $B_k\colon\GG_k\to 2^{\GG_k}$ be
maximally monotone, and suppose that $0\neq L_k\in\BL(\HH,\GG_k)$. 
The objective is to 
\begin{equation}
\label{e:p71}
\text{find}\;\;x\in\zer A\;\;\text{such that}\;\;
\bigl(\forall k\in\{1,\ldots,p\}\bigr)\;\;L_kx\in\zer B_k.
\end{equation}
\end{problem}

\begin{example}
\label{ex:62}
Suppose that, in Problem~\ref{prob:71}, $A=N_C$, where $C$ is a
nonempty closed convex subset of $\HH$, and, for every
$k\in\{1,\ldots,p\}$, $B_k=N_{D_k}$, where $D_k$ is a nonempty
closed convex subset of $\GG_k$. Then \eqref{e:p71} is the
\emph{split feasibility problem} \cite{Reic20} 
\begin{equation}
\label{e:p62}
\text{find}\;\:x\in C\;\:\text{such that}\;\:
\bigl(\forall k\in\{1,\ldots,p\}\bigr)\quad L_kx\in D_k.
\end{equation}
\end{example}

\begin{example}
\label{ex:61}
Suppose that, in Problem~\ref{prob:71}, $A=\partial f$, where
$f\in\Gamma_0(\HH)$, and, for every
$k\in\{1,\ldots,p\}$, $\GG_k=\HH$, $L_k=\Id$, and 
$B_k=\partial f_k$, where $f_k\in\Gamma_0(\HH)$. 
Then \eqref{e:p71} becomes
\begin{equation}
\label{e:p61}
\text{find}\;\:x\in\bigl(\Argmin f\bigr)
\cap\bigcap_{k=1}^p\Argmin f_k.
\end{equation}
\end{example}

\begin{example}
\label{ex:63}
Suppose that, in Problem~\ref{prob:71}, $A=N_C$, where $C$ is a
nonempty closed convex subset of $\HH$, and, for every
$k\in\{1,\ldots,p\}$, $B_k=(\Id-F_k+r_k)^{-1}-\Id$, where 
$F_k\colon\GG_k\to\GG_k$ is firmly nonexpansive and $r_k\in\GG_k$.
Then \eqref{e:p71} becomes
\begin{equation}
\label{e:p63}
\text{find}\;\:x\in C\;\:\text{such that}\;\:
\bigl(\forall k\in\{1,\ldots,p\}\bigr)\quad F_k(L_kx)=r_k.
\end{equation}
Note that the operators $(\Id-F_k+r_k)_{1\leq k\leq p}$ are firmly
nonexpansive as well, which makes the operators 
$(B_k)_{1\leq k\leq p}$ maximally monotone by 
Lemma~\ref{l:0}\ref{l:0iii}. This formulation was investigated in
\cite{Siim22} in the context of recovering a signal in $C$ from $p$
nonlinear observations modeled as outputs of Wiener systems (see
also Example~\ref{ex:z23}).
\end{example}

Our focus here is on situations in which \eqref{e:p71} is not
guaranteed to have solutions (see
\cite{Cens18,Sign99,Siim19,Gold85} for concrete illustrations). In
such environments, it is natural to approximate it by a more
general problem, which exhibits better regularity properties and
admits solutions. We propose the following relaxation of
Problem~\ref{prob:71}, in which $\dom A$ serves as a hard
constraint.

\begin{problem}
\label{prob:72}
Consider the setting of Problem~\ref{prob:71} and let 
$(\rho_k)_{1\leq k\leq p}$ and $(\omega_k)_{1\leq k\leq p}$ be in
$\RPP$. The objective is to solve the \emph{partial Yosida
approximation} 
\begin{equation}
\label{e:p72}
\text{find}\;\;x\in\HH\;\;\text{such that}\;\;
0\in Ax+\sum_{k=1}^p\omega_kL_k^*\big(\moyo{B_k}{\rho_k}(L_kx)\big)
\end{equation}
to Problem~\ref{prob:71}.
\end{problem}

The fact that Problem~\ref{prob:72} is an appropriate relaxation of
Problem~\ref{prob:71} is supported by the following argument.

\begin{proposition}
\label{p:56}
Suppose that the set of solutions to Problem~\ref{prob:71} is
not empty. Then it coincides with the set of solutions to 
Problem~\ref{prob:72}.
\end{proposition}
\begin{proof}
Let $\overline{x}$ be a solution to Problem~\ref{prob:71}. Then 
\eqref{e:yosi2} yields
\begin{equation}
\label{e:240}
0=-\sum_{k=1}^p\omega_kL_k^*\big(\moyo{B_k}{\rho_k}
(L_k\overline{x})\big)\in A\overline{x},
\end{equation}
which shows that $\overline{x}$ solves Problem~\ref{prob:72}. Now
let $x$ be a solution to Problem~\ref{prob:72}. Then 
\begin{equation}
\label{e:241}
-\sum_{k=1}^p\omega_kL_k^*\big(\moyo{B_k}{\rho_k}(L_kx)\big)\in Ax.
\end{equation}
It follows from \eqref{e:240}, \eqref{e:241}, the monotonicity
of $A$, and the cocoercivity of the operators 
$(\moyo{B_k}{\rho_k})_{1\leq k\leq p}$ (see Example~\ref{ex:13y})
that 
\begin{align}
0&\geq\Scal{x-\overline{x}}{\sum_{k=1}^p
\omega_kL_k^*\big(\moyo{B_k}{\rho_k}(L_kx)\big)-
\sum_{k=1}^p\omega_kL_k^*\big(\moyo{B_k}{\rho_k}
(L_k\overline{x})\big)}
\nonumber\\
&=\sum_{k=1}^p\omega_k\Scal{L_kx-L_k\overline{x}}{
\moyo{B_k}{\rho_k}(L_kx)-\moyo{B_k}{\rho_k}(L_k\overline{x})}
\nonumber\\
&\geq\sum_{k=1}^p\omega_k\rho_k
\big\|\moyo{B_k}{\rho_k}(L_kx)-\moyo{B_k}{\rho_k}(L_k\overline{x})
\big\|^2
\nonumber\\
&=\sum_{k=1}^p\omega_k\rho_k
\big\|\moyo{B_k}{\rho_k}(L_kx)\big\|^2.
\end{align}
Hence, we deduce from \eqref{e:yosi2} that
$(\forall k\in\{1,\ldots,p\})$ 
$L_kx\in\zer\moyo{B_k}{\rho_k}=\zer B_k$.
In view of \eqref{e:241}, we conclude that $x$ solves
Problem~\ref{prob:71}.
\end{proof}

\begin{remark}
\label{r:56}
It should be emphasized that Problem~\ref{prob:72} is a relaxation
of Problem~\ref{prob:71}, and not of the inclusion
\begin{equation}
\label{e:p73}
\text{find}\;\;x\in\HH\;\;\text{such that}\;\;
0\in Ax+\sum_{k=1}^p\omega_kL_k^*\big(B_k(L_kx)\big).
\end{equation}
In particular, $\zer(A+\moyo{B}{\rho})\neq\zer(A+B)$ when 
$\zer(A+B)\neq\emp$. However, the
problem of finding a zero of $A+\moyo{B}{\rho}$ can be regarded as
a regularization of that of finding a zero of $A+B$ in the sense
that solutions to the former approaches a particular solution of
the latter as $\rho\to 0$ \cite{Mahe93,Merc80,Moud00}.
\end{remark}

\begin{example}
\label{ex:62'}
Consider the setting of Example~\ref{ex:62} and let
$(\forall k\in\{1,\ldots,p\})$ $\rho_k=1$. Then
\eqref{e:p72} relaxes the possibly inconsistent problem
\eqref{e:p62} to the problem
\begin{equation}
\label{e:56}
\minimize{x\in C}{\sum_{k=1}^p\omega_kd_{D_k}^2(L_kx)}.
\end{equation}
\begin{enumerate}
\item
Assume that, for every $k\in\{1,\ldots,p\}$,
$\GG_k=\HH$ and $L_k=\Id$. Then \eqref{e:56} is the relaxed
formulation of \cite{Sign99}. 
\item
Assume that $\HH=\RR^N$, $C=\RR^N$, and, for every
$k\in\{1,\ldots,p\}$, $\GG_k=\RR$, $L_k\colon x\mapsto u_k^\top x$
with $u_k\in\RR^N$, and $D_k=\{\eta_k\}$ with $\eta_k\in\RR$. Let
$U\in\RR^{p\times N}$ be the matrix with rows
$u_1^\top$,\,\ldots,$\,u_p^\top$ and set 
$y=(\eta_k)_{1\leq k\leq p}$.
Then \eqref{e:p62} amounts to solving the linear system $Ux=y$ and
\eqref{e:56} to minimizing $x\mapsto\|Ux-y\|^2$. This least-squares
relaxation was proposed by Legendre \cite{Lege05} and rediscovered
by Gauss \cite{Gaus09}.
\end{enumerate}
\end{example}

\begin{example}
\label{ex:61'}
Consider the setting of Example~\ref{ex:61} and recall that
$(\forall k\in\{1,\ldots,p\})$ 
$\moyo{(\partial f_k)}{\rho_k}=\{\nabla(\moyo{f_k}{\rho_k})\}$
\cite[Example~23.3]{Livre1}. 
Thus, \eqref{e:p72} relaxes the possibly inconsistent problem
\eqref{e:p61} to the problem
\begin{equation}
\label{e:56''}
\minimize{x\in\HH}{f(x)+\sum_{k=1}^p\omega_k
\bigl(\moyo{f_k}{\rho_k}\bigr)(x)}.
\end{equation}
This formulation arises in particular in federated learning
\cite{Path20}.
\end{example}

\begin{example}
\label{ex:63'}
Consider the setting of Example~\ref{ex:63} and let
$(\forall k\in\{1,\ldots,p\})$ $\rho_k=1$. Then it follows from
Example~\ref{ex:2} and \eqref{e:yosi1} that \eqref{e:p72} relaxes
the possibly inconsistent problem \eqref{e:p63} to the variational
inequality problem (see Problem~\ref{prob:22})
\begin{equation}
\label{e:63'}
\text{find}\:\;x\in C\:\;\text{such that}\:\;
(\forall y\in C)\;\:\sum_{k=1}^p\omega_k
\scal{L_k(y-x)}{F_k(L_kx)-r_k}\geq 0,
\end{equation}
which is precisely the relaxation of \eqref{e:p63} studied in
\cite{Siim22}.
\end{example}

Let us now solve Problem~\ref{prob:72} with the forward-backward
algorithm.

\begin{proposition}
\label{p:57}
Consider the setting of Problem~\ref{prob:72}, suppose that its set
$Z$ of solutions is not empty, and set 
\begin{equation}
\alpha=\dfrac{1}{\displaystyle{\sum_{k=1}^p}
\dfrac{\omega_k\|L_k\|^2}{\rho_k}}.
\end{equation}
Let $x_0\in\HH$, let
$\varepsilon\in\left]0,\alpha/(\alpha+1)\right[$, let
$(\gamma_n)_{n\in\NN}$ be a sequence in
$\left[\varepsilon,(2-\varepsilon)\alpha\right]$,
and suppose that $(\mu_n)_{n\in\NN}$ satisfies \eqref{e:aqp63}. 
Iterate
\begin{equation}
\label{e:algo57}
\begin{array}{l}
\text{for}\;n=0,1,\ldots\\
\left\lfloor
\begin{array}{l}
\text{for}\;k=1,\ldots,p\\
\left\lfloor
\begin{array}{l}
y_{k,n}=L_kx_n\\
p_{k,n}=\rho_k^{-1}\big(y_{k,n}-J_{\rho_kB_k}y_{k,n}\big)\\
\end{array}
\right.\\
b_n^*=\gamma_n\Sum_{k=1}^p\omega_kL_k^*p_{k,n}\\
w_n=J_{\gamma_nA}(x_n-b_n^*)\\
x_{n+1}=x_n+\mu_n(w_n-x_n).
\end{array}
\right.\\
\end{array}
\end{equation}
Then $(x_n)_{n\in\NN}$ converges weakly to a point in $Z$.
\end{proposition}
\begin{proof}
Define 
\begin{equation}
B=\sum_{k=1}^p\omega_kL_k^*\circ(
\moyo{\hspace{-.3mm}B_k}{\rho_k})\circ L_k.
\end{equation}
Then it follows from \cite[Proposition~4.12]{Livre1} and
Example~\ref{ex:13y} that $B$ is $\alpha$-cocoercive. Since
\eqref{e:algo57} is a specialization of \eqref{e:fb},
Theorem~\ref{t:5}\ref{t:5i} furnishes the desired conclusion.
\end{proof}

\subsubsection{Backward-backward splitting}

We focus on the following special case of Problem~\ref{prob:72}.

\begin{problem}
\label{prob:11}
Let $A\colon\HH\to 2^{\HH}$ and $B\colon\HH\to 2^{\HH}$ be 
maximally monotone, and let $\rho\in\RPP$. The objective is to
\begin{equation}
\label{e:p11}
\text{find}\;\;x\in\HH\;\;\text{such that}\;\;
0\in Ax+\moyo{B}{\rho}x.
\end{equation}
\end{problem}

\begin{proposition}
\label{p:11}
Consider the setting of Problem~\ref{prob:11} under the assumption
that $Z=\zer(A+\moyo{B}{\rho})\neq\emp$. Let $x_0\in\HH$, let
$\varepsilon\in\left]0,\rho/(\rho+1)\right[$, let
$(\gamma_n)_{n\in\NN}$ be a sequence in
$\left[\varepsilon,(2-\varepsilon)\rho\right]$,
and suppose that $(\mu_n)_{n\in\NN}$ satisfies 
\eqref{e:aqp63} with $\alpha=\rho$. 
Iterate
\begin{equation}
\label{e:algo11}
\begin{array}{l}
\text{for}\;n=0,1,\ldots\\
\left\lfloor
\begin{array}{l}
p_n=\rho^{-1}\big(x_n-J_{\rho B}x_n\big)\\
w_n=J_{\gamma_nA}(x_n-\gamma_np_n)\\
x_{n+1}=x_n+\mu_n(w_n-x_n).
\end{array}
\right.\\
\end{array}
\end{equation}
Then $(x_n)_{n\in\NN}$ converges weakly to a point in $Z$.
\end{proposition}
\begin{proof}
Apply Proposition~\ref{p:57} with $p=1$, $\GG_1=\HH$, $L_1=\Id$,
$B_1=B$, $\omega_1=1$, and $\rho_1=\rho$.
\end{proof}

\begin{example}
\label{ex:bb1}
In particular, if we execute \eqref{e:algo11} with, for every
$n\in\NN$, $\gamma_n=\rho$ and $\mu_n=1$, then
\begin{equation}
\label{e:bb1}
(\forall n\in\NN)\quad
x_{n+1}=J_{\rho A}\bigl(J_{\rho B}x_n\bigr).
\end{equation}
This recursion is known as the \emph{backward-backward algorithm},
as it alternates two backward Euler steps. As derived above,
it is a special case of \eqref{e:algo57} and therefore of the
forward-backward algorithm \eqref{e:fb}. Its asymptotic behavior
has been studied in \cite{Nona05,Merc80} (see also
\cite{Lion78,Pass79} for ergodic convergence).
\end{example}

\begin{example}
\label{ex:bb2}
Let $f$ and $g$ be functions in $\Gamma_0(\HH)$. In
Problem~\ref{prob:11}, suppose that $A=\partial f$ and 
$B=\partial g$. Then, as in Example~\ref{ex:61'},
\eqref{e:bb1} becomes
\begin{equation}
\label{e:bb-2}
\minimize{x\in\HH}{f(x)+\moyo{g}{\rho}(x)}
\end{equation}
and \eqref{e:bb1} reduces to the
\emph{alternating proximal point algorithm}
\begin{equation}
\label{e:bb2}
(\forall n\in\NN)\quad
x_{n+1}=\prox_{\rho f}\bigl(\prox_{\rho g}x_n\bigr).
\end{equation}
This method was first investigated in \cite{Acke80}, with further
developments in \cite{Nona05}.
\end{example}

\begin{example}
\label{ex:bb3}
Let $C$ and $D$ be nonempty closed convex subsets of $\HH$. 
In Example~\ref{ex:bb2}, suppose that $f=\iota_C$ and $g=\iota_D$.
Then \eqref{e:bb2} is the problem of finding a point in $C$ at
minimal distance from $D$ and \eqref{e:bb2} yields the
\emph{alternating projection method}
\begin{equation}
\label{e:bb3}
(\forall n\in\NN)\quad
x_{n+1}=\proj_C\bigl(\proj_{D}x_n\bigr),
\end{equation}
which was first investigated in \cite{Chen59}. Its weak
convergence was established in \cite[Theorem~2]{Gubi67}
\end{example}

\begin{example}
\label{ex:bb4}
Let $f\in\Gamma_0(\HH)$, $h\in\Gamma_0(\HH)$, $z\in\HH$, 
and $\rho\in\RPP$. The problem is to 
\begin{equation}
\label{e:bb4}
\minimize{x\in\HH, w\in\HH}{f(x)+h(w)+\dfrac{1}{2\rho}
\|x+w-z\|^2}.
\end{equation}
Following \cite[Section~4.4]{Smms05}, set 
$g\colon y\mapsto h(z-y)$. Then, with the change of variable 
$y=z-w$, the objective of \eqref{e:bb4} is to
\begin{equation}
\label{e:bB4}
\minimize{x\in\HH, y\in\HH}{f(x)+g(y)+\dfrac{1}{2\rho}
\|x-y\|^2},
\end{equation}
which is precisely \eqref{e:bb-2} in terms of the variable $x$.
Now let $x_0\in\HH$, let
$\varepsilon\in\left]0,\rho/(\rho+1)\right[$, let
$(\gamma_n)_{n\in\NN}$ be a sequence in
$\left[\varepsilon,(2-\varepsilon)\rho\right]$,
and let $(\mu_n)_{n\in\NN}$ be a sequence in
$\left[\varepsilon,1\right]$.
Applying algorithm \eqref{e:algo11} to $A=\partial f$ and
$B=\partial g$, and noting that $J_{\rho B}=\prox_{\rho g}\colon
x\mapsto z-\prox_{\rho h}(z-x)$ yields
\begin{equation}
\label{e:algobb4}
\begin{array}{l}
\text{for}\;n=0,1,\ldots\\
\left\lfloor
\begin{array}{l}
p_n=\rho^{-1}\big(x_n-z+\prox_{\rho h}(z-x_n)\big)\\
w_n=\prox_{\gamma_nf}(x_n-\gamma_np_n)\\
x_{n+1}=x_n+\mu_n(w_n-x_n).
\end{array}
\right.\\
\end{array}
\end{equation}
It follows from Proposition~\ref{p:11} that $(x_n)_{n\in\NN}$
converges weakly to a point $x$ such that 
$(x,\prox_{\rho h}(z-x))$ solves \eqref{e:bb4}.
\end{example}

Next, we revisit the problem of projecting onto the Minkowski sum
of two convex sets (see Example~\ref{ex:546}).

\begin{example}
\label{ex:bb5}
Let $C$ and $D$ be nonempty closed convex subsets of $\HH$ such
that $C+D$ is closed, and let $z\in\HH$. Upon setting $f=\iota_C$,
$h=\iota_D$, and $\rho=1$ in Example~\ref{ex:bb4}, \eqref{e:bb4}
specializes to the problem of finding the projection of $z$ onto
$C+D$. Now let $x_0\in C$, let $\varepsilon\in\left]0,1/2\right[$,
let $(\gamma_n)_{n\in\NN}$ be a sequence in
$\left[\varepsilon,2-\varepsilon\right]$, and let
$(\mu_n)_{n\in\NN}$ be a sequence in $\left[\varepsilon,1\right]$.
Then \eqref{e:algobb4} assumes the form 
\begin{equation}
\label{e:algobb5}
\begin{array}{l}
\text{for}\;n=0,1,\ldots\\
\left\lfloor
\begin{array}{l}
p_n=x_n-z+\proj_{D}(z-x_n)\\
w_n=\proj_{C}(x_n-\gamma_np_n)\\
x_{n+1}=x_n+\mu_n(w_n-x_n)
\end{array}
\right.\\
\end{array}
\end{equation}
and it follows from Proposition~\ref{p:11} that $(x_n)_{n\in\NN}$
converges weakly to a point $x$ such that 
$\proj_{C+D}z=x+\proj_{D}(z-x)$. This best approximation algorithm
was first obtained in \cite[Theorem~2.1]{Seeg98} in the case when
$(\forall n\in\NN)$ $\gamma_n=\mu_n=1$, i.e., 
\begin{equation}
(\forall n\in\NN)\quad
x_{n+1}=\proj_{C}\bigl(z-\proj_{D}(z-x_n)\bigr).
\end{equation}
\end{example}

\subsubsection{Dual implementation}

We present a framework for solving strongly monotone composite 
inclusion problems by applying the forward-backward algorithm to
the dual problem. The embedding underlying this approach is that of
Example~\ref{ex:f13}.

\begin{problem}
\label{prob:26}
Let $\rho\in\RPP$, let $0<p\in\NN$, let $z\in\HH$, and let 
$A\colon\HH\to 2^{\HH}$ be maximally monotone. For every
$k\in\{1,\ldots,p\}$, let $B_k\colon\GG_k\to 2^{\GG_k}$ be
maximally monotone, let $\nu_k\in\RPP$, let 
$D_k\colon\GG_k\to 2^{\GG_k}$ be maximally 
monotone and $\nu_k$-strongly monotone, and suppose that 
$0\neq L_k\in\BL(\HH,\GG_k)$. Further, suppose that
\begin{equation}
\label{e:26a}
z\in\ran\bigg(A+\sum_{k=1}^pL_k^*\circ(B_k\infconv D_k)\circ
L_k+\rho\Id\bigg).
\end{equation}
The problem is to solve the primal inclusion
\begin{equation}
\label{e:26p}
\text{find}\;\;x\in\HH\;\;\text{such that}\;\;
z\in Ax+\sum_{k=1}^pL_k^*\big((B_k\infconv D_k)
(L_kx)\big)+\rho x,
\end{equation}
together with the dual inclusion
\begin{multline}
\label{e:26d}
\text{find}\;\;{y^*_1}\in\GG_1,\ldots,{y^*_p}\in\GG_p\;\;
\text{such that}\;\;
\bigl(\forall k\in\{1,\ldots,p\}\bigr)\\
0\in -L_k\bigg(J_{A/\rho}\bigg(\dfrac{1}{\rho}\bigg(z-
\sum_{j=1}^pL_j^*{y^*_j}\bigg)\bigg)\bigg)+
B_k^{-1}{y^*_k}+D_k^{-1}{y^*_k}.
\end{multline}
\end{problem}

We refer to \cite[Proposition~5.2(iv)]{Opti14} for sufficient
conditions that guarantee \eqref{e:26a}. The mechanism to solve
\eqref{e:26p} dually hinges on the following properties.

\begin{proposition}[{\protect{\cite[Proposition~5.2(ii)--(iii)]%
{Opti14}}}]
\label{p:26}
Consider the setting of Problem~\ref{prob:26} and set
\begin{equation}
\label{e:26u}
M=A+\sum_{k=1}^pL_k^*\circ(B_k\infconv D_k)\circ L_k
\quad\text{and}\quad
\overline{x}=J_{M/\rho}\big(z/\rho\big).
\end{equation}
Then the following hold:
\begin{enumerate}
\item
\label{p:26i}
$\overline{x}$ is the unique solution to the primal problem 
\eqref{e:26p}.
\item
\label{p:26ii}
The dual problem \eqref{e:26d} admits solutions and, if
$(\overline{y}_k^*)_{1\leq k\leq p}$ solves \eqref{e:26d}, then
\begin{equation}
\overline{x}=J_{A/\rho}\Biggl(\rho^{-1}
\biggl(z-\sum_{k=1}^pL_k^*\overline{y}^*_k\biggr)\Biggr).
\end{equation}
\end{enumerate}
\end{proposition}

We now apply the forward-backward algorithm of Theorem~\ref{t:5} 
to the dual inclusion \eqref{e:26d} to construct a sequence
$(x_n)_{n\in\NN}$ which converges strongly to the solution to
primal inclusion \eqref{e:26p}. The following result is an
adaptation of \cite[Corollary~5.4]{Opti14}.

\begin{proposition}
\label{p:26c}
Consider the setting of Problem~\ref{prob:26} and set
\begin{equation}
\label{e:beta1}
\nu=\min_{1\leq k\leq p}\nu_k\quad\text{and}\quad
\alpha=\frac{1}{\dfrac{1}{\nu}+\dfrac{1}{\rho}
\displaystyle{\displaystyle{\sum_{1\leq k\leq p}}\|L_k\|^2}}.
\end{equation}
Let $\varepsilon\in\left]0,\alpha/(\alpha+1)\right[$,
let $(\gamma_n)_{n\in\NN}$ be a 
sequence in $[\varepsilon,(2-\varepsilon)\alpha]$, 
suppose that $(\mu_n)_{n\in\NN}$ satisfies \eqref{e:aqp63},
and, for every $k\in\{1,\ldots,p\}$, let $y^*_{k,0}\in\GG_k$.
Iterate 
\begin{equation}
\label{e:-08a}
\begin{array}{l}
\text{for}\;n=0,1,\ldots\\
\left\lfloor
\begin{array}{l}
q_n=z-\sum_{k=1}^pL_k^*y^*_{k,n}\\
x_n=J_{A/\rho}(q_n/\rho)\\
\operatorname{for}\;k=1,\ldots,p\\
\left\lfloor
\begin{array}{l}
w_{k,n}=y^*_{k,n}+\gamma_n\big(L_kx_n-D_k^{-1}y^*_{k,n}
\big)\\[2mm]
y^*_{k,n+1}=y^*_{k,n}+\mu_n\big(J_{\gamma_nB_k^{-1}}
w_{k,n}-y^*_{k,n}\big).\\[1mm] 
\end{array}
\right.\\[2mm]
\end{array}
\right.
\end{array}
\end{equation}
Then the following hold for the solution $\overline{x}$ to 
\eqref{e:26p} and for some solution $\overline{\boldsymbol{y}}^*
=(\overline{y}^*_1,\ldots,\overline{y}^*_p)$ to \eqref{e:26d}:
\begin{enumerate}
\item
\label{p:26ci}
$(\forall k\in\{1,\ldots,p\})$ $y^*_{k,n}\weakly\overline{y}^*_k$. 
\item
\label{p:26cii}
$x_n\to\overline{x}$.
\end{enumerate}
\end{proposition}
\begin{proof}
We deduce from \cite[Proposition~22.11(ii)]{Livre1} that,
for every $k\in\{1,\ldots,p\}$, $D_k^{-1}$ is $\nu_k$-cocoercive
with $\dom D_k^{-1}=\GG_k$. Let us set
$\GGG=\GG_1\oplus\cdots\oplus\GG_p$ and 
\begin{equation}
\label{e:-01a}
\begin{cases}
T\colon\HH\to\HH\colon x\mapsto J_{\rho^{-1}A}
\big(\rho^{-1}(z-x)\big)\\[3mm]
{\boldsymbol A}\colon\GGG\to2^{\GGG}\colon\boldsymbol{y}^*\mapsto
{\displaystyle\bigtimes_{1\leq k\leq p}}B^{-1}_ky^*_k\\[5mm]
{\boldsymbol D}\colon\GGG\to\GGG\colon\boldsymbol{y}^*\mapsto
\bigl(D^{-1}_ky^*_k\bigr)_{1\leq k\leq p}\\[2mm]
{\boldsymbol L}\colon\HH\to\GGG\colon x\mapsto
\big(L_kx\big)_{1\leq k\leq p}\\
\boldsymbol{B}=\boldsymbol{D}-\boldsymbol{L}\circ T\circ
\boldsymbol{L}^*.
\end{cases}
\end{equation}
It follows from Lemmas~\ref{l:0602} and \ref{l:0616} that
$\boldsymbol{A}$ is maximally monotone, from \eqref{e:beta1} that
${\boldsymbol D}$ is $\nu$-cocoercive, from 
Lemma~\ref{l:0}\ref{l:0iii} that $-T$ is $\rho$-cocoercive, and
hence from \cite[Proposition~4.12]{Livre1} that 
\begin{equation}
\label{e:D+}
\boldsymbol{B}=
\boldsymbol{D}+\boldsymbol{L}\circ(-T)\circ \boldsymbol{L}^*\;
\text{is $1/(1/\nu+\|\boldsymbol{L}\|^2/\rho)$-cocoercive}.
\end{equation}
Since $\|{\boldsymbol L}\|^2\leq\sum_{k=1}^p\|L_k\|^2$,
\eqref{e:beta1} implies that $\boldsymbol{B}$ is
$\alpha$-cocoercive.
Next, let us define $(\forall n\in\NN)$
$\boldsymbol{y}^*_n=(y^*_{k,n})_{1\leq k\leq p}$ and 
$\boldsymbol{w}_n=(w_{k,n})_{1\leq k\leq p}$.
Then, upon combining \eqref{e:-01a} and Example~\ref{ex:r1},
\eqref{e:-08a} can be rewritten as 
\begin{equation}
\label{e:-08c}
\begin{array}{l}
\text{for}\;n=0,1,\ldots\\
\begin{array}{l}
\left\lfloor
\begin{array}{l}
\boldsymbol{w}_n=\boldsymbol{y}^*_n-\gamma_n 
\boldsymbol{B}\boldsymbol{y}^*_n\\[1mm]
\boldsymbol{y}^*_{n+1}={\boldsymbol y}^*_n+\mu_n
\big(J_{\gamma_n\boldsymbol{A}}\,
\boldsymbol{w}_n-\boldsymbol{y}^*_n\big),
\end{array}
\right.\\[2mm]
\end{array}
\end{array}
\end{equation}
and the dual problem \eqref{e:26d} as
\begin{equation}
\label{e:2012-04-08b}
\text{find}\;\;\boldsymbol{y}^*\in\GGG\;\;
\text{such that}\;\;
\boldsymbol{0}\in\boldsymbol{A}{\boldsymbol{y}^*}+
\boldsymbol{B}{\boldsymbol{y}^*}.
\end{equation}

\ref{p:26ci}: In view of the above, 
the claim follows from Theorem~\ref{t:5}\ref{t:5i}.

\ref{p:26cii}:
We derive from Proposition~\ref{p:26}, \eqref{e:-08a}, and
\eqref{e:-01a} that
\begin{equation}
\overline{x}=T(\boldsymbol{L}^*\overline{\boldsymbol{y}}^*)
\quad\text{and}\quad
(\forall n\in\NN)\quad x_n=T(\boldsymbol{L}^*\boldsymbol{y}^*_n).
\end{equation}
In turn, we deduce from the $\rho$-cocoercivity of $-T$,
\ref{p:26ci}, the monotonicity of ${\boldsymbol D}$, and the
Cauchy--Schwarz inequality that
\begin{align}
\label{e:2012-04-09c}
(\forall n\in\NN)\quad\rho\|x_n-\overline{x}\|^2
&=\rho\|T(\boldsymbol{L}^*\boldsymbol{y}^*_n)-
T(\boldsymbol{L}^*\overline{\boldsymbol{y}}^*)\|^2
\nonumber\\
&\leq\scal{\boldsymbol{L}^*(\boldsymbol{y}^*_n-
{\overline{\boldsymbol y}^*})}{T(\boldsymbol{L}^*
\overline{\boldsymbol{y}}^*)-T(\boldsymbol{L}^*\boldsymbol{y}^*_n)}
\nonumber\\
&=\scal{\boldsymbol{y}^*_n-\overline{\boldsymbol{y}}^*}
{(\boldsymbol{L}\circ T\circ \boldsymbol{L}^*)
\overline{\boldsymbol{y}}^*
-(\boldsymbol{L}\circ T\circ \boldsymbol{L}^*)\boldsymbol{y}^*_n}
\nonumber\\
&\leq
\scal{\boldsymbol{y}^*_n-\overline{\boldsymbol{y}}^*}
{{\boldsymbol D}\boldsymbol{y}^*_n-
{\boldsymbol D}\overline{\boldsymbol{y}}^*}\nonumber\\
&\quad-\scal{\boldsymbol{y}^*_n-\overline{\boldsymbol{y}}^*}
{(\boldsymbol{L}\circ T\circ \boldsymbol{L}^*)
\boldsymbol{y}^*_n
-(\boldsymbol{L}\circ T\circ \boldsymbol{L}^*)
\overline{\boldsymbol{y}}^*}
\nonumber\\
&=\scal{\boldsymbol{y}^*_n-\overline{\boldsymbol{y}}^*}
{\boldsymbol{B}{\boldsymbol{y}^*_n-
\boldsymbol{B}\overline{\boldsymbol{y}}^*}}
\nonumber\\
&\leq\delta\|\boldsymbol{B}\boldsymbol{y}^*_n-
\boldsymbol{B}\overline{\boldsymbol{y}}^*\|
\end{align}
where, by \ref{p:26ci},
\begin{equation}
\delta=\sup_{n\in\NN}\|\boldsymbol{y}^*_n-
\overline{\boldsymbol{y}}^*\|<\pinf.
\end{equation}
Therefore, using \eqref{e:-08c} and 
Theorem~\ref{t:5}\ref{t:5i+}--\ref{t:5ii},
we conclude that $\|x_n-\overline{x}\|\to 0$.
\end{proof}

Here is an application to strongly convex minimization problems
that arise in particular in mechanics \cite{Ekel74,Merc80} and in
signal processing \cite{Svva10,Jmaa11,Pott93}.

\begin{example}
\label{ex:26c}
Let $0<p\in\NN$, let $z\in\HH$, let $f\in\Gamma_0(\HH)$, and let
$\moyo{(f^*)}{1}$ be the Moreau envelope of $f^*$
(see \eqref{e:jjm3}). For
every $k\in\{1,\ldots,p\}$, let $g_k\in\Gamma_0(\GG_k)$, let
$\nu_k\in\RPP$, let $h_k\in\Gamma_0(\GG_k)$ be $\nu_k$-strongly
convex, and suppose that $0\neq L_k\in\BL(\HH,\GG_k)$. Define 
$\alpha$ as in \eqref{e:beta1} and suppose that  
\begin{equation}
\label{e:cq1}
z\in\ran\bigg(\partial f+\sum_{k=1}^pL_k^*\circ
(\partial g_k\infconv\partial h_k)\circ L_k+\Id\bigg).
\end{equation}
Then the primal problem 
\begin{equation}
\label{e:p26c}
\minimize{x\in\HH}{f(x)+\sum_{k=1}^p\,(g_k\infconv h_k)
(L_kx)+\frac{1}{2}\|x-z\|^2}
\end{equation}
admits a unique solution $\overline{x}$, namely
\begin{equation}
\overline{x}=\prox_{f+\sum_{k=1}^p\,(g_k\infconv h_k)\circ L_k}z,
\end{equation}
and the dual problem is
\begin{equation}
\label{e:d26c}
\minimize{{y^*_1}\in\GG_1,\:\ldots,\:{y^*_p}\in\GG_p}{
\moyo{\bigl(f^*\bigr)}{1}\biggl(z-\sum_{k=1}^pL_k^*y^*_k\biggr)
+\sum_{k=1}^p\bigl(g_k^*(y^*_k)+h_k^*(y^*_k)\bigr)}.
\end{equation}
Now let $\varepsilon\in\left]0,\alpha/(\alpha+1)\right[$,
let $(\gamma_n)_{n\in\NN}$ be a sequence in 
$[\varepsilon,(2-\varepsilon)\alpha]$, suppose that 
$(\mu_n)_{n\in\NN}$ satisfies \eqref{e:aqp63}, and, for every 
$k\in\{1,\ldots,p\}$, let $y^*_{k,0}\in\GG_k$. Iterate 
\begin{equation}
\label{e:+08c}
\begin{array}{l}
\text{for}\;n=0,1,\ldots\\
\left\lfloor
\begin{array}{l}
q_n=z-\sum_{k=1}^pL_k^*y^*_{k,n}\\
x_n=\prox_fq_n\\
\operatorname{for}\;k=1,\ldots,p\\
\left\lfloor
\begin{array}{l}
w_{k,n}=y^*_{k,n}+\gamma_n\bigl(L_kx_n-\nabla h_k^*(y^*_{k,n})
\bigr)\\[2mm]
y^*_{k,n+1}=y^*_{k,n}+\mu_n\bigl(\prox_{\gamma_ng_k^*}
w_{k,n}-y^*_{k,n}\bigr).\\[1mm] 
\end{array}
\right.\\[2mm]
\end{array}
\right.
\end{array}
\end{equation}
Then the following hold:
\begin{enumerate}
\item
\label{ex:26ci}
There exists a solution 
$(\overline{y}_1^*,\ldots,\overline{y}_p^*)$ to \eqref{e:d26c} 
such that $(\forall k\in\{1,\ldots,p\})$ 
$y^*_{k,n}\weakly\overline{y}^*_k$. 
\item
\label{ex:26cii}
$x_n\to\overline{x}$.
\end{enumerate}
\end{example}
\begin{proof}
Apply Proposition~\ref{p:26c} with $\rho=1$, $A=\partial f$, and
$(\forall k\in\{1,\ldots,p\})$ $B_k=\partial g_k$ and 
$D_k=\partial h_k$ (see \cite[Eample~5.6]{Opti14} for details).
\end{proof}

\begin{remark}
\label{r:26}
In Example~\ref{ex:26c}, suppose that $\HH=H_0^1(\Omega)$, where 
$\Omega$ is a bounded open domain in $\RR^2$, $p=1$, 
$\GG_1=L^2(\Omega)\oplus L^2(\Omega)$, $L_1=\nabla$, 
$g_1=\mu\|\cdot\|_{2,1}$ with $\mu\in\RPP$, and
$h_1=\iota_{\{0\}}$. Then \eqref{e:p26c} reduces to
\begin{equation}
\label{e:p26t}
\minimize{x\in H_0^1(\Omega)}{f(x)+
\mu\int_\Omega|\nabla x(\omega)|_2d\omega+\frac{1}{2}\|x-z\|^2}.
\end{equation}
In mechanics, \eqref{e:p26t} has been studied
for certain potentials $f$ \cite{Ekel74}. For instance,
$f=0$ yields Mossolov's problem and its dual analysis is carried
out in \cite[Section~IV.3.1]{Ekel74}. In image processing,
Mossolov's problem corresponds to the total variation denoising
problem. In 1980, Mercier \cite{Merc80} proposed a
dual projection algorithm to solve Mossolov's problem. In image
processing, this approach was rediscovered in a discrete setting
in \cite{Cham04,Cham05}.
\end{remark}

\subsubsection{Barycentric Dykstra-like algorithm}

Using Proposition~\ref{p:26c} and, thereby, the forward-backward
algorithm, we obtain a method for computing the resolvent of a 
sum of maximally monotone operators. This result, which generalizes
the barycentric Dykstra algorithm of \cite{Gaff89} for projecting
onto an intersection of closed convex sets, was originally derived
in \cite[Theorem~3.3]{Joca09} with different techniques. 

\begin{proposition}
\label{p:26w}
Let $0<p\in\NN$, let $z\in\HH$, and, for every 
$k\in\{1,\ldots,p\}$, let $A_k\colon\HH\to 2^{\HH}$ be maximally
monotone. Suppose that  
\begin{equation}
\label{e:cq4}
z\in\ran\bigg(\sum_{k=1}^p A_k+\Id\bigg)
\end{equation}
and consider the inclusion problem
\begin{equation}
\label{e:p26w}
\text{find}\;\;x\in\HH\;\;\text{such that}\;\;
z\in\sum_{k=1}^pA_kx+x.
\end{equation}
Set $x_0=z$ and $(\forall k\in\{1,\ldots,p\})$ $z_{k,0}=z$.
Iterate
\begin{equation}
\label{e:-08r}
\begin{array}{l}
\text{for}\;n=0,1,\ldots\\
\left\lfloor
\begin{array}{l}
\operatorname{for}\;k=1,\ldots,p\\
\left\lfloor
\begin{array}{l}
r_{k,n}=J_{pA_k}z_{k,n}\\
\end{array}
\right.\\[1mm]
x_{n+1}=(1/p)\sum_{k=1}^pr_{k,n}\\
\operatorname{for}\;k=1,\ldots,p\\
\left\lfloor
\begin{array}{l}
z_{k,n+1}=z_{k,n}-r_{k,n}+x_{n+1}.
\end{array}
\right.\\[2mm]
\end{array}
\right.
\end{array}
\end{equation}
Then $x_n\to J_{\sum_{k=1}^p A_k}z$.
\end{proposition}
\begin{proof}
First, we observe that \eqref{e:cq4}--\eqref{e:p26w} is the 
special case of \eqref{e:26a}--\eqref{e:26p} in which $A=0$ and, 
for every $k\in\{1,\ldots,p\}$, $\GG_k=\HH$, $B_k=A_k$, $L_k=\Id$,
and $D_k=\{0\}^{-1}$. Moreover, the cocoercivity constant in
\eqref{e:beta1} is $\alpha=1/p$. With this scenario, implementing
\eqref{e:-08a} with, for every $n\in\NN$, $\mu_n=1$ and
$\gamma_n=1/p$, and, for every $k\in\{1,\ldots,p\}$, $y^*_{k,0}=0$
leads to the recursion
\begin{equation}
\label{e:-08w}
\begin{array}{l}
\text{for}\;n=0,1,\ldots\\
\left\lfloor
\begin{array}{l}
x_n=z-\sum_{k=1}^py^*_{k,n}\\
\operatorname{for}\;k=1,\ldots,p\\
\left\lfloor
\begin{array}{l}
y^*_{k,n+1}=J_{A_k^{-1}/p}\bigl(y^*_{k,n}+x_n/p\bigr)\\[1mm] 
\end{array}
\right.\\[2mm]
\end{array}
\right.
\end{array}
\end{equation}
and Proposition~\ref{p:26c}\ref{p:26cii} guarantees that
$x_n\to J_{\sum_{k=1}^pA_k}z$. Alternatively, with the 
initialization $x_0=z$, we rewrite \eqref{e:-08w} as 
\begin{equation}
\label{e:-08s}
\begin{array}{l}
\text{for}\;n=0,1,\ldots\\
\left\lfloor
\begin{array}{l}
\operatorname{for}\;k=1,\ldots,p\\
\left\lfloor
\begin{array}{l}
y^*_{k,n+1}=J_{A_k^{-1}/p}\bigl(y^*_{k,n}+x_n/p\bigr)\\[1mm] 
\end{array}
\right.\\[2mm]
x_{n+1}=z-\sum_{k=1}^py^*_{k,n+1}.
\end{array}
\right.
\end{array}
\end{equation}
Let us introduce the variables
$(\forall n\in\NN)(\forall k\in\{1,\ldots,p\})$ 
$z_{k,n}=py^*_{k,n}+x_n$, where $z_{k,0}=x_0=z$. Then
\eqref{e:-08s} corresponds to the iterations
\begin{equation}
\label{e:-08t}
\begin{array}{l}
\text{for}\;n=0,1,\ldots\\
\left\lfloor
\begin{array}{l}
x_{n+1}=z-\sum_{k=1}^pJ_{A_k^{-1}/p}\bigl(z_{k,n}/p\bigr)\\
\operatorname{for}\;k=1,\ldots,p\\
\left\lfloor
\begin{array}{l}
z_{k,n+1}=p\,J_{A_k^{-1}/p}\bigl(z_{k,n}/p\bigr)+x_{n+1}.
\end{array}
\right.\\[2mm]
\end{array}
\right.
\end{array}
\end{equation}
By construction, 
\begin{equation}
(\forall n\in\NN)\quad\sum_{k=1}^pz_{k,n}=pz.
\end{equation}
Hence, appealing to \eqref{e:yosi1}, \eqref{e:-08t} becomes 
\begin{equation}
\label{e:-08m}
\begin{array}{l}
\text{for}\;n=0,1,\ldots\\
\left\lfloor
\begin{array}{l}
x_{n+1}=(1/p)\sum_{k=1}^pJ_{pA_k}z_{k,n}\\
\operatorname{for}\;k=1,\ldots,p\\
\left\lfloor
\begin{array}{l}
z_{k,n+1}=z_{k,n}-J_{pA_k}z_{k,n}+x_{n+1},
\end{array}
\right.\\[2mm]
\end{array}
\right.
\end{array}
\end{equation}
which is precisely \eqref{e:-08r}.
\end{proof}

\begin{example}
\label{ex:26z}
Consider the instantiation of Proposition~\ref{p:26w} in which,
for every $k\in\{1,\ldots,p\}$, $A_k=\partial f_k$, with
$f_k\in\Gamma_0(\HH)$, and execute \eqref{e:-08r}, which becomes
\begin{equation}
\label{e:-08f}
\begin{array}{l}
\text{for}\;n=0,1,\ldots\\
\left\lfloor
\begin{array}{l}
\operatorname{for}\;k=1,\ldots,p\\
\left\lfloor
\begin{array}{l}
r_{k,n}=\prox_{pf_k}z_{k,n}\\
\end{array}
\right.\\[1mm]
x_{n+1}=(1/p)\sum_{k=1}^pr_{k,n}\\
\operatorname{for}\;k=1,\ldots,p\\
\left\lfloor
\begin{array}{l}
z_{k,n+1}=z_{k,n}-r_{k,n}+x_{n+1}.
\end{array}
\right.\\[2mm]
\end{array}
\right.
\end{array}
\end{equation}
Then $x_n\to\prox_{\sum_{k=1}^pf_k}z$. 
\end{example}

Our last example addresses the barycentric Dykstra algorithm
\emph{per se}. The original Dykstra algorithm was devised in
\cite{Dyks83} to project onto the intersection of closed convex
cones (see also \cite{Hans88} for general closed convex sets whose
intersection has a nonempty interior) in Euclidean spaces using
periodic applications of the projectors onto the individual sets.
Convergence of this periodic scheme in the general case of
arbitrary closed and convex sets in Hilbert spaces was established
in \cite{Boyl86} (see \cite{Pjoo08} for an extension to
monotone operators). The barycentric version described below,
in which all the projectors are used at each iteration, was devised
in \cite[Section~6]{Gaff89}. Its connection with the
forward-backward algorithm is discussed in
\cite[Remark~3.8]{Svva10} and \cite[Remark~2.3]{Jmaa11}, and its
asymptotic behavior in the inconsistent case in
\cite[Theorem~6.1]{Baus94}.

\begin{example}
\label{ex:26n}
In Example~\ref{ex:26z}, suppose that, for every
$k\in\{1,\ldots,p\}$, $f_k=\iota_{C_k}$, where $C_k$ is a
nonempty closed convex subset of $\HH$. Then algorithm
\eqref{e:-08f} becomes
\begin{equation}
\label{e:-08n}
\begin{array}{l}
\text{for}\;n=0,1,\ldots\\
\left\lfloor
\begin{array}{l}
\operatorname{for}\;k=1,\ldots,p\\
\left\lfloor
\begin{array}{l}
r_{k,n}=\proj_{C_k}z_{k,n}\\
\end{array}
\right.\\[1mm]
x_{n+1}=(1/p)\sum_{k=1}^pr_{k,n}\\
\operatorname{for}\;k=1,\ldots,p\\
\left\lfloor
\begin{array}{l}
z_{k,n+1}=z_{k,n}-r_{k,n}+x_{n+1}
\end{array}
\right.\\[2mm]
\end{array}
\right.
\end{array}
\end{equation}
and $x_n\to\proj_{\bigcap_{k=1}^pC_k}z$. 
\end{example}

\subsubsection{Renorming}
\label{sec:ren2}

We preface our discussion with a renormed version of
Theorem~\ref{t:5}.

\begin{proposition}
\label{p:6}
Let $\alpha\in\RPP$, let $\beta\in\RPP$, let 
$A\colon\HH\to 2^{\HH}$ be maximally monotone, let 
$B\colon\HH\to\HH$ be $\alpha$-cocoercive, let
$U\in\BL(\HH)$ be self-adjoint and $\beta$-strongly monotone, and
let $\mathcal{X}$ be the real Hilbert space obtained by endowing
$\HH$ with the scalar product $(x,y)\mapsto\scal{Ux}{y}$. Let 
$\varepsilon\in\left]0,\alpha\beta/(\alpha\beta+1)\right[$, let
$(\gamma_n)_{n\in\NN}$ be a sequence in
$\left[\varepsilon,(2-\varepsilon)\alpha\beta\right]$, and let
$(\lambda_n)_{n\in\NN}$ be a sequence in 
$\left[\varepsilon,1\right]$. Suppose that the set $Z$ of solutions
to the problem 
\begin{equation}
\label{e:35p}
\text{find}\;\;x\in\HH\;\;\text{such that}\;\;0\in Ax+Bx
\end{equation}
is not empty and let $Z^*$ be the set of solutions to the dual
problem 
\begin{equation}
\label{e:35d}
\text{find}\;\;x^*\in\HH\;\;\text{such that}\;\;
0\in -A^{-1}(-x^*)+B^{-1}x^*.
\end{equation}
Let $x_0\in\HH$ and iterate
\begin{equation}
\label{e:algo35}
\begin{array}{l}
\text{for}\;n=0,1,\ldots\\
\left\lfloor
\begin{array}{l}
u_n^*=\gamma_n^{-1}Ux_n-Bx_n\\
w_n=\big(\gamma_n^{-1}U+A\big)^{-1}u_n^*\\
x_{n+1}=x_n+\lambda_n(w_n-x_n).
\end{array}
\right.\\
\end{array}
\end{equation}
Then the following hold:
\begin{enumerate}
\item
\label{p:6i}
$(x_n)_{n\in\NN}$ converges weakly to a point in $Z$.
\item
\label{p:6i+}
$Z^*$ contains a single point $\overline{x}^*$ and
$(\forall z\in Z)$ $Bz=\overline{x}^*$.
\item
\label{p:6ii}
$(Bx_n)_{n\in\NN}$ converges strongly to $\overline{x}^*$.
\end{enumerate}
\end{proposition}
\begin{proof}
We derive from Lemma~\ref{l:9} and Example~\ref{ex:b12} that 
\begin{equation}
\label{e:fu}
(\forall n\in\NN)\quad x_{n+1}=x_n+\lambda_n\Bigl(
J_{\gamma_n U^{-1}\circ A}\bigl(x_n-\gamma_nU^{-1}(Bx_n)\bigr)
-x_n\Bigr),
\end{equation}
where $U^{-1}\circ A\colon\mathcal{X}\to 2^{\mathcal{X}}$ is
maximally monotone, 
$U^{-1}\circ B\colon\mathcal{X}\to\mathcal{X}$ is
$\alpha\beta$-cocoercive, and $\zer(A+B)=\zer(U^{-1}\circ(A+B))$.
Hence the assertions follow from Theorem~\ref{t:5} applied to
$U^{-1}\circ A$ and $U^{-1}\circ B$ in $\mathcal{X}$.
\end{proof}

\begin{remark}
\label{r:wr2}
In terms of the warped resolvents of Section~\ref{sec:wr}, 
\eqref{e:algo35} can be condensed into
\begin{equation}
(\forall n\in\NN)\quad
x_{n+1}=x_n+\lambda_n\Bigl(J_{\gamma_n(A+B)}^{U_n}x_n-x_n\Bigr),
\;\;\text{where}\;\;U_n=U-\gamma_n B.
\end{equation}
\end{remark}

We present an approach proposed in \cite{Bang13}, which
revisited the primal-dual setting of \cite{Svva12} discussed in
Proposition~\ref{p:18} by replacing the monotone Lipschitz
property of the operators $C$ and $(D_k^{-1})_{1\leq k\leq p}$ with
the stronger cocoercivity property. 

\begin{proposition}[{\protect{\cite[Theorem~3.1(i)]{Bang13}}}]
\label{p:21}
Let $0<p\in\NN$, let $\alpha\in\RPP$, let $A\colon\HH\to 2^{\HH}$ 
be maximally monotone, and let $C\colon\HH\to\HH$ be 
$\alpha$-cocoercive. For every $k\in\{1,\ldots,p\}$, let 
$\beta_k\in\RPP$, let $\GG_k$ be a real Hilbert space, 
let $B_k\colon\GG_k\to 2^{\GG_k}$ be maximally monotone,
let $D_k\colon\GG_k\to 2^{\GG_k}$ be maximally monotone and
$\beta_k$-strongly monotone, and
suppose that $0\neq L_k\in\BL(\HH,\GG_k)$. Additionally, suppose
that the set $Z$ of solutions to the primal inclusion 
\begin{equation}
\label{e:bap}
\text{find}\;\;x\in\HH\;\;\text{such that}\;\;
0\in Ax+\sum_{k=1}^pL_k^*\big((B_k\infconv D_k)(L_kx)\big)+Cx
\end{equation}
is not empty and let $Z^*$ be the set of solutions to 
the dual inclusion
\begin{multline}
\label{e:bad}
\text{find}\;\;y^*_1\in\GG_1,\ldots,y^*_p\in\GG_p
\;\:\text{such that}\\
(\exi x\in\HH)\;
\begin{cases}
x\in(A+C)^{-1}\biggl(-\sum_{k=1}^pL_k^*{y^*_k}\biggr)\\
\bigl(\forall k\in\{1,\ldots,p\}\bigr)\;L_kx
\in B^{-1}_ky^*_k+D^{-1}_ky^*_k.
\end{cases}
\end{multline}
Let $\varepsilon\in\zeroun$, let $(\lambda_n)_{n\in\NN}$ be a
sequence in $[\varepsilon,1]$, let $x_0\in\HH$, let
$(y^*_{1,0},\ldots,y^*_{p,0})\in\GG_1\oplus\cdots\oplus\GG_p$, 
let $\tau\in\RPP$, and let
$(\sigma_1,\ldots,\sigma_p)\in\RPP^p$. Set
\begin{equation}
\label{e:betav}
\aleph=\min\{\alpha,\beta_1,\ldots,\beta_p\}\quad\text{and}\quad
\beta=\dfrac{1-\sqrt{\tau\sum_{k=1}^p\sigma_k\|L_k\|^2}}
{\max\{\tau,\sigma_1,\ldots,\sigma_p\}}
\end{equation}
and assume that
\begin{equation}
\label{e:z5}
\aleph\beta>\dfrac{1}{2}.
\end{equation}
Iterate
\begin{equation}
\label{e:v1}
\begin{array}{l}
\text{for}\;n=0,1,\ldots\\
\left\lfloor
\begin{array}{l}
x^*_n=\tau\bigl(\sum_{k=1}^pL_k^*y_{k,n}^*+Cx_n\bigr)\\
p_n=J_{\tau A}(x_n-x^*_n)\\
x_{n+1}=x_n+\lambda_n(p_n-x_n)\\
\text{for}\;k=1,\ldots,p\\
\left\lfloor
\begin{array}{l}
y_{k,n}=\sigma_k\bigl(L_k(2p_n-x_n)-D_k^{-1}y^*_{k,n}\bigr)\\
q^*_{k,n}=J_{\sigma_k B_k^{-1}}(y^*_{k,n}+y_{k,n})\\
y^*_{k,n+1}=y^*_{k,n}+\lambda_n(q^*_{k,n}-y^*_{k,n}).
\end{array}
\right.\\[1mm]
\end{array}
\right.\\[1mm]
\end{array}
\end{equation}
Then there exist $x\in Z$ and $(y^*_1,\ldots,y^*_p)\in Z^*$ such 
that $x_n\weakly{x}$, and, for every $k\in\{1,\ldots,p\}$,
$y^*_{k,n}\weakly{y^*_k}$.
\end{proposition}
\begin{proof}
Set $\XXX=\HH\oplus\GG_1\oplus\cdots\oplus\GG_p$ and 
\begin{equation}
\label{e:v2}
\begin{cases}
\boldsymbol{M}\colon\XXX\to 2^{\XXX}\colon
(x,y^*_1,\ldots,y^*_p)\mapsto\\
\hspace{20mm}\bigl(Ax+\sum_{k=1}^pL_k^*y^*_k\big)\times
\bigl(-L_1x+B^{-1}_1y^*_1\bigr)\times
\cdots\times\bigl(-L_px+B^{-1}_py^*_p\bigr)\\
\boldsymbol{C}\colon\XXX\to\XXX\colon(x,y^*_1,\ldots,y^*_p)
\mapsto\bigl(Cx,D_1^{-1}y^*_1,\ldots,D_p^{-1}y^*_p\bigr)\\
\boldsymbol{U}\colon\XXX\to\XXX\colon(x,y^*_1,\ldots,y^*_p)
\mapsto\\
\hspace{20mm}\bigl(\tau^{-1}x-\sum_{k=1}^pL_k^*y^*_k,
-L_1x+\sigma_1^{-1}y^*_1,\ldots,-L_px+\sigma_p^{-1}y^*_p\bigr).
\end{cases}
\end{equation}
As in \eqref{e:UU}, $\boldsymbol{M}$ is maximally monotone,
while $\boldsymbol{C}$ is $\aleph$-cocoercive.
Furthermore, $\boldsymbol{U}\in\BL(\HH)$ is self-adjoint and, as
shown in \cite[Equation~(3.20)]{Bang13}, \eqref{e:z5} implies that
it is $\beta$-strongly monotone. Now set $(\forall n\in\NN)$
$\boldsymbol{x}_n=(x_n,y^*_{1,n},\ldots,y^*_{p,n})$ and
$\boldsymbol{w}_n=(p_n,q^*_{1,n},\ldots,q^*_{p,n})$. Then, 
adopting the same pattern as in the proof of 
Example~\ref{ex:1Ub}, we rewrite \eqref{e:v1} as
\begin{equation}
\label{e:algo35b}
\begin{array}{l}
\text{for}\;n=0,1,\ldots\\
\left\lfloor
\begin{array}{l}
\boldsymbol{u}_n^*=\boldsymbol{U}
\boldsymbol{x}_n-\boldsymbol{C}\boldsymbol{x}_n\\
\boldsymbol{w}_n=\big(\boldsymbol{U}+
\boldsymbol{M}\big)^{-1}\boldsymbol{u}_n^*\\
\boldsymbol{x}_{n+1}=\boldsymbol{x}_n+
\lambda_n(\boldsymbol{w}_n-\boldsymbol{x}_n),
\end{array}
\right.\\
\end{array}
\end{equation}
and thus recover \eqref{e:algo35} with $(\forall n\in\NN)$ 
$\gamma_n=1<2\aleph\beta$. 
We therefore appeal to Proposition~\ref{p:6}\ref{p:6i} to obtain
the weak convergence of $(\boldsymbol{x}_n)_{n\in\NN}$ to a point 
$(x,y^*_1,\ldots,y_p^*)\in\zer(\boldsymbol{M}+\boldsymbol{C})$.
However, replacing $A$ with $A+C$ and 
$(B^{-1}_k)_{1\leq k\leq p}$ with
$(B^{-1}_k+D^{-1}_k)_{1\leq k\leq p}$ in
Lemma~\ref{l:37z}\ref{l:37zii} yields
$\zer(\boldsymbol{M}+\boldsymbol{C})\subset Z\times Z^*$.
\end{proof}

\begin{remark}
In terms of Framework~\ref{f:1}, the embedding underlying
Proposition~\ref{p:21} employs
$\XXX=\HH\oplus\GG_1\oplus\cdots\oplus\GG_p$,
$\MMM=\boldsymbol{M}+\boldsymbol{C}$, and
$\TTT\colon\XXX\to\HH\colon(x,y^*_1,\ldots,y^*_p)\mapsto x$.
\end{remark}

The following application to minimization revisits the setting of
Example~\ref{ex:18} and Remark~\ref{r:31}.

\begin{example}
\label{ex:20h}
Let $0<p\in\NN$, let $\alpha\in\RPP$, let $f\in\Gamma_0(\HH)$, 
and let $h\colon\HH\to\RR$ be convex, differentiable, and such
that $\nabla h$ is $1/\alpha$-Lipschitzian. For every
$k\in\{1,\ldots,p\}$, let 
$\beta_k\in\RPP$, let $\GG_k$ be a real Hilbert space, 
let $g_k\in\Gamma_0(\GG_k)$,
let $\ell_k\in\Gamma_0(\GG_k)$ be $\beta_k$-strongly convex, 
and suppose that $0\neq L_k\in\BL(\HH,\GG_k)$. Let $Z$ be the set
of solutions to the primal problem
\begin{equation}
\label{e:bap2}
\minimize{x\in\HH}
{f(x)+\sum_{k=1}^p(g_k\infconv\ell_k)(L_kx)+h(x)},
\end{equation}
let $Z^*$ be the set of solutions to the dual problem
\begin{equation}
\label{e:bad2}
\minimize{y^*_1\in\GG_1,\ldots,y^*_p\in\GG_p}
{(f^*\infconv h^*)\biggl(-\sum_{k=1}^pL_k^*{y^*_k}\biggr)
+\sum_{k=1}^p\bigl(g_k^*(y^*_k)+\ell_k^*(y^*_k)\bigr)},
\end{equation}
and suppose that
\begin{equation}
\label{e:sv12}
\zer\biggl(\partial f+\sum_{k=1}^pL_k^*\circ\bigl(
\partial g_k\infconv\partial\ell_k\bigr)\circ L_k
+\nabla h\biggr)
\neq\emp.
\end{equation}
Let $\varepsilon\in\zeroun$, let $(\lambda_n)_{n\in\NN}$ be a
sequence in $[\varepsilon,1]$, let $x_0\in\HH$, let
$(y^*_{1,0},\ldots,y^*_{p,0})\in\GG_1\oplus\cdots\oplus\GG_p$, 
let $\tau\in\RPP$, and let
$(\sigma_1,\ldots,\sigma_p)\in\RPP^p$ be such that
\eqref{e:betav}--\eqref{e:z5} hold. Iterate
\begin{equation}
\label{e:v4}
\begin{array}{l}
\text{for}\;n=0,1,\ldots\\
\left\lfloor
\begin{array}{l}
x^*_n=\tau\bigl(\sum_{k=1}^pL_k^*y_{k,n}^*+\nabla h(x_n)\bigr)\\
p_n=\prox_{\tau f}(x_n-x^*_n)\\
x_{n+1}=x_n+\lambda_n(p_n-x_n)\\
\text{for}\;k=1,\ldots,p\\
\left\lfloor
\begin{array}{l}
y_{k,n}=\sigma_k\bigl(L_k(2p_n-x_n)-\nabla\ell_k^{*}(y^*_{k,n})
\bigr)\\
q^*_{k,n}=\prox_{\sigma_k g_k^*}(y^*_{k,n}+y_{k,n})\\
y^*_{k,n+1}=y^*_{k,n}+\lambda_n(q^*_{k,n}-y^*_{k,n}).
\end{array}
\right.\\[1mm]
\end{array}
\right.\\[1mm]
\end{array}
\end{equation}
Then there exist $x\in Z$ and $(y^*_1,\ldots,y^*_p)\in Z^*$ such 
that $x_n\weakly{x}$, and, for every $k\in\{1,\ldots,p\}$,
$y^*_{k,n}\weakly{y^*_k}$.
\end{example}
\begin{proof}
It follows from the arguments presented in \cite[Section~4]{Svva12}
that this is an application of Proposition~\ref{p:21} with
$A=\partial f$, $C=\nabla h$, and $(\forall k\in\{1,\ldots,p\})$
$B_k=\partial g_k$ and $D_k=\partial \ell_k$.
\end{proof}

\begin{remark}
\label{r:cv}
If we make the additional assumptions that, for every
$k\in\{1,\ldots,p\}$, $\ell_k=\iota_{\{0\}}$ and
$\sigma_k=\sigma_1$, Example~\ref{ex:20h} was independently
obtained in \cite[Section~5]{Cond13}. For this reason,
\eqref{e:v4} in this particular setting is called the
\emph{Condat--V\~u} algorithm. 
\end{remark}

\subsection{Forward-backward-half-forward splitting}
\label{sec:fbfh}

Let $A\colon\HH\to 2^{\HH}$ be maximally monotone, let
$C\colon\HH\to\HH$ be cocoercive, and let $Q\colon\HH\to\HH$ be
monotone and Lipschitzian. Then a zero of $M=A+C+Q$ can be
constructed through the forward-backward-forward algorithms of
Theorem~\ref{t:6} or Theorem~\ref{t:6s}, applied to $A$ and the
monotone and Lipschitzian operator $B=C+Q$. These algorithms
require two applications of $B$, i.e., two applications of $C$ and
$Q$, at each iteration. However, the algorithms discussed so far
require two applications of a monotone Lipschitzian operator per
iteration, as in the Antipin--Korpelevi\v{c} method of
Section~\ref{sec:korp} and the 
forward-backward-forward methods of Sections~\ref{sec:fbfw} and
\ref{sec:fbfs}, but only one application of a cocoercive operator,
as in the Euler method of Section~\ref{sec:eul} and the
forward-backward methods of Sections~\ref{sec:fbw} and
\ref{sec:fbs}. It is therefore natural to ask whether one can find
a zero of $A+C+Q$ using only one application of $C$ per iteration.
A positive answer to this question was given in \cite{Bric18} 
with the following forward-backward-half-forward splitting
algorithm. We provide a simple proof of its convergence 
using our geometric framework. 

\begin{proposition}[{\protect{\cite[Theorem~2.3.1]{Bric18}}}]
\label{p:22}
Let $\alpha\in\RPP$, let $\beta\in\RPP$, let
$A\colon\HH\to 2^{\HH}$ be maximally monotone, let 
$C\colon\HH\to\HH$ be $\alpha$-cocoercive, let $Q\colon\HH\to\HH$
be monotone and $\beta$-Lipschitzian, and suppose that the set of
solutions $Z$ to the inclusion
\begin{equation}
\label{e:n4}
\text{find}\;\;x\in\HH\;\;\text{such that}\;\;0\in Ax+Cx+Qx
\end{equation}
is not empty. Let $x_0\in\HH$, set
$\chi=4\alpha/(1+\sqrt{1+16\alpha^2\beta^2})$, let 
$\varepsilon\in\left]0,\chi/(\chi+1)\right[$, and let
$(\gamma_n)_{n\in\NN}$ be a sequence in 
$\left[\varepsilon,(1-\varepsilon)\chi\right]$.
Iterate
\begin{equation}
\label{e:half}
\begin{array}{l}
\text{for}\;n=0,1,\ldots\\
\left\lfloor
\begin{array}{l}
c_n^*=\gamma_nCx_n\\
q_n^*=\gamma_nQx_n\\
w_n=J_{\gamma_nA}(x_n-c_n^*-q_n^*)\\
x_{n+1}=w_n-\gamma_nQw_n+q_n^*.
\end{array}
\right.\\
\end{array}
\end{equation}
Then $(x_n)_{n\in\NN}$ converges weakly to a point in $Z$.
\end{proposition}
\begin{proof}
The claims will be established as an application of
Theorem~\ref{t:8} with 
\begin{equation}
\label{e:aqp50}
W=A+Q,\;\text{and}\;\;(\forall n\in\NN)\;\;
U_n=\gamma_n^{-1}\Id-C-Q\;\:\text{and}\:\;q_n=x_n. 
\end{equation}
In this setting, \cite[Proposition~3.9]{Jmaa20} implies that 
\eqref{e:aqp27} is satisfied, we have
\begin{align}
(\forall n\in\NN)\quad J_{W+C}^{U_n}
&=\bigl(\gamma_n^{-1}\Id+A\bigr)\circ
\bigl(\gamma_n^{-1}\Id-C-Q\bigr)\nonumber\\
&=J_{\gamma_nA}\circ\bigl(\Id-\gamma_n(C+Q)\bigr),
\end{align}
and the variables of \eqref{e:fejer8} become
\begin{equation}
\label{e:aqp51}
(\forall n\in\NN)\quad 
\begin{cases}
w_n=J_{\gamma_nA}\bigl(x_n-\gamma_n(Cx_n+Qx_n)\bigr)\\[2mm]
t^*_n=\bigl(\gamma_n^{-1}\Id-Q\bigr)x_n-
\bigl(\gamma_n^{-1}\Id-Q\bigr)w_n\\[2mm]
\delta_n=\biggl(\dfrac{1}{\gamma_n}
-\dfrac{1}{4\alpha}\biggr)\|w_n-x_n\|^2
-\scal{w_n-x_n}{Qw_n-Qx_n}.
\end{cases}
\end{equation}
Now set
\begin{equation}
\label{e:aqp52}
(\forall n\in\NN)\quad\lambda_n=
\begin{cases}
\dfrac{\gamma_n\|t_n^*\|^2}{\delta_n},
&\text{if}\:\:\delta_n>0;\\
\varepsilon,&\text{otherwise}
\end{cases}
\end{equation}
and note that the assumptions yield
\begin{equation}
\label{e:kE}
\inf_{n\in\NN}\lambda_n>0\;\;\text{and}\;\;
\sup_{n\in\NN}\lambda_n<2.
\end{equation}
As a consequence of \eqref{e:aqp51} and the properties of $Q$, 
we have
\begin{align}
\label{e:aqp53}
(\forall n\in\NN)\quad\delta_n\leq 0
&\Rightarrow\;
\biggl(\dfrac{1}{\gamma_n}
-\dfrac{1}{4\alpha}-\beta\biggr)\|w_n-x_n\|^2\leq 0\nonumber\\
&\Leftrightarrow\;w_n=x_n\nonumber\\
&\Leftrightarrow\;t^*_n=0.
\end{align}
Hence, \eqref{e:fejer8} yields
\begin{equation}
\label{e:aqp54}
(\forall n\in\NN)\quad d_n=\dfrac{\gamma_n}{\lambda_n}t^*_n
=\dfrac{1}{\lambda_n}\bigl(
x_n-w_n+\gamma_n(Qw_n-Qx_n)\bigr).
\end{equation}
As a result, the sequence
$(x_n)_{n\in\NN}$ produced by \eqref{e:half} coincides with that of
\eqref{e:fejer8}. Hence, by Theorem~\ref{t:8}\ref{t:8i} and
\eqref{e:kE}, $\sum_{n\in\NN}\|d_n\|^2<\pinf$ which, in view of
\eqref{e:aqp54}, yields 
\begin{equation}
\label{e:aqp55}
(\Id-\gamma_nQ)w_n-(\Id-\gamma_nQ)x_n\to 0.
\end{equation}
However, since $\chi\leq1/\beta$, $(\gamma_n)_{n\in\NN}$ lies
in $[\varepsilon,(1-\varepsilon)/\beta]$ and 
Lemma~\ref{l:20}\ref{l:20i} implies that the
operators $(\Id-\gamma_nQ)_{n\in\NN}$ are $\varepsilon$-strongly
monotone. Hence,
\begin{equation}
(\forall n\in\NN)\quad\varepsilon\|w_n-x_n\|^2\leq
\scal{w_n-x_n}{(\Id-\gamma_nQ)w_n-(\Id-\gamma_nQ)x_n}
\end{equation}
and, by the Cauchy--Schwarz inequality and \eqref{e:aqp55},
\begin{equation}
\label{e:aqp56}
\|w_n-x_n\|\leq\varepsilon^{-1}
\|(\Id-\gamma_nQ)w_n-(\Id-\gamma_nQ)x_n\|\to 0.
\end{equation}
In turn, since $C$ is $1/\alpha$-Lipschitzian, 
these facts confirm that
\begin{align}
\|U_nw_n-U_nx_n\|
&\leq\gamma_n^{-1}
\|(\Id-\gamma_nQ)w_n-(\Id-\gamma_nQ)x_n\|+\|Cw_n-Cx_n\|
\nonumber\\
&\leq\varepsilon^{-1}
\|(\Id-\gamma_nQ)w_n-(\Id-\gamma_nQ)x_n\|+\alpha^{-1}\|w_n-x_n\|
\nonumber\\
&\to 0.
\end{align}
Thus, the assertion follows from Theorem~\ref{t:8}\ref{t:8ii} 
since its conditions \ref{t:8iid} and \ref{t:8iia} are fulfilled.
\end{proof}

\begin{remark}
\label{r:22}
We complement Proposition~\ref{p:22} with a few commentaries.
\begin{enumerate}
\item
Suppose that $C=0$. Then, since $\alpha$ can be arbitrarily large,
$\chi=1/\beta$ and \eqref{e:half} reverts to the
forward-backward-forward algorithm \eqref{e:fbf}.
\item
Suppose that $Q=0$. Then, since $\beta=0$, $\chi=2\alpha$ and
\eqref{e:half} becomes an unrelaxed version of forward-backward 
algorithm \eqref{e:fb}.
\item
Using the geometric pattern of the proof given above, a strongly
convergent version of the forward-backward-half-forward algorithm
can be derived from Theorem~\ref{t:8s}.
\end{enumerate}
\end{remark}

As an illustration, we extend the Lagrangian approach of
Proposition~\ref{p:i}.

\begin{example}
\label{ex:221}
Let $f\in\Gamma_0(\HH)$, $g\in\Gamma_0(\GG)$, and 
$L\in\BL(\HH,\GG)$ be such that 
$0\in\sri(L(\dom f)-\dom g)$. Let $\alpha\in\RPP$ and let 
$h\colon\HH\to\RR$ be convex and
differentiable and such that $\nabla h$ is
$1/\alpha$-Lipschitzian.
Suppose that the primal problem
\begin{equation}
\label{e:221p}
\minimize{x\in\HH}{f(x)+g(Lx)+h(x)}
\end{equation}
admits solutions and consider the dual problem
\begin{equation}
\label{e:221d}
\minimize{v^*\in\GG}{(f^*\infconv h^*)(-L^*v^*)+g^*(v^*)}.
\end{equation}
Let $(x_0,y_0,v^*_0)\in\HH\oplus\GG\oplus\GG$, set
$\chi=4\alpha/(1+\sqrt{1+16\alpha^2(1+\|L\|^2)}\,)$, let 
$\varepsilon\in\left]0,\chi/(\chi+1)\right[$, and let
$(\gamma_n)_{n\in\NN}$ be a sequence in 
$\left[\varepsilon,(1-\varepsilon)\chi\right]$.
Iterate
\begin{equation}
\label{e:mp9}
\begin{array}{l}
\text{for}\;n=0,1,\ldots\\
\left\lfloor
\begin{array}{l}
c_n^*=\gamma_n\nabla h(x_n)\\
q^*_{1,n}=\gamma_nL^*v^*_n\\
q^*_{2,n}=-\gamma_nv^*_n\\
q^*_{3,n}=\gamma_n(y_n-Lx_n)\\
a_{1,n}=\prox_{\gamma_n f}\big(x_n-c^*_n-q^*_{1,n}\big)\\[1mm]
a_{2,n}=\prox_{\gamma_n g}\big(y_n-q^*_{2,n}\big)\\[1mm]
x_{n+1}=a_{1,n}+\gamma_nL^*q^*_{3,n}\\[1mm]
y_{n+1}=a_{2,n}-\gamma_nq^*_{3,n}\\[1mm]
v^*_{n+1}=v^*_n+\gamma_n\big(La_{1,n}-a_{2,n}\big).
\end{array}
\right.\\
\end{array}
\end{equation}
Then $(x_n)_{n\in\NN}$ and $(v^*_n)_{n\in\NN}$ converge weakly to 
solutions to \eqref{e:221p} and \eqref{e:221d},
respectively.
\end{example}
\begin{proof}
We adapt the approach of Section~\ref{sec:Lagr}. The saddle
operator of \eqref{e:Kr2}--\eqref{e:r76} becomes
$\sad=\boldsymbol{A}+\boldsymbol{C}+\boldsymbol{Q}$, where
\begin{equation}
\label{e:r77}
\begin{cases}
\boldsymbol{A}\colon (x,y,v^*)\mapsto
\partial f(x)\times\partial g(y)\times\{0\}\\
\boldsymbol{C}\colon (x,y,v^*)\mapsto
\big(\nabla h(x),0,0\big)\\
\boldsymbol{Q}\colon (x,y,v^*)\mapsto
\big(L^*v^*,-v^*,-Lx+y\big).
\end{cases}
\end{equation}
As in Section~\ref{sec:Lagr}, $\boldsymbol{A}$ is maximally
monotone and $\boldsymbol{Q}$ is monotone and 
$\sqrt{1+\|L\|^2}$-Lipschitzian. 
Further, by virtue of Lemma~\ref{l:bh}, $\boldsymbol{C}$ is
$\alpha$-cocoercive. Now set $(\forall n\in\NN)$ 
$\boldsymbol{x}_n=(x_n,y_n,v^*_n)$,
$\boldsymbol{c}^*_n=(c^*_n,0,0)$,
$\boldsymbol{q}^*_n=(q^*_{1,n},q^*_{2,n},q^*_{3,n})$,
and $\boldsymbol{w}_n=(a_{1,n},a_{2,n},v_n^*-q^*_{3,n})$.
Then \eqref{e:mp9} assumes the form 
\begin{equation}
\label{e:half3}
\begin{array}{l}
\text{for}\;n=0,1,\ldots\\
\left\lfloor
\begin{array}{l}
\boldsymbol{c}_n^*=\gamma_n\boldsymbol{C}\boldsymbol{x}_n\\
\boldsymbol{q}_n^*=\gamma_n\boldsymbol{Q}\boldsymbol{x}_n\\
\boldsymbol{w}_n=J_{\gamma_n\boldsymbol{A}}
(\boldsymbol{x}_n-\boldsymbol{c}_n^*-\boldsymbol{q}_n^*)\\
\boldsymbol{x}_{n+1}=\boldsymbol{w}_n-
\gamma_n\boldsymbol{Q}\boldsymbol{w}_n+\boldsymbol{q}_n^*,
\end{array}
\right.\\
\end{array}
\end{equation}
which is \eqref{e:half}. Hence, by Proposition~\ref{p:22},
$(x_n,y_n,v^*_n)_{n\in\NN}$ converges weakly to a point
$(x,y,v^*)\in\zer\sad$.
\end{proof}

\begin{remark}
\label{r:dj2}
Let $\alpha\in\RPP$, let $A\colon\HH\to 2^{\HH}$ and 
$B\colon\GG\to 2^{\GG}$ be maximally monotone, let 
$C\colon\HH\to\HH$ be $\alpha$-cocoercive, 
and let $0\neq L\in\BL(\HH,\GG)$.
As in Remark~\ref{r:dj1}, the saddle approach of
Example~\ref{ex:221} has a natural
extension to the problem of finding a zero of
$A+L^*\circ B\circ L+C$ and the dual problem of finding a zero 
of $-L\circ (A+C)^{-1}\circ(-L^*)+B^{-1}$. In this setting, the 
saddle operator is
\begin{equation}
\label{e:s2}
\begin{array}{ccll}
\sad\colon&\HH\oplus\GG\oplus\GG&\to&
2^{\HH\oplus\GG\oplus\GG}\\
&(x,y,v^*)&\mapsto&
(Ax+Cx+L^*v^*)\times(By-v^*)\times\{-Lx+y\}.
\end{array}
\end{equation}
Accordingly, it suffices to replace $\nabla h$ 
with $C$, $\prox_{\gamma_n f}$ with $J_{\gamma_n A}$, and
$\prox_{\gamma_n g}$ with $J_{\gamma_n B}$ in \eqref{e:mp9} to find
primal-dual solutions.
\end{remark}

\section{Block-iterative Kuhn--Tucker projective splitting}
\label{sec:ps}

\subsection{Preview}
Unlike the methods described so far, those described in this
section were explicitly designed by employing the geometric
principle of Theorem~\ref{t:1}. The terminology
\emph{projective splitting} was coined in \cite{Svai08} in the
context of an algorithm to solve Problem~\ref{prob:0} by choosing
points in the graph of $A$ and $B$ to construct half-spaces
containing an ``extended solution set.'' In the language
of Lemma~\ref{l:31z}, this set is actually the set of zeros of 
the Kuhn--Tucker operator \eqref{e:kt3}, which collapses to
\begin{equation}
\zer\kut=\menge{(x,x^*)\in\HH\oplus\HH}
{-x^*\in Ax\;\;\text{and}\;\;x\in B^{-1}x^*}.
\end{equation}
The paper \cite{Svai08} initiated a fruitful line of work towards
more complex monotone inclusions 
\cite{Siop14,Nfao15,Bedn18,Buin22,MaPr18,Ecks17,Svai09,John19,%
John20,John21,John22,Mach18,Mach23,Sicr20}.
We use the term \emph{Kuhn--Tucker projective splitting} to
describe a method that operates through the principles of
Framework~\ref{f:1}, where $\MMM$ is a Kuhn-Tucker operator. As we
shall see, projective splitting algorithms have features quite 
different from those of the traditional methods of
Sections~\ref{sec:ppa}--\ref{sec:fb} and they display an
unprecedented level of flexibility in terms of implementation.

\subsection{Primal-dual composite inclusions}

Let us go back to the composite Problem~\ref{prob:3}. The sets of
primal and dual solutions are, respectively,
\begin{equation}
\label{e:400}
Z=\zer(A+L^*\circ B\circ L)\quad\text{and}\quad
Z^*=\zer\bigl(-L\circ A^{-1}\circ(-L^*)+B^{-1}\bigr).
\end{equation}
Moreover, as pointed out in Example~\ref{ex:f3}, an embedding 
of \eqref{e:p3} is 
$(\XXX,\kut,\TTT)$, where $\XXX=\HH\oplus\GG$, $\kut$ is 
the Kuhn--Tucker operator of \eqref{e:kt3}, that is, 
\begin{equation}
\label{e:i41}
\kut\colon\XXX\to 2^{\XXX}\colon(x,y^*)\mapsto
\big(Ax+L^*y^*\big)\times\big(B^{-1}y^*-Lx\big),
\end{equation}
and $\TTT\colon\XXX\to\HH\colon(x,y^*)\mapsto x$.
The task is therefore to find a
zero of $\kut$. This is the path followed in the monotone+skew
approach of Section~\ref{sec:m+s}. However, this method requires
knowledge of $\|L\|$ (or of a tight upper bound for it), which may
be difficult to obtain in certain problems. The renormed algorithms
of Example~\ref{ex:1Ub} and \cite{Botr13}, the saddle algorithm of
Remark~\ref{r:dj2}, or the minimal lifting algorithm of
\cite{Arag23} share the same potential limitation. On the other
hand, the method of Proposition~\ref{p:2013}, which was derived
from the method of partial inverses, requires the inversion of
linear operators, a task that may also face implementation issues. 

A strategy which circumvents the above shortcomings was proposed in
\cite{Siop14}, where the approach of \cite{Svai08} for solving
Problem~\ref{prob:0} was extended to Problem~\ref{prob:3}. More
precisely, it employs the geometric principle of
Proposition~\ref{p:1} as follows. Let us assume that, at iteration
$n$, points $(a_n,a^*_n)\in\gra A$ and $(b_n,b^*_n)\in\gra B$ are
available and set
\begin{equation}
\boldsymbol{m}_n=(a_n,b^*_n)\quad\text{and}\quad
\boldsymbol{m}^*_n=(a^*_n+L^*b^*_n,b_n-La_n).
\end{equation}
Then it is clear from \eqref{e:i41} that 
$(\boldsymbol{m}_n,\boldsymbol{m}^*_n)\in\gra\kut$.
Hence, given $\lambda_n\in\left]0,2\right[$, iteration $n$ of
algorithm \eqref{e:fejer14} updates $(x_n,y^*_n)\in\XXX$ via
the routine
\begin{equation}
\label{e:13e}
\left\lfloor
\begin{array}{l}
(t_n,t^*_n)=(b_n-La_n,a^*_n+L^*b^*_n)\\
\tau_n={\|t_n\|^2+\|t^*_n\|^2}\\
\text{if}\;\tau_n>0\\
\left\lfloor
\begin{array}{l}
\theta_n=\dfrac{\lambda_n}{\tau_n} \,
\text{\rm max} \Big\{0,\scal{x_n}{t^*_n}+\scal{t_n}{y^*_n}
-\scal{a_n}{a^*_n}-\scal{b_n}{b^*_n}\Big\} 
\end{array} 
\right.\\
\text{else}\;\theta_n=0\\
(x_{n+1},y^*_{n+1})=(x_n-\theta_nt^*_n,y^*_n-\theta_nt_n).
\end{array}
\right.\\[4mm]
\end{equation}
In view of Proposition~\ref{p:1}\ref{p:1ii}, the task is now to
specify $(a_n,a^*_n)\in\gra A$ and $(b_n,b^*_n)\in\gra B$ so as to
guarantee that $\boldsymbol{m}_n-(x_n,y_n^*)\weakly 0$ and 
$\boldsymbol{m}_n^*\to 0$, that is,  
\begin{equation}
a_n-x_n\weakly 0,\;b^*_n-y^*_n\weakly 0,\;
b_n-La_n\to 0,\;\text{and}\;a^*_n+L^*b^*_n\to 0.
\end{equation}
Given $\gamma_n$ and $\sigma_n$ in $\RPP$, choosing
\begin{equation}
\label{e:y1}
(a_n,a^*_n)=\Bigl(J_{\gamma_nA}(x_n-\gamma_nL^*y^*_n),
\gamma_n^{-1}\bigl(x_n-J_{\gamma_nA}(x_n-\gamma_nL^*y^*_n)\bigr)
-L^*y^*_n\Bigr)
\end{equation}
and
\begin{equation}
\label{e:y2}
(b_n,b^*_n)=\Bigl(J_{\sigma_nB}(Lx_n+\sigma_ny^*_n),\sigma_n^{-1}
\bigr(Lx_n-J_{\sigma_nB}(Lx_n+\sigma_ny^*_n)\bigr)+y^*_n\Bigr)
\end{equation}
satisfies this requirement, which leads to the following result.

\begin{proposition}[{\protect{\cite[Proposition~3.5]{Siop14}}}]
\label{p:42}
Let $A\colon\HH\to 2^{\HH}$ and $B\colon\GG\to 2^{\GG}$ be
maximally monotone, and let $L\in\BL(\HH,\GG)$. Suppose that the
set $Z$ of solutions to the primal inclusion
\begin{equation}
\label{e:p3f}
\text{find}\;\;x\in\HH\;\;\text{such that}\;\;
0\in Ax+L^*\big(B(Lx)\big)
\end{equation}
is not empty and let $Z^*$ be the set of solutions to 
the dual inclusion
\begin{equation}
\label{e:d3f}
\text{find}\;\;y^*\in\GG\;\;\text{such that}\;\;
0\in -L\bigl(A^{-1}(-L^*y^*)\bigr)+B^{-1}y^*.
\end{equation}
Let $\varepsilon\in\zeroun$, let $(\gamma_n)_{n\in\NN}$ and 
$(\sigma_n)_{n\in\NN}$ be sequences in
$[\varepsilon,1/\varepsilon]$, 
let $(\lambda_n)_{n\in\NN}$ be a sequence in 
$[\varepsilon,2-\varepsilon]$, let $x_0\in\HH$, and let 
$y_0^*\in\GG$. Iterate
\begin{equation}
\label{e:201310}
\begin{array}{l}
\text{for}\;n=0,1,\ldots\\
\left\lfloor
\begin{array}{l}
a_n=J_{\gamma_n A}(x_n-\gamma_n L^*y_n^*)\\
l_n=Lx_n\\
b_n=J_{\sigma_n B}(l_n+\sigma_n y_n^*)\\
t_n=b_n-La_n\\
t^*_n=\gamma_n^{-1}(x_n-a_n)+\sigma_n^{-1}L^*(l_n-b_n)\\
\tau_n=\|t_n\|^2+\|t^*_n\|^2\\
\text{if}\;\tau_n>0\\
\left\lfloor
\begin{array}{l}
\theta_n=\lambda_n\big(\gamma_n^{-1}\|x_n-a_n\|^2+\sigma_n^{-1}
\|l_n-b_n\|^2\big)/\tau_n\\
\end{array}
\right.\\
\text{else}\;\theta_n = 0\\
x_{n+1}=x_n-\theta_n t^*_n\\
y^*_{n+1}=y^*_n-\theta_n t_n.
\end{array}
\right.\\[4mm]
\end{array}
\end{equation}
Then $(x_n)_{n\in\NN}$ converges weakly to a point $x\in Z$ and
$(y_n^*)_{n\in\NN}$ converges weakly to a point $y^*\in Z^*$.
\end{proposition}

\begin{remark}
\label{r:e1}
Here are notable instantiations of Proposition~\ref{p:42}.
\begin{enumerate}
\item
The first instance of \eqref{e:201310} in the literature seems
to be that of \cite{Dong05}, where $\HH$ and $\GG$ are Euclidean
spaces, $A=0$, and $(\forall n\in\NN)$ $\gamma_n=\sigma_n=1$ and 
$\lambda_n=\lambda\in\left]0,2\right[$. Convergence of the primal
sequence $(x_n)_{n\in\NN}$ was established by different means.
\item
In the setting of Problem~\ref{prob:0} (i.e., $\GG=\HH$ and
$L=\Id$), \eqref{e:201310} was studied in \cite{Svai08}. Under the
additional assumptions that $A+B$ is maximally monotone or that
$\HH$ is finite-dimensional, weak convergence was established in
\cite[Proposition~3]{Svai08} for a version of \eqref{e:201310} 
which allows for an additional relaxation parameter in the
definition of $a_n$. 
\end{enumerate}
\end{remark}

\begin{remark}
\label{r:e2}
So far, we have presented several methods to solve 
Problem~\ref{prob:3}; see Proposition~\ref{p:2013},
Example~\ref{ex:1Ub}, Proposition~\ref{p:31}, and
Remark~\ref{r:dj2}. Some features that distinguish the 
splitting algorithm \eqref{e:201310} from them are as follows.
\begin{enumerate}
\item
At each iteration of \eqref{e:201310}, different proximal
parameters $\gamma_n$ and $\sigma_n$ can be used for the operators
$A$ and $B$ and, since $\varepsilon$ is chosen by the user, their
values can be arbitrarily large.
\item
The execution of \eqref{e:201310} does not require that $\|L\|$ or
an approximation thereof be known, or the inversion of linear
operators.
\item
\label{r:e2iii}
A variant of \eqref{e:201310} exploiting the cocoercivity of some
of the operators and activating them via Euler steps is discussed
in \cite{John21}.
\item
The complexity of certain special cases and variants of
\eqref{e:201310} is investigated in \cite{John19,Mach23}.
\end{enumerate}
\end{remark}

The following strongly convergent projective splitting algorithm
results from Proposition~\ref{p:2}.

\begin{proposition}[{\protect{\cite[Proposition~3.5]{Nfao15}}}]
\label{p:42s}
Let $A\colon\HH\to 2^{\HH}$ and $B\colon\GG\to 2^{\GG}$ be
maximally monotone, and let $L\in\BL(\HH,\GG)$. Suppose that the
set $Z$ of solutions to the primal inclusion
\begin{equation}
\label{e:p3a}
\text{find}\;\;x\in\HH\;\;\text{such that}\;\;
0\in Ax+L^*\big(B(Lx)\big)
\end{equation}
is not empty and let $Z^*$ be the set of solutions to 
the dual inclusion
\begin{equation}
\label{e:d3a}
\text{find}\;\;y^*\in\GG\;\;\text{such that}\;\;
0\in -L\bigl(A^{-1}(-L^*y^*)\bigr)+B^{-1}y^*.
\end{equation}
Let $\varepsilon\in\zeroun$, let $(\gamma_n)_{n\in\NN}$ and 
$(\sigma_n)_{n\in\NN}$ be sequences in
$[\varepsilon,1/\varepsilon]$, 
let $(\lambda_n)_{n\in\NN}$ be a sequence in 
$[\varepsilon,1]$, let $x_0\in\HH$, and let 
$y_0^*\in\GG$. Iterate
\begin{equation}
\label{e:2014}
\begin{array}{l}
\text{for}\;n=0,1,\ldots\\
\left\lfloor
\begin{array}{l}
a_n=J_{\gamma_n A}(x_n-\gamma_n L^*y_n^*)\\
l_n=Lx_n\\
b_n=J_{\sigma_n B}(l_n+\sigma_n y_n^*)\\
t_n=b_n-La_n\\
t^*_n=\gamma_n^{-1}(x_n-a_n)+\sigma_n^{-1}L^*(l_n-b_n)\\
\tau_n=\|t_n\|^2+\|t^*_n\|^2\\
\text{if}\;\tau_n>0\\
\left\lfloor
\begin{array}{l}
\theta_n=\lambda_n\big(\gamma_n^{-1}\|x_n-a_n\|^2+\sigma_n^{-1}
\|l_n-b_n\|^2\big)/\tau_n\\
\end{array}
\right.\\
\text{else}\;\theta_n = 0\\
r_n=x_n-\theta_n t^*_n\\
r^*_n=y^*_n-\theta_n t_n\\
\chi_n=\theta_n\bigl(
\scal{x_0-x_n}{t^*_n}+\scal{t_n}{y_0^*-y_n^*}\bigr)\\
\mu_n=\|x_0-x_n\|^2+\|y_0^*-y_n^*\|^2\\
\nu_n=\theta_n^2\tau_n\\
\rho_n=\mu_n\nu_n-\chi_n^2\\
\text{if}\;\rho_n=0\;\text{and}\;\chi_n\geq 0\\
\left\lfloor
\begin{array}{l}
x_{n+1}=r_n\\
y^*_{n+1}=r_n^*
\end{array}
\right.\\
\text{if}\;\rho_n>0\;\text{and}\;\chi_n\nu_n\geq\rho_n\\
\left\lfloor
\begin{array}{l}
x_{n+1}=x_0-\theta_n(1+\chi_n/\nu_n)t^*_n\\
y^*_{n+1}=y_0^*-\theta_n(1+\chi_n/\nu_n)t_n
\end{array}
\right.\\
\text{if}\;\rho_n>0\;\text{and}\;\chi_n\nu_n<\rho_n\\
\left\lfloor
\begin{array}{l}
x_{n+1}=x_n+(\nu_n/\rho_n)\bigl(\chi_n(x_0-x_n)
-\mu_n\theta_nt^*_n\bigr)\\
y^*_{n+1}=y_n^*+(\nu_n/\rho_n)\bigl(\chi_n(y_0^*-y_n^*)
-\mu_n\theta_nt_n\bigr).
\end{array}
\right.\\
\end{array}
\right.\\[4mm]
\end{array}
\end{equation}
Then $(x_n)_{n\in\NN}$ converges strongly to a point $x\in Z$ and
$(y_n^*)_{n\in\NN}$ converges strongly to a point $y^*\in Z^*$.
\end{proposition}

\subsection{Block-iterative asynchronous method}

We consider a refinement of Problem~\ref{prob:9} in which the
primal variable is specified in terms of finitely many 
coordinates, say
$\boldsymbol{x}=(x_1,\ldots,x_m)$, where each $x_i$ lies in a
Hilbert space $\HH_i$. Such coupled systems of inclusions
arise in particular in multivariate optimization 
\cite{Acke80,Atto08,Sico10,Siop13}, 
domain decomposition methods \cite{Aldu23,Nume16,Atto11},
image processing \cite{Aujo05,Jmiv11,Chau13,Vese04},
game theory \cite{Belg21,Boer21,Bric13,Joca22}, 
network flow problems \cite{Bert98,Buim22,Rock84,Rock95},
machine learning \cite{Bric19,Jena11,Mich13,Vill14},
signal processing \cite{Bric09},
mean field games \cite{Bri18b},
statistics \cite{Ejst20,Bien21},
tensor completion \cite{Gand11,Mizo19},
and semi-definite programming \cite{Huwo23,Oliv18}.

\begin{problem}
\label{prob:12}
Let $I=\{1,\ldots,m\}$ and $K=\{1,\ldots,p\}$ be nonempty finite
sets. For every $i\in I$ 
and every $k\in K$, let $\HH_i$ and $\GG_k$ be real 
Hilbert spaces, let $A_i\colon\HH_i\to 2^{\HH_i}$ and
$B_k\colon\GG_k\to 2^{\GG_k}$ be maximally monotone, 
and let $L_{ki}\in\BL(\HH_i,\GG_k)$. Set 
\begin{equation}
\HHH=\bigoplus_{i\in I}\HH_i\quad\text{and}\quad
\GGG=\bigoplus_{k\in K}\GG_k.
\end{equation}
The objective is to solve the primal inclusion
\begin{multline}
\label{e:p12}
\text{find}\;\;\boldsymbol{x}\in\HHH
\;\;\text{such that}\\
(\forall i\in I)\quad
0\in A_ix_i+\Sum_{k\in K}L_{ki}^*
\bigg(B_k\bigg(\Sum_{j\in I}L_{kj}{x}_j\bigg)\bigg)
\end{multline}
together with the dual inclusion
\begin{multline}
\label{e:d12}
\text{find}\;\;\boldsymbol{y}^*\in\GGG
\;\;\text{such that}\\
(\exi\boldsymbol{x}\in\HHH)\;\;
\begin{cases}
(\forall i\in I)\;\;
x_i\in A_i^{-1}\biggl(-\Sum_{k\in K}L_{ki}^*y^*_k\biggr)\\[5mm]
(\forall k\in K)\;\;
\Sum_{i\in I}L_{ki}x_i\in B_k^{-1}{y}^*_k.
\end{cases}
\end{multline}
\end{problem}

\begin{remark}
\label{r:5}
There is an oversight in the dual problem given in
\cite[Problem~1]{MaPr18}, the correct formulation of the
dual inclusion is \eqref{e:d12}.
\end{remark}

The counterpart of Lemma~\ref{l:37z} for Problem~\ref{prob:12} is
as follows.

\begin{lemma}
\label{l:p12}
In the setting of Problem~\ref{prob:12}, set $\XXX=\HHH\oplus\GGG$,
and let $\boldsymbol{Z}$ and $\boldsymbol{Z}^*$ be the sets of
solutions to \eqref{e:p12} and \eqref{e:d12}, respectively. Define
the Kuhn--Tucker operator of Problem~\ref{prob:12} as 
\begin{multline}
\label{e:kt39}
\kut\colon\XXX\to 2^{\XXX}\colon
(\boldsymbol{x},\boldsymbol{y}^*)\mapsto\\
\hspace{18mm}
\biggl(A_1x_1+\sum_{k\in K}L_{k1}^*y_k^*\biggr)\times\cdots\times
\biggl(A_mx_m+\sum_{k\in K}L_{km}^*y_k^*\biggr)\\
\times
\biggl(-\sum_{i\in I}L_{1i}x_i+B_1^{-1}y^*_1\biggr)
\times\cdots\times
\biggl(-\sum_{i\in I}L_{pi}x_i+B_p^{-1}y^*_p\biggr)
\end{multline}
and the set of Kuhn--Tucker points as $\zer\kut$.
Then the following hold:
\begin{enumerate}
\item 
\label{l:p12i} 
$\kut$ is maximally monotone.
\item 
\label{l:p12ii} 
$\zer\kut$ is a closed convex subset of 
$\boldsymbol{Z}\times\boldsymbol{Z}^*$.
\item 
\label{l:p12iii} 
$\boldsymbol{Z}^*\neq\emp$ $\Leftrightarrow$ 
$\zer\kut\neq\emp$ $\Rightarrow$ $\boldsymbol{Z}\neq\emp$.
\end{enumerate}
\end{lemma}

\begin{example}
\label{ex:f7}
In the setting of Problem~\ref{prob:12}, set 
$\XXX=\HHH\oplus\GGG$, let $\kut$ be the Kuhn--Tucker operator of 
\eqref{e:kt39}, and let
$\TTT\colon\XXX\to\HHH\colon(\boldsymbol{x},\boldsymbol{y}^*)
\mapsto\boldsymbol{x}$. Then it follows from 
Lemma~\ref{l:p12}\ref{l:p12ii} that $(\XXX,\kut,\TTT)$ is an
embedding of \eqref{e:p12}.
\end{example}

When the monotone operators $(A_i)_{1\leq i\leq m}$ and
$(B_k)_{1\leq k\leq p}$ are taken to be subdifferentials,
Problem~\ref{prob:12} specializes to a multivariate minimization
problem under a suitable qualification condition.

\begin{example}
\label{ex:e12}
Define $\HHH$ and $\GGG$ as in Problem~\ref{prob:12}. For every 
$i\in I$ and every $k\in K$, let $f_i\in\Gamma_0(\HH_i)$, let 
$g_k\in\Gamma_0(\GG_k)$, and let $L_{ki}\in\BL(\HH_i,\GG_k)$.
Suppose that (existence of a Kuhn--Tucker point) 
\begin{equation}
\label{e:cq5}
(\exi\boldsymbol{x}\in\HHH)(\exi\boldsymbol{y}^*\in\GGG)\quad
\begin{cases}
(\forall i\in I)\;\;
-\Sum_{k\in K}L_{ki}^*y^*_k\in\partial f_i(x_i)\\[5mm]
(\forall k\in K)\;\;
\Sum_{i\in I}L_{ki}x_i\in\partial g_k^*(y^*_k).
\end{cases}
\end{equation}
The objective is to solve the primal minimization problem
\begin{equation}
\label{e:23p}
\minimize{\boldsymbol{x}\in\HHH}
{\sum_{i\in I}f_i(x_i)+\sum_{k\in K} 
g_k\bigg(\sum_{i\in I}L_{ki}x_i\bigg)}
\end{equation}
together with its dual problem
\begin{equation}
\label{e:23d}
\minimize{\boldsymbol{y}^*\in\GGG}
{\sum_{i\in I}f_i^*\bigg(-\sum_{k\in K}L_{ki}^*y^*_k\bigg)
+\sum_{k\in K}g^*_k(y^*_k)}.
\end{equation}
\end{example}

In an attempt to recast Problem~\ref{prob:12} as a realization
of Problem~\ref{prob:3}, let us define
\begin{equation}
\label{e:GGG}
\begin{cases}
A\colon\HHH\to 2^{\HHH}\colon\boldsymbol{x}\mapsto
A_1x_1\times\cdots\times A_mx_m\\
B\colon\GGG\to 2^{\GGG}\colon\boldsymbol{y}\mapsto
B_1y_1\times\cdots\times B_py_p\\
L\colon\HHH\to\GGG\colon\boldsymbol{x}\mapsto
\bigl(\sum_{i\in I}L_{1i}x_i,\ldots,\sum_{i\in I}L_{pi}x_i\bigr).
\end{cases}
\end{equation}
Upon injecting these operators into \eqref{e:201310} and invoking
Example~\ref{ex:r1}, we obtain an algorithm that requires that
$m+p$ resolvents be evaluated at each iteration. In large-scale
problems, $m$ and/or $p$ can be huge and this requirement poses
implementation issues as the only information flow within an
iteration is from the $m$ operators $(A_i)_{i\in I}$ 
calculations to the $p$ operators $(B_k)_{k\in K}$ calculations.
This results in an algorithm in which large blocks of calculations
must be performed before any information is exchanged between
subsystems. Thus, if some small subset of the subsystems
represented by the operators $(A_i)_{i\in I}$ or
$(B_k)_{k\in K}$ are more computation-intensive than others,
load balancing can become problematic: most processors may have to
sit idle while the remaining few complete their tasks. More
generally, none of the methods discussed so far can handle
block-processing or asynchronicity.

The algorithm we present now was conceived in \cite{MaPr18} around
combined objectives which were beyond the reach of the existing
splitting algorithms:
\begin{itemize}
\item
{\bfseries Block iterations:} 
At iteration $n$, it necessitates calculation of new points in the
graphs of only some of the operators, say $(A_i)_{i\in I_n}$ and 
$(B_k)_{k\in K_n}$ with $I_n\subset I$ and $K_n\subset K$. The
deterministic control sequences $(I_n)_{n\in\NN}$ and
$(K_n)_{n\in\NN}$ dictate how frequently the various operators are
used. 
\item
{\bfseries Asynchronicity:} 
A new point $(a_{i,n},a^*_{i,n})\in\gra A_i$ being incorporated
into the calculations at iteration $n$ may be based on data
$x_{i,\pi_i(n)}$ and $(y^*_{k,\pi_i(n)})_{k\in K}$ available at 
some possibly earlier iteration $\pi_i(n)\leq n$. Therefore, the
calculation of $(a_{i,n},a^*_{i,n})$ could have been initiated at
iteration $\pi_i(n)$, with its results becoming available only at
iteration $n$. Likewise, for every $k\in K_n$, the computation of
$(b_{k,n},b^*_{k,n})\in\gra B_k$ can be initiated at some iteration
$\omega_k(n)\leq n$, based on $(x_{i,\omega_k(n)})_{i\in I}$ and
$y^*_{k,\omega_k(n)}$. 
\item
{\bfseries Convergence:} 
It guarantees (weak or strong) convergence of the iterates to
primal and dual solutions.
\end{itemize}

\begin{remark}
\label{r:13}
Regarding block iterations for Problem~\ref{prob:12}, a product 
space version of the Douglas--Rachford algorithm was introduced 
in \cite{Siop15}, which features random activation of the blocks.
A random block-iterative version of the forward-backward
algorithm was also proposed in \cite{Siop15}, which led in 
\cite{Pesq15} to algorithms for Problem~\ref{prob:12} via the
renorming techniques presented in Section~\ref{sec:ren2}
(for specialized block-iterative forward-backward 
algorithms tailored for instances of Example~\ref{ex:e12}, see
\cite{Bric19,Liuj15,Salz22,Trao23}). These methods differ from the
deterministic ones presented below in that they operate under
stochastic assumptions on the underlying processes, have a less
predictable computational load over the iterations, have less
freedom in the choice of the proximal parameters, and offer only
almost sure convergence guarantees (see also \cite{Icas22} for
numerical comparisons).
\end{remark}

Going back to \eqref{e:13e} in the setting of \eqref{e:GGG} and
Lemma~\ref{l:p12}, what is actually needed at iteration $n$ to
create the half-space containing $\zer\kut$ are points
\begin{equation}
\begin{cases}
(a_{i,n},a^*_{i,n})\in\gra A_i,&\text{for}\;i\in I;\\
(b_{k,n},b^*_{k,n})\in\gra B_k,&\text{for}\;k\in K.
\end{cases}
\end{equation}
The key observation is that not all of these points have to be new
in order to obtain a new half-space. In other words, we can update
only some of them while keeping old ones and still create a new
half-space onto which the current primal-dual iterate
$(\boldsymbol{x}_n,\boldsymbol{y}_n^*)
=(x_{1,n},\ldots,x_{m,n},y^*_{1,n},\ldots,y^*_{p,n})$ will be
projected. How often the points in the individual graphs should be
updated, and in which fashion, will be regulated by the following
rules.

\begin{assumption}
\label{a:01}
Given $0<R\in\NN$, $(I_n)_{n\in\NN}$ is a sequence of nonempty
subsets of $I$, and $(K_n)_{n\in\NN}$ is a sequence of nonempty
subsets of $K$ such that
\begin{equation}
\label{e:-24G}
I_0=I,\;K_0=K,\;\;\text{and}\;\;
(\forall n\in\NN)\;
\begin{cases}
\displaystyle\bigcup_{j=n}^{n+R-1}I_j=I\\
\displaystyle\bigcup_{j=n}^{n+R-1}K_j=K.
\end{cases}
\end{equation}
\end{assumption}

\begin{assumption}
\label{a:02}
$T\in\NN$ and, for every $i\in I$ and every $k\in K$, 
$(\pi_i(n))_{n\in\NN}$ and $(\omega_k(n))_{n\in\NN}$ are sequences
in $\NN$ such that $(\forall n\in\NN)$ $n-T\leq\pi_i(n)\leq n$ and
$n-T\leq\omega_k(n)\leq n$.
\end{assumption}

With these considerations and by making selections for the updated
points $(a_{i,n},a^*_{i,n})_{i\in I_n}$ and 
$(b^*_{k,n},b^*_{k,n})_{k\in K_n}$ akin to those of 
\eqref{e:y1} and \eqref{e:y2}, we arrive at the following 
realization of \eqref{e:13e}.

\begin{algorithm}
\label{algo:3}
Consider the setting of Problem~\ref{prob:12}, suppose that 
Assumptions~\ref{a:01} and \ref{a:02} are in force, let
$\varepsilon\in\zeroun$, and let $(\lambda_n)_{n\in\NN}$ be a
sequence in $[\varepsilon,2-\varepsilon]$.
For every $i\in I$, let
$(\gamma_{i,n})_{n\in\NN}$ be a sequence in 
$[\varepsilon,1/\varepsilon]$ and let $x_{i,0}\in\HH_i$.
For every $k\in K$, let
$(\sigma_{k,n})_{n\in\NN}$ be a sequence in 
$[\varepsilon,1/\varepsilon]$ and let $y^*_{k,0}\in\GG_k$.
Iterate
\begin{equation}
\label{e:nyc2015-04-03a}
\begin{array}{l}
\text{for}\;n=0,1,\ldots\\
\left\lfloor
\begin{array}{l}
\begin{array}{l}
\text{for every}\;i\in I_n\\
\left\lfloor
\begin{array}{l}
l^*_{i,n}=\sum_{k\in K}L_{ki}^*y_{k,\pi_i(n)}^*\\
a_{i,n}=J_{\gamma_{i,\pi_i(n)}A_i}
\big(x_{i,\pi_i(n)}-\gamma_{i,\pi_i(n)}l^*_{i,n}\big)\\
a_{i,n}^*=\gamma_{i,\pi_i(n)}^{-1}
(x_{i,\pi_i(n)}-a_{i,n})-l^*_{i,n}\\
\end{array}
\right.\\[1mm]
\text{for every}\;i\in I\smallsetminus I_n\\
\left\lfloor
\begin{array}{l}
\bigl(a_{i,n},a_{i,n}^*\bigr)=\bigl(a_{i,n-1},a_{i,n-1}^*\bigr)\\
\end{array}
\right.\\[1mm]
\text{for every}\;k\in K_n\\
\left\lfloor
\begin{array}{l}
l_{k,n}=\sum_{i\in I}L_{ki}x_{i,\omega_k(n)}\\
b_{k,n}=J_{\sigma_{k,\omega_k(n)}B_k}
\big(l_{k,n}+\sigma_{k,\omega_k(n)}y_{k,\omega_k(n)}^*\big)\\
b^*_{k,n}=y_{k,\omega_k(n)}^*+
\sigma_{k,\omega_k(n)}^{-1}(l_{k,n}-b_{k,n})\\
\end{array}
\right.\\[1mm]
\text{for every}\;k\in K\smallsetminus K_n\\
\left\lfloor
\begin{array}{l}
\bigl(b_{k,n},b^*_{k,n}\bigr)=\bigl(b_{k,n-1},b^*_{k,n-1}\bigr)\\
\end{array}
\right.\\[1mm]
\text{for every}\;i\in I\\
\left\lfloor
\begin{array}{l}
t^*_{i,n}=a^*_{i,n}+\sum_{k\in K}L_{ki}^*b^*_{k,n}\\
\end{array}
\right.\\
\text{for every}\;k\in K\\
\left\lfloor
\begin{array}{l}
t_{k,n}=b_{k,n}-\sum_{i\in I}L_{ki}a_{i,n}
\end{array}
\right.\\
\tau_n=\sum_{i\in I}\|t_{i,n}^*\|^2+\sum_{k\in K}\|t_{k,n}\|^2\\
\text{if}\;\tau_n>0\\
\left\lfloor
\begin{array}{l}
\theta_n=\dfrac{\lambda_n}{\tau_n}\,\text{\rm max} 
\Bigl\{0,\sum_{i\in I}\bigl(\scal{x_{i,n}}{t^*_{i,n}}-
\scal{a_{i,n}}{a^*_{i,n}}\bigr)\\
\hspace{26.5mm}+\sum_{k\in K}
\bigl(\scal{t_{k,n}}{y^*_{k,n}} -\scal{b_{k,n}}{b^*_{k,n}}\bigr)
\Bigr\}\\
\end{array}
\right.\\
\text{else}\;\theta_n=0\\
\text{for every}\;i\in I\\
\left\lfloor
\begin{array}{l}
x_{i,n+1}=x_{i,n}-\theta_n t^*_{i,n}\\
\end{array}
\right.\\
\text{for every}\;k\in K\\
\left\lfloor
\begin{array}{l}
y^*_{k,n+1}=y^*_{k,n}-\theta_n t_{k,n}.
\end{array}
\right.\\
\end{array}
\end{array}
\right.\\[4mm]
\end{array}
\end{equation}
\end{algorithm}

Weak convergence is obtained by applying the principles of
Proposition~\ref{p:1}\ref{p:1ii}.

\begin{theorem}[{\protect{\cite[Theorem~13]{MaPr18}}}]
\label{t:7}
Consider the setting of Problem~\ref{prob:12} and
Algorithm~\ref{algo:3}, and suppose that the Kuhn--Tucker
operator $\kut$ of \eqref{e:kt39} has zeros. Then, for
every $i\in I$, $(x_{i,n})_{n\in\NN}$ converges weakly
to a point $x_i\in\HH_i$ and, for every $k\in K$,
$(y^*_{k,n})_{n\in\NN}$ converges weakly to a point
$y^*_k\in\GG_k$. In addition, $(x_i)_{i\in I}$ solves the primal
problem \eqref{e:p12} and $(y^*_k)_{k\in K}$ solves the
dual problem \eqref{e:d12}.
\end{theorem}

\begin{remark}
\label{r:12}
Here are a few comments on algorithm \eqref{algo:3}.
\begin{enumerate}
\item
The synchronous implementation is obtained by taking, for every
$n\in\NN$, every $i\in I_n$, and every $k\in K_n$,
$\pi_i(n)=\omega_k(n)=n$.
\item
We recover \cite[Theorem~4.3]{Siop14} (and in particular
Proposition~\ref{p:42s} when $m=p=1$) in the special case when 
the implementation is synchronous, and at every iteration $n$, 
every operator is used (i.e., $I_n=I$ and 
$K_n=K$), with $\gamma_{i,n}=\gamma_n$ for every
$i\in I$ and $\sigma_{k,n}=\sigma_n$ for every 
$k\in K$.
\item
The specialization of Theorem~\ref{t:7} to the minimization setting
of Example~\ref{ex:e12} is obtained by replacing each
$J_{\gamma_{i,\pi_i(n)}A_i}$ with $\prox_{\gamma_{i,\pi_i(n)}f_i}$
and each $J_{\sigma_{k,\omega_k(n)}B_k}$ with 
$\prox_{\sigma_{k,\omega_k(n)}g_k}$.
Numerical experiments are presented in \cite{Icas22} in the  
context of signal recovery and machine learning, and in 
\cite{Ecks23} in the context of stochastic programming.
\item
For the strongly convergent variant of Theorem~\ref{t:7} based on
Proposition~\ref{p:2}, see \cite[Theorem~15]{MaPr18}.
\item
\label{r:12v}
When $m=1$ and $A=0$, a variant that takes into account the fact
that some of the operators $(B_k)_{k\in K}$ may be monotone
and Lipschitzian, and which activate them via Euler steps is
presented in \cite{John22} (see also \cite{John20}).
\end{enumerate}
\end{remark}

\section{Block-iterative saddle projective splitting}
\label{sec:sad}

\subsection{Preview}
In all the algorithms discussed so far, each monotone operator has
one of three properties: it is set-valued, single-valued and
cocoercive, or single-valued and Lipschitzian. In addition, at each
iteration, a set-valued operator is used once via its 
resolvents, a cocoercive operator once via a Euler step, 
and a Lipschitzian operator twice via Euler steps. This is 
particularly the case in the forward-backward-half-forward
algorithm of Section~\ref{sec:fbfh}, the objective of which is to
find a zero of 
\begin{equation}
\label{e:acq}
M=A+C+Q,\;\;\text{where}\;\; 
\begin{cases}
A\colon\HH\to 2^{\HH}&\text{is maximally monotone}\\
C\colon\HH\to\HH&\text{is cocoercive}\\
Q\colon\HH\to\HH&\text{is monotone and Lipschitzian.}
\end{cases}
\end{equation}
On the other hand, the Kuhn--Tucker projective
splitting techniques of Section~\ref{sec:ps} activate all the
operators via their resolvents (exceptions were noted in
Remarks~\ref{r:e2}\ref{r:e2iii} and \ref{r:12}\ref{r:12v}, but they
concern special cases of Problem~\ref{prob:12}). Furthermore, they
are not designed to handle problems such as \eqref{e:p18} or
\eqref{e:bap}, which incorporate parallel sums.

In this section, following \cite{Moor22}, we unify all the problem
formulations encountered in Sections~\ref{sec:ppa}--\ref{sec:ps} by
including parallel sums in the system of monotone inclusions of
Problem~\ref{prob:12}, and decomposing each operator in the
resulting problem as in \eqref{e:acq}. In addition, nonlinear
coupling operators $(R_i)_{i\in I}$ are incorporated.

\begin{problem}
\label{prob:20}
Let $(\HH_i)_{i\in I}$ and $(\GG_k)_{k\in K}$ be finite families of
real Hilbert spaces, and set 
\begin{equation}
\label{e:atl1}
\HHH=\bigoplus_{i\in I}\HH_i\quad\text{and}\quad
\GGG=\bigoplus_{k\in K}\GG_k. 
\end{equation}
For every $i\in I$ and every 
$k\in K$, suppose that the following are satisfied:
\begin{enumerate}[label={\rm[\alph*]}]
\item
\label{prob:20a}
$A_i\colon\HH_i\to 2^{\HH_i}$ is maximally monotone,
$C_i\colon\HH_i\to\HH_i$ is cocoercive with constant
$\alpha_i^{\CC}\in\RPP$,
$Q_i\colon\HH_i\to\HH_i$ is monotone and Lipschitzian
with constant $\alpha_i^{\LL}\in\RP$, and 
$R_i\colon\HHH\to\HH_i$.
\item
\label{prob:20b}
$B_k^{\MM}\colon\GG_k\to 2^{\GG_k}$ is maximally monotone,
$B_k^{\CC}\colon\GG_k\to\GG_k$ is cocoercive with constant
$\beta_k^{\CC}\in\RPP$, and $B_k^{\LL}\colon\GG_k\to\GG_k$
is monotone and Lipschitzian with constant $\beta_k^{\LL}\in\RP$.
\item
\label{prob:20c}
$D_k^{\MM}\colon\GG_k\to 2^{\GG_k}$ is maximally monotone,
$D_k^{\CC}\colon\GG_k\to\GG_k$ is cocoercive with constant
$\delta_k^{\CC}\in\RPP$, and $D_k^{\LL}\colon\GG_k\to\GG_k$
is monotone and Lipschitzian with constant $\delta_k^{\LL}\in\RP$.
\item
\label{prob:20d}
$L_{ki}\in\BL(\HH_i,\GG_k)$.
\end{enumerate}
In addition, 
\begin{enumerate}[resume,label={\rm[\alph*]}]
\item
\label{prob:20e}
$\boldsymbol{R}\colon\HHH\to\HHH\colon
\boldsymbol{x}\mapsto(R_i\boldsymbol{x})_{i\in I}$
is monotone and Lipschitzian with constant $\chi\in\RP$.
\end{enumerate}
The objective is to solve the primal problem
\begin{multline}
\label{e:20p}
\text{find}\;\:\boldsymbol{x}=(x_i)_{i\in I}\in\HHH
\;\:\text{such that}\;\:(\forall i\in I)\;\;
0\in A_ix_i+C_ix_i+Q_ix_i
+R_i\boldsymbol{x}\\
\hskip -20mm+\Sum_{k\in K}L_{ki}^*\Bigg(\Big(
\big(B_k^{\MM}+B_k^{\CC}+B_k^{\LL}\big)\infconv
\big(D_k^{\MM}+D_k^{\CC}+D_k^{\LL}\big)\Big)
\Bigg(\Sum_{j\in I}L_{kj}x_j\Bigg)\Bigg)
\end{multline}
and the associated dual problem
\begin{multline}
\label{e:20d}
\hskip -8mm
\text{find}\;\:\boldsymbol{y}^*=(y^*_k)_{k\in K}\in\GGG
\;\:\text{such that}\;\:
(\exi\boldsymbol{x}\in\HHH)\\
\hskip -18mm
\begin{cases}
(\forall i\in I)\;\;-\Sum_{k\in K}L_{ki}^*y_k^*\in
A_ix_i+C_ix_i+Q_ix_i+R_i\boldsymbol{x}\\
(\forall k\in K)\;\;y_k^*\in
\Big(\big(B_k^{\MM}+B_k^{\CC}+B_k^{\LL}\big)\infconv
\big(D_k^{\MM}+D_k^{\CC}+D_k^{\LL}\big)\Big)
\Bigg(\Sum_{i\in I}L_{ki}x_i\Bigg).
\end{cases}
\end{multline}
\end{problem}

Here is an instance of Problem~\ref{prob:20} which is not captured
by previous monotone inclusion models.

\begin{example}
\label{ex:80}
We consider a game theoretic minimax problem. 
Let $I$ be a finite set and suppose that $\emp\neq J\subset I$.
For every $i\in I$, the strategy $x_i$ of player $i$ belongs to a
real Hilbert space $\HH_i$. A strategy profile is a point
\begin{equation}
\boldsymbol{x}=(x_i)_{i\in I}\in\bigoplus_{i\in I}\HH_i, 
\end{equation}
and the associated profile of the players other than $i\in I$ is 
$\boldsymbol{x}_{\smallsetminus i}=
(x_j)_{j\in I\smallsetminus\{i\}}$. For every $i\in I$ and every 
\begin{equation}
(x_i,\boldsymbol{y})\in\HH_i\oplus\bigoplus_{j\in I}\HH_j,
\end{equation}
we set
$(x_i;\boldsymbol{y}_{\smallsetminus i})=(y_1,\ldots,y_{i-1},x_i,
y_{i+1},\ldots,y_p)$. Now set 
\begin{equation}
\UUU=\bigoplus_{i\in I\smallsetminus J}\HH_i,\quad
\VVV=\bigoplus_{j\in J}\HH_j,
\quad\text{and}\quad\HHH=\UUU\oplus\VVV, 
\end{equation}
and, for every $i\in I$, let $f_i\in\Gamma_0(\HH_i)$. Further, let
$\boldsymbol{F}\colon\HHH\to\RR$ be
differentiable with a Lipschitzian gradient and 
such that, for every $\boldsymbol{u}\in\UUU$ and every
$\boldsymbol{v}\in\VVV$, the functions
${-}\boldsymbol{F}(\boldsymbol{u},\cdot)$ and 
$\boldsymbol{F}(\cdot,\boldsymbol{v})$ are
convex. We consider the multivariate minimax problem
\begin{equation}
\label{e:mima}
\minmax{\boldsymbol{u}\in\UUU}{\boldsymbol{v}\in\VVV}
{\sum_{i\in I\smallsetminus J}f_i(u_i)
+\boldsymbol{F}(\boldsymbol{u},\boldsymbol{v})
-\sum_{j\in J}f_j(v_j)}.
\end{equation}
Now define
\begin{equation}
(\forall i\in I)\quad
\boldsymbol{h}_{\!i}\colon\HHH\to\RR\colon
(\boldsymbol{u},\boldsymbol{v})\mapsto
\begin{cases}
\boldsymbol{F}(\boldsymbol{u},\boldsymbol{v}),
&\text{if}\;\:i\in I\smallsetminus J;\\
{-}\boldsymbol{F}(\boldsymbol{u},\boldsymbol{v}),
&\text{if}\;\:i\in J.
\end{cases}
\end{equation}
Then \eqref{e:mima} can be put in the form
\begin{equation}
\label{e:77}
\text{find}\;\:\boldsymbol{x}\in\HHH\;\:
\text{such that}\;\:(\forall i\in I)\;\;
{x}_i\in\Argmin{f_i+\boldsymbol{h}_{\!i}(\cdot;
\boldsymbol{x}_{\smallsetminus i})}.
\end{equation}
Since
\begin{equation}
\label{e:78}
(\forall i\in I)(\forall\boldsymbol{x}\in\HHH)\quad
\pnabla{i}\boldsymbol{h}_{\!i}(\boldsymbol{x})=
\begin{cases}
\pnabla{i}\boldsymbol{F}(\boldsymbol{x}),
&\text{if}\;\:i\in I\smallsetminus J;\\
{-}\pnabla{i}\boldsymbol{F}(\boldsymbol{x}),
&\text{if}\;\:i\in J,
\end{cases}
\end{equation}
the operator
\begin{equation}
\boldsymbol{R}\colon\HHH\to\HHH\colon
\boldsymbol{x}\mapsto
\big(\pnabla{i}\boldsymbol{h}_{\!i}(\boldsymbol{x})\big)_{i\in I}=
\Bigl(\big(\pnabla{i}\boldsymbol{F}(\boldsymbol{x})
\big)_{i\in I\smallsetminus J},
\big({-}\pnabla{j}\boldsymbol{F}
(\boldsymbol{x})\big)_{j\in J}\Bigr)
\end{equation}
is monotone \cite{Roc70b,Roc71d} and Lipschitzian. 
Now, for every $i\in I$, set $A_i=\partial f_i$. Then, by
Fermat's rule, \eqref{e:77} is equivalent to 
\begin{equation}
\label{e:79}
\text{find}\;\:\boldsymbol{x}\in\HHH
\;\:\text{such that}\;\:(\forall i\in I)\;\;
0\in A_ix_i+R_i\boldsymbol{x},
\end{equation}
which shows that \eqref{e:mima} is an instantiation of 
\eqref{e:20p}. Special cases of \eqref{e:mima} under the above 
assumptions arise in 
\cite{Sico10,Sign21,Mont15,Nemi04,Rock95,Thek19,Yoon21}.
\end{example}

Our objective is to solve Problem~\ref{prob:20} with the same level
of flexibility and the same primal-dual convergence guarantees as
in Theorem~\ref{t:7}, i.e., to achieve full splitting of all the
operators using an asynchronous block-iterative algorithm without
knowledge of the norms of the linear operators or inversion of
linear operators. In addition, all the single-valued operators
should be activated via Euler steps. 

\subsection{Saddle operator formulation}

The approach adopted in Section~\ref{sec:ps} to break
Problem~\ref{prob:12} into manageable pieces hinged on the
Kuhn--Tucker operator of Lemma~\ref{l:p12} to obtain the embedding
of Framework~\ref{f:1}. This strategy does not appear to lead to a
full splitting of Problem~\ref{prob:20}, as it contains a larger
number of operators. We therefore require an embedding in a space
$\XXX$ which is bigger than the primal-dual space
$\HH_1\oplus\cdots\HH_m\oplus\GG_1\oplus\cdots\oplus\GG_p$ of
Theorem~\ref{t:7}. As discussed in Remark~\ref{r:dj2}, saddle 
operators are defined on a bigger space than Kuhn--Tucker 
operators (for instance, $\HH\oplus\GG\oplus\GG$ versus
$\HH\oplus\GG$ in \eqref{e:s2}) and their zeros still provide
primal-dual solutions. Following Framework~\ref{f:1}, as we did in
Example~\ref{ex:f5}, the methodology of 
\emph{saddle projective splitting}
is to introduce a saddle operator for Problem~\ref{prob:20}. We
shall then devise asynchronous block-iterative splitting algorithms
based on the geometric principles of Theorems~\ref{t:1c} and
\ref{t:2c} to find a zero of it, from which solutions to
Problem~\ref{prob:20} will be extracted. This is outlined in the
following lemma.

\begin{lemma}[{\protect{\cite[Proposition~1]{Moor22}}}]
\label{l:sad}
Define $\HHH$ and $\GGG$ as in \eqref{e:atl1}, set 
$\XXX=\HHH\oplus\GGG\oplus\GGG\oplus\GGG$, and define the 
saddle operator of Problem~\ref{prob:20} as
\begin{align}
\label{e:saddle}
\sad\colon\XXX\to 2^{\XXX}\colon
(\boldsymbol{x},\boldsymbol{y},\boldsymbol{z},\boldsymbol{v}^*)
&\mapsto\nonumber\\
&\Biggl(\bigtimes_{i\in I}
\biggl(A_ix_i+C_ix_i+Q_ix_i+R_i\boldsymbol{x}
+\sum_{k\in K}L^*_{ki}v^*_k\biggr),\nonumber\\
&\;\bigtimes_{k\in K}\big(B_k^{\MM}y_k
+B_k^{\CC}y_k+B_k^{\LL}y_k-v_k^*\big),\nonumber\\
&\;\bigtimes_{k\in K}\big(D_k^{\MM}z_k
+D_k^{\CC}z_k+D_k^{\LL}z_k-v_k^*\big),\nonumber\\
&\;\bigtimes_{k\in K}\bigg\{
y_k+z_k-\sum_{i\in I}L_{ki}x_i\bigg\}~\Biggr),
\end{align}
let $\boldsymbol{Z}$ be the set of solutions to \eqref{e:20p} and 
let $\boldsymbol{Z}^*$ be the set of solutions to \eqref{e:20d}.
Then the following hold:
\begin{enumerate}
\item
\label{l:sadi}
$\sad$ is maximally monotone.
\item
\label{l:sadii}
$\zer\sad$ is closed and convex.
\item
\label{l:sadiii}
Suppose that 
$(\boldsymbol{x},\boldsymbol{y},
{\boldsymbol{z}},\boldsymbol{v}^*)
\in\zer\sad$. Then $(\boldsymbol{x},
\boldsymbol{v}^*)\in\boldsymbol{Z}\times\boldsymbol{Z}^*$.
\item
\label{l:sadiv}
$\boldsymbol{Z}^*\neq\emp$ $\Leftrightarrow$ $\zer\sad\neq\emp$ 
$\Rightarrow$ $\boldsymbol{Z}\neq\emp$.
\end{enumerate}
\end{lemma}

We thus obtain the following generalization of Example~\ref{ex:f5}.

\begin{example}
\label{ex:f8}
In the setting of Problem~\ref{prob:20}, set 
\begin{equation}
\XXX=\HHH\oplus\GGG\oplus\GGG\oplus\GGG, 
\end{equation}
let $\sad$ be the
saddle operator of \eqref{e:saddle}, and let
\begin{equation}
\TTT\colon\XXX\to\HHH\colon
(\boldsymbol{x},\boldsymbol{y},\boldsymbol{z},\boldsymbol{v}^*)
\mapsto\boldsymbol{x}. 
\end{equation}
Then it follows from 
Lemma~\ref{l:sad}\ref{l:sadiii} that $(\XXX,\sad,\TTT)$ is 
an embedding of \eqref{e:20p}.
\end{example}

Thus, to solve Problem~\ref{prob:20} via 
Theorem~\ref{t:1c}, we need a decomposition of the saddle 
operator \eqref{e:saddle} as 
$\sad=\boldsymbol{\mathsf{W}}+\boldsymbol{\mathsf{C}}$, where
$\boldsymbol{\mathsf{W}}\colon\XXX\to 2^{\XXX}$ is maximally
monotone and $\boldsymbol{\mathsf{C}}\colon\XXX\to\XXX$ is
$\alpha$-cocoercive. This will be achieved with 
\begin{equation}
\label{e:2796}
\boldsymbol{\mathsf{C}}\colon\XXX\to\XXX\colon
(\boldsymbol{x},\boldsymbol{y},\boldsymbol{z},\boldsymbol{v}^*)
\mapsto\Bigl(
\bigl(C_ix_i\bigr)_{i\in I},\bigl(B_k^{\CC}y_k\bigr)_{k\in K},
\bigl(D_k^{\CC}z_k\bigr)_{k\in K},\boldsymbol{0}\Bigr)
\end{equation}
and $\alpha=\min\{\alpha_i^{\CC},\beta_k^{\CC},
\delta_k^{\CC}\}_{i\in I,k\in K}$. These considerations lead to
the following implementation of \eqref{e:fejer19}.

\begin{algorithm}
\label{algo:20}
In the setting of Problem~\ref{prob:20}, set
\begin{equation}
\alpha=\min\bigl\{\alpha_i^{\CC},\beta_k^{\CC},
\delta_k^{\CC}\bigr\}_{\substack{i\in I\\ k\in K}}, 
\end{equation}
let $\sigma\in\left]1/(4\alpha),\pinf\right[$ and 
$\varepsilon\in\zeroun$ be such that
\begin{equation}
\label{e:a2i'}
\dfrac{1}{\varepsilon}>\sigma+\max\Bigl\{\alpha_i^{\LL}+\chi,
\beta_k^{\LL},\delta_k^{\LL}\Bigr\}_{\substack{i\in I\\k\in K}},
\end{equation}
and let $(\lambda_n)_{n\in\NN}$ be a sequence in
$\left[\varepsilon,2-\varepsilon\right]$.
For every $i\in I$, let
$(\gamma_{i,n})_{n\in\NN}$ be a sequence in 
$\left[\varepsilon,1/(\alpha_i^{\LL}+\chi+\sigma)\right]$ and let 
$x_{i,0}\in\HH_i$. For every $k\in K$, let
$(\mu_{k,n})_{n\in\NN}$ be a sequence in 
$\left[\varepsilon,1/(\beta_k^{\LL}+\sigma)\right]$, let
$(\rho_{k,n})_{n\in\NN}$ be a sequence in 
$\left[\varepsilon,1/(\delta_k^{\LL}+\sigma)\right]$, let
$(\sigma_{k,n})_{n\in\NN}$ be a sequence in 
$\left[\varepsilon,1/\varepsilon\right]$, and let
$\{y_{k,0},z_{k,0},v_{k,0}^*\}\subset\GG_k$.
Suppose that Assumptions~\ref{a:01} and \ref{a:02} are in force
and iterate
\newpage
\begin{align}
\label{e:long1}
\hspace{-11mm}
\begin{array}{l}
\text{for}\;n=0,1,\ldots\\
\left\lfloor
\begin{array}{l}
\text{for every}\;i\in I_n\\
\left\lfloor
\begin{array}{l}
l_{i,n}^*=Q_ix_{i,\pi_i(n)}
+R_i\boldsymbol{x}_{\pi_i(n)}
+\sum_{k\in K}L_{ki}^*v_{k,\pi_i(n)}^*;\\
a_{i,n}=J_{\gamma_{i,\pi_i(n)}A_i}\big(
x_{i,\pi_i(n)}-\gamma_{i,\pi_i(n)}
(l_{i,n}^*+C_ix_{i,\pi_i(n)})\big);\\
a_{i,n}^*=\gamma_{i,\pi_i(n)}^{-1}(x_{i,\pi_i(n)}
-a_{i,n})-l_{i,n}^*+Q_ia_{i,n};\\
\xi_{i,n}=\|a_{i,n}-x_{i,\pi_i(n)}\|^2;
\end{array}
\right.\\
\text{for every}\;i\in I\smallsetminus I_n\\
\left\lfloor
\begin{array}{l}
a_{i,n}=a_{i,n-1};\;a_{i,n}^*=a_{i,n-1}^*;\;
\xi_{i,n}=\xi_{i,n-1};\\
\end{array}
\right.\\
\text{for every}\;k\in K_n\\
\left\lfloor
\begin{array}{l}
u_{k,n}^*=v_{k,\omega_k(n)}^*-B_k^{\LL}y_{k,\omega_k(n)};
w_{k,n}^*=v_{k,\omega_k(n)}^*-D_k^{\LL}z_{k,\omega_k(n)};\\
b_{k,n}=J_{\mu_{k,\omega_k(n)}B_k^{\MM}}\big(y_{k,\omega_k(n)}
+\mu_{k,\omega_k(n)}(u_{k,n}^*-B_k^{\CC}y_{k,\omega_k(n)})
\big);\\
d_{k,n}=J_{\rho_{k,\omega_k(n)}D_k^{\MM}}\big(z_{k,\omega_k(n)}
+\rho_{k,\omega_k(n)}(w_{k,n}^*-D_k^{\CC}z_{k,\omega_k(n)})\big);\\
e_{k,n}^*=\sigma_{k,\omega_k(n)}\big(
\sum_{i\in I}L_{ki}x_{i,\omega_k(n)}
-y_{k,\omega_k(n)}-z_{k,\omega_k(n)}\big)\\
\hspace{10mm}\;+\:v_{k,\omega_k(n)}^*;\\
q_{k,n}^*=\mu_{k,\omega_k(n)}^{-1}(y_{k,\omega_k(n)}-b_{k,n})
+u_{k,n}^*+B_k^{\LL}b_{k,n}-e_{k,n}^*;\\
t_{k,n}^*=\rho_{k,\omega_k(n)}^{-1}(z_{k,\omega_k(n)}-d_{k,n})
+w_{k,n}^*+D_k^{\LL}d_{k,n}-e_{k,n}^*;\\
\eta_{k,n}=\|b_{k,n}-y_{k,\omega_k(n)}\|^2
+\|d_{k,n}-z_{k,\omega_k(n)}\|^2;\\
e_{k,n}=b_{k,n}+d_{k,n}-\sum_{i\in I}L_{ki}a_{i,n};
\end{array}
\right.\\
\text{for every}\;k\in K\smallsetminus K_n\\
\left\lfloor
\begin{array}{l}
b_{k,n}=b_{k,n-1};\;
d_{k,n}=d_{k,n-1};\;
e_{k,n}^*=e_{k,n-1}^*;\\
q_{k,n}^*=q_{k,n-1}^*;\;
t_{k,n}^*=t_{k,n-1}^*;\;\eta_{k,n}=\eta_{k,n-1};\\
e_{k,n}=b_{k,n}+d_{k,n}-\sum_{i\in I}L_{ki}a_{i,n};
\end{array}
\right.\\
\text{for every}\;i\in I\\
\left\lfloor
\begin{array}{l}
p_{i,n}^*=a_{i,n}^*
+R_i\boldsymbol{a}_n
+\sum_{k\in K}L_{ki}^*e_{k,n}^*;
\end{array}
\right.\\
\begin{aligned}
\Delta_n&=\textstyle
{-}(4\alpha)^{-1}\big(\sum_{i\in I}\xi_{i,n}
+\sum_{k\in K}\eta_{k,n}\big)
+\sum_{i\in I}\scal{x_{i,n}-a_{i,n}}{p_{i,n}^*}\\
&\textstyle
\quad\;+\sum_{k\in K}\big(\scal{y_{k,n}-b_{k,n}}{q_{k,n}^*}
+\scal{z_{k,n}-d_{k,n}}{t_{k,n}^*}\\
&\hspace{20mm}+\scal{e_{k,n}}{v_{k,n}^*-e_{k,n}^*}\big);
\end{aligned}\\
\text{if}\;\Delta_n>0\\
\left\lfloor
\begin{array}{l}
\theta_n=\lambda_n\Delta_n/
\big(\sum_{i\in I}\|p_{i,n}^*\|^2\!+\!\sum_{k\in K}\big(
\|q_{k,n}^*\|^2\!+\!\|t_{k,n}^*\|^2\!+\!\|e_{k,n}\|^2\big)\big);\\
\text{for every}\;i\in I\\
\left\lfloor
\begin{array}{l}
x_{i,n+1}=x_{i,n}-\theta_np_{i,n}^*;
\end{array}
\right.\\
\text{for every}\;k\in K\\
\left\lfloor
\begin{array}{l}
y_{k,n+1}=y_{k,n}-\theta_nq_{k,n}^*;\;
z_{k,n+1}=z_{k,n}-\theta_nt_{k,n}^*;\\
v_{k,n+1}^*=v_{k,n}^*-\theta_ne_{k,n};
\end{array}
\right.\\[1mm]
\end{array}
\right.\\
\text{else}\\
\left\lfloor
\begin{array}{l}
\text{for every}\;i\in I\\
\left\lfloor
\begin{array}{l}
x_{i,n+1}=x_{i,n};
\end{array}
\right.\\
\text{for every}\;k\in K\\
\left\lfloor
\begin{array}{l}
y_{k,n+1}=y_{k,n};\;z_{k,n+1}=z_{k,n};\;v_{k,n+1}^*=v_{k,n}^*.
\end{array}
\right.\\[1mm]
\end{array}
\right.\\[9mm]
\end{array}
\right.
\end{array}
\end{align}
\end{algorithm}

\subsection{Convergence}
The convergence properties of Algorithm~\ref{algo:20} are 
laid out in the following theorem.

\begin{theorem}[{\protect{\cite[Theorem~1(iv)]{Moor22}}}]
\label{t:20}
Consider the setting of Problem~\ref{prob:20} and
Algorithm~\ref{algo:20}, and suppose that the saddle operator
$\sad$ of \eqref{e:saddle} has zeros. Then, for every $i\in I$,
$(x_{i,n})_{n\in\NN}$ converges weakly to a point $x_i\in\HH_i$
and, for every $k\in K$, $(v^*_{k,n})_{n\in\NN}$ converges weakly
to a point $v^*_k\in\GG_k$. In addition, $(x_i)_{i\in I}$ solves
the primal problem \eqref{e:20p} and $(v^*_k)_{k\in K}$ solves the
dual problem \eqref{e:20d}.
\end{theorem}

\begin{remark}
\label{r:20}
The strongly convergent variant of Theorem~\ref{t:20} based on
Theorem~\ref{t:2c} is proposed in \cite[Theorem~2(iv)]{Moor22}.
\end{remark}

\begin{remark}
\label{r:21}
A fact that has not be appreciated previously is that 
Theorem~\ref{t:20} contains as special cases various weak
convergence results of Sections~\ref{sec:fbf}--\ref{sec:fb}. Thus,
suppose that
\begin{equation}
\label{e:59}
I=K=\{1\},\;R_1=0,\;\text{and}\;\;L_{11}=0.
\end{equation}
Then Problem~\ref{prob:20} reduces to finding a zero of 
$A_1+C_1+Q_1$ (see \eqref{e:n4}), \eqref{e:long1} reduces to the 
forward-backward-half-forward
algorithm \eqref{e:half}, and Theorem~\ref{t:20} reduces to 
Proposition~\ref{p:22}. This covers both the 
forward-backward-forward algorithm \eqref{e:fbf} for $C_1=0$
(Theorem~\ref{t:6}) and the unrelaxed forward-backward algorithm 
\eqref{e:fb} for $Q_1=0$ (Theorem~\ref{t:5}).
In a similar fashion, we can recover the multivariate
forward-backward-forward algorithm of \cite{Siop13} 
by choosing
\begin{equation}
\label{e:459}
(\forall i\in I)(\forall k\in K)\;\;C_i=R_i=0
\;\;\text{and}\;\;
B_k^{\CC}=B_k^{\LL}=D_k^{\CC}=D_k^{\LL}=0.
\end{equation}
Going back to the simple inclusion problem 
\eqref{e:n4}, Theorem~\ref{t:20} offers several other
possibilities, for instance by implementing it with 
\begin{multline}
I=K=\{1\},
A_1=A,\; 
R_1=C_1=Q_1=0,\; 
L_{11}=\Id,\\ 
B_1^{\MM}=0\;,
B_1^{\CC}=C,\;B_1^{\LL}=Q, \;\text{and}\;
D_1^{\MM}=D_1^{\CC}=D_1^{\LL}=\{0\}^{-1}.
\end{multline}
\end{remark}

As mentioned earlier, Problem~\ref{prob:20} encompasses all the
problems discussed earlier. Theorem~\ref{t:20} can therefore be
used to provide alternative algorithms to solve them in an
asynchronous and block-iterative manner, and with
operator-dependent proximal parameters (these features are absent
from the algorithms of Sections~\ref{sec:ppa}--\ref{sec:fb}). Here
is an example.

\begin{example}
\label{ex:x1}
In Problem~\ref{prob:20}, suppose that 
\begin{multline}
\label{e:x1}
I=\{1\},\;K=\{1,\ldots,p\},\;A_1=A, C_1=R_1=0, Q_1=Q,\;\text{and}\;
(\forall k\in K)\\
L_{k1}=L_k, B_k^{\MM}=B_k, B_k^{\CC}=B_k^{\LL}=0,
D_k^{\MM}=D_k, \text{and}\; D_k^{\CC}=D_k^{\LL}=0.
\end{multline}
Then we obtain the primal-dual inclusions
\eqref{e:p18}--\eqref{e:d18} of Proposition~\ref{p:18},
and Theorem~\ref{t:20} furnishes a
flexible alternative to Proposition~\ref{p:18} which, in addition,
places no restriction on the operators $(D_k)_{k\in K}$,
with the algorithm
\begin{align}
\label{e:long2}
\begin{array}{l}
\text{for}\;n=0,1,\ldots\\
\left\lfloor
\begin{array}{l}
l_n^*=Qx_{\pi(n)}
+\sum_{k\in K}L_k^*v_{k,\pi(n)}^*;\\
a_n=J_{\gamma_{\pi(n)}A}\big(
x_{\pi(n)}-\gamma_{\pi(n)}l_n^*\big);\\
a_n^*=\gamma_{\pi(n)}^{-1}(x_{\pi(n)}-a_n)-l_n^*+Qa_n;\\
\text{for every}\;k\in K_n\\
\left\lfloor
\begin{array}{l}
b_{k,n}=J_{\mu_{k,\omega_k(n)}B_k}\big(y_{k,\omega_k(n)}
+\mu_{k,\omega_k(n)}v_{k,\omega_k(n)}^*\big);\\
d_{k,n}=J_{\rho_{k,\omega_k(n)}D_k}\big(z_{k,\omega_k(n)}
+\rho_{k,\omega_k(n)}v_{k,\omega_k(n)}^*\big);\\
e_{k,n}^*=\sigma_{k,\omega_k(n)}\big(L_kx_{\omega_k(n)}
-y_{k,\omega_k(n)}-z_{k,\omega_k(n)}\big)+v_{k,\omega_k(n)}^*
;\\
q_{k,n}^*=\mu_{k,\omega_k(n)}^{-1}(y_{k,\omega_k(n)}-b_{k,n})
+v_{k,\omega_k(n)}^*-e_{k,n}^*;\\
t_{k,n}^*=\rho_{k,\omega_k(n)}^{-1}(z_{k,\omega_k(n)}-d_{k,n})
+v_{k,\omega_k(n)}^*-e_{k,n}^*;\\
\eta_{k,n}=\|b_{k,n}-y_{k,\omega_k(n)}\|^2
+\|d_{k,n}-z_{k,\omega_k(n)}\|^2;\\
e_{k,n}=b_{k,n}+d_{k,n}-L_ka_n;
\end{array}
\right.\\
\text{for every}\;k\in K\smallsetminus K_n\\
\left\lfloor
\begin{array}{l}
b_{k,n}=b_{k,n-1};\;
d_{k,n}=d_{k,n-1};\;
e_{k,n}^*=e_{k,n-1}^*;\\
q_{k,n}^*=q_{k,n-1}^*;\;
t_{k,n}^*=t_{k,n-1}^*;\;\eta_{k,n}=\eta_{k,n-1};\\
e_{k,n}=b_{k,n}+d_{k,n}-L_{k}a_{n};
\end{array}
\right.\\
p_{n}^*=a_{n}^*+\sum_{k\in K}L_{k}^*e_{k,n}^*;\\
\begin{aligned}
\Delta_n&=\textstyle
{-}(4\alpha)^{-1}\big(\|a_n-x_{\pi(n)}\|^2
+\sum_{k\in K}\eta_{k,n}\big)
+\scal{x_n-a_n}{p_n^*}\\
&\textstyle
\quad\;+\sum_{k\in K}\big(\scal{y_{k,n}-b_{k,n}}{q_{k,n}^*}
+\scal{z_{k,n}-d_{k,n}}{t_{k,n}^*}\\
&\hspace{20mm}+\scal{e_{k,n}}{v_{k,n}^*-e_{k,n}^*}\big);
\end{aligned}\\
\text{if}\;\Delta_n>0\\
\left\lfloor
\begin{array}{l}
\theta_n=\lambda_n\Delta_n/
\big(\|p_n^*\|^2+\sum_{k\in K}\big(
\|q_{k,n}^*\|^2+\|t_{k,n}^*\|^2+\|e_{k,n}\|^2\big)\big);\\
x_{n+1}=x_n-\theta_np_{n}^*;\\
\text{for every}\;k\in K\\
\left\lfloor
\begin{array}{l}
y_{k,n+1}=y_{k,n}-\theta_nq_{k,n}^*;\;
z_{k,n+1}=z_{k,n}-\theta_nt_{k,n}^*;\\
v_{k,n+1}^*=v_{k,n}^*-\theta_ne_{k,n};
\end{array}
\right.\\[1mm]
\end{array}
\right.\\
\text{else}\\
\left\lfloor
\begin{array}{l}
x_{n+1}=x_{n};\\
\text{for every}\;k\in K\\
\left\lfloor
\begin{array}{l}
y_{k,n+1}=y_{k,n};\;z_{k,n+1}=z_{k,n};\;v_{k,n+1}^*=v_{k,n}^*.
\end{array}
\right.\\[1mm]
\end{array}
\right.\\[1mm]
\end{array}
\right.
\end{array}
\end{align}
\end{example}

\begin{remark}
In the same vein as Example~\ref{ex:x1}, we can solve the 
primal-dual inclusions \eqref{e:bap}--\eqref{e:bad} of
Proposition~\ref{p:21} via Theorem~\ref{t:20} by making the
modifications $C_1=C$ and $Q_1=0$ in \eqref{e:x1}.
\end{remark}

\section{Extensions and variants}
\label{sec:var}

The following flowchart summarizes the articulation of the main
splitting methods presented in the previous sections (a similar
flowchart can be drawn for the chain of strong convergence results
starting with the Haugazeau principle of Theorem~\ref{t:2}, 
then Theorem~\ref{t:2c}, etc.).
\begin{align}
\label{e:j1}
&\bullet\;
\text{Cutting plane Fej\'er principle (Theorem~\ref{t:1})}
\nonumber\\
&\qquad\Downarrow\nonumber\\
&\qquad\bullet\;\text{Graph-based cuts (Theorem~\ref{t:1c})}
\nonumber\\
&\qquad\qquad\bullet\;
\text{Section~\ref{sec:ps} (Block-iterative Kuhn--Tucker 
projective splitting)}
\nonumber\\
&\qquad\qquad\bullet\;
\text{Section~\ref{sec:sad} (Block-iterative saddle 
projective splitting)}
\nonumber\\
&\qquad\qquad\bullet\;
\text{Warped resolvent splitting (Theorem~\ref{t:8})}
\nonumber\\
&\qquad\qquad\qquad\Downarrow\nonumber\\
&\qquad\qquad\qquad\bullet\;
\text{Section~\ref{sec:ppa} (Proximal point algorithm)}
\nonumber\\
&\qquad\qquad\qquad\bullet\;
\text{Section~\ref{sec:dr} (Douglas--Rachford splitting)}
\nonumber\\
&\qquad\qquad\qquad\bullet\;
\text{Section~\ref{sec:fbf} (Forward-backward-forward splitting)}
\nonumber\\
&\qquad\qquad\qquad\bullet\;
\text{Section~\ref{sec:fb} (Forward-backward splitting).}
\end{align}
This flowchart suggests that any extension or variant of the main
theorems of Section~\ref{sec:4} (Theorems~\ref{t:1}, \ref{t:1c},
and \ref{t:8}) will lead to further splitting methods or, at least,
different implementations of them. We discuss some of the possible
variations on the basic geometric principles we have employed.

The basic operating principle of Theorem~\ref{t:1} is 
Fej\'er-monotonicity, i.e., its property \ref{t:1i}. There are
extensions of this notion which preserve the main weak convergence
conclusions. For instance the notion of quasi-Fej\'er monotonicity,
introduced in \cite{Ermo68} and studied in detail in 
\cite{Else01}, requires that there exist a summable sequence 
$(\varepsilon_n)_{n\in\NN}$ in $\RP$ such that
\begin{equation}
\label{e:qf}
(\forall z\in Z)(\forall n\in\NN)\quad
\|x_{n+1}-z\|^2\leq\|x_n-z\|^2+\varepsilon_n.
\end{equation}
It follows from \cite[Section~3]{Else01} that Theorem~\ref{t:1}
remains valid if, for some sequence $(e_n)_{n\in\NN}$ in $\HH$
such that $\sum_{n\in\NN}\lambda_n\|e_n\|<\pinf$, we use an
approximate projection $p_n=\proj_{H_n}x_n+e_n$ in \eqref{e:5}
(see also \cite{Siop15} for a stochastic version of
this result that allows for random iteration modeling). 
This summable error framework can be propagated in \eqref{e:j1} to
recover approximate implementation results from 
\cite{Botr13,Opti04,Siop13,Svva12,Cond13,Roc76a,Bang13}. 
Variable metric quasi-Fej\'er-monotonicity is an extension of
\eqref{e:qf} described by
\begin{equation}
\label{e:qf2}
(\forall z\in Z)(\forall n\in\NN)\quad
\|x_{n+1}-z\|_{U_{n+1}}^2\leq\|x_n-z\|_{U_{n}}^2+\varepsilon_n,
\end{equation}
where $(U_n)_{n\in\NN}$ is a sequence of strongly monotone
operators in $\BL(\HH)$ satisfying certain properties 
\cite{Nonl13}. It follows from \cite[Theorem~3.3]{Nonl13} that
the conclusions of Theorem~\ref{t:1} remain valid in this setting,
which amounts to changing the metric of $\HH$ at each iteration.
See \cite{Chen97,Opti14} for applications to forward-backward
splitting, \cite{Rock24} for applications to multiplier methods,
and \cite{Ragu15} for considerations on the choice of the variable
metrics. All the results derived from Theorem~\ref{t:1} can be
revisited in this variable-metric context. Another extension of
\eqref{e:qf} of interest is the multi-step
quasi-Fej\'er-monotonicity notion
\begin{equation}
(\forall z\in Z)(\forall n\in\NN)\quad
\|x_{n+1}-x\|^2\leq\Sum_{j=0}^{n}\mu_{n,j}\|x_j-x\|^2
+\varepsilon_n
\end{equation}
of \cite[Lemma~2.2]{Anon21}, where  
$(\mu_{n,j})_{n\in\NN,0\leq j\leq n}$ is an array in
$\RP$ satisfying certain properties. This setting led to
deterministic block-iterative implementations of the
forward-backward algorithm \cite[Proposition~4.9]{Anon21} in the
spirit of methods found in \cite{Mish20,Mokh18} in the minimization
case.

The hybrid proximal-extragradient/projection methods of 
\cite{Solo04,Sol99b,Sol99c,Solo01} revolve around a variant of 
Proposition~\ref{p:1} in which, at iteration $n$, $(m_n,m_n^*)$ is
merely required to be in the graph of a perturbed version of $M$,
which permits us to recover certain iterative methods beyond the
proximal point algorithm. See also \cite{Svai14} for more recent 
work along these lines, where approximate resolvents are used to
recover an instance of the forward-backward algorithm.

As is apparent from \eqref{e:j1}, many convergence results we have
discussed follow from Theorem~\ref{t:8}. We now present a 
perturbed extension of it in which, at iteration $n$, the warped
resolvent is applied at a point $\widetilde{x}_n$ and not
necessarily at the current iterate $x_n$. The special case when
$C=0$, $(q_n)_{n\in\NN}=(w_n)_{n\in\NN}$, and conditions
\ref{t:8iid} and \ref{t:8iia} of Theorem~\ref{t:8} are fulfilled
appears in \cite[Theorem~4.2]{Jmaa20}.

\begin{theorem}
\label{t:9}
Let $\alpha\in\RPP$, 
let $W\colon\HH\to 2^{\HH}$ be maximally monotone, 
let $C\colon\HH\to\HH$ be $\alpha$-cocoercive and
such that $Z=\zer(W+C)\neq\emp$, let $x_0\in\HH$, and let
$(\lambda_n)_{n\in\NN}$ be a sequence in $\left]0,2\right[$.
Further, for every $n\in\NN$, let $\widetilde{x}_n\in\HH$ and 
let $U_n\colon\HH\to\HH$ be an operator such that 
$\ran U_n\subset\ran(U_n+W+C)$ and $U_n+W+C$ is injective. Iterate
\begin{equation}
\label{e:fejer9}
\begin{array}{l}
\text{for}\;n=0,1,\ldots\\
\left\lfloor
\begin{array}{l}
w_n=J_{W+C}^{U_n}\widetilde{x}_n\\
w_n^*=U_n\widetilde{x}_n-U_nw_n-Cw_n\\
q_n\in\HH\\
t_n^*=w_n^*+Cq_n\\
\delta_n=\scal{x_n-w_n}{t_n^*}-\|w_n-q_n\|^2/(4\alpha)\\[1mm]
d_n=
\begin{cases}
\dfrac{\delta_n}
{\|t_n^*\|^2}t_n^*,&\text{if}\:\:\delta_n>0;\\
0,&\text{otherwise}\\
\end{cases}\\
x_{n+1}=x_n-\lambda_nd_n.
\end{array} 
\right. 
\end{array} 
\end{equation}
Suppose that $\widetilde{x}_n-x_n\to 0$. Then the conclusions of
Theorem~\ref{t:8} remain valid if the condition 
$U_nw_n-U_nx_n\to 0$ in \ref{t:8iia} is replaced by
$U_nw_n-U_n\widetilde{x}_n\to 0$.
\end{theorem}
\begin{proof}
Adapt the pattern of the proof of Theorem~\ref{t:8}.
\end{proof}

\begin{remark}
\label{r:200}
The auxiliary sequence $(\widetilde{x}_n)_{n\in\NN}$ in 
Theorem~\ref{t:9} adds considerable breadth to the scope 
of the algorithm, compared to that of Theorem~\ref{t:8}. 
Here are some illustrations of the condition
$\widetilde{x}_n-x_n\to 0$, where we assume that
$\inf_{n\in\NN}\lambda_n>0$ and $\sup_{n\in\NN}\lambda_n<2$.
\begin{enumerate}
\item
At iteration $n$, $\widetilde{x}_n$ can model an additive 
perturbation of $x_n$, say $\widetilde{x}_n=x_n+e_n$. Here, the
error sequence $(e_n)_{n\in\NN}$ need only satisfy $\|e_n\|\to 0$
and not the usual summability condition
$\sum_{n\in\NN}\|e_n\|<\pinf$ required in the quasi-Fej\'erian
splitting methods of
\cite{Botr13,Else01,Opti04,Siop13,Svva12,Bang13}.
\item
\label{r:2ii}
In the spirit of inertial methods 
\cite{Atto20,Bec09a,Cham15,Siop17,Poly64}, 
let $(\alpha_n)_{n\in\NN}$ be a sequence in $\RR$ and set
$(\forall n\in\NN\smallsetminus\{0\})$ 
$\widetilde{x}_n=x_n+\alpha_n(x_n-x_{n-1})$. In these methods,
$\alpha_n(x_n-x_{n-1})\to 0$, which guarantees that
$\|\widetilde{x}_n-x_n\|\to 0$, as required.
\item
More generally, weak convergence results can be derived from
Theorem~\ref{t:9} for iterations with memory, that is,
\begin{multline}
(\forall n\in\NN)\quad \widetilde{x}_n=\sum_{j=0}^n\mu_{n,j}x_j,
\quad\text{where}\\
(\mu_{n,j})_{0\leq j\leq n}\in\RR^{n+1}
\;\;\text{and}\;\;\sum_{j=0}^n\mu_{n,j}=1.
\end{multline}
Here we have $\widetilde{x}_n-x_n\to 0$ if
$(1-\mu_{n,n})x_n-\sum_{j=0}^{n-1}\mu_{n,j}x_j\to 0$.
In the case of standard inertial methods, weak convergence requires
more stringent conditions on the weights 
$(\mu_{n,j})_{n\in\NN,0\leq j\leq n}$ \cite{Siop17}.
\item
As indicated in \eqref{e:j1}, Theorem~\ref{t:7} on the Kuhn--Tucker
projective splitting algorithm was derived from 
Proposition~\ref{p:1}, hence from Theorem~\ref{t:1c}, and it
does not appear possible to derive it from Theorem~\ref{t:8}.
However, as shown in \cite[Corollary~4]{Buin22}, Theorem~\ref{t:7}
follows from Theorem~\ref{t:9} (implemented with $C=0$ and
$q_n=w_n$) through a suitable choice of the auxiliary sequence
$(\widetilde{x}_n)_{n\in\NN}$. This last example provides further
confirmation of the effectiveness of warped resolvents.
\end{enumerate}
\end{remark}

\medskip
{\bfseries Acknowledgment.} The author thanks Minh N. B\`ui for 
his careful proofreading of the paper and his suggestions.

\label{lastpage}
\end{document}